\documentclass{amsart}
\usepackage{graphicx}
\usepackage{multirow}
\usepackage{amsmath,amssymb,amsfonts}
\usepackage{amsthm}
\usepackage{mathrsfs}
\usepackage[title]{appendix}
\usepackage{xcolor}
\usepackage{textcomp}
\usepackage{manyfoot}
\usepackage{booktabs}
\usepackage{algorithm}
\usepackage{algorithmicx}
\usepackage{algpseudocode}
\usepackage{listings}
\usepackage{geometry}
\usepackage{indentfirst}
\usepackage{cases}
\usepackage{tikz}
\usepackage{dsfont} 
\usepackage{float}
\usepackage{lmodern}
\usepackage{geometry}
\usepackage{indentfirst}
\usepackage{cases}
\newtheorem{thm}{Theorem}[section]
\newtheorem{lem}[thm]{Lemma}

\newtheorem{prop}[thm]{Proposition}
\newtheorem{rem}[thm]{Remark}
\numberwithin{equation}{section}
\title[Three-phase Muskat problem: uniform lifespan]{Three-phase Muskat problem: uniform lifespan with respect to the width of the strip between interfaces}
\author{\'Angel Castro and Liangchen Zou}

\address[\'Angel Castro]{ICMAT-CSIC. C\ Nicol\'as Cabrera 13-15, 28049, Madrid}
\email{angel\_castro@icmat.es}

\address[Liangchen Zou]{School of Mathematical Sciences, University of Science and Technology of China, Hefei, Anhui, 230026, PR China}
\email{zlc0601@mail.ustc.edu.cn }

\date{}

\begin{document}

\begin{abstract}
We consider the three-phase Muskat problem with different densities and the same viscosities.  The lifespan of the solutions with respect to the width of the strip between interfaces is studied. Indeed, the  interfaces are parameterized by the graph of two functions $f(x,t)$ and $g(x,t)$ and we impose that $||f(\cdot,0)-g(\cdot,0)||_{L^\infty}\leq C\sigma$ and $\inf_x |f(\cdot,0)-g(\cdot,0)|\geq c\sigma.$ It is shown, under stronger assumption on $f(x,0)$ and $g(x,0)$, local existence independent of the parameter $\sigma$ (with $\sigma$ small enough). In order to prove such a result, we need to work in analytic spaces.
\end{abstract}
\maketitle

\tableofcontents

\section{Introduction}

We consider the Incompressible Porous Media (IPM) system in two-dimensional space, which describes the dynamics of incompressible fluids in a porous media. For simplicity, we take the gravity $g=1$, the permeability of the media $\kappa=1$ and the viscosity $\nu=1$ so that the fluids satisfy the following equations on $(x,y)\in\mathbb{R}^2$ and $t\geq0$:
\begin{equation}
\left\{\begin{aligned} &\partial_t\rho+\nabla\cdot(\rho u)=0,\quad\\
&u+\nabla p+\left(0,\rho\right)=0 \\
&\nabla\cdot u=0
\end{aligned}\right.,\label{1.1}\end{equation}
which take into account mass conservation, Darcy's law \cite{Darcy1856} and incompressibility. Here $\rho=\rho(x,y,t)$ is the density, $p=p(x,y,t)$ is the pressure, and $u=(u_1(x,y,t),u_2(x,y,t))$ is the velocity field. This system is mathematically equivalent to the motion of an incompressible fluid in a Hele-Shaw cell \cite{SaffmanTaylor1958}.

In particular, we shall study the three-phase Muskat problem \cite{Muskat34} which considers that the density is a piecewise constant function given by
\begin{align*}
    \rho(x,y,t)=\left\{\begin{array}{cc}\rho_0 & y>\sigma+f(x,t)\\ \rho_1 & -\sigma+g(x,t)<y<\sigma+f(x,t) \\
    \rho_2 & y<-\sigma+g(x,t)\end{array}\right.
\end{align*}
where $\sigma>0$, $\rho_0<\rho_1<\rho_2$ and $f$ and $g$ are smooth functions that decay at infinity and satisfy $2\sigma+f(x,t)-g(x,t)>0$.  

From \eqref{1.1} a system of coupled evolution equations can be obtained for $f$ and $g$. Indeed, 
\begin{align}
&\partial_t f = \frac{\rho_1-\rho_0}{2\pi}\int_{ \mathbb{R}}\frac{y(\partial_x f(x)-\partial_x f(x-y))}{y^2+(f(x,t)-f(x-y,t))^2}dy
+ \frac{\rho_2-\rho_1}{2\pi}\int_{ \mathbb{R}}\frac{y(\partial_x g(x)-\partial_x g(x-y))}{y^2+(2\sigma+f(x)-g(x-y))^2}dy,\label{f}\\
&f(x,0)=f_0(x)\nonumber,\\
&\partial_t g = \frac{\rho_2-\rho_1}{2\pi}\int_{ \mathbb{R}}\frac{y(\partial_x g(x)-\partial_x g(x-y))}{y^2+(g(x,t)-g(x-y,t))^2}dy
+ \frac{\rho_1-\rho_0}{2\pi}\int_{ \mathbb{R}}\frac{y(\partial_x f(x)-\partial_x f(x-y))}{y^2+(-2\sigma+g(x)-f(x-y))^2}dy\label{g},\\
&g(x,0)=g_0(x)\nonumber.
\end{align}
System \eqref{f}-\eqref{g} was obtained in \cite{CordobaGancedo2010}, where C\'ordoba and Gancedo also showed 
local existence of solutions in Sobolev spaces $H^k(\mathbb{\mathbb{R}})$, $k\geq 4$,  and a criterion for the absence of squirt singularities. Another interesting result related to our paper is the criterion for the lack of splash singularities in  \cite{GancedoStrain2014}. In that paper, Gancedo and Strain prove that there are no splash singularities for the system \eqref{f}-\eqref{g} as far as  $
\|f\|_{C^2}$ and $\|g\|_{C^2}$ remain bounded. Roughly speaking,  a splash singularity in this setting is a collision of the interfaces in a point and a squirt (or splat) singularity in a segment.

In this paper we would like to study how system \eqref{f}-\eqref{g} behaves with respect to width between the interfaces $\sigma+f(x,t)$ and $-\sigma+g(x,t)$. More specifically, let us assume that $X$ is some Banach space to be defined later, that  $f_0$, $g_0\in X$, $f_0-g_0=\theta_0$, $\|\theta_0\|_{X}<C\sigma$, for a universal small constant $C$ independent of $\sigma$. Also $\inf_{x\in\mathbb{R}}(\theta_0(x))/\sigma>-1$. In this situation, is the system \eqref{f}-\eqref{g} locally well-posed in $X$ with time of existence uniform in $\sigma$?  Notice that, if the embedding $\|\cdot\|_{L^\infty}\lesssim\|\cdot\|_{X}$ holds, then $$||2\sigma +f_0-g_0||_{L^\infty}\leq C\sigma$$ and the distance $d_0$ between the initial interfaces
\begin{align*}
d_0\equiv 2\sigma+\inf_{x\in \mathbb{R}}(f_0(x)-g_0(x))
\end{align*}
satisfies
\begin{align*}
     d_0\geq  \sigma.
\end{align*}

\begin{rem}
    Notice that here we do not ask the interfaces to make flat when $\sigma$ goes to zero. Furthermore, although our main result, Theorem \ref{main}, needs some smallness condition on $f_0$ and $g_0$, these conditions do not depend on the parameter $\sigma$.  
\end{rem}

Here it is  important to remark that, at least formally, if one takes $\sigma=0$, thus $f=g$, \eqref{f}-\eqref{g} becomes
\begin{align}
    \partial_t f = \frac{\rho_2-\rho_0}{2\pi}\int_{ \mathbb{R}}\frac{y(\partial_x f(x)-\partial_x f(x-y))}{y^2+(f(x,t)-f(x-y,t))^2}dy,\label{1pmuskat}
\end{align}
which is the well-known Muskat equation in the two-phase case.

The main result of this paper is Theorem \ref{main}. Next, we present a softer and friendly version:
\begin{thm}Let $f_0$ and $g_0$ be analytic functions in $H^k_{\gamma_0}$ (see Section \ref{preliminaries}), $\rho_0<\rho_1<\rho_2$, $0<\sigma<1$, and $\mu_1$, $\mu_2$ be given by (\ref{mu}). Then there exist small constants $\epsilon_0$, $\epsilon_1>0$, independent of $\sigma$, such that if 
\begin{align*}
    \mu_2||f_0||^2_{H^k_\gamma}+\mu_1||g_0||^2_{H^k_\gamma}\leq \epsilon_0^2,\quad \text{and}\quad
    ||f_0-g_0||_{H^{k-3}_{\gamma}}\leq \epsilon_1\sigma,
\end{align*}
there exists a unique analytic (in space) solution $(f(x,t), g(x,t))$ to \eqref{f}-\eqref{g} in an interval $t\in [0, T]$, with $T>0$  independent of $\sigma$.
\end{thm}

\subsection{A brief summary of previous results on two-phase stable case.}

The local existence theory is well-established for equation \eqref{1pmuskat} with $\rho_2>\rho_0$ (stable case). Here we just review some results in this direction concerning the two-phase stable case, with gravity, the same viscosities and constant permeability and without either surface tension or boundaries. For most of them, one could expect that  an analogous version holds for \eqref{f}-\eqref{g} as far as the initial distance between the interfaces is larger than zero. 

C\'ordoba and Gancedo \cite{CordobaGancedo2007} proved local existence  in Sobolev spaces $H^s(\mathbb R), s\geq 3$. See also \cite{Ambrose2004}. In \cite{CGSV2017}, Constantin, Gancedo, Shvydkoy and Vicol proved  local  existence for  initial data in  $W^{2,p}(\mathbb R)$ for $p\in(1,\infty]$.  In \cite{Matioc2019}, Matioc proved  local existence for initial data $ H^s(\mathbb R)$ with $s\in (3/2,2).$ In \cite{AbelMatioc},  Abels and Matioc established local existence for initial data in $W^{s,p}(\mathbb R)$ with $p\in (1,\infty)$ and  $s\in (1+1/p,2).$ Alazard and Lazar established in \cite{AlazardLazard2019} local  existence for  initial data in $\dot H^1(\mathbb R)\cap \dot H^s(\mathbb R)$ with $s>3/2$. Same result can be found in \cite{NguyenPausader2019} thanks to Nguyen and Pausader. In  \cite{AlazardNguyen02-2021} Alazard and  Nguyen  proved  local existence for an initial data in the critical space $\dot W^{1,\infty}(\mathbb R)\cap H^{3/2}(\mathbb R)$.  In  \cite{AlazardNguyen04-2021} the same authors showed local  existence   in $H^{3/2+\log}$ and, in  \cite{AlazardNguyen2023},  in $H^{3/2}(\mathbb R)$. 	In \cite{G-JG-SNP}, Garc\'ia-Ju\'arez,  G\'omez-Serrano, Nguyen and Pausader proved the existence of  self-similar solutions. In  \cite{G-JG-SHP2024}, Garc\'ia-Ju\'arez,  G\'omez-Serrano,  Haziot and  Pausader proved  local existence when the initial interface  has multiple corners and linear growth at infinity. Recently, S\'anchez, in \cite{Omar} proves the existence of solutions with quadratic growth at the infinity.

The 3D case (two-dimensional interface) has been treated, for example, in \cite{CordobaGancedo2007}, \cite{GancedoLazar2020}, \cite{AlazarNguyen3d} and \cite{NguyenPausader2019}.

The first author together with C\'ordoba, Fefferman, Gancedo  and  L\'opez-Fern\'andez proved the existence of turning singularities in \cite{CCFGL-F2012} and 	together with  C\'ordoba, Fefferman and  Gancedo, breakdown of smoothness in \cite{CCFG2013}. C\'ordoba, 	G\'omez-Serrano and	 Zlato\v{s} proved  in \cite{CG-SZ2015} the existence of solutions  that start in the unstable regime, then become stable and finally return to the unstable regime. The same authors in  \cite{CG-SZ2017} established the existence of solutions  that start in the stable regime, then become unstable and finally return to the stable regime.

There are some results proving global existence for small initial data. It seems doable to extend some of them to the multi-phase case. There are also  medium-sized initial data. These ones look, in general, more difficult to adapt to more than one interface. The interested reader could consult, for example, \cite{CG-BS2016,CCGS2013,CCGR-PS2016,cameron2019,Cameron2020,CordobaLazar2018,Lin2017} .

\subsection{Almost-sharp fronts for SQG}

A related problem with a similar flavor is the lifespan of almost-sharp fronts for the Surface Quasi-Geostrophic equation (SQG). This equation is also  a transport equation for the the temperature $\theta(x,t)$ with incompressible velocity $u=R^\perp\theta$, where $R_i$, $i=1,2$, are the Riesz transform in direction $x_i$. If one assumes that the temperature takes only two constant values separated by a curve, an integro-differential equation for the evolution of such curve can be found. In the case, for example, in which,
\begin{align*}
\theta(x,y,t)=\left\{ \begin{array}{cc}\theta_1 & (x,y)\in \Omega(t),\\ \theta_2 & (x,y)\in \Omega^c(t),\end{array}\right.
\end{align*}
where $\Omega(t)$ is a bounded and simply connected domain in $\mathbb{R}^2$, with boundary $\partial \Omega(t)=\{z(\alpha,t)\in \mathbb{R}^2,\,\, \alpha \in \mathbb{T}\}$, it can be computed that
\begin{align}\label{patchsqg}
    \partial_t z(\alpha,t)= \frac{\theta_1-\theta_2}{2\pi}\int_{\mathbb{T}}\frac{\partial_\alpha z(\alpha,t)-\partial_\beta z(\beta,t)}{|z(\alpha,t)-z(\beta,t)|} d\beta.
\end{align}
Expression \eqref{patchsqg} is known as the sharp front equation. Local existence has been proven for \eqref{patchsqg}, in $C^\infty$ \cite{Rodrigo2005} by Rodrigo, in Sobolev spaces \cite{Gancedo2008} by Gancedo, and in analytic spaces \cite{almost2} by Fefferman and Rodrigo.

 In this context, an almost-sharp front is a solution of SQG with initial data
\begin{align*}
    \theta(x,y,0)=\left\{\begin{array}{cc} \theta_1 & (x,y)\in \Omega^1,\\ \theta^m_{0} & (x,y)\in \Omega^m,\\ \theta_2 & (x,y)\in \Omega^2,\end{array}\right.
\end{align*}
where $\theta_i$, $i=1,2$, are constants, and $\theta^m_0$ is smooth in $\Omega_m$, monotone and such that $\theta(x,y,0)$ is smooth in $\mathbb{R}^2$. The domains $\Omega^1$, $\Omega^m$ and $\Omega^2$ are open, simply connected, disjoint, and $\Omega^1\cup \overline{\Omega}^m\cup \Omega^2=\mathbb{R}^2$.  In addition, we consider that the domain in the middle, $\Omega^m$, has a width of order $\sigma$. 

Thus, the question here is whether the almost-sharp fronts are well-posed  with existence time independent of $\sigma$. This problem has been treated in the series of papers \cite{almost1,almost3,almost4} where the authors studied the asymptotic properties of the equation for the almost-sharp front when $\sigma\to 0$. Finally, in \cite{almost5}, Fefferman and Rodrigo were able to prove the existence of families of real-analytic sharp fronts whose lifespan is uniform in $\sigma.$ Actually, in these families, the scalar $\theta^m(\cdot,\cdot,t)$ is who needs to be real-analytic (go to \cite{almost5} for a rigorous version of this statement).

Khor and Rodrigo have also considered this problem for a more singular version of SQG ($\alpha$-SQG, $1<\alpha<2$) in \cite{singularalmost1} and \cite{singularalmost2}.

In our paper, we deal with IPM instead of SQG and in a different scenario since our $\rho_0^m=\rho_0$ is constant and $\rho$ has two jumps. Furthermore, we try a different strategy that takes advantage of the parabolicity in the Muskat problem, which is not present in \eqref{patchsqg}. However, in order to make the analysis uniform in $\sigma$, we lose a derivative that forces us to use analytic spaces, as it  also happens for the almost-sharp fronts in SQG.  
 
\subsection{Organization of the paper.} The paper is organized as follows. In Section \ref{setting} we present the functional setting in which we will solve the equations and Theorem \ref{main} which is the main result of this paper. The proof of Theorem \ref{main} runs along Sections \ref{apriori}, \ref{distance} and \ref{proof}. 

Section \ref{apriori} contains the a priori energy estimates. We have decided to present all the details in the computations and pay the price of a long extension. Let us sketch it to emphasize what the important points are:

The most important sections are \ref{top} and \ref{Asymmetry}. In \ref{top} we deal with the higher order terms in derivatives in the equation, that is, the most singular terms. In order to estimate them, we just use parabolicity, and then we could produce similar estimates in classical Sobolev spaces $H^k(\mathbb{R})$. In Section \ref{Asymmetry}, we deal with the terms involving $(P_{12}-P_{21})$ in \eqref{2.9}-\eqref{2.12}, or equivalently, the left-hand side of \eqref{2.53}-\eqref{2.54}. In these terms, we lose a derivative that cannot be canceled out by the parabolicity uniformly in $\sigma$. This is the only term we would not be able to handle in the case we work in classical Sobolev spaces. Thus, we are forced to introduce analytic spaces $H^k_\gamma$ just because of the terms involving $(P_{12}-P_{21})$ in \eqref{2.9}-\eqref{2.12}. 

In Sections \ref{lower}, \ref{transport}, \ref{commutators1}, \ref{commutators2}, \ref{energy1} and \ref{energy2} we deal with lower-order terms. These sections could be avoided in a first reading since they are rather technical. 

In Section \ref{distance} we control the distance between the interface by using Section \ref{apriori}. We need this control to close the proof of Theorem \ref{main}.

Finally, in Section \ref{proof} we show how estimates in Section \ref{apriori} and \ref{distance} yield the proof of Theorem \ref{main}. 

\section{Functional setting and main Theorem}\label{setting}
For the sake of completeness and to make clear our notation, next we  give a derivation of the equations \eqref{f}-\eqref{g} in the three-phase case. As we said before, a similar exposition (slightly different) can be found in \cite{CordobaGancedo2010}. After that we will present the functional setting and the main result of this paper. 

\subsection{Equations for multi-phase Muskat problem.}
 We suppose that $\rho$ is given by a profile function $\overline{\rho}(y)$ and the level set function $\eta=\eta(x,y,t)$ in the sense that
\begin{equation}
\rho(x,y+\eta(x,y,t),t)=\overline{\rho}(y)
\label{1.2}\end{equation}
for $(x,y)\in\mathbb{R}^2$ and $t\geq0$.
Taking derivatives in (\ref{1.2}) yields
\begin{equation}\begin{aligned}
\partial_t\rho+\partial_y\rho\partial_t\eta&=0,\\
\partial_x\rho+\partial_y\rho\partial_x\eta&=0,\\
\partial_y\rho\left(1+\partial_y\eta\right)&=\partial_y\overline{\rho}.
\label{1.3}\end{aligned}\end{equation}
Then we derive the equation satisfied by $\eta$ from (\ref{1.3}) and (\ref{1.1}):
\begin{equation}
\partial_t\eta+u_1\partial_x\eta-u_2=0.
\label{1.4}\end{equation}
Meanwhile, assuming the hydrostatic condition for the pressure, $u$ can be expressed in terms of $\rho$ through the Riesz transform:
$$u_1(x,y,t)=-\left(-\Delta\right)^{-1}\partial_y\partial_x\rho=\frac{1}{2\pi}P.V.\int_{\mathbb{R}^2}\frac{\alpha\partial_y\rho\left(x-\alpha,\overline{y},t\right)}{\alpha^2+\left(y-\overline{y}\right)^2}d\alpha d\overline{y},$$
$$u_2(x,y,t)=\left(-\Delta\right)^{-1}\partial_x^2\rho=-\frac{1}{2\pi}P.V.\int_{\mathbb{R}^2}\frac{\alpha\partial_x\rho\left(x-\alpha,\overline{y},t\right)}{\alpha^2+\left(y-\overline{y}\right)^2}d\alpha d\overline{y}.$$
 Replacing $y$ by $y+\eta(x,y,t)$ and applying the change of variable $\left(\alpha, \overline{y}\right):=\left(\alpha, y_1+\eta(x-\alpha,y_1,t)\right)$ in the above two equations yields
\begin{equation}
u_1(x,y+\eta,t)=\frac{1}{2\pi}P.V.\int_{\mathbb{R}^2}\frac{\alpha\partial_y\overline{\rho}(y_1)}{\alpha^2+\left(y-y_1+\eta(x,y,t)-\eta(x-\alpha,y_1,t)\right)^2}d\alpha dy_1,
\label{1.5}\end{equation}
\begin{equation}
u_2(x,y+\eta,t)=\frac{1}{2\pi}P.V.\int_{\mathbb{R}^2}\frac{\alpha\partial_x\eta(x-\alpha,y_1,t)\partial_y\overline{\rho}(y_1)}{\alpha^2+\left(y-y_1+\eta(x,y,t)-\eta(x-\alpha,y_1,t)\right)^2}d\alpha dy_1.
\label{1.6}\end{equation}
Then expressing $u_1$ and $u_2$ by (\ref{1.5}-\ref{1.6}) in (\ref{1.4}) gives
\begin{equation}\begin{aligned}
\partial_t\eta(x,y,t)=&-\frac{1}{2\pi}P.V.\int_{\mathbb{R}^2}\frac{\alpha\left(\partial_x\eta(x,y,t)-\partial_x\eta(x-\alpha,y_1,t)\right)\partial_y\overline{\rho}(y_1)}{\alpha^2+\left(y-y_1+\eta(x,y,t)-\eta(x-\alpha,y_1,t)\right)^2}d\alpha dy_1\\
=&-u_1(x,y+\eta,t)\partial_x\eta(x,y,t)\\&+\frac{1}{2\pi}P.V.\int_{\mathbb{R}^2}\frac{\alpha\cdot\partial_x\eta(x-\alpha,y_1,t)\partial_y\overline{\rho}(y_1)}{\alpha^2+\left(y-y_1+\eta(x,y,t)-\eta(x-\alpha,y_1,t)\right)^2}d\alpha dy_1.
\end{aligned}\label{1.7}\end{equation}
We consider the case in which the profile $\overline{\rho}$ is a step function:
\begin{equation}
\overline{\rho}(y)=\left\{\begin{aligned}&\rho_0,\quad y>\sigma,\\ &\rho_1,\quad\sigma>y>-\sigma,\\ &\rho_2,\quad-\sigma>y,
\end{aligned}\right.\label{1.8}\end{equation}
where  $\sigma>0$ is a given constant and $\rho_0<\rho_1<\rho_2$ are real constants. In this case, the derivative $\partial_y\overline{\rho}$ is the sum of two dirac functions:
$$\partial_y\overline{\rho}=-2\pi\Delta\rho\left(\mu_2\delta(y-\sigma)+\mu_1\delta(y+\sigma)\right),$$
where we denote 
\begin{equation}
\Delta\rho:=\frac{\rho_2-\rho_0}{2\pi},\;\mu_1:=(\rho_2-\rho_0)^{-1}\left(\rho_2-\rho_1\right),\;\mu_2:=(\rho_2-\rho_0)^{-1}\left(\rho_1-\rho_0\right).
\label{mu}\end{equation}
Note that $\mu_1+\mu_2=1$. Now let 
$$f(x,t):=\eta(x,\sigma,t),\quad g(x,t):=\eta(x,-\sigma,t),$$
$$u_+(x,t):=u_1(x,\sigma+\eta(x,\sigma,t),t),\quad u_-(x,t):=u_1(x,-\sigma+\eta(x,-\sigma,t),t).$$
From (\ref{1.5}-\ref{1.7}), it follows
\begin{equation}\begin{aligned}
u_+(x,t)=&-\mu_2\Delta\rho \,P.V.\int_{\mathbb{R}}\frac{\alpha}{\alpha^2+\left(\Delta f(x,x-\alpha,t)\right)^2}d\alpha\\
&-\mu_1\Delta\rho \,P.V.\int_{\mathbb{R}}\frac{\alpha}{\alpha^2+\left(2\sigma+f(x,t)-g(x-\alpha,t)\right)^2}d\alpha,
\label{1.9}\end{aligned}\end{equation}
\begin{equation}\begin{aligned}
u_-(x,t)=&-\mu_2\Delta\rho \,P.V.\int_{\mathbb{R}}\frac{\alpha}{\alpha^2+\left(-2\sigma+g(x,t)-f(x-\alpha,t)\right)^2}d\alpha\\
&-\mu_1\Delta\rho \,P.V.\int_{\mathbb{R}}\frac{\alpha}{\alpha^2+\left(\Delta g(x,x-\alpha,t)\right)^2}d\alpha,
\label{1.10}\end{aligned}\end{equation} 
\begin{equation}\begin{aligned}
\partial_t f+u_+\partial_xf=&-\mu_2\Delta\rho \,P.V.\int_{\mathbb{R}}\frac{\alpha\cdot\partial_x f(x-\alpha,t)}{\alpha^2+\left(\Delta f(x,x-\alpha,t)\right)^2}d\alpha\\
&-\mu_1\Delta\rho \,P.V.\int_{\mathbb{R}}\frac{\alpha\cdot\partial_x g(x-\alpha,t)}{\alpha^2+\left(2\sigma+f(x,t)-g(x-\alpha,t)\right)^2}d\alpha,
\label{1.11}\end{aligned}\end{equation}
\begin{equation}\begin{aligned}
\partial_t g+u_-\partial_x g=&-\mu_2\Delta\rho \,P.V.\int_{\mathbb{R}}\frac{\alpha\cdot\partial_x f(x-\alpha,t)}{\alpha^2+\left(-2\sigma+g(x,t)-f(x-\alpha,t)\right)^2}d\alpha\\
&-\mu_1\Delta\rho \,P.V.\int_{\mathbb{R}}\frac{\alpha\cdot\partial_x g(x-\alpha,t)}{\alpha^2+\left(\Delta g(x,x-\alpha,t)\right)^2}d\alpha.
\label{1.12}\end{aligned}\end{equation}

Here we have introduced the notation
\begin{equation}
\Delta w(x,x_1,t):=w(x,t)-w(x_1,t)
\label{deltaw}\end{equation}
for arbitrary function $w$. In particular, $\Delta x=x-x_1$. Along the paper we will also use
\begin{align*}
\tilde{\Delta}w(x,\alpha,t):=\Delta w(x,x-\alpha,t)=w(x,t)-w(x-\alpha,t).
\end{align*}
In most of the paper we will not make explicit the dependence on time. When no confusion is possible, we will also omit the dependence on the spatial variables and just write either $\Delta w$ or $\tilde{\Delta}w$.

For simplicity, we write
\begin{equation}\begin{aligned}
&P_{11}(x,x-\alpha):=\frac{\alpha}{\alpha^2+(\Delta f(x,x-\alpha))^2},\quad P_{12}(x,x-\alpha):=\frac{\alpha}{\alpha^2+(2\sigma+f(x,t)-g(x-\alpha))^2},\\
&P_{21}(x,x-\alpha):=\frac{\alpha}{\alpha^2+(2\sigma+f(x-\alpha)-g(x))^2},\quad P_{22}(x,x-\alpha):=\frac{\alpha}{\alpha^2+(\Delta g(x,x-\alpha)})^2,\end{aligned}\label{Pij}\end{equation}

and subsequently
$$u_+(x,t)=-\mu_2\Delta\rho\,P.V.\int_{\mathbb{R}}P_{11}(x,x-\alpha)d\alpha-\mu_1\Delta\rho\,P.V.\int_{\mathbb{R}}P_{12}(x,x-\alpha)d\alpha,$$
$$u_-(x,t)=-\mu_2\Delta\rho\,P.V.\int_{\mathbb{R}}P_{21}(x,x-\alpha)d\alpha-\mu_1\Delta\rho\,P.V.\int_{\mathbb{R}}P_{22}(x,x-\alpha)d\alpha.$$

To apply a further transformation, we set 
$$h:=\mu_2f+\mu_1g,\quad \theta:=f-g.$$
Then we deduce the equations satisfied by $h$ and $\theta$ from (\ref{1.11}-\ref{1.12}):
\begin{equation}\begin{aligned}
&\partial_t h+\left(\mu_2u_++\mu_1u_-\right)\partial_xh+\mu_1\mu_2\left(u_+-u_-\right)\partial_x\theta\\
=&-\Delta\rho \,P.V.\int_{\mathbb{R}}\left(\mu_2^2P_{11}+\mu_1\mu_2 P_{12}+\mu_1\mu_2 P_{21}+\mu_1^2 P_{22}\right)(x,x-\alpha)\partial_x h(x-\alpha)d\alpha\\
&-\mu_1\mu_2\Delta\rho\,P.V.\int_{\mathbb{R}}\left(\mu_2(P_{11}-P_{12})-\mu_1(P_{22}-P_{21})\right)(x,x-\alpha)\partial_x\theta (x-\alpha)d\alpha,
\end{aligned}\label{1.13}\end{equation}
\begin{equation}\begin{aligned}
&\partial_t\theta+\left(u_+-u_-\right)\partial_xh+\left(\mu_1u_++\mu_2u_-\right)\partial_x\theta\\
=&-\Delta\rho\,P.V.\int_{\mathbb{R}}\left(\mu_2(P_{11}-P_{21})-\mu_1(P_{22}-P_{12})\right)(x,x-\alpha)\partial_xh(x-\alpha)d\alpha\\
&-\mu_1\mu_2\Delta\rho\,P.V.\int_{\mathbb{R}}\left(P_{11}+P_{22}-P_{12}-P_{21}\right)(x,x-\alpha)\partial_x\theta(x-\alpha)d\alpha.
\end{aligned}\label{1.14}\end{equation}

\subsection{Preliminaries}\label{preliminaries}
We will consider the functional space that consists of functions which can be extended analytically to a strip of width $\gamma$ on the complex plane $\Omega_\gamma:=\left\{x\in\mathbb{C}\mid |\Im x|<\gamma\right\},$ and whose restriction to the boundary of $\Omega_\gamma$ belongs to $L^p$. We equipped this space with the norm $\|\cdot\|_{L^p_\gamma}$ defined by
$$\|f\|^p_{L^p_\gamma}:=\|f(\cdot+i\gamma)\|^p_{L^p(\mathbb{R})}+\|f(\cdot-i\gamma)\|^p_{L^p(\mathbb{R})}.$$
Specifically, we denote this functional space by
$$L^p_{\gamma}:=\left\{f\in L^p(\mathbb{R})\mid f \text{ can be extended analytically to } \Omega_\gamma,\; \|f\|_{L^p_{\gamma}}<+\infty\right\}.$$
We also denote the corresponding Sobolev space by
$$H^k_{\gamma}:=\left\{f\in L^2_{\gamma}\mid \partial_x^j f\in L^2_{\gamma} \text{ for each } j\in[0,k]\cap\mathbb{Z}\right\}$$
with the norm defined by
$$\|f\|_{H^k_{\gamma}}^2:=\|f(\cdot+i\gamma)\|_{H^k(\mathbb{R})}^2+\|f(\cdot-i\gamma)\|_{H^k(\mathbb{R})}^2.$$
For later use, we define the operator $\Lambda$ as $\hat{\Lambda f}(\xi)=|\xi|\hat{f}(\xi)$, where $\hat{f}$ denotes the Fourier transform of $f$. An important inequality in our analysis will be the following.
\begin{lem}
Let $f(\cdot,t)\in C^1\left([0,T];H^k_{\gamma(t)}\right)$ with $k\geq1$ and $\gamma(t)$ be a positive time-dependent decreasing $C^1$ function. Then
\begin{equation}\begin{aligned}
&\partial_t\|f\|_{L^2_{\gamma(t)}}^2\\
\leq     &2\Re\int_{\mathbb{R}}\partial_tf(x+i\gamma(t),t)\overline{f(x+i\gamma(t),t)}dx+2\Re\int_{\mathbb{R}}\partial_tf(x-i\gamma(t),t)\overline{f(x-i\gamma(t),t)}dx\\
&+2\gamma^\prime(t)\tanh{(2\gamma(t))}\|\Lambda^\frac{1}{2}f\|_{L^2_{\gamma(t)}}^2-2\gamma^\prime(t)\tanh{(2\gamma(t))}\|f\|_{L^2_{\gamma(t)}}^2.
\end{aligned}\label{1.15}\end{equation}
\label{lem 1.1}\end{lem}
\begin{proof}
For $f(\cdot,t)\in C^1\left([0,T];H^k_{\gamma(t)}\right)$ and a time-dependent function $\gamma(t)$, there is
$$\begin{aligned}
\partial_t\|f\|_{L^2_{\gamma(t)}}^2=&\partial_t\left( \|f(\cdot+i\gamma(t),t)\|^2_{L^2(\mathbb{R})}+\|f(\cdot-i\gamma(t),t)\|^2_{L^2(\mathbb{R})}\right)\\
=&\frac{1}{2\pi}\partial_t\int_{\mathbb{R}}|\hat{f}(\xi,t)|^2(e^{2\gamma(t)\xi}+e^{-2\gamma(t)\xi})d\xi\\
=&\frac{1}{\pi}\Re\int_{\mathbb{R}}\partial_t\hat{f}
(\xi,t)\overline{\hat{f}(\xi,t)}(e^{2\gamma(t)\xi}+e^{-2\gamma(t)\xi})d\xi\\
&+\frac{1}{\pi}\gamma^\prime(t)\int_{\mathbb{R}}|\hat{f}(\xi,t)|^2|\xi|(e^{2\gamma(t)|\xi|}-e^{-2\gamma(t)|\xi|})d\xi\\
=&2\Re\int_{\mathbb{R}}\partial_tf(x+i\gamma(t),t)\overline{f(x+i\gamma(t),t)}dx+2\Re\int_{\mathbb{R}}\partial_tf(x-i\gamma(t),t)\overline{f(x-i\gamma(t),t)}dx\\
&+\frac{2}{\pi}\gamma^\prime(t)\int_{\mathbb{R}}|\hat{f}(\xi,t)|^2|\xi|\sinh{(2\gamma(t)|\xi|)}d\xi.
\end{aligned}$$
Note that 
$$\begin{aligned}
|\xi|\sinh{(2\gamma(t)|\xi|)}&=|\xi|\cosh{(2\gamma(t)|\xi|)\tanh{(2\gamma(t)|\xi|)}}\\
&\geq|\xi|\cosh{(2\gamma(t)|\xi|)}\tanh{(2\gamma(t))}-\cosh{(2\gamma(t)|\xi|)}\tanh{(2\gamma(t))}.
\end{aligned}$$
Hence for decreasing $\gamma(t)$ 
$$\begin{aligned}
&\partial_t\|f\|_{L^2_{\gamma(t)}}^2\\
\leq& 2\Re\int_{\mathbb{R}}\partial_tf(x+i\gamma(t),t)\overline{f(x+i\gamma(t),t)}dx+2\Re\int_{\mathbb{R}}\partial_tf(x-i\gamma(t),t)\overline{f(x-i\gamma(t),t)}dx\\
&+\frac{2}{\pi}\gamma^\prime(t)\tanh{(2\gamma(t))}\int_{\mathbb{R}}|\hat{f}(\xi,t)|^2|\xi|\cosh{(2\gamma(t)\xi)}d\xi\\
&-\frac{2}{\pi}\gamma^\prime(t)\tanh{(2\gamma(t))}\int_{\mathbb{R}}|\hat{f}(\xi,t)|^2\cosh{(2\gamma(t)\xi)}d\xi\\
=&2\Re\int_{\mathbb{R}}\partial_tf(x+i\gamma(t),t)\overline{f(x+i\gamma(t),t)}dx+2\Re\int_{\mathbb{R}}\partial_tf(x-i\gamma(t),t)\overline{f(x-i\gamma(t),t)}dx\\
&+2\gamma^\prime(t)\tanh{(2\gamma(t))}\|\Lambda^\frac{1}{2}f\|_{L^2_{\gamma(t)}}^2-2\gamma^\prime(t)\tanh{(2\gamma(t))}\|f\|_{L^2_{\gamma(t)}}^2.
\end{aligned}$$
\end{proof}
\subsection{Statement of the main result}
In this paper, we consider the initial
data $(h_0,\theta_0)\in H^{k}_{\gamma_0}$ for $k\geq10$ and $\gamma_0>0$ 
with the control of norms 
\begin{equation}
\|h_0\|^2_{H^{k}_{\gamma_0}}+\mu_1\mu_2\|\theta_0\|^2_{H^{k}_{\gamma_0}}\leq \epsilon^2_0.
\label{1.16}\end{equation}
Here $\mu_1$, $\mu_2$ are given by (\ref{mu}), and $\epsilon_0$ is a small constant independent on $\sigma$ which will be determined later.
Moreover, we suppose that the initial distance between the two interfaces is proportional to $\sigma$:
\begin{equation}
\|\theta_0\|_{H^{k-3}_{\gamma_0}}\leq\sigma\epsilon_1
\label{1.17}\end{equation}
for some small constant $\epsilon_1$ independent of $\sigma$.
The main result is the uniform local existence of solutions to (\ref{1.13}-\ref{1.14}) as $\sigma$ tends to $0$ with initial data satisfying (\ref{1.16}-\ref{1.17}). Since we are only concerned with small $\sigma$, assume that $\sigma<1$ without loss of generality.
\begin{thm}\label{main}
Suppose that $\sigma<1$. Let $h_0$ and $\theta_0$ be analytic functions that belong to $H_{\gamma_0}^{k}$ with $k\geq 10$ and satisfy (\ref{1.16}-\ref{1.17}) for small constants $\epsilon_0$, $\epsilon_1$ and a given constant $\gamma_0>0$. Then there exists $T>0$ independent of $\sigma$ and a unique solution $(h(t),\theta(t))$ to (\ref{1.13}-\ref{1.14}) on $[0,T]$ satisfying
\begin{equation}
\|h(t)\|^2_{H^{k}_{\gamma(t)}}+\mu_1\mu_2\|\theta(t)\|^2_{H^{k}_{\gamma(t)}}\lesssim\epsilon_0^2,
\label{1.18}\end{equation}
\begin{equation}
\|\theta(t)\|^2_{H^{k-3}_{\gamma(t)}}\lesssim\sigma(\epsilon_1^2+\epsilon_0^2).
\label{1.19}\end{equation}
Here $\mu_1$, $\mu_2$ are given by (\ref{mu}), and $\gamma(t)$ is a time-dependent positive width function satisfying $\gamma(0)=\gamma_0$ and
\begin{equation}
2\gamma^\prime(t)\tanh(2\gamma(t))=-C_2\left(\|\partial_x h(t)\|_{L^\infty_{\gamma(t)}}+\|\partial_x\theta(t)\|_{L^\infty_{\gamma(t)}}\right),
\label{1.20}\end{equation}
where $C_2$ is a universal constant. 
\label{thm1.1}\end{thm}

\begin{rem}[Parabolicity and choice of the unknown]
We notice that, when linearized at $f=g=0$, the equation (\ref{1.11}) and (\ref{1.12}) becomes
$$\partial_t f=-\mu_2\Delta\rho\,P.V.\int_{\mathbb{R}}\frac{\partial_xf(x-\alpha,t)}{\alpha}d\alpha-\mu_1\Delta\rho\,P.V.\int_{\mathbb{R}}\frac{\alpha\cdot\partial_xg(x-\alpha),t)}{\alpha^2+(2\sigma)^2}d\alpha,$$
$$\partial_tg=-\mu_2\Delta\rho\,P.V.\int_{\mathbb{R}}\frac{\alpha\cdot\partial_xf(x-\alpha,t)}{\alpha^2+(2\sigma)^2}d\alpha-\mu_1\Delta\rho\,P.V.\int_{\mathbb{R}}\frac{\partial_xg(x-\alpha,t)}{\alpha}d\alpha.$$
If one further takes $\sigma=0$, the right-hand sides of the above two equations become the same:
$$-\Delta\rho\,P.V.\int_{\mathbb{R}}\frac{(\mu_2\partial_xf+\mu_1\partial_xg)(x-\alpha,t)}{\alpha}d\alpha=-\Delta\rho\Lambda h.$$
Consequently, we get a first order parabolic equation of $h$:
$\partial_th=-\Delta\rho \Lambda h.$ Therefore, one can expect to gain half a derivative in $h$ in the energy estimate. Moreover, by multiplying linearized \eqref{1.11} and \eqref{1.12} by $\mu_2\overline{f}$ and $\mu_1\overline{g}$ respectively, integrating over $x\in\Gamma_{\pm}(t)$ and taking the real parts, the left-hand sides will give rise to
 \begin{equation}\begin{aligned}
&\mu_2\Re\int_{\mathbb{R}}\partial_t f(x+i\gamma(t),t)\overline{f(x+i\gamma(t),t)}dx+\mu_2\Re\int_{\mathbb{R}}\partial_t f(x-i\gamma(t),t)\overline{f(x-i\gamma(t),t)}dx\\
&+\mu_1\Re\int_{\mathbb{R}}\partial_t g(x+i\gamma(t),t)\overline{g(x+i\gamma(t),t)}dx+\mu_1\Re\int_{\mathbb{R}}\partial_t g(x-i\gamma(t),t)\overline{g(x-i\gamma(t),t)}dx\\
=&
\frac{1}{2}\partial_t\left(\mu_2\|f\|_{H^{k}_{\gamma(t)}}^2+\mu_1\|g\|_{H^k_{\gamma(t)}}^2\right)-\frac{\mu_2}{\pi}\gamma^\prime(t)\int_{\mathbb{R}}|\hat{f}(\xi,t)|^2|\xi|\sinh\left(2\gamma(t)|\xi|\right)d\xi\\&-\frac{\mu_1}{\pi}\gamma^\prime(t)\int_{\mathbb{R}}|\hat{g}(\xi,t)|^2|\xi|\sinh\left(2\gamma(t)|\xi|\right)d\xi,
\end{aligned}\label{1.22}\end{equation}
and the terms on the right-hand sides will give rise to 
$$\begin{aligned}
-\mu_2^2\,\Re P.V.\int_{\mathbb{R}}\int_{\Gamma_{\pm}(t)}\frac{\partial_x f(x-\alpha)\overline{f(x)}}{\alpha}d\alpha dx-\mu_1\mu_2\,\Re P.V.\int_{\mathbb{R}}\int_{\Gamma_{\pm}(t)}\frac{\alpha\cdot \partial_xg(x-\alpha)\overline{f(x)}}{\alpha^2+(2\sigma)^2}d\alpha dx,
\end{aligned}$$
$$\begin{aligned}
-\mu_1\mu_2\Re\,P.V.\int_{\mathbb{R}}\int_{\Gamma_{\pm}(t)}\frac{\alpha\partial_xf(x-\alpha)\overline{g(x)}}{\alpha^2+(2\sigma)^2}d\alpha dx-\mu_1^2\Re\,P.V.\int_{\mathbb{R}}\int_{\Gamma_{\pm}(t)}\frac{\partial_xg(x-\alpha)\overline{g(x)}}{\alpha}d\alpha dx.
\end{aligned}$$
Here we dropped the common coefficient $\Delta\rho$ and the varriable $t$ for simplicity. Using the identity $\partial_xf(x-\alpha)=-\partial_\alpha\left(f(x-\alpha)-f(x)\right)$ and the symmetry between $x$ and $x-\alpha$, after integrating by parts in $\alpha$, the sum of the above terms becomes
$$\begin{aligned}
&-\frac{1}{2}P.V.\int_{\mathbb{R}}\int_{\Gamma_{\pm(t)}}\left(\mu_2^2\frac{\left|\Delta f(x,x-\alpha)\right|^2}{\alpha^2}+\mu_1^2\frac{\left|\Delta g(x,x-\alpha)\right|^2}{\alpha^2}\right.\\
&+\left.2\mu_1\mu_2\frac{\alpha^2-(2\sigma)^2}{\left(\alpha^2+(2\sigma)^2\right)^2}\Re\left(\Delta f(x,x-\alpha)\cdot\overline{\Delta g(x,x-\alpha)}\right)\right)d\alpha dx,
\end{aligned}$$
which will be a dissipation term in the energy estimate.
If one takes $\sigma=0$, the integrand is exactly $\alpha^{-2}\left|\mu_2 \Delta f+\mu_1\Delta g\right|^2=\alpha^{-2}\left|\Delta h\right|^2$. Since we are concerned with small $\sigma>0$, in general cases, the integrand can be viewed as a perturbation of $\alpha^{-2}\left|\Delta h\right|^2$. In fact, by setting $\theta=f-g$, it can be shown that
$$\mu_2^2\frac{\left|\Delta f\right|^2}{\alpha^2}+\mu_1^2\frac{\left|\Delta g\right|^2}{\alpha^2}+2\mu_1\mu_2\frac{\alpha^2-(2\sigma)^2}{\left(\alpha^2+(2\sigma)^2\right)^2}\Re\left(\Delta f\cdot\overline{\Delta g}\right)\simeq \frac{|\Delta h|^2}{\alpha^2}+\mu_1^2\mu_2^2\frac{(2\sigma)^2}{\alpha^2+(2\sigma)^2}\frac{|\Delta\theta|^2}{\alpha^2}.
$$
One reason to introduce $h$ and $\theta$  is that they are the eigenfunctions in the previous diagonalization. Indeed, the eigenvalue corresponding to $\Delta h$ has lower and upper bounds independent of $\sigma$, and the eigenvalue corresponding to $\theta$ depends on $\sigma$. While the above quadratic form is positive-definite, the factor $\frac{(2\sigma)^2}{\alpha^2+(2\sigma)^2}$ vanishes ununiformly in $\alpha$ as $\sigma$ goes to $0$. Moreover, the first term on the right-hand side of (\ref{1.22}) can be expressed in terms of $h$ and $\theta$:
$$\mu_2\|f\|_{H^{k}_{\gamma(t)}}^2+\mu_1\|g\|_{H^k_{\gamma(t)}}^2=\|h\|_{H^{k}_{\gamma(t)}}^2+\mu_1\mu_2\|\theta\|_{H^k_{\gamma(t)}}^2,$$
which are regarded as higher order energy throughout the paper. The other reason is that in order to bound the nonlinear terms, we will have to control $f-g$ and show that $\sigma^{-1}\|f-g\|_{H^{k-3}_{\gamma(t)}}\leq C$ with a constant $C$ independent of $\sigma$.
\label{rem2.3}\end{rem}

\begin{rem}\label{rem2.4} A very important remark for this paper is the following: in order to prove Theorem \ref{main} we will use energy estimates. Thus, we will have to fight against the loss of derivatives which appears in many terms throughout the analysis. Actually, to handle all these terms, one can simply use the gain of derivatives from losing analyticity by choosing $-\gamma^\prime$ large enough in Lemma \ref{lem 1.1}. For example, the integral of the quadratic form (\ref{2.15}) can be absorbed regardless of its sign if we choose $-\gamma^\prime>C$ for certain large constant $C$. However, in order to avoid losing analyticity and make $-\gamma^\prime$ as small as possible, we will use the parabolicity in (\ref{2.15}) instead of  losing analyticity  whenever it is possible to do so. Indeed, the only essential terms that cannot be absorbed by the parabolicity in (\ref{2.15}) are the anti-symmetric terms involving $(P_{12}-P_{21})$ in (\ref{2.9}-\ref{2.12}). These terms force us to make use of the analytic functional space and  handle them by losing analyticity. Subsequently, we will also use the loss of analyticity to handle the transport terms, since the velocities are then not real-valued, and we are no longer able to deal with the transport terms with integration by parts. Notice that these transport terms would not cause any trouble if we work in classical Sobolev spaces. In other words, the anti-symmetric terms in (\ref{2.9}-\ref{2.12}) are the only reason why we cannot prove Theorem \ref{main} in classical Sobolev spaces. 
\end{rem}

\section{A priori energy estimate}\label{apriori}
In this section, we will derive the a priori energy estimate for equations (\ref{1.13}-\ref{1.14}).
Indeed, we look for the estimate of the type
$$\begin{aligned}
&\frac{d}{dt}\left(\|\partial_x^kh\|^2_{L^2_{\gamma(t)}}+\mu_1\mu_2\|\partial_x^k\theta\|^2_{L^2_{\gamma(t)}}\right)+c_0 \text{Diss}_k^2\\
&-2\gamma^\prime(t)\tanh(2\gamma(t))\left(\|\Lambda^\frac{1}{2}\partial_x^k\|_{L^2_{\gamma(t))}}+\mu_1\mu_2\|\Lambda^\frac{1}{2}\partial_x^k\theta\|_{L^2_{\gamma(t)}}^2\right)\\
\lesssim&\text{nonlinear terms}.
\end{aligned}$$
The term $\text{Diss}_k$ stands for the dissipation term which we have mentioned in Remark \ref{rem2.3}.
In order to get this inequality we will need to take $k$-derivatives in the equations for $h$ and $\theta$ and test with $\overline{\partial_x^kh}$ and $\overline{\partial_x^k\theta}$. After integrating over $\Gamma_{\pm}(t)$, we will have to control terms of the form
\begin{align*}
    P.V.\int_{\Gamma_{\pm(t)}}\int_{\Gamma_{\pm(t)}} P(x,x_1)\partial_x^j  w(x)\partial_x^k w(x_1)dxdx_1.
\end{align*}
Here $w$ stands for $h$ or $\theta$, and $P$ stands for an integral kernel satisfying $|P|\simeq\frac{1}{|x-x_1|}$ which will be specified later for each term.
Depending on the numbers of the derivatives, i.e. depending on $j\in[0,k+1]\cap\mathbb{Z}$ and the behaviour of the kernel $P$,  we will split them in several classes:
\begin{enumerate}
\item Most singular terms and the transport term. They are the terms with $j=k+1$. Section \ref{decomposition}-\ref{transport} are devoted to them. In fact, in section \ref{top}, we will show that the symmetric part of the most singular terms will give rise to the dissipation $\text{Diss}_{k}$.
\item Safe term. They are the terms with $j=k$. These terms may be controlled by using $\text{Diss}_k$, or controlled by $\|\partial_x^kw\|_{L^2_{\gamma(t)}}^2$ through the T1 theorem, which will be presented in  section \ref{commutators1}.
\item Easy term. They are the terms with $j<k$. These terms can be controlled  by $\|h\|^2_{H^k_{\gamma(t)}}+\|\theta\|_{H^k_{\gamma(t)}}^2$, which is given in section \ref{commutators2}.
\end{enumerate}

To begin with, applying $\partial_x^k$ on each side of (\ref{1.13}-\ref{1.14}) yields that
\begin{equation}\begin{aligned}
&\partial_t\partial_x^kh+\left(\mu_2u_++\mu_1u_-\right)\partial^{k+1}_xh+\mu_1\mu_2\left(u_+-u_-\right)\partial_x^{k+1}\theta\\
=&-\Delta\rho \,P.V.\int_{\mathbb{R}}\left(\mu_2^2P_{11}+\mu_1\mu_2 P_{12}+\mu_1\mu_2 P_{21}+\mu_1^2 P_{22}\right)(x,x-\alpha)\partial^{k+1}_x h(x-\alpha)d\alpha\\
&-\mu_1\mu_2\Delta\rho\,P.V.\int_{\mathbb{R}}\left(\mu_2(P_{11}-P_{12})-\mu_1(P_{22}-P_{21})\right)(x,x-\alpha)\partial_x^{k+1}\theta (x-\alpha)d\alpha\\
&+\sum_{m=1}^4 N^h_{m}
\end{aligned}\label{2.1}\end{equation}
\begin{equation}\begin{aligned}
&\partial_t\partial_x^{k}\theta+\left(u_+-u_-\right)\partial^{k+1}_xh+\left(\mu_1u_++\mu_2u_-\right)\partial^{k+1}_x\theta\\
=&-\Delta\rho\,P.V.\int_{\mathbb{R}}\left(\mu_2(P_{11}-P_{21})-\mu_1(P_{22}-P_{12})\right)(x,x-\alpha)\partial^{k+1}_xh(x-\alpha)d\alpha\\
&-\mu_1\mu_2\Delta\rho\,P.V.\int_{\mathbb{R}}\left(P_{11}+P_{22}-P_{12}-P_{21}\right)(x,x-\alpha)\partial_x^{k+1}\theta(x-\alpha)d\alpha\\
&+\sum_{m=1}^4 N^\theta_{m},
\end{aligned}\label{2.2}\end{equation}
where $P_{ij}$ are given by (\ref{Pij}), and $N^h_{m}$ and $N^\theta_{m}$ are commutators given by
$$N_{1}^h:=-\left[\partial_x^k, \mu_2u_++\mu_1u_-\right]\partial_x h,\quad N_{2}^h:=-\mu_1\mu_2\left[\partial_x^k, u_+-u_-\right]\partial_x \theta,$$
$$N_{3}^h:=-\Delta\rho\,P.V.\int_{\mathbb{R}}\left[\partial_x^k, \left(\mu_2^2P_{11}+\mu_1\mu_2 P_{12}+\mu_1\mu_2 P_{21}+\mu_1^2 P_{22}\right)(x,x-\alpha)\right]\partial_xh(x-\alpha)d\alpha,$$
$$N_{4}^h:=-\mu_1\mu_2\Delta\rho\,P.V.\int_{\mathbb{R}}\left[\partial_x^k, \left(\mu_2(P_{11}-P_{12})-\mu_1(P_{22}-P_{21})\right)(x,x-\alpha)\right]\partial_x\theta(x-\alpha)d\alpha,$$
$$N_{1}^\theta:=-\left[\partial_x^k, u_+-u_-\right]\partial_xh,\quad N_{2}^\theta:=-\left[\partial_x^k, \mu_1u_++\mu_2 u_-\right]\partial_x\theta,$$
$$N_{3}^\theta:=-\Delta\rho\,P.V.\int_{\mathbb{R}}\left[\partial_x^k,\left(\mu_2(P_{11}-P_{21})-\mu_1(P_{22}-P_{12})\right)(x,x-\alpha)\right]\partial_xh(x-\alpha)d\alpha,$$
$$N_{4}^\theta:=-\mu_1\mu_2\Delta\rho\,P.V.\int_{\mathbb{R}}\left[\partial_x^k, \left(P_{11}+P_{22}-P_{12}-P_{21}\right)(x,x-\alpha)\right]\partial_x\theta(x-\alpha)d\alpha.$$
Then we multiply equation (\ref{2.1}) by $\overline{\partial^k_xh}$, multiply equation (\ref{2.2}) by $\mu_1\mu_2\overline{\partial^k_x\theta}$, and integrate the real part of the sum over $\Gamma_{\pm}(t):=\left\{x\in\mathbb{C}\mid x=x^\prime\pm i\gamma(t),\,x^\prime\in\mathbb{R}\right\}.$
To obtain the estimate over the resulted equation, we will look at the contribution of each term in (\ref{2.1})(\ref{2.2}) in this section.\\ 
\indent The organization of this section is as follows. In section \ref{decomposition}, the most singular terms\textemdash the contribution of the first two lines on the right-hand side of \eqref{2.1}\eqref{2.2} are decomposed into three classes after an integration by parts: the symmetric part, the anti-symmetric part and the lower order terms. In Section \ref{top}, we prove that the symmetric part is negative definite and provides dissipation terms. The lower order terms are handled in Section \ref{lower} in an easier fashion. In Section \ref{Asymmetry}, using the properties of the Hilbert transform, we are able to control the anti-symmetric part with loss of half a derivative in both $h$ and $\theta$. The second and third terms on the left-hand side of \eqref{2.1}\eqref{2.2}, which we refer to as transport terms, are treated in section \ref{transport}. Then we control the commutators $N_{m}^h$ and $N_{m}^\theta$ in Section \ref{commutators1} and \ref{commutators2}. The bounds obtained above are collected to complete the energy estimate in Section \ref{energy1} and \ref{energy2}.

\subsection{Decomposition of the most singular terms} \label{decomposition}
In this section, the most singular terms are decomposed into the aforementioned three classes. Recall that $$\Gamma_{\pm}(t)=\left\{x\in\mathbb{C}\mid x=x^\prime\pm i\gamma(t),\,x^\prime\in\mathbb{R}\right\}.$$
\noindent For $x\in\Gamma_{\pm}(t)$ and $(i,j)\in\left\{1,2\right\}^2$, applying the change of variable $x_1=x-\alpha$ and integrating by parts gives
\begin{equation}\begin{aligned}
P.V.\int_{\mathbb{R}}P_{ij}(x,x-\alpha)\partial_x^{k+1}h(x-\alpha)d\alpha&=P.V.\int_{\Gamma_{\pm}(t)}P_{ij}(x,x_1)\partial_x^{k+1}h(x_1)dx_1\\\
&=P.V.\int_{\Gamma_{\pm}(t)}\partial_{x_1}P_{ij}(x,x_1)\cdot\Delta\partial_x^kh(x,x_1)\,dx_1.
\end{aligned}\label{2.3}\end{equation}
 Then we compute that
$$\partial_{x_1}P_{11}(x,x_1)=K_{11}(x,x_1)+J_{11}(x,x_1),$$
$$K_{11}(x,x_1):=\frac{1}{(\Delta x)^2+(\Delta f)^2},\quad 
J_{11}(x,x_1):=\frac{2\Delta f(\partial_xf(x_1)\Delta x-\Delta f)}{\left((\Delta x)^2+(\Delta f)^2\right)^2}.$$
$$\partial_{x_1}P_{22}(x,x_1)=K_{22}(x,x_1)+J_{22}(x,x_1),$$
$$K_{22}(x,x_1):=\frac{1}{(\Delta x)^2+(\Delta g)^2},\quad J_{22}(x,x_1):=\frac{2\Delta g(\partial_x g(x_1)\Delta x-\Delta g)}{\left((\Delta x)^2+(\Delta g)^2\right)}.$$
For $\partial_{x_1}P_{12}$ and $\partial_{x_1}P_{21}$, we will use two equivalent expressions for each of them:
$$\partial_{x_1}P_{12}(x,x_1)=K_{12}(x,x_1)+J_{12}(x,x_1)=\tilde{K}_{12}(x,x_1)+\tilde{J}_{12}(x,x_1),$$
$$\partial_{x_1}P_{21}(x,x_1)=K_{21}(x,x_1)+J_{21}(x,x_1)=\tilde{K}_{21}(x,x_1)+\tilde{J}_{21}(x,x_1),$$
$$K_{12}(x,x_1):=\frac{(\Delta x)^2+(\Delta f)^2-(2\sigma+\theta(x_1))^2}{\left((\Delta x)^2+(\Delta f+2\sigma+\theta(x_1))^2\right)^2},\quad \tilde{K}_{12}(x,x_1):=\frac{(\Delta x)^2+(\Delta g)^2-(2\sigma+\theta(x))^2}{\left((\Delta x)^2+(\Delta g+2\sigma+\theta(x))^2\right)^2},$$
$$J_{12}(x,x_1):=\frac{2(\Delta f+2\sigma+\theta(x_1))(\partial_x f(x_1)\Delta x-\Delta f)}{\left((\Delta x)^2+(\Delta f+2\sigma+\theta(x_1))^2\right)^2}-\frac{2\Delta x\partial_x\theta(x_1)(\Delta f+2\sigma+\theta(x_1))}{\left((\Delta x)^2+(\Delta f+2\sigma+\theta(x_1))^2\right)^2},$$
$$\tilde{J}_{12}(x,x_1)=\frac{2(\Delta g+2\sigma+\theta(x))(\partial_x g(x_1)\Delta x-\Delta g)}{\left((\Delta x)^2+(\Delta g+2\sigma+\theta(x))^2\right)^2}.$$
$$K_{21}(x,x_1):=\frac{(\Delta x)^2+(\Delta f)^2-(2\sigma+\theta(x))^2}{\left((\Delta x)^2+(\Delta f-2\sigma-\theta(x))^2\right)^2},\quad \tilde{K}_{21}(x,x_1):=\frac{(\Delta x)^2+(\Delta g)^2-(2\sigma+\theta(x_1))^2}{\left((\Delta x)^2+(\Delta g-2\sigma-\theta(x_1))^2\right)^2},$$
$$J_{21}(x,x_1):=\frac{2(\Delta f-2\sigma-\theta(x))(\partial_x f(x_1)\Delta x-\Delta f)}{\left((\Delta x)^2+(\Delta f-2\sigma-\theta(x))^2\right)^2},$$
$$\tilde{J}_{21}(x,x_1):=\frac{2(\Delta g-2\sigma-\theta(x_1))(\partial_x g(x_1)\Delta x-\Delta g)}{\left((\Delta x)^2+(\Delta g-2\sigma-\theta(x_1))^2\right)^2}+\frac{2\Delta x\partial_x\theta(x_1)(\Delta g-2\sigma-\theta(x_1))}{\left((\Delta x)^2+(\Delta g-2\sigma-\theta(x_1))^2\right)^2}.$$
Actually we will proof the following lemma in this section.
\begin{lem}\label{summary3.1}\allowdisplaybreaks[4]
The contribution of the most singular terms can be decomposed into the symmetric parts, the anti-symmetric parts and the lower order terms:
\begin{align*}
&-\int_{\Gamma_{\pm}(t)}P.V.\int_{\Gamma_{\pm}(t)}\left(\mu_2^2P_{11}+\mu_1\mu_2(P_{12}+P_{21})+\mu_1^2P_{22}\right)(x,x_1)\partial_x^{k+1}h(x_1)dx_1\,\overline{\partial_x^{k}h(x)}dx\\
&-\mu_1\mu_2^2\int_{\Gamma_{\pm}(t)}P.V.\int_{\Gamma_{\pm}(t)}(P_{11}-P_{12})(x,x_1)\partial_x^{k+1}\theta(x_1)dx_1\,\overline{\partial_x^kh(x)}dx\\
&+\mu_1^2\mu_2\int_{\Gamma_{\pm}(t)}P.V.\int_{\Gamma_{\pm}(t)}(P_{22}-P_{21})(x,x_1)\partial_x^{k+1}\theta(x_1)dx_1\,\overline{\partial_x^kh(x)}dx\\
&-\mu_1^2\mu_2^2\int_{\Gamma_{\pm}(t)}P.V.\int_{\Gamma_{\pm}(t)}\left(P_{11}+P_{22}-P_{12}-P_{21}\right)(x,x_1)\partial_x^{k+1}\theta(x_1)dx_1\,\overline{\partial_x^k\theta(x)}dx\\
&-\mu_1\mu_2^2\int_{\Gamma_{\pm}(t)}P.V.\int_{\Gamma_{\pm}(t)}(P_{11}-P_{21})(x,x_1)\partial_x^{k+1}h(x_1)dx_1\,\overline{\partial_x^k\theta(x)}dx\\
&+\mu_1^2\mu_2\int_{\Gamma_{\pm}(t)}P.V.\int_{\Gamma_{\pm}(t)}(P_{22}-P_{12})(x,x_1)\partial_x^{k+1}h(x_1)dx_1\,\overline{\partial_x^k\theta(x)}dx\\
=&T_{sym}+T_{low}+T_{asym},
\end{align*}
where $T_{sym}$, $T_{asym}$ and $T_{low}$ denote the symmetric part, anti-symmetric part and the lower order terms respectively. They are given as follows.
\begin{align*}
T_{sym}
:=&-\int_{\Gamma_{\pm}(t)}\int_{\Gamma_{\pm}(t)}\left[\frac{1}{2}\Re D_{11}|\Delta\partial_x^kh|^2\right.\\
&+\mu_1\mu_2\Re\left(\mu_2\left(K_{11}-\frac{1}{2}K_{12}-\frac{1}{2}K_{21}\right)-\mu_1\left(K_{22}-\frac{1}{2}\tilde{K}_{12}-\frac{1}{2}\tilde{K}_{21}\right)\right)\Re\left(\Delta\partial_x^kh\cdot\overline{\Delta\partial_x^k\theta}\right)\\
&+\left.\frac{\mu_1^2\mu_2^2}{2}\Re\left(K_{11}-\frac{1}{2}K_{12}-\frac{1}{2}K_{21}+K_{22}-\frac{1}{2}\tilde{K}_{12}-\frac{1}{2}\tilde{K}_{21}\right)|\Delta\partial_x^k\theta|^2\right]dxdx_1,
\end{align*}
where we denote
$$D_{11}(x,x_1):=\left(\mu_2^2K_{11}+\frac{1}{2}\mu_1\mu_2K_{12}+\frac{1}{2}\mu_1\mu_2K_{21}+\frac{1}{2}\mu_1\mu_2\tilde{K}_{12}+\frac{1}{2}\mu_1\mu_2\tilde{K}_{21}+\mu_1^2K_{22}\right)(x,x_1).$$
\begin{align*}
T_{low}
:=&-\int_{\Gamma_{\pm}(t)}\int_{\Gamma_{\pm}(t)}\left(\mu_2^2J_{11}+\mu_1^2J_{22}\right)(x,x_1)\cdot\Delta\partial_x^kh\cdot\overline{\partial_x^kh(x)}dxdx_1\\
&-\int_{\Gamma_{\pm}(t)}\int_{\Gamma_{\pm}(t)}\frac{\mu_1\mu_2}{2}\left(J_{12}+J_{21}+\tilde{J}_{12}+\tilde{J}_{21}\right)(x,x_1)\cdot\Delta\partial_x^kh\cdot\overline{\partial_x^kh(x)}dxdx_1\\
&-\mu_1^2\mu_2^2\int_{\Gamma_{\pm}(t)}\int_{\Gamma_{\pm}(t)}\left[J_{11}+J_{22}-\frac{1}{2}\left(J_{12}+J_{21}+\tilde{J}_{12}+\tilde{J}_{21}\right)\right](x,x_1)\cdot\Delta\partial_x^k\theta\cdot\overline{\partial_x^k\theta(x)}dxdx_1\\
&+\mu_1^2\mu_2\int_{\Gamma_{\pm}(t)}\int_{\Gamma_{\pm}(t)}\left(J_{22}-\frac{1}{2}\tilde{J}_{12}-\frac{1}{2}\tilde{J}_{21}\right)(x,x_1)\cdot\Delta\partial_x^k\theta\cdot\overline{\partial_x^kh(x)}dxdx_1\\
&-\mu_1\mu_2^2\int_{\Gamma_{\pm}(t)}\int_{\Gamma_{\pm}(t)}\left(J_{11}-\frac{1}{2}J_{12}-\frac{1}{2}J_{21}\right)(x,x_1)\cdot\Delta\partial_x^k\theta\cdot\overline{\partial_x^kh(x)}dxdx_1\\
&-\mu_1\mu_2^2\int_{\Gamma_{\pm}(t)}\int_{\Gamma_{\pm}(t)}\left(J_{11}-\frac{1}{2}J_{12}-\frac{1}{2}J_{21}\right)(x,x_1)\cdot\Delta\partial_x^kh\cdot\overline{\partial_x^k\theta(x)}dxdx_1\\
&+\mu_1^2\mu_2\int_{\Gamma_{\pm}(t)}\int_{\Gamma_{\pm}(t)}\left(J_{22}-\frac{1}{2}\tilde{J}_{12}-\frac{1}{2}\tilde{J}_{21}\right)(x,x_1)\cdot\Delta\partial_x^kh\cdot\overline{\partial_x^k\theta(x)}dxdx_1.
\end{align*}
\begin{align*}
T_{asym}:=
&\mu_1\mu_2^2P.V.\int_{\Gamma_{\pm}(t)}\int_{\Gamma_{\pm}(t)}\frac{1}{2}(P_{12}-P_{21})(x,x_1)\partial_x^{k+1}\theta(x_1)\overline{\partial_x^kh(x)}dxdx_1\\
&+\mu_1^2\mu_2\int_{\Gamma_{\pm}(t)}\int_{\Gamma_{\pm}(t)}\frac{1}{2}(P_{12}-P_{21})(x,x_1)\partial_x^{k+1}\theta(x_1)\overline{\partial_x^kh(x)}dxdx_1\\
&-\mu_1\mu_2^2P.V.\int_{\Gamma_{\pm}(t)}\int_{\Gamma_{\pm}(t)}\frac{1}{2}(P_{12}-P_{21})(x,x_1)\partial_x^{k+1}h(x_1)\overline{\partial_x^k\theta(x)}dxdx_1\\
&-\mu_1^2\mu_2\int_{\Gamma_{\pm}(t)}\int_{\Gamma_{\pm}(t)}\frac{1}{2}(P_{12}-P_{21})(x,x_1)\partial_x^{k+1}h(x_1)\overline{\partial_x^k\theta(x)}dxdx_1.
\end{align*}\allowdisplaybreaks[0]
\end{lem}

\begin{proof}[Proof of Lemma \ref{summary3.1}]
To derive the above decomposition, by (\ref{2.3}) the terms from (\ref{2.1}) involving the highest order derivative of $h$ are
\begin{equation}\begin{aligned}
&-\int_{\Gamma_{\pm}(t)}P.V.\int_{\Gamma_{\pm}(t)}\left(\mu_2^2P_{11}+\mu_1\mu_2(P_{12}+P_{21})+\mu_1^2P_{22}\right)(x,x_1)\partial_x^{k+1}h(x_1)dx_1\,\overline{\partial_x^{k}h(x)}dx\\
=&-P.V.\int_{\Gamma_{\pm}(t)}\int_{\Gamma_{\pm}(t)}D_{11}(x,x_1)\cdot\Delta\partial_x^kh\cdot\overline{\partial_x^kh(x)}dxdx_1\\
&-\int_{\Gamma_{\pm}(t)}\int_{\Gamma_{\pm}(t)}\left(\mu_2^2J_{11}+\mu_1^2J_{22}\right)(x,x_1)\cdot\Delta\partial_x^kh\cdot\overline{\partial_x^kh(x)}dxdx_1\\
&-\int_{\Gamma_{\pm}(t)}\int_{\Gamma_{\pm}(t)}\frac{\mu_1\mu_2}{2}\left(J_{12}+J_{21}+\tilde{J}_{12}+\tilde{J}_{21}\right)(x,x_1)\cdot\Delta\partial_x^kh\cdot\overline{\partial_x^kh(x)}dxdx_1.
\end{aligned}\label{2.4}\end{equation}
By noticing $K_{12}(x,x_1)=K_{21}(x_1,x)$ and $\tilde{K}_{12}(x,x_1)=\tilde{K}_{21}(x,x_1)$, we see that $K_{11}$, $K_{22}$, $K_{12}+K_{21}$ and $\tilde{K}_{12}+\tilde{K}_{21}$ are symmetric in the sense that
$$K_{11}(x,x_1)=K_{11}(x_1,x),\quad K_{22}(x,x_1)=K_{22}(x_1,x),$$
$$(K_{12}+K_{21})(x,x_1)=(K_{12}+K_{21})(x_1,x),\quad (\tilde{K}_{12}+\tilde{K}_{21})(x,x_1)=(\tilde{K}_{12}+\tilde{K}_{21})(x,x_1).$$
In particular, $D_{11}(x,x_1)=D_{11}(x_1,x)$. Hence 
\begin{equation}\begin{aligned}
&\Re\,P.V.\int_{\Gamma_{\pm}(t)}\int_{\Gamma_{\pm}(t)}D_{11}(x,x_1)\cdot\Delta\partial_x^k h\cdot\overline{\partial_x^kh(x)}dxdx_1\\
=&\frac{1}{2}\Re\int_{\Gamma_{\pm}(t)}\int_{\Gamma_{\pm}(t)}D_{11}(x,x_1)\left|\Delta\partial_x^kh\right|^2dxdx_1\\
=&\frac{1}{2}\int_{\Gamma_{\pm}(t)}\int_{\Gamma_{\pm}(t)}\Re D_{11}(x,x_1)\left|\Delta\partial_x^kh\right|^2dxdx_1.
\end{aligned}\label{2.5}\end{equation}
Replacing $h$ by $\theta$ in (\ref{2.3}), the terms from (\ref{2.2}) involving the highest order of $\theta$ are
\begin{equation}\begin{aligned}
&-\mu_1^2\mu_2^2\int_{\Gamma_{\pm}(t)}P.V.\int_{\Gamma_{\pm}(t)}\left(P_{11}+P_{22}-P_{12}-P_{21}\right)(x,x_1)\partial_x^{k+1}\theta(x_1)dx_1\,\overline{\partial_x^k\theta(x)}dx\\
=&-\mu_1^2\mu_2^2P.V.\int_{\Gamma_{\pm}(t)}\int_{\Gamma_{\pm}(t)}\left(K_{11}-\frac{1}{2}K_{12}-\frac{1}{2}K_{21}\right)(x,x_1)\cdot\Delta\partial_x^k\theta\cdot\overline{\partial_x^k\theta(x)}dxdx_1\\
&-\mu_1^2\mu_2^2P.V.\int_{\Gamma_{\pm}(t)}\int_{\Gamma_{\pm}(t)}\left(K_{22}-\frac{1}{2}\tilde{K}_{12}-\frac{1}{2}\tilde{K}_{21}\right)(x,x_1)\cdot\Delta\partial_x^k\theta\cdot\overline{\partial_x^k\theta(x)}dxdx_1\\
&-\mu_1^2\mu_2^2\int_{\Gamma_{\pm}(t)}\int_{\Gamma_{\pm}(t)}\left(J_{11}+J_{22}\right)(x,x_1)\cdot\Delta\partial_x^k\theta\cdot\overline{\partial_x^k\theta(x)}dxdx_1.\\
&+\frac{\mu_1^2\mu_2^2}{2}\int_{\Gamma_{\pm}(t)}\int_{\Gamma_{\pm}(t)}(J_{12}+J_{21}+\tilde{J}_{12}+\tilde{J}_{21})(x,x_1)\cdot\Delta\partial_x^k\theta\cdot\overline{\partial_x^k\theta(x)}dxdx_1.
\end{aligned}\label{2.6}\end{equation}
By symmetry, it follows
\begin{equation}\begin{aligned}
&\mu_1^2\mu_2^2\Re\,P.V.\int_{\Gamma_{\pm}(t)}\int_{\Gamma_{\pm}(t)}\left(K_{11}-\frac{1}{2}K_{12}-\frac{1}{2}K_{21}\right)(x,x_1)\cdot\Delta\partial_x^k\theta\cdot\overline{\partial_x^k\theta(x)}dxdx_1\\
=&\frac{1}{2}\mu_1^2\mu_2^2\int_{\Gamma_{\pm}(t)}\int_{\Gamma_{\pm}(t)}\Re\left(K_{11}-\frac{1}{2}K_{12}-\frac{1}{2}K_{21}\right)(x,x_1)|\Delta\partial_x^k\theta|^2dxdx_1,
\end{aligned}\label{2.7}\end{equation}
\begin{equation}\begin{aligned}
&\mu_1^2\mu_2^2\Re\,P.V.\int_{\Gamma_{\pm}(t)}\int_{\Gamma_{\pm}(t)}\left(K_{22}-\frac{1}{2}\tilde{K}_{12}-\frac{1}{2}\tilde{K}_{21}\right)(x,x_1)\cdot\Delta\partial_x^k\theta\cdot\overline{\partial_x^k\theta(x)}dxdx_1\\
=&\frac{1}{2}\mu_1^2\mu_2^2\int_{\Gamma_{\pm}(t)}\int_{\Gamma_{\pm}(t)}\Re\left(K_{22}-\frac{1}{2}\tilde{K}_{12}-\frac{1}{2}\tilde{K}_{21}\right)(x,x_1)|\Delta\partial_x^k\theta|^2dxdx_1.
\end{aligned}\label{2.8}\end{equation}
For the terms from (\ref{2.1}) involving the highest order derivative of $\theta$, we decompose them into symmetric parts and anti-symmetric parts:
$$P_{11}-P_{12}=\left(P_{11}-\frac{1}{2}P_{12}-\frac{1}{2}P_{21}\right)-\frac{1}{2}\left(P_{12}-P_{21}\right),$$
$$P_{22}-P_{21}=\left(P_{22}-\frac{1}{2}P_{12}-\frac{1}{2}P_{21}\right)+\frac{1}{2}\left(P_{12}-P_{21}\right).$$
Then use (\ref{2.3}) to get
\begin{equation}\begin{aligned}
&-\mu_1\mu_2^2\int_{\Gamma_{\pm}(t)}P.V.\int_{\Gamma_{\pm}(t)}(P_{11}-P_{12})(x,x_1)\partial_x^{k+1}\theta(x_1)dx_1\,\overline{\partial_x^kh(x)}dx\\
=&-\frac{1}{2}\mu_1\mu_2^2 \int_{\Gamma_{\pm}(t)}\int_{\Gamma_{\pm}(t)}\left(K_{11}-\frac{1}{2}K_{12}-\frac{1}{2}K_{21}\right)(x,x_1)\cdot\Delta\partial_x^k\theta\cdot\overline{\Delta\partial_x^kh}\,dxdx_1\\
&-\mu_1\mu_2^2\int_{\Gamma_{\pm}(t)}\int_{\Gamma_{\pm}(t)}\left(J_{11}-\frac{1}{2}J_{12}-\frac{1}{2}J_{21}\right)(x,x_1)\cdot\Delta\partial_x^k\theta\cdot\overline{\partial_x^kh(x)}dxdx_1\\
&+\mu_1\mu_2^2P.V.\int_{\Gamma_{\pm}(t)}\int_{\Gamma_{\pm}(t)}\frac{1}{2}(P_{12}-P_{21})(x,x_1)\partial_x^{k+1}\theta(x_1)\overline{\partial_x^kh(x)}dxdx_1
\end{aligned}\label{2.9}\end{equation}
\begin{equation}\begin{aligned}
&\mu_1^2\mu_2\int_{\Gamma_{\pm}(t)}P.V.\int_{\Gamma_{\pm}(t)}(P_{22}-P_{21})(x,x_1)\partial_x^{k+1}\theta(x_1)dx_1\,\overline{\partial_x^kh(x)}dx\\
=&\frac{1}{2}\mu_1^2\mu_2\int_{\Gamma_{\pm}(t)}\int_{\Gamma_{\pm}(t)}\left(K_{22}-\frac{1}{2}\tilde{K}_{12}-\frac{1}{2}\tilde{K}_{21}\right)(x,x_1)\cdot\Delta\partial_x^k\theta\cdot\overline{\Delta\partial_x^kh}\,dxdx_1\\
&+\mu_1^2\mu_2\int_{\Gamma_{\pm}(t)}\int_{\Gamma_{\pm}(t)}\left(J_{22}-\frac{1}{2}\tilde{J}_{12}-\frac{1}{2}\tilde{J}_{21}\right)(x,x_1)\cdot\Delta\partial_x^k\theta\cdot\overline{\partial_x^kh(x)}dxdx_1\\
&+\mu_1^2\mu_2\int_{\Gamma_{\pm}(t)}\int_{\Gamma_{\pm}(t)}\frac{1}{2}(P_{12}-P_{21})(x,x_1)\partial_x^{k+1}\theta(x_1)\overline{\partial_x^kh(x)}dxdx_1.
\end{aligned}\label{2.10}\end{equation}
Similarly, for the terms from (\ref{2.2}) involving the  highest order derivative of $h$, we apply the decomposition
$P_{11}-P_{21}=\left(P_{11}-\frac{1}{2}P_{12}-\frac{1}{2}P_{21}\right)+\frac{1}{2}(P_{12}-P_{21})$ and
$P_{22}-P_{12}=\left(P_{22}-\frac{1}{2}P_{12}-\frac{1}{2}P_{21}\right)-\frac{1}{2}(P_{12}-P_{21})$
so that
\begin{equation}\begin{aligned}
&-\mu_1\mu_2^2\int_{\Gamma_{\pm}(t)}P.V.\int_{\Gamma_{\pm}(t)}(P_{11}-P_{21})(x,x_1)\partial_x^{k+1}h(x_1)dx_1\,\overline{\partial_x^k\theta(x)}dx\\
=&-\frac{1}{2}\mu_1\mu_2^2 \int_{\Gamma_{\pm}(t)}\int_{\Gamma_{\pm}(t)}\left(K_{11}-\frac{1}{2}K_{12}-\frac{1}{2}K_{21}\right)(x,x_1)\cdot\Delta\partial_x^kh\cdot\overline{\Delta\partial_x^k\theta}\,dxdx_1\\
&-\mu_1\mu_2^2\int_{\Gamma_{\pm}(t)}\int_{\Gamma_{\pm}(t)}\left(J_{11}-\frac{1}{2}J_{12}-\frac{1}{2}J_{21}\right)(x,x_1)\cdot\Delta\partial_x^kh\cdot\overline{\partial_x^k\theta(x)}dxdx_1\\
&-\mu_1\mu_2^2P.V.\int_{\Gamma_{\pm}(t)}\int_{\Gamma_{\pm}(t)}\frac{1}{2}(P_{12}-P_{21})(x,x_1)\partial_x^{k+1}h(x_1)\overline{\partial_x^k\theta(x)}dxdx_1,
\end{aligned}\label{2.11}\end{equation}
\begin{equation}\begin{aligned}
&\mu_1^2\mu_2\int_{\Gamma_{\pm}(t)}P.V.\int_{\Gamma_{\pm}(t)}(P_{22}-P_{12})(x,x_1)\partial_x^{k+1}h(x_1)dx_1\,\overline{\partial_x^k\theta(x)}dx\\
=&\frac{1}{2}\mu_1^2\mu_2\int_{\Gamma_{\pm}(t)}\int_{\Gamma_{\pm}(t)}\left(K_{22}-\frac{1}{2}\tilde{K}_{12}-\frac{1}{2}\tilde{K}_{21}\right)(x,x_1)\cdot\Delta\partial_x^kh\cdot\overline{\Delta\partial_x^k\theta}\,dxdx_1\\
&+\mu_1^2\mu_2\int_{\Gamma_{\pm}(t)}\int_{\Gamma_{\pm}(t)}\left(J_{22}-\frac{1}{2}\tilde{J}_{12}-\frac{1}{2}\tilde{J}_{21}\right)(x,x_1)\cdot\Delta\partial_x^kh\cdot\overline{\partial_x^k\theta(x)}dxdx_1\\
&-\mu_1^2\mu_2\int_{\Gamma_{\pm}(t)}\int_{\Gamma_{\pm}(t)}\frac{1}{2}(P_{12}-P_{21})(x,x_1)\partial_x^{k+1}h(x_1)\overline{\partial_x^k\theta(x)}dxdx_1.
\end{aligned}\label{2.12}\end{equation}
Moreover, the real part of the sum of the $K$'s terms in (\ref{2.9}) and (\ref{2.11}) is
\begin{equation}\begin{aligned}
&-\frac{1}{2}\mu_1\mu_2^2\Re\int_{\Gamma_{\pm}(t)}\int_{\Gamma_{\pm}(t)}\left(K_{11}-\frac{1}{2}K_{12}-\frac{1}{2}K_{21}\right)(x,x_1)\cdot\Delta\partial_x^k\theta\cdot\overline{\Delta\partial_x^kh}\,dxdx_1\\
&-\frac{1}{2}\mu_1\mu_2^2\Re\int_{\Gamma_{\pm}(t)}\int_{\Gamma_{\pm}(t)}\left(K_{11}-\frac{1}{2}K_{12}-\frac{1}{2}K_{21}\right)(x,x_1)\cdot\Delta\partial_x^kh\cdot\overline{\Delta\partial_x^k\theta}\,dxdx_1\\
=&-\mu_1\mu_2^2\int_{\Gamma_{\pm}(t)}\int_{\Gamma_{\pm}(t)}\Re\left(K_{11}-\frac{1}{2}K_{12}-\frac{1}{2}K_{21}\right)(x,x_1)\cdot\Re\left(\Delta\partial_x^k h\cdot\overline{\Delta\partial_x^k\theta}\right)dxdx_1,
\end{aligned}\label{2.13}\end{equation}
while the real part of the sum of the $K$'s terms in (\ref{2.10}) and (\ref{2.12}) is
\begin{equation}\begin{aligned}
&\frac{1}{2}\mu_1^2\mu_2\Re\int_{\Gamma_{\pm}(t)}\int_{\Gamma_{\pm}(t)}\left(K_{22}-\frac{1}{2}\tilde{K}_{12}-\frac{1}{2}\tilde{K}_{21}\right)(x,x_1)\cdot\Delta\partial_x^k\theta\cdot\overline{\Delta\partial_x^kh}\,dxdx_1\\
&+\frac{1}{2}\mu_1^2\mu_2\Re\int_{\Gamma_{\pm}(t)}\int_{\Gamma_{\pm}(t)}\left(K_{22}-\frac{1}{2}\tilde{K}_{12}-\frac{1}{2}\tilde{K}_{21}\right)(x,x_1)\cdot\Delta\partial_x^kh\cdot\overline{\Delta\partial_x^k\theta}\,dxdx_1\\
=&\mu_1^2\mu_2\int_{\Gamma_{\pm}(t)}\int_{\Gamma_{\pm}(t)}\Re\left(K_{22}-\frac{1}{2}\tilde{K}_{12}-\frac{1}{2}\tilde{K}_{21}\right)(x,x_1)\cdot\Re\left(\Delta\partial_x^k h\cdot\overline{\Delta\partial_x^k\theta}\right)dxdx_1.
\end{aligned}\label{2.14}\end{equation}
Summing up equations (\ref{2.4})(\ref{2.6}) and (\ref{2.9}-\ref{2.12}), we arrive at the desired decomposition in Lemma \ref{summary3.1}.
\end{proof}
\subsection{Symmetric part and parabolicity}\label{top}
The symmetric part $T_{sym}$ consists of an integral of a quadratic form of $\Delta\partial_x^kh$ and $\Delta\partial_x^k\theta$:
\begin{equation}\begin{aligned}
&\frac{1}{2}\Re D_{11}|\Delta\partial_x^kh|^2+\frac{\mu_1^2\mu_2^2}{2}\Re\left(K_{11}-\frac{1}{2}K_{12}-\frac{1}{2}K_{21}+K_{22}-\frac{1}{2}\tilde{K}_{12}-\frac{1}{2}\tilde{K}_{21}\right)|\Delta\partial_x^k\theta|^2\\
&+\mu_1\mu_2\Re\left(\mu_2\left(K_{11}-\frac{1}{2}K_{12}-\frac{1}{2}K_{21}\right)-\mu_1\left(K_{22}-\frac{1}{2}\tilde{K}_{12}-\frac{1}{2}\tilde{K}_{21}\right)\right)\Re\left(\Delta\partial_x^kh\cdot\overline{\Delta\partial_x^k\theta}\right).
\end{aligned}\label{2.15}\end{equation}
The goal of this section is to prove that this quadratic form is positive definite up to some other lower order terms, and therefore the symmetric part $T_{sym}$ gives rise to a dissipation term. The main result of this section is as follows.
\begin{lem}\label{summary3.2}
Let $D_{11}^0$ and $D_{22}^0$ be given by
$$D_{11}^0(x,x_1)=\mu_2^2\frac{1}{(\Delta x)^2}+2\mu_1\mu_2\frac{(\Delta x)^2-(2\sigma)^2}{\left((\Delta x)^2+(2\sigma)^2\right)^2}+\mu_1^2\frac{1}{(\Delta x)^2},$$
$$D_{22}^0(x,x_1)=2G_0(\Delta x,2\sigma)^{-1}F_0(\Delta x,2\sigma),$$
where
$$G_0(\Delta x,2\sigma)=(\Delta x)^2\left[(\Delta x)^2+(2\sigma)^2\right]^4,$$
$$F_0(\Delta x,2\sigma)=(2\sigma)^8+5(\Delta x)^2(2\sigma)^6+7(\Delta x)^4(2\sigma)^4+3(\Delta x)^6(2\sigma)^2.$$
The symmetric part $T_{sym}$ satisfies the following control.
$$\begin{aligned}
-T_{sym}
\geq&C\int_{\Gamma_{\pm}(t)}\int_{\Gamma_{\pm}(t)}D_{11}^0(x,x_1)|\Delta\partial_x^kh|^2dxdx_1+C\mu_1^2\mu_2^2\int_{\Gamma_{\pm}(t)}\int_{\Gamma_{\pm}(t)}D_{22}^0(x,x_1)|\Delta\partial_x^k\theta|^2dxdx_1\\
&-C\sigma^{-1}\|\partial_x\theta\|_{L^\infty_{\gamma(t)}}\left(\|\partial_xf\|_{L^\infty_{\gamma(t)}}+\|\partial_xg\|_{L^\infty_{\gamma(t)}}\right)\|\partial_x^kh\|_{L^2_{\gamma(t)}}\|\partial_x^k\theta\|_{L^2_{\gamma(t)}}\\
&-C\sigma^{-1}\|\partial_x\theta\|_{L^\infty_{\gamma(t)}}\left(\|\partial_xf\|_{L^\infty_{\gamma(t)}}+\|\partial_xg\|_{L^\infty_{\gamma(t)}}\right)\|\partial_x^k\theta\|_{L^2_{\gamma(t)}}^2.
\end{aligned}$$
\end{lem}
\begin{rem}\label{parabolicity}
As a complement of Remark \ref{rem2.3}, the lower bound of $-T_{sym}$ given by Lemma \ref{summary3.2} shows that the parabolicity of the equations (\ref{1.13})(\ref{1.14}) gives rise to not only dissipation in $h$ but also in $\theta$. These dissipation terms will be used to cancel other terms from section \ref{lower}, section \ref{commutators1} and section \ref{commutators2}, see Lemma \ref{summary3.3} and Lemma \ref{summary3.6}-\ref{summary3.7-3}. Moreover, the lower bound of $-T_{sym}$ allows us to choose $-\gamma^\prime$  as small as it is given in (\ref{1.20}) instead of constant size, see also Remark \ref{rem2.4}.
\end{rem}
In order to prove Lemma \ref{summary3.2}, the first step is decomposing $K_{11}-\frac{1}{2}K_{12}-\frac{1}{2}K_{21}$ and $K_{22}-\frac{1}{2}\tilde{K}_{12}-\frac{1}{2}\tilde{K}_{21}$ into their principal parts and lower order parts.
\begin{lem}
There exist functions $K_f$, $K_g$, $e_f$, $e_g$ such that\\
(1) $K_{11}-\frac{1}{2}K_{12}-\frac{1}{2}K_{21}=K_f+e_f$, $K_{22}-\frac{1}{2}\tilde{K}_{12}-\frac{1}{2}\tilde{K}_{21}=K_g+e_g$.\\
(2) $K_f$ and $K_g$ are perturbations of positive rational functions of $\Delta x$ and $2\sigma$:
$$K_f(x,x_1)=G\left(\Delta x, 2\sigma; \frac{\Delta f}{\Delta x},\frac{\theta(x)}{2\sigma},\frac{\theta(x_1)}{2\sigma}\right)^{-1}F\left(\Delta x, 2\sigma;\frac{\Delta f}{\Delta x},\frac{\theta(x)}{2\sigma},\frac{\theta(x_1)}{2\sigma}\right),$$
$$K_g(x,x_1)=G\left(\Delta x, 2\sigma; \frac{\Delta g}{\Delta x},\frac{\theta(x_1)}{2\sigma},\frac{\theta(x)}{2\sigma}\right)^{-1}F\left(\Delta x, 2\sigma;\frac{\Delta g}{\Delta x},\frac{\theta(x_1)}{2\sigma},\frac{\theta(x)}{2\sigma}\right).$$
(3) $e_f$ and $e_g$ have $\Delta\theta$ as a factor:
$$e_f(x,x_1)=\Delta\theta\Delta f\cdot G\left(\Delta x, 2\sigma; \frac{\Delta f}{\Delta x},\frac{\theta(x)}{2\sigma},\frac{\theta(x_1)}{2\sigma}\right)^{-1}E\left(\Delta x,2\sigma;\frac{\Delta f}{\Delta x},\frac{\theta(x)}{2\sigma},\frac{\theta(x_1)}{2\sigma}\right),$$
$$e_g(x,x_1)=-\Delta\theta\Delta g\cdot G\left(\Delta x, 2\sigma; \frac{\Delta g}{\Delta x},\frac{\theta(x_1)}{2\sigma},\frac{\theta(x)}{2\sigma}\right)^{-1}E\left(\Delta x,2\sigma;\frac{\Delta g}{\Delta x},\frac{\theta(x_1)}{2\sigma},\frac{\theta(x)}{2\sigma}\right).$$
Here $E\left(\Delta x,2\sigma;w_1,w_2,w_3\right)$, $F\left(\Delta x,2\sigma;w_1,w_2,w_3\right)$, $G\left(\Delta x,2\sigma;w_1,w_2,w_3\right)$ are two-element polynomials of $\Delta x$ and $2\sigma$ with coefficients depending on $(w_1,w_2,w_3)$. Moreover, $G(\Delta x,2\sigma;0,0,0)\geq0$, $F(\Delta x,2\sigma;0,0,0)\geq0$.
\label{lem2.1}\end{lem}
\begin{proof}
Let 
$$\begin{aligned}
&G\left(\Delta x, 2\sigma;,w_1,w_2,w_3\right)\\
:=&\left(1+w_1^2\right)\left(\Delta x\right)^2\left[\left(\Delta x\right)^2+\left(w_1\Delta x+2(1+w_3)\sigma\right)^2\right]^2\left[\left(\Delta x\right)^2+\left(-w_1\Delta x+2(1+w_2)\sigma\right)^2\right]^2,
\end{aligned}$$
so that
$$\begin{aligned}
&G\left(\Delta x,2\sigma;\frac{\Delta f}{\Delta x},\frac{\theta(x)}{2\sigma},\frac{\theta(x_1)}{2\sigma}\right)\\
=&\left((\Delta x)^2+(\Delta f)^2\right)\left((\Delta x)^2+(\Delta f+2\sigma+\theta(x_1))^2\right)^2\left((\Delta x)^2+(-\Delta f+2\sigma+\theta(x))^2\right)^2.
\end{aligned}$$
Then we compute by brute force to get
$$\begin{aligned}
&G\left(\Delta x,2\sigma;\frac{\Delta f}{\Delta x},\frac{\theta(x)}{2\sigma},\frac{\theta(x_1)}{2\sigma}\right)\left(K_{11}-\frac{1}{2}K_{12}-\frac{1}{2}K_{21}\right)(x,x_1)\\
=&c_0(2\sigma)^8+c_{2}(\Delta x)^2(2\sigma)^6+c_4(\Delta x)^4(2\sigma)^4+c_6(\Delta x)^6(2\sigma)^2\\
&+c_1(\Delta x)(2\sigma)^7+c_3(\Delta x)^3(2\sigma)^5+c_5(\Delta x)^5(2\sigma)^3+c_7(\Delta x)^7(2\sigma),
\end{aligned}$$
where $c_j=c_j\left(\frac{\Delta f}{\Delta x},\frac{\theta(x)}{2\sigma},\frac{\theta(x_1)}{2\sigma}\right)$, $0\leq j\leq8$ are functions of $\left(\frac{\Delta f}{\Delta x},\frac{\theta(x)}{2\sigma},\frac{\theta(x_1)}{2\sigma}\right)$:
$$c_0(w_1,w_2,w_3)=(1+w_2)^4(1+w_3)^4,$$
$$\begin{aligned}
c_2(w_1,w_2,w_3)=&\frac{1}{2}(5-3w_1^2)(1+w_2)^2(1+w_3)^2\left[(1+w_2)^2+(1+w_3)^2\right]\\
&+8w_1^2(1+w_2)^2(1+w_3)^2(w_2-w_3)^2,
\end{aligned}$$
$$\begin{aligned}
c_4(w_1,w_2,w_3)=&\frac{1}{2}(1+w_1)^2(1-15w_1)^2\left[(1+w_2)^4+(1+w_3)^4\right]\\
&+(6+16w_1^2+34w_1^4)(1+w_2^2)(1+w_3^2)\\
&+8w_1^2(1+w_1^2)\left[(1+w_2)^2+(1+w_3)^2\right](w_2-w_3)^2,
\end{aligned}$$
$$\begin{aligned}
c_6(w_1,w_2,w_3)=\frac{3}{2}(1+w_1^2)^2(1-3w_1^2)\left[(1+w_2)^2+(1+w_3)^2\right]+8w_1^2(1+w_1^2)^2(w_2-w_3)^2,
\end{aligned}$$
$$c_1(w_1,w_2,w_3)=4w_1(w_2-w_3)(1+w_2)^3(1+w_3)^3,$$
$$\begin{aligned}
c_3(w_1,w_2,w_3)=&w_1(w_2-w_3)\left[-(10+26w_1^2)(1+w_2)^2(1+w_3)^2\right.\\
&\left.+4(1+w_1^2)(1+w_2)(1+w_3)\frac{(1+w_2)^3-(1+w_3)^3}{w_2-w_3}\right],
\end{aligned}$$
$$\begin{aligned}
c_5(w_1,w_2,w_3)&=w_1(w_2-w_3)\left[2(1+w_1^2)^2\frac{(1+w_2)^3-(1+w_3)^3}{w_2-w_3}\right.\\
&\left.+(1+w_1^2)(10+26w_1^2)(1+w_2)(1+w_3)\right],
\end{aligned}$$
$$\begin{aligned}
c_7(w_1,w_2,w_3)=-2w_1(w_2-w_3)(1+w_1^2)^3.
\end{aligned}$$
Now we let 
$$F(\Delta x,2\sigma;w_1,w_2,w_3):=\sum_{j=0}^3c_{2j}(w_1,w_2,w_3)(\Delta x)^{2j}(2\sigma)^{8-2j},$$
$$E(\Delta x,2\sigma;w_1,w_2,w_3):=w_1^{-1}(w_2-w_3)^{-1}\sum_{j=0}^3c_{1+2j}(w_1,w_2,w_3)(\Delta x)^{2j}(2\sigma)^{6-2j}.$$
Note that $c_{1+2j}$, $j=0,1,2,3$ have $w_1(w_2-w_3)$ as a factor, so $E$ is well-defined. Then we get
$$\left(K_{11}-\frac{1}{2}K_{12}-\frac{1}{2}K_{21}\right)(x,x_1)=K_f(x,x_1)+e_f(x,x_1),$$
$$\begin{aligned}
K_f(x,x_1)=&G\left(\Delta x,2\sigma;\frac{\Delta f}{\Delta x},\frac{\theta(x)}{2\sigma},\frac{\theta(x_1)}{2\sigma}\right)^{-1}F\left(\Delta x,2\sigma;\frac{\Delta f}{\Delta x},\frac{\theta(x)}{2\sigma},\frac{\theta(x_1)}{2\sigma}\right),
\end{aligned}$$
$$\begin{aligned}
e_f(x,x_1)=\Delta\theta\Delta f\cdot G\left(\Delta x,2\sigma;\frac{\Delta f}{\Delta x},\frac{\theta(x)}{2\sigma},\frac{\theta(x_1)}{2\sigma}\right)^{-1} E\left(\Delta x,2\sigma;\frac{\Delta f}{\Delta x},\frac{\theta(x)}{2\sigma},\frac{\theta(x_1)}{2\sigma}\right).
\end{aligned}$$
To obtain the corresponding identities for $K_{22}-\frac{1}{2}\tilde{K}_{12}-\frac{1}{2}\tilde{K}_{21}$, it suffices to observe that the expression of $K_{22}-\frac{1}{2}\tilde{K}_{12}-\frac{1}{2}\tilde{K}_{21}$ is the same as $K_{11}-\frac{1}{2}K_{12}-\frac{1}{2}K_{21}$ if we replace $\Delta g$ by $\Delta f$ and swap $\theta(x)$ and $\theta(x_1)$. Finally, we see that
$$G_0(\Delta x,2\sigma):=G(\Delta x,2\sigma;0,0,0)=(\Delta x)^2\left[(\Delta x)^2+(2\sigma)^2\right]^4\geq0,$$
$$F_0(\Delta x,2\sigma):=F(\Delta x,2\sigma;0,0,0)=(2\sigma)^8+5(\Delta x)^2(2\sigma)^6+7(\Delta x)^4(2\sigma)^4+3(\Delta x)^6(2\sigma)^2\geq0.$$
\end{proof}
Using Lemma \ref{lem2.1}, we decompose the quadratic form (\ref{2.15}) into its principal part and lower-order parts:
\begin{equation}\begin{aligned}
&\frac{1}{2}\Re D_{11}|\Delta\partial_x^kh|^2+\frac{\mu_1^2\mu_2^2}{2}\Re\left(K_{11}-\frac{1}{2}K_{12}-\frac{1}{2}K_{21}+K_{22}-\frac{1}{2}\tilde{K}_{12}-\frac{1}{2}\tilde{K}_{21}\right)|\Delta\partial_x^k\theta|^2\\
&+\mu_1\mu_2\Re\left(\mu_2\left(K_{11}-\frac{1}{2}K_{12}-\frac{1}{2}K_{21}\right)-\mu_1\left(K_{22}-\frac{1}{2}\tilde{K}_{12}-\frac{1}{2}\tilde{K}_{21}\right)\right)\Re\left(\Delta\partial_x^kh\cdot\overline{\Delta\partial_x^k\theta}\right)\\
=&\frac{1}{2}\Re D_{11}|\Delta\partial_x^kh|^2+\mu_1\mu_2\Re D_{12}\cdot\Re\left(\Delta\partial_x^kh\cdot\overline{\Delta\partial_x^k\theta}\right)+\frac{\mu_1^2\mu_2^2}{2}\Re D_{22}|\Delta\partial_x^k\theta|^2\\
&+\mu_1\mu_2\Re\left(\mu_2e_f-\mu_1e_g\right)\cdot\Re\left(\Delta\partial_x^kh\cdot\overline{\Delta\partial_x^k\theta}\right)+\frac{\mu_1^2\mu_2^2}{2}\Re\left(e_f+e_g\right)|\Delta\partial_x^k\theta|^2,
\end{aligned}\label{2.16}\end{equation}
where we denote
$D_{12}:=\mu_2K_f-\mu_1 K_g$, $D_{22}=K_f+K_g.$
The last two terms on the right-hand side of (\ref{2.16}) are considered lower order for that $e_f$ and $e_g$ have $\Delta\theta$ as a factor, which is expected to be proportional to $\sigma\Delta x$. In the next, we show that the principle part is positive definite for small data, for which we need the following algebraic inequality.
\begin{lem}
Suppose that $a_j$ and $b_j$, $1\leq j\leq m$ are complex valued. Then we have the inequality
$$\left|\prod_{j=1}^na_j-\prod_{j=1}^nb_j\right|\leq\prod_{j=1}^n\max\left\{|a_j|,|b_j|\right\}\sum_{j=1}^n\frac{|a_j-b_j|}{\max\left\{|a_j|,|b_j|\right\}}.$$
\label{lem2.2}\end{lem}
\begin{proof}
We prove the desired inequality by induction. The case $n=1$ is trivial. Suppose that the above inequality holds for $n-1$. Then 
$$\begin{aligned}
&\left|\prod_{j=1}^na_j-\prod_{j=1}^nb_j\right|\\
\leq&|a_n-b_n|\prod_{j=1}^{n-1}|a_j|+|b_n|\left|\prod_{j=1}^{n-1}a_j-\prod_{j=1}^{n-1}b_j\right|\\
\leq&\max\left\{|a_n|,|b_n|\right\}\frac{|a_n-b_n|}{\max\left\{|a_n|,|b_n|\right\}}\prod_{j=1}^{n-1}|a_j|+|b_n|\sum_{j=1}^{n-1}\frac{|a_j-b_j|}{\max\left\{|a_j|,|b_j|\right\}}\prod_{j=1}^{n-1}\max\left\{|a_j|,|b_j|\right\}\\
\leq&\prod_{j=1}^n\max\left\{|a_j|,|b_j|\right\}\sum_{j=1}^n\frac{|a_j-b_j|}{\max\left\{|a_j|,|b_j|\right\}}.
\end{aligned}$$
\end{proof}
\begin{prop}
There exists a small constant $w_0>0$ such that if 
\begin{equation}    \|\partial_xf\|_{L^\infty_\gamma}+\|\partial_xg\|_{L^\infty_\gamma}+\sigma^{-1}\left\|\theta\right\|_{L^\infty_\gamma}\leq w_0,
\label{2.17}\end{equation}
then 
$$\begin{aligned}
&\frac{1}{2}\Re D_{11}|\Delta\partial_x^kh|^2+\mu_1\mu_2\Re D_{12}\Re\left(\Delta\partial_x^kh\cdot\overline{\Delta\partial_x^k\theta}\right)+\frac{1}{2}\mu_1^2\mu_2^2\Re D_{22}|\Delta\partial_x^k\theta|^2\\
\geq&\left(\frac{7}{50}-Cw_0\right)D_{11}^0|\Delta\partial_x^kh|^2+\left(\frac{7}{64}-Cw_0\right)\mu_1^2\mu_2^2D_{22}^0|\Delta\partial_x^k\theta|^2.
\end{aligned}$$
Here $D_{11}^0$ and $D_{22}^0$ are given by (\ref{2.18})(\ref{2.19}), and satisfy the following estimates:
$$D_{11}^0(x,x_1)\simeq\frac{1}{(\Delta x)^2}, \quad D_{22}^0(x,x_1)\simeq\frac{1}{(\Delta x)^2}\frac{(2\sigma)^2}{(\Delta x)^2+(2\sigma)^2}.$$
\label{prop2.3}\end{prop}
We remark that the two constants $\frac{7}{50}$ and $\frac{7}{64}$ are neither essential nor sharp. In the subsequent proof, we will only use the fact that the two leading coefficients in the lower bound given by Proposition \ref{prop2.3} are positive.\\
\indent To prove the desired inequality, we look at the determinant of the quadratic form:
\begin{equation}\begin{aligned}
&\Re D_{11}\Re D_{22}-(\Re D_{12})^2\\
=&\Re\left(D_{11}-\mu_2^2K_f-\mu_1^2K_g\right)\cdot\Re(K_f+K_g)+\Re K_f\cdot \Re K_g\\
=&\Re K_f\cdot\Re\left(\frac{1}{2}K_g+D_{11}-\mu_2^2K_f-\mu_1^2K_g\right)+\Re K_g\cdot\Re\left(\frac{1}{2}K_f+D_{11}-\mu_2^2K_f-\mu_1^2K_g\right).
\end{aligned}\label{det1}\end{equation}
By Lemma \ref{lem2.1}, 
\begin{equation}\begin{aligned}
&\frac{1}{2}K_g+D_{11}-\mu_2^2K_f-\mu_1^2K_g\\
=&\frac{1}{2}K_{22}+\frac{1}{2}\mu_2(K_{12}+K_{21})+\left(\frac{1}{2}\mu_1-\frac{1}{4}\right)(\tilde{K}_{12}+\tilde{K}_{21})+\mu_2^2e_f+\left(\mu_1^2-\frac{1}{2}\right)e_g\\
=&\frac{1}{2}K_{22}+\frac{1}{4}\tilde{K}_{12}+\frac{1}{4}\tilde{K}_{21}+\frac{1}{2}\mu_2(K_{12}+K_{21}-\tilde{K}_{12}-\tilde{K}_{21})+\mu_2^2e_f+\left(\mu_1^2-\frac{1}{2}\right)e_g,
\end{aligned}\label{det2}\end{equation}
and similarly,
\begin{equation}\begin{aligned}
&\frac{1}{2}K_f+D_{11}-\mu_2^2K_f-\mu_1^2K_g\\
=&\frac{1}{2}K_{11}+\left(\frac{1}{2}\mu_2-\frac{1}{4}\right)(K_{12}+K_{21})+\frac{1}{2}\mu_1(\tilde{K}_{12}+\tilde{K}_{21})+\left(\mu_2^2-\frac{1}{2}\right)e_f+\mu_1^2e_g\\
=&\frac{1}{2}K_{11}+\frac{1}{4}K_{12}+\frac{1}{4}K_{21}-\frac{1}{2}\mu_1(K_{12}+K_{21}-\tilde{K}_{12}-\tilde{K}_{21})+\left(\mu_2^2-\frac{1}{2}\right)e_f+\mu_1^2e_g.
\end{aligned}\label{det3}\end{equation}
In order to prove Proposition \ref{prop2.3}, the key point is to find a lower bound for the determinant. To this end, let $K_f^0$ and $K_g^0$ denote the unperturbed versions of $K_f$ and $K_g$:
$$K_f^0(x,x_1)=K_g^0(x,x_1):=G_0(\Delta x,2\sigma)^{-1}F_0(\Delta x,2\sigma),$$
where $F$ and $G$ are given by Lemma \ref{lem2.1}. 
Also, we denote $D_{11}^0$ and $D_{22}^0$ the unperturbed versions of $D_{11}$ and $D_{22}$:
\begin{equation}
D_{11}^0(x,x_1)=\mu_2^2\frac{1}{(\Delta x)^2}+2\mu_1\mu_2\frac{(\Delta x)^2-(2\sigma)^2}{\left((\Delta x)^2+(2\sigma)^2\right)^2}+\mu_1^2\frac{1}{(\Delta x)^2},
\label{2.18}\end{equation}
\begin{equation}
D_{22}^0(x,x_1):=K_f^0+K_g^0=2G_0(\Delta x,2\sigma)^{-1}F_0(\Delta x,2\sigma).
\label{2.19}\end{equation}
We will need the following estimates.
\begin{lem}\label{estimatesdetermminant}
Suppose that condition (\ref{2.17}) holds. Then the following estimates hold.\\
(i) $\Re K_f$ and $\Re K_g$ are perturbations of $K_f^0$ and $K_g^0$:
$$(1-Cw_0)K_f^0(x,x_1)\leq\Re K_f(x,x_1)\leq(1+Cw_0)K_f^0(x,x_1),$$
$$(1-Cw_0)K_g^0(x,x_1)\leq\Re K_g(x,x_1)\leq(1+Cw_0)K_g^0(x,x_1).$$
In particular,
\begin{equation}
\frac{1}{2}-Cw_0\leq\frac{\Re K_f}{\Re K_f+\Re K_g}\leq\frac{1}{2}+Cw_0.
\label{2.27}\end{equation}
(ii) $D_{11}^0$ can be bounded from below and above:
\begin{equation}\begin{aligned}
\frac{1}{(\Delta x)^2}\geq D_{11}^0(x,x_1)\geq\frac{7}{16}\frac{1}{(\Delta x)^2},
\end{aligned}\label{2.28}\end{equation}
and the following bound holds:
\begin{equation}\begin{aligned}
\Re(K_{11}+K_{12})(x,x_1)\geq\left(\frac{7}{8}-Cw_0\right)\frac{1}{(\Delta x)^2}.
\end{aligned}\label{2.29}\end{equation}
The same lower bound also holds for $\Re(K_{11}+K_{21})$, $\Re(K_{22}+\tilde{K}_{12})$ and $\Re(K_{22}+\tilde{K}_{21})$.  Moreover, $D_{11}$ is a perturbation of $D_{11}^0$:
\begin{equation}\begin{aligned}
\left|D_{11}(x,x_1)-D_{11}^0(x,x_1)\right|\leq Cw_0D_{11}^0(x,x_1).
\end{aligned}\label{2.30}\end{equation}
(iii) Controls for nonlinear terms:
\begin{equation}|e_f|+|e_g|\leq Cw_0^2(\Delta x)^{-2},\label{efeg}\end{equation}
\begin{equation}|K_{12}-\tilde{K}_{12}|+|K_{21}-\tilde{K}_{21}|\leq Cw_0(\Delta x)^{-2}.\label{K-K}\end{equation}
\end{lem}
\begin{proof}
(i) Recall that $K_f^0(x,x_1)=K_g^0(x,x_1)=G_0(\Delta x,2\sigma)^{-1}F_0(\Delta x,2\sigma)$ and
$$G_0(\Delta x,2\sigma)=(\Delta x)^2\left[(\Delta x)^2+(2\sigma)^2\right]^4,$$
$$F_0(\Delta x,2\sigma)=(2\sigma)^8+5(\Delta x)^2(2\sigma)^6+7(\Delta x)^4(2\sigma)^4+3(\Delta x)^6(2\sigma)^2\simeq(2\sigma)^2\left[(\Delta x)^2+(2\sigma)^2\right]^3.$$
Hence $D_{22}^0(x,x_1)\simeq\frac{1}{(\Delta x)^2}\frac{(2\sigma)^2}{(\Delta x)^2+(2\sigma)^2}.$
By virtue of $w_0\ll1$, the following inequalities hold:
\begin{equation}
\left|\left|(\Delta x)^2+(\Delta f)^2\right|-|\Delta x|^2\right|\leq\left|(\Delta f)^2\right|\leq w_0^2(\Delta x)^2,
\label{2.20}\end{equation}
\begin{equation}\begin{aligned}
\left|\left|(\Delta x)^2+(\Delta f+2\sigma+\theta(x_1))^2\right|-\left((\Delta x)^2+(2\sigma)^2\right)\right|\leq&\left|\Delta f+\theta(x_1)\right|\left|\Delta f+\theta(x_1)+4\sigma\right|\\
\leq& Cw_0\left((\Delta x)^2+(2\sigma)^2\right),
\end{aligned}\label{2.21}\end{equation}
\begin{equation}\begin{aligned}
\left|\left|(\Delta x)^2+(-\Delta f+2\sigma+\theta(x))^2\right|-\left((\Delta x)^2+(2\sigma)^2\right)\right|\leq&\left|\Delta f-\theta(x)\right|\left|-\Delta f+\theta(x)+4\sigma\right|\\
\leq& Cw_0\left((\Delta x)^2+(2\sigma)^2\right),
\end{aligned}\label{2.22}\end{equation}
\begin{equation}
(\Delta x)^2\left|\frac{1}{(\Delta x)^2+(\Delta f)^2}-\frac{1}{(\Delta x)^2}\right|\leq\left|\frac{(\Delta f)^2}{(\Delta x)^2+(\Delta f)^2}\right|\leq\frac{w_0^2}{1-w_0^2}\leq Cw_0^2,
\label{2.23}\end{equation}
\begin{equation}\begin{aligned}
&\left((\Delta x)^2+(2\sigma)^2\right)\left|\frac{1}{(\Delta x)^2+(2\sigma)^2}-\frac{1}{(\Delta x)^2+(\Delta f+2\sigma+\theta(x_1))^2}\right|\\
=&\left|\frac{(\Delta f+4\sigma+\theta(x_1))(\Delta f+\theta(x_1))}{(\Delta x)^2+(\Delta f+2\sigma+\theta(x_1))^2}\right|\\
\leq&\frac{Cw_0\left((\Delta x)^2+(2\sigma)^2\right)}{(1-Cw_0)\left((\Delta x)^2+(2\sigma)^2\right)}\\
\leq&Cw_0,
\end{aligned}\label{2.24}\end{equation}
and similarly
\begin{equation}\begin{aligned}
\left((\Delta x)^2+(2\sigma)^2\right)\left|\frac{1}{(\Delta x)^2+(2\sigma)^2}-\frac{1}{(\Delta x)^2+(-\Delta f+2\sigma+\theta(x))^2}\right|\leq Cw_0.
\end{aligned}\label{2.25}\end{equation}
Using Lemma \ref{lem2.2} with inequalities (\ref{2.20}-\ref{2.25}), we obtain
$$\begin{aligned}
&\left|G_0(\Delta x,2\sigma)^{-1}-G\left(\Delta x,2\sigma;\frac{\Delta f}{\Delta x},\frac{\theta(x)}{2\sigma},\frac{\theta(x_1)}{2\sigma}\right)^{-1}\right|\leq Cw_0G_0(\Delta x,2\sigma)^{-1}.
\end{aligned}$$
Recall that $$F(\Delta x,2\sigma;w_1,w_2,w_3)=\sum_{j=0}^3c_{2j}(w_1,w_2,w_3)(\Delta x)^{2j}(2\sigma)^{8-2j}$$ with $c_{2j}$ depending smoothly on $w_1,w_2,w_3$, and that all coefficients of $F_0$ are positive. Hence
$$\left|F\left(\Delta x,2\sigma;\frac{\Delta f}{\Delta x},\frac{\theta(x_1)}{2\sigma},\frac{\theta(x)}{2\sigma}\right)-F_0(\Delta x,2\sigma)\right|\leq Cw_0F_0(\Delta x,2\sigma).$$
By Lemma \ref{lem2.2} again, we obtain
\begin{equation}\begin{aligned}
\left|K_f(x,x_1)-K_f^0(x,x_1)\right|\leq Cw_0K_f^0(x,x_1).
\end{aligned}\label{2.26}\end{equation}
Since $K_f^0(x,x_1)$ is real-valued, it follows that
$$(1-Cw_0)K_f^0(x,x_1)\leq\Re K_f(x,x_1)\leq(1+Cw_0)K_f^0(x,x_1).$$
By repeating the above arguments, we also have
$$(1-Cw_0)K_g^0(x,x_1)\leq\Re K_g(x,x_1)\leq(1+Cw_0)K_g^0(x,x_1),$$
and thus
$\frac{1}{2}-Cw_0\leq\frac{\Re K_f}{\Re K_f+\Re K_g}\leq\frac{1}{2}+Cw_0.$\\
(ii) Using the inequalities
$$\frac{1}{8(\Delta x)^2}+\frac{(\Delta x)^2-(2\sigma)^2}{\left((\Delta x)^2+(2\sigma)^2\right)^2}=\frac{\left(3\Delta x-2\sigma\right)^2}{8(\Delta x)^2\left((\Delta x)^2+(2\sigma)^2\right)^2}\geq0,\quad\frac{(\Delta x)^2-(2\sigma)^2}{\left((\Delta x)^2+(2\sigma)^2\right)^2}\leq\frac{1}{(\Delta x)^2},$$
$D_{11}^0$ can be bounded from below and above:
$$\begin{aligned}
\frac{1}{(\Delta x)^2}\geq D_{11}^0(x,x_1)\geq\left(\mu_1^2+\mu_2^2-\frac{1}{4}\mu_1\mu_2\right)\frac{1}{(\Delta x)^2}=\left(\frac{7}{16}+\frac{9}{16}(\mu_1-\mu_2)^2\right)\frac{1}{(\Delta x)^2}.
\end{aligned}$$
Next, we compute
$$\begin{aligned}
&\left|\left((\Delta x)^2+(\Delta f)^2-(2\sigma+\theta(x_1))^2\right)-\left((\Delta x)^2-(2\sigma)^2\right)\right|\\
\leq&|\Delta f|^2+|4\sigma+\theta(x_1)||\theta(x_1)|\\
\leq&Cw_0\left((\Delta x)^2+(2\sigma)^2\right).
\end{aligned}$$
The similar bounds hold with $f$ replaced by $g$ or $\theta(x_1)$ replaced by $\theta(x)$.
Hence by Lemma \ref{lem2.2}
$$\begin{aligned}
\left|(K_{11}+K_{12})(x,x_1)-\frac{1}{(\Delta x)^2}-\frac{(\Delta x)^2-(2\sigma)^2}{\left((\Delta x)^2+(2\sigma)^2\right)^2}\right|\leq\frac{Cw_0}{(\Delta x)^2},
\end{aligned}$$
and consequently,
$$\begin{aligned}
\Re(K_{11}+K_{12})(x,x_1)\geq\frac{1}{(\Delta x)^2}+\frac{(\Delta x)^2-(2\sigma)^2}{\left((\Delta x)^2+(2\sigma)^2\right)^2}-\frac{Cw_0}{(\Delta x)^2}\geq\left(\frac{7}{8}-Cw_0\right)\frac{1}{(\Delta x)^2}.
\end{aligned}$$
The lower bounds of $\Re(K_{11}+K_{21})$, $\Re(K_{22}+\tilde{K}_{12})$, $\Re(K_{22}+\tilde{K}_{21})$ can be proved in the same way. Moreover, from (\ref{2.28}) it follows
$$\begin{aligned}
\left|D_{11}(x,x_1)-D_{11}^0(x,x_1)\right|\leq Cw_0\frac{1}{(\Delta x)^2} \leq Cw_0D_{11}^0(x,x_1).
\end{aligned}$$
(iii) Using Lemma \ref{lem2.1} and that $G\left(\Delta x,2\sigma;\frac{\Delta f}{\Delta x},\frac{\theta(x)}{2\sigma},\frac{\theta(x_1)}{2\sigma}\right)\simeq G_0(\Delta x,2\sigma)$, we see that
\begin{equation}\begin{aligned}
|e_f(x,x_1)|=&\left|\Delta\theta \Delta f\cdot G\left(\Delta x,2\sigma;\frac{\Delta f}{\Delta x},\frac{\theta(x)}{2\sigma},\frac{\theta(x_1)}{2\sigma}\right)^{-1}E\left(\Delta x,2\sigma;\frac{\Delta f}{\Delta x},\frac{\theta(x)}{2\sigma},\frac{\theta(x_1)}{2\sigma}\right)\right|\\
\leq&C\left|\frac{\Delta f}{\Delta x}\frac{\Delta\theta}{\Delta x}\right|(\Delta x)^2G_0(\Delta x,2\sigma)^{-1}((\Delta x)^2+(2\sigma)^2)^3\\
\leq&C\left|\frac{\Delta f}{\Delta x}\frac{\Delta\theta}{\Delta x}\right|\left((\Delta x)^2+(2\sigma)^2\right)^{-1}.
\end{aligned}\label{2.31}\end{equation}
In the same manner, we have
\begin{equation}|e_g|\leq C\left|\frac{\Delta g}{\Delta x}\frac{\Delta\theta}{\Delta x}\right|\left((\Delta x)^2+(2\sigma)^2\right)^{-1}.
\label{2.32}\end{equation}
In particular, it follows
$|e_f|+|e_g|\leq Cw_0^2(\Delta x)^{-2}.$
Recall that $\partial_{x_1}P_{12}=K_{12}+J_{12}=\tilde{K}_{12}+\tilde{J}_{12}$, so 
$$\begin{aligned}
\left|K_{12}-\tilde{K}_{12}\right|=&\left|\tilde{J}_{12}-J_{12}\right|\\
\leq&\left|\frac{2(\Delta g+2\sigma+\theta(x))(\partial_xg(x_1)\Delta x-\Delta g)}{\left((\Delta x)^2+(\Delta g+2\sigma+\theta(x))^2\right)^2}\right|+\left|\frac{2(\Delta f+2\sigma+\theta(x_1))(\partial_xg(x_1)\Delta x-\Delta f)}{\left((\Delta x)^2+(\Delta f+2\sigma+\theta(x_1))^2\right)^2}\right|\\
\leq& Cw_0(\Delta x)^{-2},
\end{aligned}$$
and in the same way we have $\left|K_{21}-\tilde{K}_{21}\right|\leq Cw_0(\Delta x)^{-2}.$
\end{proof}
\begin{proof}[Proof of Proposition \ref{prop2.3}]
With the help of (\ref{2.29}), (\ref{efeg}) and (\ref{K-K}), (\ref{det2})(\ref{det3}) yield that
$$\Re\left(\frac{1}{2}K_g+D_{11}-\mu_2^2K_f-\mu_1^2K_g\right)\geq\left(\frac{7}{16}-Cw_0\right)(\Delta x)^{-2},$$
$$\Re\left(\frac{1}{2}K_f+D_{11}-\mu_2^2K_f-\mu_1^2K_g\right)\geq\left(\frac{7}{16}-Cw_0\right)(\Delta x)^{-2}.$$
Then from (\ref{2.27})(\ref{2.28})(\ref{2.30}), it follows  from (\ref{det1}) that
\begin{equation}\begin{aligned}
\Re D_{11}\Re D_{22}-(\Re D_{12})^2
\geq&\left(\frac{1}{2}-Cw_0\right)(\Re K_f+\Re K_g)\cdot \left(\frac{7}{8}-Cw_0\right)(\Delta x)^{-2}\\
\geq&\left(\frac{7}{16}-Cw_0\right)\Re D_{22}\Re D_{11}.
\end{aligned}\label{2.33}\end{equation}
Finally, by (\ref{2.30})(\ref{2.26})(\ref{2.33}) we obtain
\allowdisplaybreaks[4]\begin{align*}
&\frac{1}{2}\Re D_{11}|\Delta\partial_x^kh|^2+\mu_1\mu_2\Re D_{12}\Re\left(\Delta\partial_x^kh\cdot\overline{\Delta\partial_x^k\theta}\right)+\frac{1}{2}\mu_1^2\mu_2^2\Re D_{22}|\Delta\partial_x^k\theta|^2\\
=&\frac{1}{2}\Re D_{11}\left|\Delta\partial_x^kh-\mu_1\mu_2\frac{\Re D_{12}}{\Re D_{11}}\Delta\partial_x^k\theta\right|^2+\frac{\mu_1^2\mu_2^2}{2}\left(\Re D_{22}-\frac{(\Re D_{12})^2}{\Re D_{11}}\right)|\Delta\partial_x^k\theta|^2\\
=&\frac{1}{2}\Re D_{11}\left(\left|\Delta\partial_x^kh-\mu_1\mu_2\frac{\Re D_{12}}{\Re D_{11}}\Delta\partial_x^k\theta\right|^2+\mu_1^2\mu_2^2\frac{\Re D_{11}\Re D_{22}-(\Re D_{12})^2}{2(\Re D_{11})^2}|\Delta\partial_x^k\theta|^2\right)\\
&+\frac{\mu_1^2\mu_2^2}{4}\left(\Re D_{22}-\frac{(\Re D_{12})^2}{\Re D_{11}}\right)|\Delta\partial_x^k\theta|^2\\
\geq&\frac{1}{2}\Re D_{11}\left(1+\frac{2(\Re D_{12})^2}{\Re D_{11}\Re D_{22}-(\Re D_{12})^2}\right)^{-1}|\Delta\partial_x^kh|^2+\frac{\mu_1^2\mu_2^2}{4}\left(\Re D_{22}-\frac{(\Re D_{12})^2}{\Re D_{11}}\right)|\Delta\partial_x^k\theta|^2\\
\geq&\frac{1}{2}\left(1+\frac{18}{7}+Cw_0\right)^{-1}\Re D_{11}|\Delta\partial_x^kh|^2+\mu_1^2\mu_2^2\left(\frac{7}{64}-Cw_0\right)\Re D_{22}|\Delta\partial_x^k\theta|^2\\
\geq&\left(\frac{7}{50}-Cw_0\right)D_{11}^0|\Delta\partial_x^kh|^2+\mu_1^2\mu_2^2\left(\frac{7}{64}-Cw_0\right)D_{22}^0|\Delta\partial_x^k\theta|^2.
\end{align*}\allowdisplaybreaks[0]
\end{proof}
\begin{proof}[Proof of Lemma \ref{summary3.2}]
To obtain the control of the integral of quadratic form (\ref{2.15}), it remains to bound the terms involving $e_f$ and $e_g$ on the right-hand side of (\ref{2.16}), for which we have
$$\begin{aligned}
&\left|\int_{\Gamma_{\pm}(t)}\int_{\Gamma_{\pm}(t)}\Re e_f(x,x_1)\cdot\Re\left(\Delta\partial_x^kh\cdot\overline{\Delta\partial_x^k\theta}\right)dxdx_1\right|\\
\lesssim&\int_{\Gamma_{\pm}(t)}\int_{\Gamma_{\pm}(t)}\left|\frac{\Delta f}{\Delta x}\frac{\Delta\theta}{\Delta x}\right|\left((\Delta x)^2+(2\sigma)^2\right)^{-1}|\Delta\partial_x^kh\|\Delta\partial_x^k\theta|dxdx_1\\
\lesssim&\sigma^{-1}\|\partial_x\theta\|_{L^\infty_{\gamma(t)}}\|\partial_xf\|_{L^\infty_{\gamma(t)}}\|\partial_x^k\theta\|_{L^2_{\gamma(t)}}\|\partial_x^kh\|_{L^2_{\gamma(t)}},
\end{aligned}$$
where we used $\int_{\mathbb{R}}(\alpha^2+(2\sigma)^2)^{-1}d\alpha=C\sigma^{-1}$ and
$|\Delta\partial_x^k\theta|\leq|\partial_x^k\theta(x)|+|\partial_x^k\theta(x_1)|$, $|\Delta\partial_x^kh|\leq|\partial_x^kh(x)|+|\partial_x^kh(x_1)|$.
As analogues to the above inequality, we also have
$$\begin{aligned}
\left|\int_{\Gamma_{\pm}(t)}\int_{\Gamma_{\pm}(t)}\Re e_g(x,x_1)\Re\left(\Delta\partial_x^kh\cdot\overline{\Delta\partial_x^k\theta}\right)dxdx_1\right|
\lesssim\frac{\|\partial_x\theta\|_{L^\infty_{\gamma(t)}}}{\sigma}\|\partial_xg\|_{L^\infty_{\gamma(t)}}\|\partial_x^k\theta\|_{L^2_{\gamma(t)}}\|\partial_x^kh\|_{L^2_{\gamma(t)}},
\end{aligned}$$
$$\begin{aligned}
\left|\int_{\Gamma_{\pm}(t)}\int_{\Gamma_{\pm}(t)}\Re \left(e_f+e_g\right)(x,x_1)|\Delta\partial_x^k\theta|^2 dxdx_1\right|
\lesssim\frac{\|\partial_x\theta\|_{L^\infty_{\gamma(t)}}}{\sigma}\left(\|\partial_xf\|_{L^\infty_{\gamma(t)}}+\|\partial_xg\|_{L^\infty_{\gamma(t)}}\right)\|\partial_x^k\theta\|_{L^2_{\gamma(t)}}^2.
\end{aligned}$$
Collecting the above bounds and using Proposition \ref{prop2.3} with a small enough $w_0$ yields that
\begin{equation}\begin{aligned}
&\int_{\Gamma_{\pm}(t)}\int_{\Gamma_{\pm}(t)}\left\{\Re D_{11}|\Delta\partial_x^kh|^2+\frac{\mu_1^2\mu_2^2}{2}\Re\left(2K_{11}-K_{12}-K_{21}+2K_{22}-\tilde{K}_{12}-\tilde{K}_{21}\right)|\Delta\partial_x^k\theta|^2\right.\\
&\left.+\mu_1\mu_2\Re\left(\mu_2\left(2K_{11}-K_{12}-K_{21}\right)-\mu_1\left(2K_{22}-\tilde{K}_{12}-\tilde{K}_{21}\right)\right)\Re\left(\Delta\partial_x^kh\cdot\overline{\Delta\partial_x^k\theta}\right)\right\}dxdx_1\\
\geq&C\int_{\Gamma_{\pm}(t)}\int_{\Gamma_{\pm}(t)}D_{11}^0(x,x_1)|\Delta\partial_x^kh|^2dxdx_1+C\mu_1^2\mu_2^2\int_{\Gamma_{\pm}(t)}\int_{\Gamma_{\pm}(t)}D_{22}^0(x,x_1)|\Delta\partial_x^k\theta|^2dxdx_1\\
&-C\sigma^{-1}\|\partial_x\theta\|_{L^\infty_{\gamma(t)}}\left(\|\partial_xf\|_{L^\infty_{\gamma(t)}}+\|\partial_xg\|_{L^\infty_{\gamma(t)}}\right)\|\partial_x^kh\|_{L^2_{\gamma(t)}}\|\partial_x^k\theta\|_{L^2_{\gamma(t)}}\\
&-C\sigma^{-1}\|\partial_x\theta\|_{L^\infty_{\gamma(t)}}\left(\|\partial_xf\|_{L^\infty_{\gamma(t)}}+\|\partial_xg\|_{L^\infty_{\gamma(t)}}\right)\|\partial_x^k\theta\|_{L^2_{\gamma(t)}}^2.
\end{aligned}\label{2.34}\end{equation}
\end{proof}
\subsection{Lower order terms}\label{lower}
In this part, we give the control over $T_{low}$, i.e. the terms involving $J_{ij}$ and $\tilde{J}_{ij}$, $(i,j)\in\left\{1,2\right\}^2$ in (\ref{2.4})(\ref{2.6}) and (\ref{2.9}-\ref{2.12}). These terms are considered lower order for that $J_{ij}$ and $\tilde{J}_{ij}$ have $\Delta f-\partial_xf(x_1)\Delta x$ or $\Delta g-\partial_xg(x_1)\Delta x$ as a factor, which contributes $(\Delta x)^2$ when $\Delta x$ is close to $0$:
$$
\left|f(x)-f(x_1)-\partial_xf(x_1)\Delta x\right|=\left|\int_{x_1}^x\partial_x^2f(y)(x-y)dy\right|\leq(\Delta x)^2M[\partial_x^2f](x_1).
$$
\begin{lem}\label{summary3.3}
Suppose that condition (\ref{2.17}) holds. Then 
$$\begin{aligned}
|T_{low}|\lesssim&\left[\left(\|\partial_xf\|_{L^\infty_{\gamma(t)}}+\|\partial_xg\|_{L^\infty_{\gamma(t)}}+\sigma^{-1}\|\theta\|_{L^\infty_{\gamma(t)}}\right)\left(\|\partial_x^2f\|_{L^2_{\gamma(t)}}+\|\partial_x^2g\|_{L^2_{\gamma(t)}}\right)\right.\\
&+\left.\sigma^{-1}\|\partial_x\theta\|_{L^\infty_{\gamma(t)}}\left(\|\partial_xf\|_{L^2_{\gamma(t)}}+\|\partial_xg\|_{L^2_{\gamma(t)}}\right)\right]\left(\|\partial_x^kh\|_{L^2_{\gamma(t)}}+\|\partial_x^k\theta\|_{L^2_{\gamma(t)}}\right)\\
&\cdot\left(\int_{\Gamma_{\pm(t)}}\int_{\Gamma_{\pm}(t)}D_{11}^0(x,x_1)|\Delta\partial_x^kh|^2dxdx_1+\int_{\Gamma_{\pm(t)}}\int_{\Gamma_{\pm}(t)}D_{22}^0(x,x_1)|\Delta\partial_x^kh|^2dxdx_1\right).
\end{aligned}$$
\end{lem}
In the remainder of this section, we assume that the solution satisfies (\ref{2.17}) so that
$$|J_{11}|\leq2\left|\frac{\Delta f}{\Delta x}\right|\left|\frac{\Delta f-\partial_x f(x_1)\Delta x}{(\Delta x)^2}\right|\left|\frac{(\Delta x)^3}{\left((\Delta x)^2+(\Delta f)^2\right)^2}\right|\lesssim\frac{1}{|\Delta x|}M[\partial_x^2 f](x_1)M[\partial_xf](x_1),$$
$$|J_{22}|\leq2\left|\frac{\Delta g}{\Delta x}\right|\left|\frac{\Delta g-\partial_x g(x_1)\Delta x}{(\Delta x)^2}\right|\left|\frac{(\Delta x)^3}{\left((\Delta x)^2+(\Delta g)^2\right)^2}\right|\lesssim\frac{1}{|\Delta x|}M[\partial_x^2 g](x_1)M[\partial_xg](x_1).$$
Here we denote $M[v]$ the Hardy-Littlewood maximal function of $v$ for an arbitrary function $v$. Then using H\"{o}lder's inequality we obtain
\begin{equation}\begin{aligned}
&\left|\int_{\Gamma_{\pm}(t)}\int_{\Gamma_{\pm}(t)}J_{11}(x,x_1)\Delta\partial_x^kh\cdot\overline{\partial_x^kh(x)}dxdx_1\right|\\
\lesssim&\int_{\Gamma_{\pm}(t)}\int_{\Gamma_{\pm}(t)}M[\partial_xf](x_1)M[\partial_x^2f](x_1)|\partial_x^kh(x)|\left|\frac{\Delta\partial_x^k h}{\Delta x}\right|dxdx_1\\
\lesssim&\left(\int_{\Gamma_{\pm}(t)}\int_{\Gamma_{\pm}(t)}\left|\frac{\Delta\partial_x^kh}{\Delta x}\right|^2dxdx_1\right)^\frac{1}{2}\left(\int_{\Gamma_{\pm}(t)}M[\partial_xf](x_1)^2M[\partial_x^2f](x_1)^2dx_1\right)^\frac{1}{2}\|\partial_x^kh\|_{L^2_{\gamma(t)}}\\
\lesssim&\|\partial_xf\|_{L^\infty_{\gamma(t)}}\|\partial_x^2f\|_{L^2_{\gamma(t)}}\|\partial_x^kh\|_{L^2_{\gamma(t)}}\|\partial_x^kh\|_{\dot{H}^{\frac{1}{2}}_{\gamma(t)}},
\end{aligned}\label{2.35}\end{equation}
\begin{equation}\begin{aligned}
&\left|\int_{\Gamma_{\pm}(t)}\int_{\Gamma_{\pm}(t)}J_{22}(x,x_1)\Delta\partial_x^kh\cdot\overline{\partial_x^kh(x)}dxdx_1\right|
\lesssim\|\partial_xg\|_{L^\infty_{\gamma(t)}}\|\partial_x^2g\|_{L^2_{\gamma(t)}}\|\partial_x^kh\|_{L^2_{\gamma(t)}}\|\partial_x^kh\|_{\dot{H}^{\frac{1}{2}}_{\gamma(t)}}.
\end{aligned}\label{2.36}\end{equation}
To control $J_{12}+J_{21}$, we have
$$\begin{aligned}
&\left|\frac{2(\Delta f+2\sigma+\theta(x_1))}{\left((\Delta x)^2+(\Delta f+2\sigma+\theta(x_1))^2\right)^2}+\frac{2(\Delta f-2\sigma-\theta(x))}{\left((\Delta x)^2+(\Delta f -2\sigma-\theta(x))^2\right)^2}\right|\\
\leq&|\Delta f+\Delta g|\left|\frac{1}{\left((\Delta x)^2+(\Delta f+2\sigma+\theta(x_1))^2\right)^2}+\frac{1}{\left((\Delta x)^2+(\Delta f-2\sigma-\theta(x))^2\right)^2}\right|\\
&+|4\sigma+\theta(x_1)+\theta(x)|\left|\frac{1}{\left((\Delta x)^2+(\Delta f+2\sigma+\theta(x_1))^2\right)^2}-\frac{1}{\left((\Delta x)^2+(\Delta f-2\sigma-\theta(x))^2\right)^2}\right|\\
\lesssim&\frac{|\Delta f+\Delta g|}{\left((\Delta x)^2+(2\sigma)^2\right)^2},
\end{aligned}$$
$$\begin{aligned}
\left|\frac{2\Delta x\partial_x\theta(x_1)(\Delta f+2\sigma+\theta(x_1))}{\left((\Delta x)^2+(\Delta f+2\sigma+\theta(x_1))^2\right)^2}\right|\lesssim\frac{|\partial_x\theta(x_1)|}{(\Delta x)^2+(2\sigma)^2}.
\end{aligned}$$
Hence
$$|J_{12}+J_{21}|\lesssim\frac{M[\partial_xf+\partial_xg](x_1)M[\partial_x^2f](x_1)}{|\Delta x|}+\frac{|\partial_x\theta(x_1)|}{(\Delta x)^2+(2\sigma)^2},$$ and as an analogue,
$$|\tilde{J}_{12}+\tilde{J}_{21}|\lesssim\frac{M[\partial_xf+\partial_xg](x_1)M[\partial_x^2g](x_1)}{|\Delta x|}+\frac{|\partial_x\theta(x_1)|}{(\Delta x)^2+(2\sigma)^2}.$$
Therefore, we get
\begin{equation}\begin{aligned}
&\left|\int_{\Gamma_{\pm}(t)}\int_{\Gamma_{\pm}(t)}(J_{12}+J_{21}+\tilde{J}_{12}+\tilde{J}_{21})(x,x_1)\cdot\Delta\partial_x^k h\cdot\overline{\partial_x^k h(x)}dxdx_1\right|\\
\lesssim&\int_{\Gamma_{\pm}(t)}\int_{\Gamma_{\pm}(t)}M[\partial_xf+\partial_xg](x_1)\left(M[\partial_x^2f](x_1)+M[\partial_x^2g](x_1)\right)\frac{|\Delta\partial_x^kh|}{|\Delta x|}|\partial_x^kh(x)|dxdx_1\\
&+\int_{\Gamma_{\pm}(t)}\int_{\Gamma_{\pm}(t)}\frac{|\partial_x\theta(x_1)|}{(\Delta x)^2+(2\sigma)^2}|\Delta\partial_x^kh||\partial_x^kh(x)|dxdx_1\\
\lesssim&\|\partial_xf+\partial_xg\|_{L^\infty_{\gamma(t)}}\left(\|\partial_x^2f\|_{L^2_{\gamma(t)}}+\|\partial_x^2g\|_{L^2_{\gamma(t)}}\right)\|\partial_x^kh\|_{L^2_{\gamma(t)}}\|\partial_x^kh\|_{\dot{H}^{\frac{1}{2}}_{\gamma(t)}}\\
&+\left(\int_{\mathbb{R}}\frac{\alpha^2}{(\alpha^2+(2\sigma)^2)^2}d\alpha\right)^\frac{1}{2}\|\partial_x\theta\|_{L^\infty_{\gamma(t)}}\|\partial_x^kh\|_{L^2_{\gamma(t)}}\|\partial_x^kh\|_{\dot{H}^{\frac{1}{2}}_{\gamma(t)}}\\
\lesssim&\left[\sigma^{-\frac{1}{2}}\|\partial_x\theta\|_{L^\infty_{\gamma(t)}}+\|\partial_xf+\partial_xg\|_{L^\infty_{\gamma(t)}}\left(\|\partial_x^2f\|_{L^2_{\gamma(t)}}+\|\partial_x^2g\|_{L^2_{\gamma(t)}}\right)\right]\|\partial_x^kh\|_{L^2_{\gamma(t)}}\|\partial_x^kh\|_{\dot{H}^{\frac{1}{2}}_{\gamma(t)}}.
\end{aligned}\label{2.37}\end{equation}
Since we do not have enough dissipation in $\theta$, it is necessary to carefully make use of the cancellation within $J_{11}-\frac{1}{2}J_{12}-\frac{1}{2}J_{21}$ and $J_{22}-\frac{1}{2}\tilde{J}_{12}-\frac{1}{2}\tilde{J}_{21}$ to control the terms involving $J_{ij}$ and $\tilde{J}_{ij}$ in (\ref{2.6}) and (\ref{2.9}-\ref{2.12}).
\begin{lem}
Suppose that (\ref{2.17}) holds, then for $(x,x_1)\in\Gamma_{\pm}(t)^2$
$$\begin{aligned}
&\left|J_{11}-\frac{1}{2}J_{12}-\frac{1}{2}J_{21}\right|\\
\lesssim&\frac{|\Delta f||\Delta f-\partial_xf(x_1)\Delta x|}{(\Delta x)^4}\cdot\frac{(2\sigma)^2}{(\Delta x)^2+(2\sigma)^2}+\frac{|\Delta\theta||\Delta f-\partial_xf(x_1)\Delta x|}{(\Delta x)^2((\Delta x)^2+(2\sigma)^2)}+\frac{|\partial_x\theta(x_1)||\Delta x|(|\Delta f|+\sigma)}{((\Delta x)^2+(2\sigma)^2)^2},
\end{aligned}$$
$$\begin{aligned}
&\left|J_{22}-\frac{1}{2}\tilde{J}_{12}-\frac{1}{2}\tilde{J}_{21}\right|\\
\lesssim&\frac{|\Delta g||\Delta g-\partial_xg(x_1)\Delta x|}{(\Delta x)^4}\cdot\frac{(2\sigma)^2}{(\Delta x)^2+(2\sigma)^2}+\frac{|\Delta\theta||\Delta g-\partial_xg(x_1)\Delta x|}{(\Delta x)^2((\Delta x)^2+(2\sigma)^2)}+\frac{|\partial_x\theta(x)||\Delta x|(|\Delta g|+\sigma)}{((\Delta x)^2+(2\sigma)^2)^2}.
\end{aligned}$$
\label{lem2.4}\end{lem}
\begin{proof}
Denote $$\begin{aligned}
&\tilde{G}\left(\Delta x,2\sigma;\frac{\Delta f}{\Delta x},\frac{\theta(x)}{2\sigma},\frac{\theta(x_1)}{2\sigma}\right)\\
:=&\left((\Delta x)^2+(\Delta f)^2\right)^{2}\left((\Delta x)^2+(\Delta f+2\sigma+\theta(x_1))^2\right)^{2}\left((\Delta x)^2+(-\Delta f+2\sigma+\theta(x))^2\right)^{2}
\end{aligned}$$
We compute by brute force:
$$\begin{aligned}
&\frac{2\Delta f}{\left((\Delta x)^2+(\Delta f)^2\right)^2}-\frac{\Delta f+2\sigma+\theta(x_1)}{\left((\Delta x)^2+(\Delta f+2\sigma+\theta(x_1))^2\right)^2}-\frac{\Delta f-2\sigma-\theta(x)}{\left((\Delta x)^2+(-\Delta f+2\sigma+\theta(x))^2\right)^2}\\
=&\tilde{G}\left(\Delta x,2\sigma;\frac{\Delta f}{\Delta x},\frac{\theta(x)}{2\sigma},\frac{\theta(x_1)}{2\sigma}\right)^{-1}\cdot\sum_{j=1}^8\tilde{c}_{j}\left(\frac{\Delta f}{\Delta x},\frac{\theta(x)}{2\sigma},\frac{\theta(x_1)}{2\sigma}\right)(\Delta x)^j(2\sigma)^{9-j},
\end{aligned}$$
where the coefficients $c_j$ are given by
$$\tilde{c}_1(w_1,w_2,w_3):=2w_1(1+w_2)^2(1+w_3)^2,$$
$$\tilde{c}_2(w_1,w_2,w_3):=8w_1^2(1+w_2)^3(1+w_3)^3(w_2-w_3),$$
$$\begin{aligned}
\tilde{c}_3(w_1,w_2,w_3):=&4w_1(1+w_2)^2(1+w_3)^2(1+3w_1^2)\left[(1+w_2)^2+(1+w_3)^2\right]\\
&-32w_1^3(1+w_2)^3(1+w_3)^3,
\end{aligned}$$
$$\begin{aligned}
\tilde{c}_4(w_1,w_2,w_3):=&-(1+w_1^2)(1-7w_1^2)(1+w_2)(1+w_3)\left[(1+w_2)^3-(1+w_3)^3\right]\\
&+6w_1^2(1+3w_1^2)(1+w_2)^2(1+w_3)^2(w_2-w_3),
\end{aligned}$$
$$\begin{aligned}
\tilde{c}_5(w_1,w_2,w_3):=&w_1(1+w_1^2)^2\left[(1+w_2)^4+(1+w_3)^4\right]+8w_1(1+3w_1^2)(1+w_2)^2(1+w_3)^2\\
&+4w_1(1+w_1^2)(1-7w_1^2)(1+w_2)(1+w_3)\left[(1+w_2)^2+(1+w_3)^2\right],
\end{aligned}$$
$$\begin{aligned}
\tilde{c}_6(w_1,w_2,w_3):=&4w_1^2(1+w_1^2)^2\left[(1+w_3)^3-(1+w_2)^3\right]\\
&-2(1+w_1^2)(1+3w_1^2)(1-7w_1^2)(1+w_2)(1+w_3)(w_2-w_3),
\end{aligned}$$
$$\begin{aligned}
\tilde{c}_7(w_1,w_2,w_3):=&2w_1(1+3w_1^2)(1+w_1^2)^2\left[(1+w_2)^2+(1+w_3)^2\right]\\
&+4w_1(1+w_1^2)^2(1-7w_1^2)(1+w_2)(1+w_3),
\end{aligned}$$
$$\begin{aligned}
\tilde{c}_8(w_1,w_2,w_3):=(w_2-w_3)(1+w_1^2)^3(1-3w_1^2).
\end{aligned}$$
By virtue of (\ref{2.17}), the following controls hold:
$$\begin{aligned}
&\sum_{j=1}^4\left|\tilde{c}_{2j-1}\left(\frac{\Delta f}{\Delta x},\frac{\theta(x)}{2\sigma},\frac{\theta(x_1)}{2\sigma}\right)(\Delta x)^{2j-1}(2\sigma)^{10-2j}\right|\lesssim|\Delta f|(2\sigma)^2\left((\Delta x)^2+(2\sigma)^2\right)^3,
\end{aligned}$$
$$\begin{aligned}
\sum_{j=1}^4\left|\tilde{c}_{2j}\left(\frac{\Delta f}{\Delta x},\frac{\theta(x)}{2\sigma},\frac{\theta(x_1)}{2\sigma}\right)(\Delta x)^{2j}(2\sigma)^{9-2j}\right|\lesssim|\Delta \theta|(\Delta x)^2\left((\Delta x)^2+(2\sigma)^2\right)^3,
\end{aligned}$$
$$\begin{aligned}
\left|\tilde{G}\left(\Delta x,2\sigma;\frac{\Delta f}{\Delta x},\frac{\theta(x)}{2\sigma},\frac{\theta(x_1)}{2\sigma}\right)\right|\simeq (\Delta x)^4\left((\Delta x)^2+(2\sigma)^2\right)^4.
\end{aligned}$$
Therefore,
$$\begin{aligned}
&\left|J_{11}-\frac{1}{2}J_{12}-\frac{1}{2}J_{21}\right| \\
\leq&\left|\Delta f-\partial_xf(x_1)\Delta x\right|\left|\tilde{G}\left(\Delta x,2\sigma;\frac{\Delta f}{\Delta x},\frac{\theta(x)}{2\sigma},\frac{\theta(x_1)}{2\sigma}\right)\right|^{-1}\left|\sum_{j=1}^8\tilde{c}_{j}\left(\frac{\Delta f}{\Delta x},\frac{\theta(x)}{2\sigma},\frac{\theta(x_1)}{2\sigma}\right)(\Delta x)^j(2\sigma)^{9-j}\right|\\
&+\left|\frac{2\Delta x\partial_x\theta(x_1)(\Delta f+2\sigma+\theta(x_1))}{\left((\Delta x)^2+(\Delta f+2\sigma+\theta(x_1))^2\right)^2}\right|\\
\lesssim&\frac{|\Delta f||\Delta f-\partial_xf(x_1)\Delta x|}{(\Delta x)^4}\cdot\frac{(2\sigma)^2}{(\Delta x)^2+(2\sigma)^2}+\frac{|\Delta\theta||\Delta f-\partial_xf(x_1)\Delta x|}{(\Delta x)^2((\Delta x)^2+(2\sigma)^2)}+\frac{|\partial_x\theta(x_1)||\Delta x|(|\Delta f|+\sigma)}{((\Delta x)^2+(2\sigma)^2)^2}.
\end{aligned}$$
Replacing $f$ by $g$ and swapping $\theta(x)$ with $\theta(x_1)$ yields the desired control of $\left|J_{22}-\tilde{J}_{12}-\tilde{J}_{21}\right|$.
\end{proof}
Using Lemma \ref{lem2.4} and H\"older inequality, we estimate the lower order terms in (\ref{2.6}):
$$\begin{aligned}
&\left|\int_{\Gamma_{\pm}(t)}\int_{\Gamma_{\pm}(t)}\left(J_{11}-\frac{1}{2}J_{12}-\frac{1}{2}J_{21}\right)(x,x_1)\cdot\Delta\partial_x^k\theta\cdot\overline{\partial_x^k\theta(x)}dxdx_1\right|\\
\lesssim&\int_{\Gamma_{\pm}(t)}\int_{\Gamma_{\pm}(t)}\left|\Delta f-\partial_xf(x_1)\Delta x\right|\frac{|\Delta f|}{(\Delta x)^4}\frac{(2\sigma)^2}{(\Delta x)^2+(2\sigma)^2}|\Delta\partial_x^k\theta||\partial_x^k\theta(x)|dxdx_1\\
&+\int_{\Gamma_{\pm}(t)}\int_{\Gamma_{\pm}(t)}\left|\Delta f-\partial_xf(x_1)\Delta x\right|\frac{|\Delta\theta|}{(\Delta x)^2\left((\Delta x)^2+(2\sigma)^2\right)}|\Delta\partial_x^k\theta||\partial_x^k\theta(x)|dxdx_1\\
&+\int_{\Gamma_{\pm}(t)}\int_{\Gamma_{\pm(t)}}\frac{|\partial_x\theta(x_1)||\Delta x|(|\Delta f|+\sigma)}{((\Delta x)^2+(2\sigma)^2)^2}|\Delta\partial_x^k\theta||\partial_x^k\theta(x)|dxdx_1\\
\lesssim&\left(\int_{\Gamma_{\pm}(t)}\int_{\Gamma_{\pm}(t)}\left|\frac{\Delta\partial_x^k\theta}{\Delta x}\right|^2\frac{(2\sigma)^2}{(\Delta x)^2+(2\sigma)^2}dxdx_1\right)^\frac{1}{2}\\
&\cdot\left[\left(\int_{\Gamma_{\pm}(t)}\int_{\Gamma_{\pm}(t)}\frac{|\Delta f-\partial_xf(x_1)\Delta x|^2|\Delta f|^2}{(\Delta x)^6}\frac{(2\sigma)^2}{(\Delta x)^2+(2\sigma)^2}|\partial_x^k\theta(x)|^2dxdx_1\right)^\frac{1}{2}\right.\\
&+\left.\left(\int_{\Gamma_{\pm}(t)}\int_{\Gamma_{\pm}(t)}\frac{|\Delta f-\partial_xf(x_1)\Delta x|^2|\Delta\theta|^2}{(2\sigma)^2(\Delta x)^2\left((\Delta x)^2+(2\sigma)^2\right)}|\partial_x^k\theta(x)|^2dxdx_1\right)^\frac{1}{2}\right.\\
&+\left.\left(\int_{\Gamma_{\pm}(t)}\int_{\Gamma_{\pm}(t)}\frac{|\partial_x\theta(x_1)|^2|\Delta x|^4(|\Delta f|+\sigma)^2}{\sigma^2((\Delta x)^2+(2\sigma)^2)^3}|\partial_x^k\theta(x)|^2dxdx_1\right)^\frac{1}{2}\right].
\end{aligned}$$
For each $x\in\Gamma_{\pm}(t)$, 
$$\begin{aligned}
\int_{\Gamma_{\pm}(t)}\frac{|\Delta f-\partial_xf(x_1)\Delta x|^2|\Delta f|^2}{(\Delta x)^6}\frac{(2\sigma)^2}{(\Delta x)^2+(2\sigma)^2}dx_1\leq&\int_{\Gamma_{\pm(t)}}M[\partial_x^2f](x_1)^2M[\partial_xf](x_1)^2dx_1\\
\leq&\|\partial_xf\|_{L^\infty_{\gamma(t)}}^2\|\partial_x^2f\|_{L^2_{\gamma(t)}}^2,
\end{aligned}$$
$$\begin{aligned}
\int_{\Gamma_{\pm}(t)}\frac{|\Delta f-\partial_xf(x_1)\Delta x|^2|\Delta \theta|^2}{(2\sigma)^2(\Delta x)^2\left((\Delta x)^2+(2\sigma)^2\right)}dx_1\leq\int_{\Gamma_{\pm(t)}}M[\partial_x^2f](x_1)^2\left|\frac{\Delta\theta}{2\sigma}\right|^2dx_1
\lesssim\sigma^{-2}\left\|\theta\right\|^2_{L^\infty_{\gamma(t)}}\|\partial_x^2f\|_{L^2_{\gamma(t)}}^2,
\end{aligned}$$
$$\begin{aligned}
\int_{\Gamma_{\pm}(t)}\frac{|\partial_x\theta(x_1)|^2|\Delta x|^4(|\Delta f|+\sigma)^2}{\sigma^2((\Delta x)^2+(2\sigma)^2)^3}dx_1\leq&\sigma^{-2}\|\partial_x\theta\|_{L^\infty_{\gamma(t)}}^2\int_{\Gamma_{\pm}(t)}\left(M[\partial_xf](x_1)^2+\frac{\sigma^2}{(\Delta x)^2+(2\sigma)^2}\right)dx_1\\
\lesssim&\sigma^{-2}\|\partial_x\theta\|_{L^\infty_{\gamma(t)}}^2\left(\|\partial_xf\|_{L^2_{\gamma(t)}}^2+\sigma\right).
\end{aligned}$$
Therefore,
\begin{equation}\begin{aligned}
&\left|\int_{\Gamma_{\pm}(t)}\int_{\Gamma_{\pm}(t)}\left(J_{11}-\frac{1}{2}J_{12}-\frac{1}{2}J_{21}\right)(x,x_1)\cdot\Delta\partial_x^k\theta\cdot\overline{\partial_x^k\theta(x)}dxdx_1\right|\\
\lesssim&\left[\left(\|\partial_xf\|_{L^\infty_{\gamma(t)}}+\sigma^{-1}\left\|\theta\right\|_{L^\infty_{\gamma(t)}}\right)\|\partial_x^2f\|_{L^2_{\gamma(t)}}+\sigma^{-1}\|\partial_x\theta\|_{L^\infty_{\gamma(t)}}\left(\|\partial_xf\|_{L^2_{\gamma(t)}}+\sigma^{\frac{1}{2}}\right)\right]\\
&\cdot\|\partial_x^k\theta\|_{L^2_{\gamma(t)}}\left(\int_{\Gamma_{\pm}(t)}\int_{\Gamma_{\pm}(t)}D_{22}^0|\Delta\partial_x^k\theta|^2dxdx_1\right)^\frac{1}{2}.
\end{aligned}\label{2.38}\end{equation}
By replacing $J_{11}-\frac{1}{2}J_{12}-\frac{1}{2}J_{21}$ by $J_{22}-\frac{1}{2}\tilde{J}_{12}-\frac{1}{2}\tilde{J}_{21}$ or replacing $\partial_x^k\theta$ by $\partial_x^kh$, we can also obtain
\begin{equation}\begin{aligned}
&\left|\int_{\Gamma_{\pm}(t)}\int_{\Gamma_{\pm}(t)}\left(J_{22}-\frac{1}{2}\tilde{J}_{12}-\frac{1}{2}\tilde{J}_{21}\right)(x,x_1)\cdot\Delta\partial_x^k\theta\cdot\overline{\partial_x^k\theta(x)}dxdx_1\right|\\
\lesssim&\left[\left(\|\partial_xg\|_{L^\infty_{\gamma(t)}}+\sigma^{-1}\left\|\theta\right\|_{L^\infty_{\gamma(t)}}\right)\|\partial_x^2g\|_{L^2_{\gamma(t)}}+\sigma^{-1}\|\partial_x\theta\|_{L^\infty_{\gamma(t)}}\left(\|\partial_xg\|_{L^2_{\gamma(t)}}+\sigma^{\frac{1}{2}}\right)\right]\\
&\cdot\|\partial_x^k\theta\|_{L^2_{\gamma(t)}}\left(\int_{\Gamma_{\pm}(t)}\int_{\Gamma_{\pm}(t)}D_{22}^0|\Delta\partial_x^k\theta|^2dxdx_1\right)^\frac{1}{2}.
\end{aligned}\label{2.39}\end{equation}
\begin{equation}\begin{aligned}
&\left|\int_{\Gamma_{\pm}(t)}\int_{\Gamma_{\pm}(t)}\left(J_{11}-\frac{1}{2}J_{12}-\frac{1}{2}J_{21}\right)(x,x_1)\cdot\Delta\partial_x^k\theta\cdot\overline{\partial_x^kh(x)}dxdx_1\right|\\
\lesssim&\left[\left(\|\partial_xf\|_{L^\infty_{\gamma(t)}}+\sigma^{-1}\left\|\theta\right\|_{L^\infty_{\gamma(t)}}\right)\|\partial_x^2f\|_{L^2_{\gamma(t)}}+\sigma^{-1}\|\partial_x\theta\|_{L^\infty_{\gamma(t)}}\left(\|\partial_xf\|_{L^2_{\gamma(t)}}+\sigma^{\frac{1}{2}}\right)\right]\\
&\cdot\|\partial_x^kh\|_{L^2_{\gamma(t)}}\left(\int_{\Gamma_{\pm}(t)}\int_{\Gamma_{\pm}(t)}D_{22}^0|\Delta\partial_x^k\theta|^2dxdx_1\right)^\frac{1}{2}.
\end{aligned}\label{2.40}\end{equation}
\begin{equation}\begin{aligned}
&\left|\int_{\Gamma_{\pm}(t)}\int_{\Gamma_{\pm}(t)}\left(J_{22}-\frac{1}{2}\tilde{J}_{12}-\frac{1}{2}\tilde{J}_{21}\right)(x,x_1)\cdot\Delta\partial_x^k\theta\cdot\overline{\partial_x^kh(x)}dxdx_1\right|\\
\lesssim&\left[\left(\|\partial_xg\|_{L^\infty_{\gamma(t)}}+\sigma^{-1}\left\|\theta\right\|_{L^\infty_{\gamma(t)}}\right)\|\partial_x^2g\|_{L^2_{\gamma(t)}}+\sigma^{-1}\|\partial_x\theta\|_{L^\infty_{\gamma(t)}}\left(\|\partial_xg\|_{L^2_{\gamma(t)}}+\sigma^{\frac{1}{2}}\right)\right]\\
&\cdot\|\partial_x^kh\|_{L^2_{\gamma(t)}}\left(\int_{\Gamma_{\pm}(t)}\int_{\Gamma_{\pm}(t)}D_{22}^0|\Delta\partial_x^k\theta|^2dxdx_1\right)^\frac{1}{2}.
\end{aligned}\label{2.41}\end{equation}
For the lower order terms in (\ref{2.11})(\ref{2.12}), we write\allowdisplaybreaks[4]
\begin{align*}
&\left|\int_{\Gamma_{\pm}(t)}\int_{\Gamma_{\pm}(t)}\left(J_{11}-\frac{1}{2}J_{12}-\frac{1}{2}J_{21}\right)(x,x_1)\cdot\Delta\partial_x^kh\cdot\overline{\partial_x^k\theta(x)}dxdx_1\right|\\
\lesssim&\int_{\Gamma_{\pm}(t)}\int_{\Gamma_{\pm}(t)}\left|\Delta f-\partial_xf(x_1)\Delta x\right|\frac{|\Delta f|}{(\Delta x)^4}\frac{(2\sigma)^2}{(\Delta x)^2+(2\sigma)^2}|\Delta\partial_x^kh||\partial_x^k\theta(x)|dxdx_1\\
&+\int_{\Gamma_{\pm}(t)}\int_{\Gamma_{\pm}(t)}\left|\Delta f-\partial_xf(x_1)\Delta x\right|\frac{|\Delta\theta|}{(\Delta x)^2\left((\Delta x)^2+(2\sigma)^2\right)}|\Delta\partial_x^kh||\partial_x^k\theta(x)|dxdx_1\\
&+\int_{\Gamma_{\pm}(t)}\int_{\Gamma_{\pm}(t)}\frac{|\partial_x\theta(x_1)||\Delta x|(|\Delta f|+\sigma)}{((\Delta x)^2+(2\sigma)^2)^2}|\Delta\partial_x^kh||\partial_x^k\theta(x)|dxdx_1\\
\lesssim&\left(\int_{\Gamma_{\pm}(t)}\int_{\Gamma_{\pm}(t)}\left|\frac{\Delta\partial_x^kh}{\Delta x}\right|dxdx_1\right)^\frac{1}{2}\\
&\cdot\left[\left(\int_{\Gamma_{\pm}(t)}\int_{\Gamma_{\pm}(t)}\frac{|\Delta f-\partial_xf(x_1)\Delta x|^2|\Delta f|^2}{(\Delta x)^6}\frac{(2\sigma)^4}{\left((\Delta x)^2+(2\sigma)^2\right)^2}|\partial_x^k\theta(x)|^2dxdx_1\right)^\frac{1}{2}\right.\\
&+\left.\left(\int_{\Gamma_{\pm}(t)}\int_{\Gamma_{\pm}(t)}\frac{|\Delta f-\partial_xf(x_1)\Delta x|^2|\Delta\theta|^2}{(\Delta x)^2\left((\Delta x)^2+(2\sigma)^2\right)^2}|\partial_x^k\theta(x)|^2dxdx_1\right)^\frac{1}{2}\right.\\
&+\left.\left(\int_{\Gamma_{\pm}(t)}\int_{\Gamma_{\pm}(t)}\frac{|\partial_x\theta(x_1)|^2|\Delta x|^4(|\Delta f|+\sigma)^2}{((\Delta x)^2+(2\sigma)^2)^4}|\partial_x^k\theta(x)|^2dxdx_1\right)^\frac{1}{2}\right].
\end{align*}
\allowdisplaybreaks[0]
For any $x\in\Gamma_{\pm}(t)$ we have the following controls:
$$\begin{aligned}
\int_{\Gamma_{\pm}(t)}\frac{|\Delta f-\partial_xf(x_1)\Delta x|^2|\Delta f|^2}{(\Delta x)^6}\frac{(2\sigma)^4}{\left((\Delta x)^2+(2\sigma)^2\right)^2}dx_1
\leq\|\partial_xf\|_{L^\infty_{\gamma(t)}}^2\|\partial_x^2f\|_{L^2_{\gamma(t)}}^2,
\end{aligned}$$
$$\begin{aligned}
\int_{\Gamma_{\pm}(t)}\frac{|\Delta f-\partial_xf(x_1)\Delta x|^2|\Delta \theta|^2}{(\Delta x)^2\left((\Delta x)^2+(2\sigma)^2\right)^2}dx_1\leq&\int_{\Gamma_{\pm(t)}}M[\partial_x^2f](x_1)^2M[\partial_x\theta](x_1)^2dx_1\\
\leq&\|\partial_x\theta\|^2_{L^\infty_{\gamma(t)}}\|\partial_x^2f\|_{L^2_{\gamma(t)}}^2,
\end{aligned}$$
$$\begin{aligned}
\int_{\Gamma_{\pm}(t)}\frac{|\partial_x\theta(x_1)|^2|\Delta x|^4(|\Delta f|+\sigma)^2}{((\Delta x)^2+(2\sigma)^2)^4}dx_1\lesssim\int_{\Gamma_{\pm}(t)}\frac{|\partial_x\theta(x_1)|^2}{(\Delta x)^2+(2\sigma)^2}dx_1\lesssim\sigma^{-1}\|\partial_x\theta\|_{L^\infty_{\gamma(t)}}^2.
\end{aligned}$$
Therefore,
\begin{equation}\begin{aligned}
&\left|\int_{\Gamma_{\pm}(t)}\int_{\Gamma_{\pm}(t)}\left(J_{11}-\frac{1}{2}J_{12}-\frac{1}{2}J_{21}\right)(x,x_1)\cdot\Delta\partial_x^kh\cdot\overline{\partial_x^k\theta(x)}dxdx_1\right|\\
\lesssim&\left[\left(\|\partial_xf\|_{L^\infty_{\gamma(t)}}+\|\partial_x\theta\|_{L^\infty_{\gamma(t)}}\right)\|\partial_x^2f\|_{L^2_{\gamma(t)}}+\sigma^{-\frac{1}{2}}\|\partial_x\theta\|_{L^\infty_{\gamma(t)}}\right]\|\partial_x^k\theta\|_{L^2_{\gamma(t)}}\|\partial_x^kh\|_{\dot{H}^{\frac{1}{2}}_{\gamma(t)}},
\end{aligned}\label{2.42}\end{equation}
and likewise
\begin{equation}\begin{aligned}
&\left|\int_{\Gamma_{\pm}(t)}\int_{\Gamma_{\pm}(t)}\left(J_{22}-\frac{1}{2}\tilde{J}_{12}-\frac{1}{2}\tilde{J}_{21}\right)(x,x_1)\cdot\Delta\partial_x^kh\cdot\overline{\partial_x^k\theta(x)}dxdx_1\right|\\
\lesssim&\left[\left(\|\partial_xg\|_{L^\infty_{\gamma(t)}}+\|\partial_x\theta\|_{L^\infty_{\gamma(t)}}\right)\|\partial_x^2g\|_{L^2_{\gamma(t)}}+\sigma^{-\frac{1}{2}}\|\partial_x\theta\|_{L^\infty_{\gamma(t)}}\right]\|\partial_x^k\theta\|_{L^2_{\gamma(t)}}\|\partial_x^kh\|_{\dot{H}^{\frac{1}{2}}_{\gamma(t)}}.
\end{aligned}\label{2.43}\end{equation}
Summing up (\ref{2.35}-\ref{2.43}), we arrive at Lemma \ref{summary3.3}.
\subsection{Anti-symmetric part}\label{Asymmetry}
In this part, we give the control of the last terms on the right-hand side of (\ref{2.9}-\ref{2.12}) involving $P_{12}-P_{21}$. The main result of this section is the following.
\begin{lem}\label{summary3.4}
Suppose that condition (\ref{2.17}) holds. The anti-symmetric part $T_{asym}$ satisfies
$$\begin{aligned}
|T_{asym}|\lesssim
&\left|\int_{\Gamma_{\pm}(t)}\int_{\Gamma_{\pm}(t)}(P_{12}-P_{21})(x,x_1)\partial_x^{k+1}\theta(x_1)\overline{\partial_x^kh(x)}dxdx_1\right|\\
\lesssim&\|\partial_xf+\partial_xg\|_{L^\infty_{\gamma(t)}}\|\partial_x^k\theta\|_{H^\frac{1}{2}_{\gamma(t)}}\|\partial_x^kh\|_{H^\frac{1}{2}_{\gamma(t)}}\\
&+\left(\|\partial_x^2f+\partial_x^2g\|_{L^\infty_{\gamma(t)}}+\|\partial_xf+\partial_xg\|_{L^\infty_{\gamma(t)}}\left\|\frac{\partial_x\theta}{2\sigma}\right\|_{L^\infty_{\gamma(t)}}\right)\|\partial_x^k\theta\|_{\dot{H}^{\frac{1}{2}}_{\gamma(t)}}\|\partial_x^kh\|_{L^2_{\gamma(t)}}.
\end{aligned}$$
\end{lem}
\begin{proof}
To begin with, we recall the definition of $P_{12}$ and $P_{21}$ to get
$$
P_{12}(x,x_1)-P_{21}(x,x_1)=-\frac{\Delta x(4\sigma+\theta(x)+\theta(x_1))(\Delta f+\Delta g)}{\left((\Delta x)^2+(\Delta f+2\sigma+\theta(x_1))^2\right)\left((\Delta x)^2+(-\Delta f+2\sigma+\theta(x))^2\right)},
$$
and under the condition (\ref{2.17})
\allowdisplaybreaks[4]\begin{align*}
&\partial_{x_1}\left(P_{12}(x,x_1)-P_{21}(x,x_1)\right)\\
=&\frac{\Delta x(\Delta f+\Delta g)\partial_x\theta(x_1)-(4\sigma+\theta(x)+\theta(x_1))\left(\Delta x(\partial_xf(x_1)+\partial_xg(x_1))+\Delta f+\Delta g\right)}{\left((\Delta x)^2+(\Delta f+2\sigma+\theta(x_1))^2\right)\left((\Delta x)^2+(-\Delta f+2\sigma+\theta(x))^2\right)}\\
&+\frac{\Delta x(4\sigma+\theta(x)+\theta(x_1))(\Delta f+\Delta g)}{\left((\Delta x)^2+(\Delta f+2\sigma+\theta(x_1))^2\right)\left((\Delta x)^2+(-\Delta f+2\sigma+\theta(x))^2\right)}\\
&\cdot\left[\frac{2\Delta x+2\partial_xg(x_1)(\Delta f+2\sigma+\theta(x_1))}{\left((\Delta x)^2+(\Delta f+2\sigma+\theta(x_1))^2\right)}+\frac{2\Delta x+2\partial_xf(x_1)(\Delta f-2\sigma-\theta(x))}{\left((\Delta x)^2+(-\Delta f+2\sigma+\theta(x))^2\right)}\right]\\
\lesssim&\|\partial_xf+\partial_xg\|_{L^\infty_{\gamma(t)}}\left[\|\partial_x\theta\|_{L^\infty_{\gamma(t)}}\frac{(\Delta x)^2}{\left((\Delta x)^2+(2\sigma)^2\right)^2}+\frac{2\sigma|\Delta x|}{\left((\Delta x)^2+(2\sigma)^2\right)^2}+\frac{2\sigma(\Delta x)^2(|\Delta x|+|2\sigma|)}{\left((\Delta x)^2+(2\sigma)^2\right)^3}\right]\\
\lesssim&\|\partial_xf+\partial_xg\|_{L^\infty_{\gamma(t)}}\left[\|\partial_x\theta\|_{L^\infty_{\gamma(t)}}\frac{(\Delta x)^2}{\left((\Delta x)^2+(2\sigma)^2\right)^2}+\frac{2\sigma|\Delta x|}{\left((\Delta x)^2+(2\sigma)^2\right)^2}\right].
\end{align*}\allowdisplaybreaks[0]
Note that $\|\partial_x\theta\|_{L^\infty_{\gamma(t)}}\leq\|\partial_xf\|_{L^\infty_{\gamma(t)}}+\|\partial_xg\|_{L^\infty_{\gamma(t)}}\leq w_0$. By Young's inequality, it follows that
\begin{equation}\begin{aligned}
&\left|\int_{\Gamma_{\pm}(t)}\int_{\Gamma_{\pm}(t)}(P_{12}-P_{21})(x,x_1)\partial_x^{k+1}\theta(x_1)\overline{\partial_x^kh(x)}dxdx_1\right|\\
=&\left|\int_{\Gamma_{\pm}(t)}\int_{\Gamma_{\pm}(t)}\partial_{x_1}(P_{12}-P_{21})(x,x_1)\partial_x^{k}\theta(x_1)\overline{\partial_x^kh(x)}dxdx_1\right|\\
\lesssim&\|\partial_xf+\partial_xg\|_{L^\infty_{\gamma(t)}}\int_{\Gamma_{\pm}(t)}\left[\int_{\Gamma_{\pm}(t)}\left(\frac{2\sigma|\Delta x|+\|\partial_x\theta\|_{L^\infty_{\gamma(t)}}|\Delta x|^2}{\left((\Delta x)^2+(2\sigma)^2\right)^2}\right)|\partial_x^k\theta(x_1)|dx_1\right]|\partial_x^kh(x)|dx\\
\lesssim&\|\partial_xf+\partial_xg\|_{L^\infty_{\gamma(t)}}(2\sigma)^{-1}\|\partial_x^k\theta\|_{L^2_{\gamma(t)}}\|\partial_x^kh\|_{L^2_{\gamma(t)}}
\end{aligned}\label{2.44}\end{equation}
where we used that $\left\|\frac{\alpha^2}{(\alpha^2+(2\sigma)^2)^2}\right\|_{L^1}\simeq\left\|\frac{2\sigma\alpha}{(\alpha^2+(2\sigma)^2)^2}\right\|_{L^1}\simeq(2\sigma)^{-1}.$
Likewise, for the terms involving $P_{12}-P_{21}$ in (\ref{2.11})(\ref{2.12}) we have
\begin{equation}\begin{aligned}
&\left|\int_{\Gamma_{\pm}(t)}\int_{\Gamma_{\pm}(t)}(P_{12}-P_{21})(x,x_1)\partial_x^{k+1}h(x_1)\overline{\partial_x^k\theta(x)}dxdx_1\right|\\
\lesssim&\|\partial_xf+\partial_xg\|_{L^\infty_{\gamma(t)}}(2\sigma)^{-1}\|\partial_x^k\theta\|_{L^2_{\gamma(t)}}\|\partial_x^kh\|_{L^2_{\gamma(t)}}.
\end{aligned}\label{2.45}\end{equation}
However, the bounds (\ref{2.44})(\ref{2.45}) are not uniform in $\sigma$ for $\sigma$ close to $0$. In order to obtain a uniform estimate, we write $\partial_x^{k+1}\theta=-\Lambda H\partial_x^k\theta$, where $H$ is the Hilbert transform, and $\Lambda:=H\partial_x.$ Note that 
$\Lambda v(x)=\frac{1}{\pi}P.V.\int_{\mathbb{R}}\frac{v(x)-v(x-\alpha)}{\alpha^2}d\alpha$ for any Schwartz class function $v$.
Using this identity, we get
$$\begin{aligned}
&\int_{\Gamma_{\pm}(t)}\int_{\Gamma_{\pm}(t)}(P_{12}-P_{21})(x,x_1)\partial_x^{k+1}\theta(x_1)\overline{\partial_x^kh(x)}dxdx_1\\
=&-\int_{\Gamma_{\pm}(t)}\int_{\Gamma_{\pm}(t)}(P_{12}-P_{21})(x,x_1)\Lambda H\partial_x^{k}\theta(x_1)\overline{\partial_x^kh(x)}dxdx_1\\
=&\frac{1}{2\pi}P.V.\int_{\Gamma_{\pm}(t)}\int_{\Gamma_{\pm}(t)}\int_{\mathbb{R}}(P_{12}-P_{21})(x,x_1)\frac{H\partial_x^k\theta(x_1-\alpha)-H\partial_x^k\theta(x_1)}{\alpha^2}\overline{\partial_x^kh(x)}d\alpha dxdx_1.
\end{aligned}$$
Then apply a change of variables by setting $\alpha=-\tilde{\alpha}$, $x_1=\tilde{x}_1-\tilde{\alpha}$, $x=\tilde{x}-\tilde{\alpha}$:
$$\begin{aligned}
&P.V.\int_{\Gamma_{\pm}(t)}\int_{\Gamma_{\pm}(t)}\int_{\mathbb{R}}(P_{12}-P_{21})(x,x_1)\frac{H\partial_x^k\theta(x_1-\alpha)-H\partial_x^k\theta(x_1)}{\alpha^2}\overline{\partial_x^kh(x)}d\alpha dxdx_1\\
=&P.V.\int_{\Gamma_{\pm}(t)}\int_{\Gamma_{\pm}(t)}\int_{\mathbb{R}}(P_{12}-P_{21})(\tilde{x}-\tilde{\alpha},\tilde{x}_1-\tilde{\alpha})\frac{H\partial_x^k\theta(\tilde{x}_1)-H\partial_x^k\theta(\tilde{x}_1-\tilde{\alpha})}{\tilde{\alpha}^2}\overline{\partial_x^kh(\tilde{x}-\tilde{\alpha})}d\tilde{\alpha}d\tilde{x}d\tilde{x_1}.
\end{aligned}$$
Hence
\begin{equation}\begin{aligned}
&\int_{\Gamma_{\pm}(t)}\int_{\Gamma_{\pm}(t)}(P_{12}-P_{21})(x,x_1)\partial_x^{k+1}\theta(x_1)\overline{\partial_x^kh(x)}dxdx_1\\
=&-\frac{1}{4\pi}P.V.\int_{\Gamma_{\pm}(t)}\int_{\Gamma_{\pm}(t)}\int_{\mathbb{R}}(P_{12}-P_{21})(x-\alpha,x_1-\alpha)\\
&\cdot\frac{(H\partial_x^k\theta(x_1)-H\partial_x^k\theta(x_1-\alpha))\overline{(\partial_x^kh(x)-\partial_x^kh(x-\alpha))}}{\alpha^2}d\alpha dxdx_1\\
&-\frac{1}{4\pi}P.V.\int_{\Gamma_{\pm}(t)}\int_{\Gamma_{\pm}(t)}\int_{\mathbb{R}}\frac{(P_{12}-P_{21})(x,x_1)-(P_{12}-P_{21})(x-\alpha,x_1-\alpha)}{\alpha}\\
&\cdot\frac{H\partial_x^k\theta(x_1)-H\partial_x^k\theta(x_1-\alpha)}{\alpha}\overline{\partial_x^kh(x)}d\alpha dxdx_1.
\end{aligned}\label{2.46}\end{equation}
For the first term on the right-hand side, we see from (\ref{2.17}) that
$$|P_{12}(x-\alpha,x_1-\alpha)-P_{21}(x-\alpha,x_1-\alpha)|\lesssim\|\partial_xf+\partial_xg\|_{L^\infty_{\gamma(t)}}\frac{2\sigma(\Delta x)^2}{\left((\Delta x)^2+(2\sigma)^2\right)^2}.$$
Hence by Young's inequality and the $L^2$ boundedness of the Hilbert transform, it follows
$$\begin{aligned}
&\left\|\int_{\Gamma_{\pm}(t)}(P_{12}-P_{21})(x-\alpha,x_1-\alpha)\frac{H\partial_x^k\theta(x_1)-H\partial_x^k\theta(x_1-\alpha)}{\alpha}dx_1\right\|_{L^2_{x,\alpha}(\Gamma_{\pm}(t)\times\mathbb{R})}\\
\lesssim&\|\partial_xf+\partial_xg\|_{L^\infty_{\gamma(t)}}\left\|\int_{\mathbb{R}}\frac{2\sigma(\Delta x)^2}{\left((\Delta x)^2+(2\sigma)^2\right)^2}\left|\frac{H\partial_x^k\theta(x_1)-H\partial_x^k\theta(x_1-\alpha)}{\alpha}\right|dx_1\right\|_{L^2_{x,\alpha}(\Gamma_{\pm}(t)\times\mathbb{R})}\\
\lesssim&\|\partial_xf+\partial_xg\|_{L^\infty_{\gamma(t)}}\left\|\frac{H\partial_x^k\theta(x)-H\partial_x^k\theta(x-\alpha)}{\alpha}\right\|_{L^2_{x,\alpha}(\Gamma_{\pm}(t)\times\mathbb{R})}\\
\lesssim&\|\partial_xf+\partial_xg\|_{L^\infty_{\gamma(t)}}\|\partial_x^k\theta\|_{\dot{H}^{\frac{1}{2}}_{\gamma(t)}}.
\end{aligned}$$
Then by H\"older's inequality we obtain
\begin{equation}\begin{aligned}
&\left|\int_{\Gamma_{\pm}(t)}\int_{\Gamma_{\pm}(t)}\int_{\mathbb{R}}(P_{12}-P_{21})(x-\alpha,x_1-\alpha)\right.\\
&\cdot\left.\frac{(H\partial_x^k\theta(x_1)-H\partial_x^k\theta(x_1-\alpha))\overline{(\partial_x^kh(x)-\partial_x^kh(x-\alpha))}}{\alpha^2}d\alpha dxdx_1\right|\\
\lesssim&\|\partial_xf+\partial_xg\|_{L^\infty_{\gamma(t)}}\|\partial_x^k\theta\|_{\dot{H}^{\frac{1}{2}}_{\gamma(t)}}\|\partial_x^kh\|_{{\dot{H}^{\frac{1}{2}}_{\gamma(t)}}}.
\end{aligned}\label{2.47}\end{equation}
To control $(P_{12}-P_{21})(x,x_1)-(P_{12}-P_{21})(x-\alpha,x_1-\alpha)$, we have the following estimates:
$$|(4\sigma+\theta(x)+\theta(x_1))-(4\sigma+\theta(x-\alpha)+\theta(x_1-\alpha))|\leq2\|\partial_x\theta\|_{L^\infty_{\gamma(t)}}|\alpha|,$$
$$\begin{aligned}
&\left|\left(f(x)-f(x_1)+g(x)-g(x_1)\right)-\left(f(x-\alpha)-f(x_1-\alpha)+g(x-\alpha)-g(x_1-\alpha\right)\right|\\
=&\left|\int_{x_1}^x\int_0^\alpha\left(\partial_x^2f+\partial_x^2g\right)(y-\beta)d\beta dy\right|\\
\leq&\|\partial_x^2f+\partial_x^2g\|_{L^\infty_{\gamma(t)}}|x-x_1||\alpha|,
\end{aligned}$$
$$\begin{aligned}
&\left|\alpha^2+\left(f(x)-f(x_1)+2\sigma+\theta(x_1)\right)^2-\alpha^2-\left(f(x-\alpha)-f(x_1-\alpha)+2\sigma+\theta(x_1-\alpha)\right)^2\right|\\
=&\left|f(x)+f(x-\alpha)-f(x_1)-f(x_1-\alpha)+4\sigma+\theta(x_1)+\theta(x_1-\alpha)\right|\\
&\cdot\left|f(x)-f(x-\alpha)-f(x_1)+f(x_1-\alpha)+\theta(x_1)-\theta(x_1-\alpha)\right|\\
=&\frac{1}{2}\left|f(x)+f(x-\alpha)-f(x_1)-f(x_1-\alpha)+4\sigma+\theta(x_1)+\theta(x_1-\alpha)\right|\\
&\cdot\left|\int_{x_1}^x\int_0^\alpha\left(\partial_x^2f+\partial_x^2g\right)(y-\beta)d\beta dy+ \theta(x)+\theta(x_1)-\theta(x-\alpha)-\theta(x_1-\alpha)\right|\\
\leq&\left(\|\partial_xf\|_{L^\infty_{\gamma(t)}}|x-x_1|+4\sigma+2\|\theta\|_{L^\infty_{\gamma(t)}}\right)\left(\|\partial_x^2f+\partial_x^2g\|_{L^\infty_{\gamma(t)}}|x-x_1|+\|\partial_x\theta\|_{L^\infty_{\gamma(t)}}\right)|\alpha|,
\end{aligned}$$
and likewise
$$\begin{aligned}
&\left|\alpha^2+\left(f(x_1)-f(x)+2\sigma+\theta(x)\right)^2-\alpha^2-\left(f(x_1-\alpha)-f(x-\alpha)+2\sigma+\theta(x-\alpha)\right)^2\right|\\
\leq&\left(\|\partial_xf\|_{L^\infty_{\gamma(t)}}|x-x_1|+4\sigma+2\|\theta\|_{L^\infty_{\gamma(t)}}\right)\left(\|\partial_x^2f+\partial_x^2g\|_{L^\infty_{\gamma(t)}}|x-x_1|+\|\partial_x\theta\|_{L^\infty_{\gamma(t)}}\right)|\alpha|.
\end{aligned}$$
Using Lemma \ref{lem2.2} and (\ref{2.17}), we obtain
\begin{equation}\begin{aligned}
&\left|(P_{12}-P_{21})(x,x_1)-(P_{12}-P_{21})(x-\alpha,x_1-\alpha)\right|\\
\lesssim&\left(\|\partial_x^2f+\partial_x^2g\|_{L^\infty_{\gamma(t)}}+\|\partial_xf+\partial_xg\|_{L^\infty_{\gamma(t)}}\left\|\frac{\partial_x\theta}{2\sigma}\right\|_{L^\infty_{\gamma(t)}}\right)\frac{(2\sigma)(\Delta x)^2|\alpha|}{\left((\Delta x)^2+(2\sigma)^2\right)^2}.
\end{aligned}\label{2.48}\end{equation}
To control the second term on the right-hand side of (\ref{2.46}) we decompose the integral into the region $|\alpha|\leq1$ and $|\alpha|>1$. In the region $|\alpha|\leq1$, we use
\begin{equation}\begin{aligned}
&\int_{\Gamma_{\pm}(t)}\int_{\Gamma_{\pm}(t)}\int_{|\alpha|\leq1}\frac{(2\sigma)(\Delta x)^2}{\left((\Delta x)^2+(2\sigma)^2\right)^2}\left|\frac{H\partial_x^k\theta(x_1)-H\partial_x^k\theta(x_1-\alpha)}{\alpha}\right||\partial_x^kh(x)|dxdx_1d\alpha\\
\leq&\left(\int_{\Gamma_{\pm}(t)}\int_{|\alpha|\leq1}\left|\frac{H\partial_x^k\theta(x_1)-H\partial_x^k\theta(x_1-\alpha)}{\alpha}\right|^2dx_1d\alpha\right)^\frac{1}{2}\\
&\cdot\left(\int_{\Gamma_{\pm}(t)}\int_{|\alpha|\leq1}\left(\int_{\Gamma_{\pm}(t)}\frac{(2\sigma)(\Delta x)^2}{\left((\Delta x)^2+(2\sigma)^2\right)^2}|\partial_x^kh(x)|dx\right)^2dx_1d\alpha\right)^\frac{1}{2}\\
\lesssim&\|\partial_x^k\theta\|_{\dot{H}^{\frac{1}{2}}_{\gamma(t)}}\|\partial_x^kh\|_{L^2_{\gamma(t)}}.
\end{aligned}\label{2.49}\end{equation}
In the region $|\alpha|>1$, we use
\begin{equation}\begin{aligned}
&\left|\frac{(P_{12}-P_{21})(x,x_1)-(P_{12}-P_{21})(x-\alpha,x_1-\alpha)}{\alpha}\frac{H\partial_x^k\theta(x_1)-H\partial_x^k\theta(x_1-\alpha)}{\alpha}\right|\\
\lesssim&\|\partial_xf+\partial_xg\|_{L^\infty_{\gamma(t)}}\frac{(2\sigma)(\Delta x)^2}{\alpha^2\left((\Delta x)^2+(2\sigma)^2\right)^2}\left(|H\partial_x^k\theta(x_1)|+|H\partial_x^k\theta(x_1-\alpha)|\right).
\end{aligned}\label{2.50}\end{equation}
Then by H\"older and Young's inequality
\begin{equation}\begin{aligned}
&\int_{\Gamma_{\pm}(t)}\int_{\Gamma_{\pm}(t)}\int_{|\alpha|>1}\frac{(2\sigma)(\Delta x)^2}{\alpha^2\left((\Delta x)^2+(2\sigma)^2\right)^2}|H\partial_x^k\theta(x_1)||\partial_x^kh(x)|dxdx_1d\alpha\\
\leq&\int_{\Gamma_{\pm}(t)}\left(\int_{\Gamma_{\pm}(t)}\frac{(2\sigma)(\Delta x)^2}{\left((\Delta x)^2+(2\sigma)^2\right)^2}|\partial_x^kh(x)|dx\right)|H\partial_x^k\theta(x_1)|dx_1\\
\lesssim&\|\partial_x^k\theta\|_{L^2_{\gamma(t)}}\|\partial_x^kh\|_{L^2_{\gamma(t)}},
\end{aligned}\label{2.51}\end{equation}
\begin{equation}\begin{aligned}
&\int_{\Gamma_{\pm}(t)}\int_{\Gamma_{\pm}(t)}\int_{|\alpha|>1}\frac{(2\sigma)(\Delta x)^2}{\alpha^2\left((\Delta x)^2+(2\sigma)^2\right)^2}|H\partial_x^k\theta(x_1-\alpha)||\partial_x^kh(x)|dxdx_1d\alpha\\
\leq&\int_{\Gamma_{\pm}(t)}\left(\int_{\Gamma_{\pm}(t)}\frac{(2\sigma)(\Delta x)^2}{\left((\Delta x)^2+(2\sigma)^2\right)^2}|\partial_x^kh(x)|dx\right)\left(\int_{|\alpha|>1}\frac{1}{\alpha^2}|H\partial_x^k\theta(x_1-\alpha)|d\alpha\right)dx_1\\
\lesssim&\|\partial_x^k\theta\|_{L^2_{\gamma(t)}}\|\partial_x^kh\|_{L^2_{\gamma(t)}}.
\end{aligned}\label{2.52}\end{equation}
Collecting the controls in (\ref{2.47}-\ref{2.52}), we conclude that
\begin{equation}\begin{aligned}
&\left|\int_{\Gamma_{\pm}(t)}\int_{\Gamma_{\pm}(t)}(P_{12}-P_{21})(x,x_1)\partial_x^{k+1}\theta(x_1)\overline{\partial_x^kh(x)}dxdx_1\right|\\
\lesssim&\|\partial_xf+\partial_xg\|_{L^\infty_{\gamma(t)}}\|\partial_x^k\theta\|_{H^\frac{1}{2}_{\gamma(t)}}\|\partial_x^kh\|_{H^\frac{1}{2}_{\gamma(t)}}\\
&+\left(\|\partial_x^2f+\partial_x^2g\|_{L^\infty_{\gamma(t)}}+\|\partial_xf+\partial_xg\|_{L^\infty_{\gamma(t)}}\left\|\frac{\partial_x\theta}{2\sigma}\right\|_{L^\infty_{\gamma(t)}}\right)\|\partial_x^k\theta\|_{\dot{H}^{\frac{1}{2}}_{\gamma(t)}}\|\partial_x^kh\|_{L^2_{\gamma(t)}}.
\end{aligned}\label{2.53}\end{equation}
As an analogue, the terms involving $P_{12}-P_{21}$ in (\ref{2.11})(\ref{2.12}) satisfy the control
\begin{equation}\begin{aligned}
&\left|\int_{\Gamma_{\pm}(t)}\int_{\Gamma_{\pm}(t)}(P_{12}-P_{21})(x,x_1)\partial_x^{k+1}h(x_1)\overline{\partial_x^k\theta(x)}dxdx_1\right|\\
\lesssim&\|\partial_xf+\partial_xg\|_{L^\infty_{\gamma(t)}}\|\partial_x^k\theta\|_{H^\frac{1}{2}_{\gamma(t)}}\|\partial_x^kh\|_{H^\frac{1}{2}_{\gamma(t)}}\\
&+\left(\|\partial_x^2f+\partial_x^2g\|_{L^\infty_{\gamma(t)}}+\|\partial_xf+\partial_xg\|_{L^\infty_{\gamma(t)}}\left\|\frac{\partial_x\theta}{2\sigma}\right\|_{L^\infty_{\gamma(t)}}\right)\|\partial_x^kh\|_{\dot{H}^{\frac{1}{2}}_{\gamma(t)}}\|\partial_x^k\theta\|_{L^2_{\gamma(t)}}.
\end{aligned}\label{2.54}\end{equation}
\end{proof}
\subsection{Transport terms}\label{transport}
The contribution of the transport terms in (\ref{2.1})(\ref{2.2}) to the energy estimate is as follows:
$$\begin{aligned}
&-\Re\int_{\Gamma_{\pm}(t)}(\mu_2u_++\mu_1u_-)\partial_x^{k+1}h\,\overline{\partial_x^kh}\,dx-\mu_1\mu_2\Re\int_{\Gamma_{\pm}(t)}(\mu_1u_++\mu_2u_-)\partial_x^{k+1}\theta\,\overline{\partial_x^k\theta}\,dx\\
&-\mu_1\mu_2\Re\int_{\Gamma_{\pm}(t)}(u_+-u_-)(\partial_x^{k+1}\theta\,\overline{\partial_x^kh}+\partial_x^{k+1}h\,\overline{\partial_x^k\theta})dx.
\end{aligned}$$
We remark that $u_+$ and $u_-$ are not necessarily real-valued, so integration by parts is not applicable. The aim of this section is to bound the transport terms by losing at most half a derivative in $h$ and $\theta$, see Lemma \ref{summary3.5}.
To this end, for $u=u_+$ or $u=u_-$ we write 
$$\begin{aligned}
&\left|\int_{\Gamma_{\pm}(t)}u\partial_x^{k+1}h\,\overline{\partial_x^kh}\,dx\right|=\left|\int_{\Gamma_{\pm}(t)}\ u\Lambda H\partial_x^kh\,\overline{\partial_x^kh}\,dx\right|=\left|\int_{\Gamma_{\pm}(t)}\Lambda^\frac{1}{2}(u\overline{\partial_x^kh})\cdot\Lambda^\frac{1}{2}H\partial_x^kh\,dx\right|,
\end{aligned}$$
$$\begin{aligned}
&\left|\int_{\Gamma_{\pm}(t)}u\partial_x^{k+1}\theta\,\overline{\partial_x^kh}\,dx\right|=\left|\int_{\Gamma_{\pm}(t)}\ u\Lambda H\partial_x^k\theta\,\overline{\partial_x^kh}\,dx\right|=\left|\int_{\Gamma_{\pm}(t)}\Lambda^\frac{1}{2}(u\overline{\partial_x^kh})\cdot\Lambda^\frac{1}{2}H\partial_x^k\theta\,dx\right|,
\end{aligned}$$
and subsequently
$$\left|\int_{\Gamma_{\pm}(t)}u\partial_x^{k+1}h\,\overline{\partial_x^kh}\,dx\right|\leq\|\Lambda^\frac{1}{2}(u\overline{\partial_x^kh})\|_{L^2_{\gamma(t)}}\|\partial_x^kh\|_{\dot{H}_{\gamma(t)}^\frac{1}{2}},$$
$$\left|\int_{\Gamma_{\pm}(t)}u\partial_x^{k+1}\theta\,\overline{\partial_x^kh}\,dx\right|\leq\|\Lambda^\frac{1}{2}(u\overline{\partial_x^kh})\|_{L^2_{\gamma(t)}}\|\partial_x^k\theta\|_{\dot{H}_{\gamma(t)}^\frac{1}{2}}.$$
The main issue is to find a control for $\|\Lambda^\frac{1}{2}(u\overline{\partial_x^kh})\|_{L^2_{\gamma(t)}}$. To start with, we have
$$\|\Lambda^\frac{1}{2}(u\overline{\partial_x^kh})\|_{L^2_{\gamma(t)}}\leq\|u\|_{L^\infty_{\gamma(t)}}\|\partial_x^kh\|_{\dot{H}^\frac{1}{2}_{\gamma(t)}}+\|\Lambda^\frac{1}{2}u\|_{L^6_{\gamma(t)}}\|\partial_x^kh\|_{L^3_{\gamma(t)}}.$$
To bound $\|\Lambda^\frac{1}{2}u\|_{L^6_{\gamma(t))}}$, we use the Gagliardo-Nirenberg interpolation inequality 
$\|\Lambda^\frac{1}{2}u\|_{L^6_{\gamma(t)}}\leq\|u\|_{L^{12}_{\gamma(t)}}^\frac{1}{2}\|\Lambda u\|_{L^4_{\gamma(t)}}^\frac{1}{2}$ and the following Lemma.
\begin{lem}
Suppose the condition (\ref{2.17}) holds, then for any $p>1$
$$\|u_+\|_{L^p_{\gamma(t)}}+\|u_-\|_{L^p_{\gamma(t)}}\lesssim \left(\|\partial_xf\|_{L^p_{\gamma(t)}}+\|\partial_xg\|_{L^p_{\gamma(t)}}\right)\left(1+\|\partial_xf\|_{H^1_{\gamma(t)}}+\|\partial_xg\|_{H^1_{\gamma(t)}}\right),$$
and for $p>2$
$$\begin{aligned}
&\|\partial_xu_+\|_{L^p_{\gamma(t)}}+\|\partial_xu_-\|_{L^p_{\gamma(t)}}\\
\lesssim&\left(\|\partial_xf\|_{L^\infty_{\gamma(t)}}+\|\partial_xg\|_{L^\infty_{\gamma(t)}}\right)\left(\|\partial_x^2f\|_{H^1_{\gamma(t)}}+\|\partial_x^2g\|_{H^1_{\gamma(t)}}+\sigma^{-1}\|\partial_x\theta\|_{H^1_{\gamma(t)}}\right)\\
&+\left(\|\partial_x^2f\|_{L^p_{\gamma(t)}}+\|\partial_x^2g\|_{L^p_{\gamma(t)}}\right)\left(1+\|\partial_x^2f\|_{L^2_{\gamma(t)}}+\|\partial_x^2g\|_{L^2_{\gamma(t)}}\right).
\end{aligned}$$
\label{lem 2.5}\end{lem}
\begin{proof}
Recall the expression of $u_+$ and $u_-$ (\ref{1.9})(\ref{1.10}). It suffices to consider only $u_+$ since the proof of the bound for $u_-$ is almost a repetition. We begin with the first term of (\ref{1.9}). Using the cancellation between the integral over the region $\alpha<0$ and the region $\alpha>0$, we write
$$\begin{aligned}
&P.V.\int_{\mathbb{R}}\frac{\alpha}{\alpha^2+(f(x)-f(x-\alpha))^2}d\alpha\\
=&\int_{\alpha>0}\frac{\alpha(f(x+\alpha)-f(x-\alpha))(f(x-\alpha)+f(x+\alpha)-2f(x))}{\left(\alpha^2+(f(x)-f(x-\alpha))^2\right)\left(\alpha^2+(f(x+\alpha)-f(x))^2\right)}d\alpha
\end{aligned}$$
For the part $\alpha>1$, it holds
\begin{equation}\begin{aligned}
&\left|\int_{\alpha>1}\frac{\alpha(f(x+\alpha)-f(x-\alpha))(f(x-\alpha)+f(x+\alpha)-2f(x))}{\left(\alpha^2+(f(x)-f(x-\alpha))^2\right)\left(\alpha^2+(f(x+\alpha)-f(x))^2\right)}d\alpha\right|\\
\lesssim&\int_{\alpha>1}\frac{1}{\alpha}\left|\frac{f(x+\alpha)-f(x-\alpha)}{\alpha}\right|\left|\frac{f(x+\alpha)-f(x)}{\alpha}-\frac{f(x)-f(x-\alpha)}{\alpha}\right|d\alpha\\
\lesssim&M[\partial_xf](x)\int_{\alpha>1}\frac{M[\partial_xf](x+\alpha)+M[\partial_xf](x-\alpha)}{\alpha}d\alpha\\
\lesssim&M[\partial_xf](x)\|\partial_xf\|_{L^2_{\gamma(t)}},
\end{aligned}\label{2.55}\end{equation}
where we used 
$$|f(x+\alpha)-f(x-\alpha)|\leq2\alpha M[\partial_xf](x),$$
$$ |f(x+\alpha)-f(x)|\leq\alpha M[\partial_xf](x+a),\quad |f(x-\alpha)-f(x)|\leq\alpha M[\partial_xf](x-\alpha).$$
For the part $0<\alpha<1$, to cancel the singularity $\alpha^{-1}$ in the integrand, we use the inequality 
$$\begin{aligned}
\left|f(x-\alpha)+f(x+\alpha)-2f(x)\right|=&\left|\int_{-\alpha}^\alpha(\alpha-|\beta|)\partial_x^2f(x+\beta)d\beta\right|\\
\leq&|\alpha|^2(M[\partial^2f](x-\alpha)+M[\partial_x^2f](x+\alpha)),
\end{aligned}$$
and thus
\begin{equation}\begin{aligned}
&\left|\int_{0<\alpha<1}\frac{\alpha(f(x+\alpha)-f(x-\alpha))(f(x-\alpha)+f(x+\alpha)-2f(x))}{\left(\alpha^2+(f(x)-f(x-\alpha))^2\right)\left(\alpha^2+(f(x+\alpha)-f(x))^2\right)}d\alpha\right|\\
\lesssim&\int_{0<\alpha<1}M[\partial_xf](x)(M[\partial_x^2f](x-\alpha)+M[\partial_x^2f](x+\alpha))d\alpha\\
\lesssim&M[\partial_xf](x)\|\partial_x^2f\|_{L^2_{\gamma(t)}}.
\end{aligned}\label{2.56}\end{equation}
Hence $\left|P.V.\int_{\mathbb{R}}\frac{\alpha}{\alpha^2+(f(x)-f(x-\alpha))^2}d\alpha\right|\leq C\|\partial_xf\|_{H^1_{\gamma(t)}}M[\partial_xf](x)$.
For the second term of (\ref{1.9}), 
\allowdisplaybreaks[4]\begin{align*}
&P.V.\int_{\mathbb{R}}\frac{\alpha}{\alpha^2+(f(x)-g(x-\alpha)+2\sigma)^2}d\alpha\\
=&P.V.\int_{\mathbb{R}}\frac{\alpha}{\alpha^2+(g(x)-g(x-\alpha)+2\sigma+\theta(x))^2}d\alpha\\
=&\int_{\alpha>0}\frac{\alpha\left(\left(g(x)-g(x+\alpha)^2+2\sigma+\theta(x)\right)^2-\left(g(x)-g(x-\alpha)+2\sigma+\theta(x)\right)^2\right)}{\left(\alpha^2+(g(x)-g(x-\alpha)+2\sigma+\theta(x))^2\right)\left(\alpha^2+(g(x)-g(x+\alpha)+2\sigma+\theta(x))^2\right)}d\alpha\\
=&\int_{\alpha>0}\frac{\alpha(g(x-\alpha)-g(x+\alpha))(2g(x)-g(x-\alpha)-g(x+\alpha))}{\left(\alpha^2+(g(x)-g(x-\alpha)+2\sigma+\theta(x))^2\right)\left(\alpha^2+(g(x)-g(x+\alpha)+2\sigma+\theta(x))^2\right)}d\alpha\\
&+\int_{\alpha>0}\frac{\alpha(4\sigma+2\theta(x))(g(x-\alpha)-g(x+\alpha))}{\left(\alpha^2+(g(x)-g(x-\alpha)+2\sigma+\theta(x))^2\right)\left(\alpha^2+(g(x)-g(x+\alpha)+2\sigma+\theta(x))^2\right)}d\alpha.
\end{align*}\allowdisplaybreaks[0]
Since $|\alpha^2+(g(x)-g(x\pm\alpha)+2\sigma+\theta(x))^2|^{-1}\leq C(\alpha^2+(2\sigma)^2)^{-1}\leq C\alpha^{-2}$, the first term on the right-hand side of the above identity can be bounded in absolute value by $CM[\partial_xg] (x)\|\partial_xg\|_{H^1_{\gamma(t)}}$ in the same way as in (\ref{2.55})(\ref{2.56}). For the second term, in view of the condition (\ref{2.17}), it holds
$$\begin{aligned}
&\left|\int_{\alpha>0}\frac{\alpha(4\sigma+2\theta(x))(g(x-\alpha)-g(x+\alpha))}{\left(\alpha^2+(g(x)-g(x-\alpha)+2\sigma+\theta(x))^2\right)\left(\alpha^2+(g(x)-g(x+\alpha)+2\sigma+\theta(x))^2\right)}d\alpha\right|\\
\lesssim&M[\partial_xg](x)\int_{\alpha>0}\frac{\alpha^2(2\sigma)}{(\alpha^2+(2\sigma)^2)^2}d\alpha\\
\lesssim&M[\partial_xg](x).
\end{aligned}$$
Taking the sum yields
$\left|P.V.\int_{\mathbb{R}}\frac{\alpha}{\alpha^2+(f(x)-g(x-\alpha)+2\sigma)^2}d\alpha\right|\lesssim(1+\|\partial_xg\|_{H^1_{\gamma(t)}})M[\partial_xg](x),$ and consequently
$$|u_+(x)|\lesssim(1+\|\partial_xg\|_{H^1_{\gamma(t)}})(M[\partial_xf](x)+M[\partial_xg](x)).$$
As an analogue, $u_-$ satisfies the bound
$$|u_-(x)|\lesssim (1+\|\partial_xf\|_{H^1_{\gamma(t)}})(M[\partial_xf](x)+M[\partial_xg](x)).$$
By the $L^p$ boundedness of the maximal function for $p>1$ we obtain the desired bound for $\|u_+\|_{L^p_{\gamma(t)}}+\|u_-\|_{L^p_{\gamma(t)}}$.
To obtain the estimate of $\partial_xu_{\pm}$, we take derivative in (\ref{1.9}) to get
\begin{equation}\begin{aligned}
\partial_xu_+
=&\mu_2\Delta\rho\,P.V\int_{\mathbb{R}}\frac{2\alpha\Delta f(x,x-\alpha)\Delta\partial_xf(x,x-\alpha)}{(\alpha^2+(\Delta f(x,x-\alpha))^2)^2}d\alpha\\&+\mu_1\Delta\rho\,P.V.\int_{\mathbb{R}}\frac{2\alpha(\Delta g(x,x-\alpha)+2\sigma+\theta(x))(\Delta\partial_xg(x,x-\alpha)+\partial_x\theta(x))}{(\alpha^2+(2\sigma+f(x)-g(x-\alpha))^2)^2}d\alpha.
\end{aligned}\label{2.57}\end{equation}
Again, we divide the integrals into the region $|\alpha|>1$ and the region $|\alpha|<1$. When $|\alpha|>1$, it holds
\begin{equation}\begin{aligned}
&\left|\int_{|\alpha|>1}\frac{2\alpha(f(x)-f(x-\alpha))(\partial_xf(x)-\partial_xf(x-\alpha))}{(\alpha^2+(\Delta f(x,x-\alpha))^2)^2}d\alpha\right|\\
\lesssim&\|\partial_xf\|_{L^\infty_{\gamma(t)}}\int_{|\alpha|>1}\frac{M[\partial_x^2f](x-\alpha)}{|\alpha|}d\alpha.
\end{aligned}\label{2.58}\end{equation}
Since $p>2$, Young's inequality gives
$$\left\|\int_{|\alpha|>1}\frac{1}{\alpha}M[\partial_x^2f](x-\alpha)d\alpha\right\|_{L^p_{\gamma(t)}}\lesssim\|\partial_x^2f\|_{L^2_{\gamma(t)}}\||\alpha|^{-1}\mathrm{1}_{|\alpha|>1}\|_{L^\frac{2p}{p+2}}\lesssim\|\partial_x^2f\|_{L^2_{\gamma(t)}}.$$
When $|\alpha|<1$, we write 
$$\begin{aligned}
&\left|\int_{|\alpha|<1}\frac{2\alpha(f(x)-f(x-\alpha))(\partial_xf(x)-\partial_xf(x-\alpha))}{(\alpha^2+(\Delta f(x,x-\alpha))^2)^2}d\alpha\right|\\
\leq&\int_{0<\alpha<1}\left|\frac{2\alpha(f(x)-f(x-\alpha))(\partial_xf(x)-\partial_xf(x-\alpha))}{(\alpha^2+(f(x)-f(x-\alpha))^2)^2}\right.\\
&\left.-\frac{2\alpha(f(x+\alpha)-f(x))(\partial_xf(x+\alpha)-\partial_xf(x))}{(\alpha^2+(f(x+\alpha)-f(x))^2)^2}\right|d\alpha.
\end{aligned}$$
Using Lemma (\ref{lem2.2}) with the inequalities
$$|(\partial_xf(x)-\partial_xf(x-\alpha))-(\partial_xf(x+\alpha)-\partial_xf(x))|\lesssim\alpha^2\left(M[\partial_x^3f](x-\alpha)+M[\partial_x^3f](x+\alpha)\right),$$
$$|(f(x)-f(x-\alpha))-(f(x+\alpha)-f(x))|\lesssim\alpha^2\left(M[\partial_x^2f](x-\alpha)+M[\partial_x^2f](x+\alpha)\right),$$
$$\begin{aligned}
&\left|\frac{1}{(\alpha^2+(f(x)-f(x-\alpha))^2)^2}-\frac{1}{(\alpha^2+(f(x)-f(x+\alpha))^2)^2}\right|\\
\lesssim&\frac{|2f(x)-f(x-\alpha)-f(x+\alpha)||f(x+\alpha)-f(x-\alpha)|}{\alpha^6}\\
\lesssim&\frac{M[\partial_x^2f](x)M[\partial_xf](x)}{|\alpha|^3},
\end{aligned}$$
we obtain 
$$\begin{aligned}
&\left|\frac{2\alpha(f(x)-f(x-\alpha))(\partial_xf(x)-\partial_xf(x-\alpha))}{(\alpha^2+(f(x)-f(x-\alpha))^2)^2}-\frac{2\alpha(f(x+\alpha)-f(x))(\partial_xf(x+\alpha)-\partial_xf(x))}{(\alpha^2+(f(x+\alpha)-f(x))^2)^2}\right|\\
\lesssim&M[\partial_x^2f](x)(M[\partial_x^2f](x+\alpha)+M[\partial_x^2f](x-\alpha))+M[\partial_xf](x)(M[\partial_x^3f](x+\alpha)+M[\partial_x^3f](x-\alpha))\\
&+M[\partial_xf](x)^2M[\partial_x^2f](x)(M[\partial_x^2f](x+\alpha)+M[\partial_x^2f](x-\alpha)).
\end{aligned}$$
Hence
\begin{equation}\begin{aligned}
&\left|\int_{|\alpha|<1}\frac{2\alpha(f(x)-f(x-\alpha))(\partial_xf(x)-\partial_xf(x-\alpha))}{(\alpha^2+(\Delta f(x,x-\alpha))^2)^2}d\alpha\right|\\
\lesssim&M[\partial_x^2f](x)\|\partial_x^2f\|_{L^2_{\gamma(t)}}(1+\|\partial_xf\|_{L^\infty_{\gamma(t)}}^2)\\
&+\|\partial_xf\|_{L^\infty_{\gamma(t)}}\int_{0<\alpha<1}\left(M[\partial_x^3f](x-\alpha)+M[\partial_x^3f](x+\alpha)\right)d\alpha.
\end{aligned}\label{2.59}\end{equation}
By Young's inequality again, we have 
$$\left\|\int_{0<\alpha<1}\left(M[\partial_x^3f](x-\alpha)+M[\partial_x^3f](x+\alpha)\right)d\alpha\right\|_{L^p_{\gamma(t)}}\lesssim\|\partial_x^3f\|_{L^2_{\gamma(t)}}\|\mathrm{1}_{0<\alpha<1}\|_{L^\frac{2p}{p+2}}\lesssim\|\partial_x^3f\|_{L^2_{\gamma(t)}}.$$
The second term in (\ref{2.57}) is the sum of four terms:
$$\begin{aligned}
&P.V.\int_{\mathbb{R}}\frac{2\alpha(g(x)-g(x-\alpha)+2\sigma+\theta(x))(\partial_xg(x)-\partial_xg(x-\alpha)+\partial_x\theta(x))}{(\alpha^2+(2\sigma+f(x)-g(x-\alpha))^2)^2}\\
=&P.V.\int_{\mathbb{R}}\frac{2\alpha(g(x)-g(x-\alpha))(\partial_xg(x)-\partial_xg(x-\alpha))}{(\alpha^2+(2\sigma+f(x)-g(x-\alpha))^2)^2}d\alpha\\
&+(2\sigma+\theta(x))P.V.\int_{\mathbb{R}}\frac{2\alpha(\partial_xg(x)-\partial_xg(x-\alpha))}{(\alpha^2+(2\sigma+f(x)-g(x-\alpha))^2)^2}d\alpha\\
&+\partial_x\theta(x)P.V.\int_{\mathbb{R}}\frac{2\alpha(g(x)-g(x-\alpha))}{(\alpha^2+(2\sigma+f(x)-g(x-\alpha))^2)^2}d\alpha\\
&+(2\sigma+\theta(x))\partial_x\theta(x)P.V.\int_{\mathbb{R}}\frac{2\alpha}{(\alpha^2+(2\sigma+f(x)-g(x-\alpha))^2)^2}d\alpha.
\end{aligned}$$
The estimate of the first term can be obtained in the same way as in (\ref{2.58})(\ref{2.59}) except that 
$$\begin{aligned}
&\left|\frac{1}{(\alpha^2+(2\sigma+f(x)-g(x-\alpha))^2)^2}-\frac{1}{(\alpha^2+(2\sigma+f(x)-g(x+\alpha))^2)^2}\right|\\
\lesssim&\frac{|4\sigma+2\theta(x)+2g(x)-g(x-\alpha)-g(x+\alpha)||g(x+\alpha)-g(x-\alpha)|}{(\alpha^2+(2\sigma)^2)^3}\\
\lesssim&\frac{(2\sigma)|\alpha|}{(\alpha^2+(2\sigma)^2)^3}M[\partial_xg](x)+\frac{|\alpha|^3M[\partial_x^2g](x)M[\partial_xg](x)}{(\alpha^2+(2\sigma)^2)^3}.
\end{aligned}$$
The additional term $\frac{(2\sigma)\alpha M[\partial_xg](x)}{(\alpha^2+(2\sigma)^2)^3}$ gives rise to $$\int_{\mathbb{R}}M[\partial_xg](x)^2M[\partial_x^2g](x)\frac{(2\sigma)\alpha^4}{(\alpha^2+(2\sigma)^2)^3}d\alpha\lesssim M[\partial_xg](x)^2M[\partial_x^2g](x).$$
Consequently,
$$\begin{aligned}
&\left\|\int_{\mathbb{R}}\frac{2\alpha(g(x)-g(x-\alpha))(\partial_xg(x)-\partial_xg(x-\alpha))}{(\alpha^2+(2\sigma+f(x)-g(x-\alpha))^2)^2}d\alpha\right\|_{L^p_{\gamma(t)}}\\
\lesssim&\|\partial_x^2g\|_{L^p_{\gamma(t)}}\left(\|\partial_x^2g\|_{L^2_{\gamma(t)}}(1+\|\partial_xg\|_{L^\infty_{\gamma(t)}}^2)+\|\partial_xg\|_{L^\infty_{\gamma(t)}}^2\right)+\|\partial_xg\|_{L^\infty_{\gamma(t)}}\|\partial_x^2g\|_{H^1_{\gamma(t)}}.
\end{aligned}$$
The second and third terms are controlled by
$$\begin{aligned}
&\left\|(2\sigma+\theta(x))P.V.\int_{\mathbb{R}}\frac{2\alpha(\partial_xg(x)-\partial_xg(x-\alpha))}{(\alpha^2+(2\sigma+f(x)-g(x-\alpha))^2)^2}d\alpha\right\|_{L^p_{\gamma(t)}}\\
\lesssim&\sigma\left\|M[\partial_x^2g](x)\int_{\mathbb{R}}\frac{\alpha^2}{(\alpha^2+(2\sigma)^2)^2}d\alpha\right\|_{L^p_{\gamma(t)}}\\
\lesssim&\|\partial_x^2g\|_{L^p_{\gamma(t)}},
\end{aligned}$$
$$\begin{aligned}
&\left\|\partial_x\theta(x)P.V.\int_{\mathbb{R}}\frac{2\alpha(g(x)-g(x-\alpha))}{(\alpha^2+(2\sigma+f(x)-g(x-\alpha))^2)^2}d\alpha\right\|_{L^p_{\gamma(t)}}\\
\lesssim&\left\|\partial_x\theta(x)M[\partial_xg](x)\int_{\mathbb{R}}\frac{\alpha^2}{(\alpha^2+(2\sigma)^2)^2}d\alpha\right\|_{L^p_{\gamma(t)}}\\
\lesssim&\sigma^{-1}\|\partial_x\theta\|_{L^p_{\gamma(t)}}\|\partial_xg\|_{L^\infty_{\gamma(t)}}.
\end{aligned}$$
For the last term, we take advantage of the cancellation between the part $\alpha<0$ and $\alpha>0$:
$$\begin{aligned}
&\left|(2\sigma+\theta(x))\partial_x\theta(x)P.V.\int_{\mathbb{R}}\frac{2\alpha}{(\alpha^2+(2\sigma+f(x)-g(x-\alpha))^2)^2}d\alpha\right|\\
=&|(2\sigma+\theta(x))\partial_x\theta(x)|\left|\int_{\alpha>0}\frac{2\alpha}{(\alpha^2+(2\sigma+f(x)-g(x-\alpha))^2)^2}-\frac{2\alpha}{(\alpha^2+(2\sigma+f(x)-g(x+\alpha))^2}d\alpha\right|\\
\lesssim&\sigma|\partial_x\theta(x)|\int_{\alpha>0}\frac{\alpha|4\sigma+2\theta(x)+2g(x)-g(x-\alpha)-g(x+\alpha)||g(x+\alpha)-g(x-\alpha)|}{(\alpha^2+(2\sigma)^2)^3}d\alpha\\
\lesssim&\sigma|\partial_x\theta(x)|M[\partial_xg](x)\int_{\alpha>0}\frac{\alpha^2(2\sigma+\|\partial_xg\|_{L^\infty_{\gamma(t)}}\alpha)}{(\alpha^2+(2\sigma)^2)^3}d\alpha\\
\lesssim&\sigma^{-1}|\partial_x\theta(x)|M[\partial_xg](x).
\end{aligned}$$
Hence the $L^p_{\gamma(t)}$ norm of this term is bounded by $C(2\sigma)^{-1}\|\partial_x\theta\|_{L^p_{\gamma(t)}}\|\partial_xg\|_{L^\infty_{\gamma(t)}}$.
Collecting the above bounds yields
$$\begin{aligned}
&\left\|P.V.\int_{\mathbb{R}}\frac{2\alpha(\Delta g(x,x-\alpha)+2\sigma+\theta(x))(\Delta\partial_xg(x,x-\alpha)+\partial_x\theta(x))}{(\alpha^2+(2\sigma+f(x)-g(x-\alpha))^2)^2}d\alpha\right\|_{L^p_{\gamma(t)}}\\
\lesssim&\|\partial_x^2g\|_{L^p_{\gamma(t)}}\left(1+\|\partial_x^2g\|_{L^2_{\gamma(t)}}\right)\left(1+\|\partial_xg\|_{L^\infty_{\gamma(t)}}\right)+\|\partial_xg\|_{L^\infty_{\gamma(t)}}\left(\|\partial_x^2g\|_{H^1_{\gamma(t)}}+\sigma^{-1}\|\partial_x\theta\|_{L^p_{\gamma(t)}}\right).
\end{aligned}$$
Then we conclude in view of the condition (\ref{2.17}) and the embedding $L^p\subset H^1$ that
$$\begin{aligned}
\|\partial_xu_+\|_{L^p_{\gamma(t)}}\lesssim&\left(\|\partial_x^2f\|_{L^p_{\gamma(t)}}+\|\partial_x^2g\|_{L^p_{\gamma(t)}}\right)\left(1+\|\partial_x^2f\|_{L^2_{\gamma(t)}}+\|\partial_x^2g\|_{L^2_{\gamma(t)}}\right)\\
&+\left(\|\partial_xf\|_{L^\infty_{\gamma(t)}}+\|\partial_xg\|_{L^\infty_{\gamma(t)}}\right)\left(\|\partial_x^2f\|_{H^1_{\gamma(t)}}+\|\partial_x^2g\|_{H^1_{\gamma(t)}}+\sigma^{-1}\|\partial_x\theta\|_{H^1_{\gamma(t)}}\right).
\end{aligned}$$
By repeating all the above arguments, the same bound also holds for $\partial_xu_-$.
\end{proof}
For simplicity, let $B_1(t)=1+\|\partial_xf\|_{H^1_{\gamma(t)}}+\|\partial_xg\|_{H^1_{\gamma(t)}}$ and $B_2(t)=\|\partial_x^2f\|_{H^1_{\gamma(t)}}+\|\partial_x^2g\|_{H^1_{\gamma(t)}}+\sigma^{-1}\|\partial_x\theta\|_{H^1_{\gamma(t)}}$ so that 
$$\begin{aligned}
&\|u\|_{L^{12}_{\gamma(t)}}^\frac{1}{2}\|\Lambda u\|_{L^4_{\gamma(t)}}^\frac{1}{2}\\
\lesssim&B_1^\frac{1}{2}B_2^\frac{1}{2}\left(\|\partial_xf\|_{L^{12}_{\gamma(t)}}+\|\partial_xg\|_{L^{12}_{\gamma(t)}}\right)^\frac{1}{2}\left(\|\partial_xf\|_{L^\infty_{\gamma(t)}}+\|\partial_xg\|_{L^\infty_{\gamma(t)}}\right)^\frac{1}{2}\\
&+B_1\left(\|\partial_xf\|_{L^{12}_{\gamma(t)}}+\|\partial_xg\|_{L^{12}_{\gamma(t)}}\right)^\frac{1}{2}\left(\|\partial_x^2f\|_{L^{4}_{\gamma(t)}}+\|\partial_x^2g\|_{L^{4}_{\gamma(t)}}\right)^\frac{1}{2}\\
\lesssim&B_1^\frac{1}{2}B_2^\frac{1}{2}\left(\|\partial_xf\|_{L^{\infty}_{\gamma(t)}}+\|\partial_xg\|_{L^{\infty}_{\gamma(t)}}\right)^\frac{11}{12}\left(\|\partial_xf\|_{L^{2}_{\gamma(t)}}+\|\partial_xg\|_{L^{2}_{\gamma(t)}}\right)^\frac{1}{12}\\
&+B_1\left(\|\partial_xf\|_{L^{\infty}_{\gamma(t)}}+\|\partial_xg\|_{L^{\infty}_{\gamma(t)}}\right)^\frac{2}{3}\left(\|\partial_xf\|_{L^{2}_{\gamma(t)}}+\|\partial_xg\|_{L^{2}_{\gamma(t)}}\right)^\frac{1}{12}\left(\|\partial_x^3f\|_{L^{2}_{\gamma(t)}}+\|\partial_x^3g\|_{L^{2}_{\gamma(t)}}\right)^\frac{1}{4},
\end{aligned}$$
where we used the inequalities
$$\|\partial_xf\|_{L^{12}_{\gamma(t)}}\leq\|\partial_xf\|_{L^\infty_{\gamma(t)}}^\frac{5}{6}\|\partial_xf\|_{L^2_{\gamma(t)}}^\frac{1}{6},\quad\|\partial_xf\|_{L^4_{\gamma(t)}}\leq\|\partial_xf\|_{L^\infty_{\gamma(t)}}^\frac{1}{2}\|\partial_xf\|_{L^2_{\gamma(t)}}^\frac{1}{2},$$
$$\|\partial_x^2f\|_{L^4_{\gamma(t)}}\leq\|\partial_xf\|_{L^\infty_{\gamma(t)}}^\frac{1}{2}\|\partial_x^3f\|_{L^2_{\gamma(t)}}^\frac{1}{2}.$$
By Young's inequality and using $\|\partial_x^kh\|_{L^3_{\gamma(t)}}\lesssim\|\partial_x^kh\|_{L^2_{\gamma(t)}}^\frac{2}{3}\|\partial_x^kh\|_{\dot{H}^\frac{1}{2}_{\gamma(t)}}^\frac{1}{3}$, we obtain
\begin{equation}\begin{aligned}
&\|\Lambda^\frac{1}{2}(u\overline{\partial_x^kh})\|_{L^2_{\gamma(t)}}\|\partial_x^kh\|_{\dot{H}^\frac{1}{2}_{\gamma(t)}}\\
\lesssim&\left[B_1\left(\|\partial_xf\|_{L^{\infty}_{\gamma(t)}}+\|\partial_xg\|_{L^{\infty}_{\gamma(t)}}\right)^\frac{2}{3}\left(\|\partial_xf\|_{L^{2}_{\gamma(t)}}+\|\partial_xg\|_{L^{2}_{\gamma(t)}}\right)^\frac{1}{12}\left(\|\partial_x^3f\|_{L^{2}_{\gamma(t)}}+\|\partial_x^3g\|_{L^{2}_{\gamma(t)}}\right)^\frac{1}{4}\right.\\
&\left.+B_1^\frac{1}{2}B_2^\frac{1}{2}\left(\|\partial_xf\|_{L^{\infty}_{\gamma(t)}}+\|\partial_xg\|_{L^{\infty}_{\gamma(t)}}\right)^\frac{11}{12}\left(\|\partial_xf\|_{L^{2}_{\gamma(t)}}+\|\partial_xg\|_{L^{2}_{\gamma(t)}}\right)^\frac{1}{12}\right]\|\partial_x^kh\|_{L^2_{\gamma(t)}}^\frac{2}{3}\|\partial_x^kh\|_{\dot{H}^\frac{1}{2}_{\gamma(t)}}^\frac{4}{3}\\
&+\left(\|\partial_xf\|_{L^\infty_{\gamma(t)}}+\|\partial_xg\|_{L^\infty_{\gamma(t)}}\right)\|\partial_x^kh\|_{\dot{H}^\frac{1}{2}_{\gamma(t)}}^2\\
\lesssim&B_3\left(\|\partial_xf\|_{L^\infty_{\gamma(t)}}+\|\partial_xg\|_{L^\infty_{\gamma(t)}}\right)\|\partial_x^kh\|_{\dot{H}^\frac{1}{2}_{\gamma(t)}}^2+B_4\|\partial_x^kh\|_{L^2_{\gamma(t)}}^2,
\end{aligned}\label{2.60}\end{equation}
where $B_3:=B_1+B_1^\frac{1}{2}B_2^\frac{1}{2}$, and
$$\begin{aligned}
B_4:=&B_1\left(\|\partial_xf\|_{L^2_{\gamma(t)}}+\|\partial_xg\|_{L^2_{\gamma(t)}}\right)^\frac{1}{4}\left(\|\partial_x^3f\|_{L^2_{\gamma(t)}}+\|\partial_x^3g\|_{L^2_{\gamma(t)}}\right)^\frac{3}{4}\\
&+B_1^\frac{1}{2}B_2^\frac{1}{2}\left(\|\partial_xf\|_{L^2_{\gamma(t)}}+\|\partial_xg\|_{L^2_{\gamma(t)}}\right)^\frac{1}{4}\left(\|\partial_xf\|_{L^\infty_{\gamma(t)}}+\|\partial_xg\|_{L^\infty_{\gamma(t)}}\right)^\frac{3}{4}.
\end{aligned}$$
Likewise, 
\begin{equation}\begin{aligned}
&\|\Lambda^\frac{1}{2}(u\overline{\partial_x^k\theta})\|_{L^2_{\gamma(t)}}\|\partial_x^k\theta\|_{\dot{H}^\frac{1}{2}_{\gamma(t)}} 
\lesssim B_3\left(\|\partial_xf\|_{L^\infty_{\gamma(t)}}+\|\partial_xg\|_{L^\infty_{\gamma(t)}}\right)\|\partial_x^k\theta\|_{\dot{H}^\frac{1}{2}_{\gamma(t)}}^2+B_4\|\partial_x^k\theta\|_{L^2_{\gamma(t)}}^2.
\end{aligned}\label{2.61}\end{equation}
Collecting the bounds in (\ref{2.60}-\ref{2.61}) we conclude with the control for all the transport terms in (\ref{2.1})(\ref{2.2}):
\begin{lem}\label{summary3.5}
Let $B_i$, $1\leq i\leq 4$ be given by
\begin{equation}\begin{aligned}
B_1(t):=&1+\|\partial_xf\|_{H^1_{\gamma(t)}}+\|\partial_xg\|_{H^1_{\gamma(t)}},\;B_2(t):=\|\partial_x^2f\|_{H^1_{\gamma(t)}}+\|\partial_x^2g\|_{H^1_{\gamma(t)}}+\sigma^{-1}\|\partial_x\theta\|_{H^1_{\gamma(t)}},\\
B_3(t):=&B_1+B_1^\frac{1}{2}B_2^\frac{1}{2},\\
B_4(t):=&B_1\left(\|\partial_xf\|_{L^2_{\gamma(t)}}+\|\partial_xg\|_{L^2_{\gamma(t)}}\right)^\frac{1}{4}\left(\|\partial_x^3f\|_{L^2_{\gamma(t)}}+\|\partial_x^3g\|_{L^2_{\gamma(t)}}\right)^\frac{3}{4}\\
&+B_1^\frac{1}{2}B_2^\frac{1}{2}\left(\|\partial_xf\|_{L^2_{\gamma(t)}}+\|\partial_xg\|_{L^2_{\gamma(t)}}\right)^\frac{1}{4}\left(\|\partial_xf\|_{L^\infty_{\gamma(t)}}+\|\partial_xg\|_{L^\infty_{\gamma(t)}}\right)^\frac{3}{4}.
\end{aligned}\label{B_i}\end{equation}
Then the transport terms can be controlled by
\begin{equation}\begin{aligned}
&\left|-\Re\int_{\Gamma_{\pm}(t)}(\mu_2u_++\mu_1u_-)\partial_x^{k+1}h\,\overline{\partial_x^kh}\,dx-\mu_1\mu_2\Re\int_{\Gamma_{\pm}(t)}(\mu_1u_++\mu_2u_-)\partial_x^{k+1}\theta\,\overline{\partial_x^k\theta}\,dx\right.\\
&\left.-\mu_1\mu_2\Re\int_{\Gamma_{\pm}(t)}(u_+-u_-)(\partial_x^{k+1}\theta\,\overline{\partial_x^kh}+\partial_x^{k+1}h\,\overline{\partial_x^k\theta})dx\right|\\
\lesssim&B_3\left(\|\partial_xf\|_{L^\infty_{\gamma(t)}}+\|\partial_xg\|_{L^\infty_{\gamma(t)}}\right)\left(\|\partial_x^kh\|_{\dot{H}^\frac{1}{2}_{\gamma(t)}}^2+\|\partial_x^k\theta\|_{\dot{H}^\frac{1}{2}_{\gamma(t)}}^2\right)\\
&+B_4\left(\|\partial_x^kh\|_{L^2_{\gamma(t)}}^2+\|\partial_x^k\theta\|_{L^2_{\gamma(t)}}^2\right).
\end{aligned}\label{2.63}\end{equation}
\end{lem}

\subsection{Commutators-Safe terms}\label{commutators1}
The following two sections are devoted to control the commutators. Actually, we will split the commutators into safe terms and easy terms based on the number of derivatives. In this section, we will handle the safe terms, and the main result is given in Lemma \ref{summary3.6}. To begin with, we observe that the sum of $N_1^h$ and $N_3^h$ has a nice structure:
$$\begin{aligned}
&(\Delta\rho)^{-1}(N^h_1+N^h_3)\\
=&P.V.\int_{\mathbb{R}}[\partial_x^k,(\mu_2P_{11}+\mu_1\mu_2P_{12}+\mu_1\mu_2P_{21}+\mu_1^2P_{22})(x,x-\alpha)](\partial_xh(x)-\partial_xh(x-\alpha))d\alpha,
\end{aligned}$$
Meanwhile, the sum of $N_2^h$and $N_4^h$ can be written as
$$\begin{aligned}
&(\Delta\rho)^{-1}(N_2^h+N_4^h)\\
=&\mu_1\mu_2\,P.V.\int_{\mathbb{R}}[\partial_x^k,(\mu_2(P_{11}-P_{12})-\mu_1(P_{22}-P_{21}))(x,x-\alpha)](\partial_x^k\theta(x)-\partial_x^k\theta(x-\alpha))d\alpha\\
&+\mu_1\mu_2\,P.V.\int_{\mathbb{R}}[\partial_x^k,(P_{12}-P_{21})(x,x-\alpha)]\partial_x\theta(x)d\alpha.
\end{aligned}$$
As an analogue, combining $N_1^\theta$ with $N_3^\theta$ and combining $N_2^\theta$ with $N_4^\theta$ yields
$$\begin{aligned}
&(\Delta\rho)^{-1}(N_1^\theta+N_3^\theta)\\
=&P.V.\int_{\mathbb{R}}[\partial_x^k,(\mu_2(P_{11}-P_{12})-\mu_1(P_{22}-P_{21}))(x,x-\alpha)](\partial_xh(x)-\partial_xh(x-\alpha))d\alpha,
\end{aligned}$$
$$\begin{aligned}
&(\Delta\rho)^{-1}(N_2^\theta+N_4^\theta)\\
=&\mu_1\mu_2\,P.V.\int_{\mathbb{R}}[\partial_x^k,(P_{11}+P_{22}-P_{12}-P_{21})(x,x-\alpha)](\partial_x\theta(x)-\partial_x\theta(x-\alpha))d\alpha\\
&+\,P.V.\int_{\mathbb{R}}[\partial_x^k,(\mu_2P_{21}+\mu_1P_{12})(x,x-\alpha)]\partial_x\theta(x)d\alpha.
\end{aligned}$$
Therefore, we see that the commutators are linear combinations of the following elements:
$$P.V.\int_{\mathbb{R}}[\partial_x^k, P_{ij}(x,x-\alpha)](\partial_xh(x)-\partial_xh(x-\alpha))d\alpha,\;(i,j)\in\left\{1,2\right\}^2,$$
$$P.V.\int_{\mathbb{R}}[\partial_x^k, P_{ij}(x,x-\alpha)](\partial_x\theta(x)-\partial_x\theta(x-\alpha))d\alpha,\;(i,j)\in\left\{1,2\right\}^2,$$
$$P.V.\int_{\mathbb{R}}[\partial_x^k, P_{12}(x,x-\alpha)]\partial_x\theta(x)d\alpha,\;P.V.\int_{\mathbb{R}}[\partial_x^k,P_{21}(x,x-\alpha)]\partial_x\theta(x)d\alpha.$$
For convenience, let $q(x):=\frac{1}{1+x^2}$ and $\tilde{\Delta}w(x,\alpha)=w(x)-w(x-\alpha)$ for any function $w$, so that $\frac{\alpha}{\alpha^2+(\tilde{\Delta}f)^2}=\frac{1}{\alpha}q\left(\frac{\tilde{\Delta} f}{\alpha}\right).$ Then we expand the commutator by Fa\`{a} di Bruno's formula, 
$$\begin{aligned}
&\int_{\mathbb{R}}[\partial_x^k, P_{11}(x,x-\alpha)](\partial_xh(x)-\partial_xh(x-\alpha))d\alpha,\\
=&\sum_{1\leq l\leq j\leq k}\sum_{S_{j,l}}A^{k,j,l}_{r_1,\cdots,r_j}\int_{\mathbb{R}}\alpha^{-1}q^{(l)}\left(\frac{\tilde{\Delta} f}{\alpha}\right)\prod_{a=1}^j\left(\frac{\tilde{\Delta}\partial_x^af}{\alpha}\right)^{r_a}\tilde{\Delta}\partial_x^{k+1-j}h\,d\alpha
\end{aligned}$$
where $A^{k,j,l}_{r_1,\cdots,r_j}$ are harmless constants, and the index set $S_{j,l}$ is defined by
$$S_{j,l}:=\left\{(r_1,r_2,\cdots,r_j)\in\mathbb{Z}^j\mid\sum_{1\leq a\leq j}ar_a=j,\,\sum_{1\leq a\leq j}r_a=l,\,r_a\geq0\right\}.$$
Likewise, the same identities hold if we replace $h$ by $\theta$, or replace $P_{11}$ by $P_{22}$ and replace $f$ by $g$.
To express $\partial_x^j(P_{12}(x,x-\alpha))$ and $\partial_x^j(P_{21}(x,x-\alpha))$, we write
$$P_{12}(x,x-\alpha)=\frac{\alpha}{\alpha^2+(2\sigma+\theta(x)+\tilde{\Delta} g)^2}=\alpha^{-1}q\left(\frac{\tilde{\Delta} g+2\sigma+\theta(x)}{\alpha}\right),$$
$$P_{21}(x,x-\alpha)=\frac{\alpha}{\alpha^2+(\tilde{\Delta} f-2\sigma-\theta(x))^2}=\alpha^{-1}q\left(\frac{\tilde{\Delta} f-2\sigma-\theta(x)}{\alpha}\right),$$
so that
$$\begin{aligned}
&\int_{\mathbb{R}}[\partial_x^k, P_{12}(x,x-\alpha)](\partial_xh(x)-\partial_xh(x-\alpha))d\alpha,\\\
=&\sum_{1\leq l\leq j\leq k}\sum_{S_{j,l}}A^{k,j,l}_{r_1,\cdots,r_j}\int_{\mathbb{R}}\alpha^{-1}q^{(l)}\left(\frac{\tilde{\Delta} g+2\sigma+\theta(x)}{\alpha}\right)\prod_{a=1}^j\left(\frac{\tilde{\Delta}\partial_x^ag+\partial_x^a\theta(x)}{\alpha}\right)^{r_a}\tilde{\Delta}\partial_x^{k+1-j}h\,d\alpha,
\end{aligned}$$
and the same identities hold if we replace $h$ by $\theta$, or replace $P_{12}$ by $P_{21}$ and replace $\tilde{\Delta} g+2\sigma+\theta(x)$ by $\tilde{\Delta} f-2\sigma-\theta(x)$. The rest two terms can be expressed in the same manner:
$$\begin{aligned}
&\int_{\mathbb{R}}[\partial_x^k, P_{12}(x,x-\alpha)]\partial_x\theta(x)d\alpha,\\\
=&\sum_{1\leq l\leq j\leq k}\sum_{S_{j,l}}A^{k,j,l}_{r_1,\cdots,r_j}\int_{\mathbb{R}}\alpha^{-1}q^{(l)}\left(\frac{\tilde{\Delta} g+2\sigma+\theta(x)}{\alpha}\right)\prod_{a=1}^j\left(\frac{\tilde{\Delta}\partial_x^ag+\partial_x^a\theta(x)}{\alpha}\right)^{r_a}\partial_x^{k+1-j}\theta(x)\,d\alpha,
\end{aligned}$$
$$\begin{aligned}
&\int_{\mathbb{R}}[\partial_x^k, P_{21}(x,x-\alpha)]\partial_x\theta(x)d\alpha,\\\
=&\sum_{1\leq l\leq j\leq k}\sum_{S_{j,l}}A^{k,j,l}_{r_1,\cdots,r_j}\int_{\mathbb{R}}\alpha^{-1}q^{(l)}\left(\frac{\tilde{\Delta} f-2\sigma-\theta(x)}{\alpha}\right)\prod_{a=1}^j\left(\frac{\tilde{\Delta}\partial_x^af-\partial_x^a\theta(x)}{\alpha}\right)^{r_a}\partial_x^{k+1-j}\theta(x)\,d\alpha.
\end{aligned}$$
To sum up, in order to obtain controls over the commutators, it suffices to consider the following three classes of terms.
\begin{itemize}
\item[(\romannumeral1)] 
$\int_{\mathbb{R}}\alpha^{-1}q^{(l)}\left(\frac{\tilde{\Delta} f}{\alpha}\right)\prod_{a=1}^j\left(\frac{\tilde{\Delta}\partial_x^af}{\alpha}\right)^{r_a}\tilde{\Delta}\partial_x^{k+1-j}h\,d\alpha$ for $r:=(r_1,r_2,\cdots,r_j)\in S_{j,l}$, $1\leq l\leq j\leq k$, and the terms with $f$ replaced $g$ or $h$ replaced by $\theta$.
\item[(\romannumeral2)]
$\int_{\mathbb{R}}\alpha^{-1}q^{(l)}\left(\frac{\tilde{\Delta} g+2\sigma+\theta(x)}{\alpha}\right)\prod_{a=1}^j\left(\frac{\tilde{\Delta}\partial_x^ag+\partial_x^a\theta(x)}{\alpha}\right)^{r_a}\tilde{\Delta}\partial_x^{k+1-j}h\,d\alpha$ for $r:=(r_1,r_2,\cdots,r_j)\in S_{j,l}$, $1\leq l\leq j\leq k$ and the terms with $h$ replaced by $\theta$ or $\tilde{\Delta} g+2\sigma+\theta(x)$ and its derivatives replaced by $\tilde{\Delta} f-2\sigma-\theta(x)$ and its derivatives, respectively.
\item[(\romannumeral3)]
$\int_{\mathbb{R}}\alpha^{-1}q^{(l)}\left(\frac{\tilde{\Delta} g+2\sigma+\theta(x)}{\alpha}\right)\prod_{a=1}^j\left(\frac{\tilde{\Delta}\partial_x^ag+\partial_x^a\theta(x)}{\alpha}\right)^{r_a}\partial_x^{k+1-j}\theta(x)\,d\alpha$ for $1\leq l\leq j\leq k$, $r\in S_{j,l}$ and the terms with $\tilde{\Delta} g+2\sigma+\theta(x)$ and its derivatives replaced by $\tilde{\Delta} f-2\sigma-\theta(x)$ and its derivatives, respectively.
\end{itemize}
Moreover, we distinguish the end point terms from each class, which are the terms involving the derivatives of $k^{\text{th}}$ order of $h$, $\theta$, $f$ or $g$. In the following, we will refer to these terms as safe terms and refer to the other terms as easy terms. To be specific, we list all the safe terms from each class in the following.
\begin{itemize}
\item Safe terms in class (\romannumeral1):\\
When $l=j=1$, $r=r_1=1$, we get
$$Safe^{11}_{h,1}(x):=\int_{\mathbb{R}}q^{(1)}\left(\frac{\tilde{\Delta} f}{\alpha}\right)\frac{\tilde{\Delta}\partial_x f}{\alpha}\frac{\tilde{\Delta}\partial_x^kh}{\alpha}d\alpha,\;
Safe^{22}_{h,1}(x):=\int_{\mathbb{R}}q^{(1)}\left(\frac{\tilde{\Delta} g}{\alpha}\right)\frac{\tilde{\Delta}\partial_x g}{\alpha}\frac{\tilde{\Delta}\partial_x^kh}{\alpha}d\alpha,$$
and $Safe^{11}_{\theta,1}$, $Safe^{22}_{\theta,1}$, which are defined by replacing $h$ by $\theta$ in the above expressions.\\
When $j=k$, $l=1$, $r=(0,0,\cdots,1)$, we get
$$Safe^{11}_{h,2}(x):=\int_{\mathbb{R}}q^{(1)}\left(\frac{\tilde{\Delta} f}{\alpha}\right)\frac{\tilde{\Delta}\partial_xh}{\alpha}\frac{\tilde{\Delta}\partial_x^kf}{\alpha}d\alpha,\;
Safe^{22}_{h,2}(x):=\int_{\mathbb{R}}q^{(1)}\left(\frac{\tilde{\Delta} g}{\alpha}\right)\frac{\tilde{\Delta}\partial_xh}{\alpha}\frac{\tilde{\Delta}\partial_x^kg}{\alpha}d\alpha,$$
and $Safe^{11}_{\theta,2}$, $Safe^{22}_{\theta,2}$, which are defined by replacing $h$ by $\theta$ in the above expressions.
\item Safe terms in class (\romannumeral2):\\
When $l=j=1$, $r=r_1=1$, we get
$$\int_{\mathbb{R}}q^{(1)}\left(\frac{\tilde{\Delta} g+2\sigma+\theta(x)}{\alpha}\right)\frac{\tilde{\Delta}\partial_x g+\partial_x\theta(x)}{\alpha}\frac{\tilde{\Delta}\partial_x^k h}{\alpha}d\alpha=Safe_{h,1}^{12}(x)+Safe_{h,2}^{12}(x),$$
where 
$$Safe_{h,1}^{12}(x):=\int_{\mathbb{R}}q^{(1)}\left(\frac{\tilde{\Delta} g+2\sigma+\theta(x)}{\alpha}\right)\frac{\tilde{\Delta}\partial_x g}{\alpha}\frac{\tilde{\Delta}\partial_x^k h}{\alpha}d\alpha,$$
$$Safe_{h,2}^{12}(x):=\int_{\mathbb{R}}q^{(1)}\left(\frac{\tilde{\Delta} g+2\sigma+\theta(x)}{\alpha}\right)\frac{\partial_x\theta(x)}{\alpha}\frac{\tilde{\Delta}\partial_x^k h}{\alpha}d\alpha.$$
Replacing $\tilde{\Delta} g+2\sigma+\theta(x)$ by $\tilde{\Delta} f-2\sigma-\theta(x)$ yields
$$Safe_{h,1}^{21}(x):=\int_{\mathbb{R}}q^{(1)}\left(\frac{\tilde{\Delta} f-2\sigma-\theta(x)}{\alpha}\right)\frac{\tilde{\Delta}\partial_x f}{\alpha}\frac{\tilde{\Delta}\partial_x^k h}{\alpha}d\alpha,$$
$$Safe_{h,2}^{21}(x):=-\int_{\mathbb{R}}q^{(1)}\left(\frac{\tilde{\Delta} f-2\sigma-\theta(x)}{\alpha}\right)\frac{\partial_x\theta(x)}{\alpha}\frac{\tilde{\Delta}\partial_x^k h}{\alpha}d\alpha.$$
When $j=k$, $l=1$, $r=(0,0,\cdots,1)$, we get
$$\int_{\mathbb{R}}q^{(1)}\left(\frac{\tilde{\Delta} g+2\sigma+\theta(x)}{\alpha}\right)\frac{\tilde{\Delta}\partial_xh}{\alpha}\frac{\tilde{\Delta}\partial_x^kg+\partial_x^k\theta(x)}{\alpha}d\alpha=Safe_{h,3}^{12}(x)+Safe_{h,4}^{12}(x),$$
where
$$Safe_{h,3}^{12}(x):=\int_{\mathbb{R}}q^{(1)}\left(\frac{\tilde{\Delta} g+2\sigma+\theta(x)}{\alpha}\right)\frac{\tilde{\Delta}\partial_xh}{\alpha}\frac{\tilde{\Delta}\partial_x^kg}{\alpha}d\alpha,$$
$$Safe_{h,4}^{12}(x):=\int_{\mathbb{R}}q^{(1)}\left(\frac{\tilde{\Delta} g+2\sigma+\theta(x)}{\alpha}\right)\frac{\tilde{\Delta}\partial_xh}{\alpha}\frac{\partial_x^k\theta(x)}{\alpha}d\alpha.$$
Replacing $\tilde{\Delta} g+2\sigma+\theta(x)$ by $\tilde{\Delta} f-2\sigma-\theta(x)$ gives
$$Safe_{h,3}^{21}(x):=\int_{\mathbb{R}}q^{(1)}\left(\frac{\tilde{\Delta} f-2\sigma-\theta(x)}{\alpha}\right)\frac{\tilde{\Delta}\partial_xh}{\alpha}\frac{\tilde{\Delta}\partial_x^kf}{\alpha}d\alpha,$$
$$Safe_{h,4}^{21}(x):=-\int_{\mathbb{R}}q^{(1)}\left(\frac{\tilde{\Delta} f-2\sigma-\theta(x)}{\alpha}\right)\frac{\tilde{\Delta}\partial_xh}{\alpha}\frac{\partial_x^k\theta(x)}{\alpha}d\alpha.$$
We also define $Safe_{\theta,1}^{12}$, $Safe_{\theta,2}^{12}$, $Safe_{\theta,1}^{21}$, $Safe_{\theta,2}^{21}$ by replacing $\frac{\tilde{\Delta}\partial_x^kh}{\alpha}$ by $\frac{\tilde{\Delta}\partial_x^k\theta}{\alpha}$ in the expressions of $Safe_{h,1}^{12}$, $Safe_{h,2}^{12}$, $Safe_{h,1}^{21}$, $Safe_{h,2}^{21}$, and define $Safe_{\theta,3}^{12}$, $Safe_{\theta,4}^{12}$, $Safe_{\theta,3}^{21}$, $Safe_{\theta,4}^{21}$ by replacing $\frac{\tilde{\Delta}\partial_xh}{\alpha}$ by $\frac{\tilde{\Delta}\partial_x\theta}{\alpha}$ in the expressions of $Safe_{h,3}^{12}$, $Safe_{h,4}^{12}$, $Safe_{h,3}^{21}$, $Safe_{h,4}^{21}$, respectively. Hence we have in total 16 safe terms in class (\romannumeral2).
\item Safe terms in class (\romannumeral3):
When $l=j=1$, $r=r_1=1$, we get
$$\int_{\mathbb{R}}q^{(1)}\left(\frac{\tilde{\Delta} g+2\sigma+\theta(x)}{\alpha}\right)\frac{\tilde{\Delta}\partial_xg+\partial_x\theta(x)}{\alpha}\frac{\partial_x^k\theta(x)}{\alpha}d\alpha=Safe_{\theta,5}^{12}(x)+Safe_{\theta,6}^{12}(x),$$
where
$$Safe_{\theta,5}^{12}(x):=\int_{\mathbb{R}}q^{(1)}\left(\frac{\tilde{\Delta} g+2\sigma+\theta(x)}{\alpha}\right)\frac{\tilde{\Delta}\partial_xg}{\alpha}\frac{\partial_x^k\theta(x)}{\alpha}d\alpha,$$
$$Safe_{\theta,6}^{12}(x):=\int_{\mathbb{R}}q^{(1)}\left(\frac{\tilde{\Delta} g+2\sigma+\theta(x)}{\alpha}\right)\frac{\partial_x\theta(x)}{\alpha}\frac{\partial_x^k\theta(x)}{\alpha}d\alpha.$$
When $j=k$, $l=1$, $r=(0,0,\cdots,1)$, we get
$$\int_{\mathbb{R}}q^{(1)}\left(\frac{\tilde{\Delta} g+2\sigma+\theta(x)}{\alpha}\right)\frac{\partial_x\theta(x)}{\alpha}\frac{\tilde{\Delta}\partial_x^kg+\partial_x^k\theta(x)}{\alpha}d\alpha=Safe_{\theta,7}^{12}(x)+Safe_{\theta,6}^{12}(x),$$
where $$Safe_{\theta,7}^{12}(x):=\int_{\mathbb{R}}q^{(1)}\left(\frac{\tilde{\Delta} g+2\sigma+\theta(x)}{\alpha}\right)\frac{\partial_x\theta(x)}{\alpha}\frac{\tilde{\Delta}\partial_x^kg}{\alpha}d\alpha.$$
Then we replace $\tilde{\Delta} g+2\sigma+\theta(x)$ by $\tilde{\Delta} f-2\sigma-\theta(x)$ and obtain $Safe^{21}_{\theta,5}$, $Safe^{21}_{\theta,6}$, $Safe^{21}_{\theta,7}$.
\end{itemize}
Now we derive the controls over each of the above safe terms.\\
\textbf{Controls for class (\romannumeral1)}. For $Safe_{h,1}^{11}$, we decompose it into three terms:
$$\begin{aligned}
&Safe_{h,1}^{11}(x)\\
=&q^{(1)}(\partial_xf(x))\partial_x^kh(x)P.V.\int_{\mathbb{R}}\frac{\tilde{\Delta}\partial_xf}{\alpha^2}d\alpha-q^{(1)}(\partial_xf(x))P.V.\int_{\mathbb{R}}\frac{\tilde{\Delta}\partial_xf}{\alpha^2}\partial_x^kh(x-\alpha)d\alpha\\
&+\int_{\mathbb{R}}\left(q^{(1)}\left(\frac{\tilde{\Delta} f}{\alpha}\right)-q^{(1)}(\partial_xf(x))\right)\frac{\tilde{\Delta}\partial_xf}{\alpha^2}\left(\partial_x^kh(x)-\partial_x^kh(x-\alpha)\right)d\alpha.
\end{aligned}$$
For the first term, we use $|q^{(1)}(\partial_xf(x))|\lesssim|\partial_xf(x)|$ and 
$$\begin{aligned}
\left|P.V.\int_{\mathbb{R}}\frac{\tilde{\Delta}\partial_xf}{\alpha^2}d\alpha\right|=&\left|\int_{|\alpha|>1}\frac{\tilde{\Delta}\partial_xf}{\alpha^2}d\alpha+\int_{0<\alpha\leq1}\frac{2\partial_xf(x)-\partial_xf(x-\alpha)-\partial_xf(x+\alpha)}{\alpha^2}d\alpha\right|\\
\lesssim&\int_{|\alpha|>1}\alpha^{-1}M[\partial_x^2f](x-\alpha)d\alpha+\int_{0<\alpha\leq1}\left(M[\partial_x^3f](x-\alpha)+M[\partial_x^3f](x+\alpha)\right)d\alpha\\
\lesssim&\|\partial_x^2f\|_{L^2_{\gamma(t)}}+\|\partial_x^3f\|_{L^2_{\gamma(t)}}
\end{aligned}$$
to obtain that
$$\begin{aligned}
\left\|q^{(1)}(\partial_xf(x))\partial_x^kh(x)P.V.\int_{\mathbb{R}}\frac{\tilde{\Delta}\partial_xf}{\alpha^2}d\alpha\right\|_{L^2_{\gamma(t)}}\lesssim\|\partial_xf\|_{L^\infty_\gamma(t)}\left(\|\partial_x^2f\|_{L^2_{\gamma(t)}}+\|\partial_x^3f\|_{L^2_\gamma(t)}\right)\|\partial_x^kh\|_{L^2_\gamma(t)}.
\end{aligned}$$
For the second term, we notice that $\int_{\mathbb{R}}\frac{\tilde{\Delta}\partial_xf}{\alpha^2}\partial_x^kh(x-\alpha)d\alpha$ is the first Calder\'{o}n commutator \cite[Page 254]{grafakos2014modern} applied on $\partial_x^kh$. Since the first Calder\'{o}n commutator is $L^2$ bounded with norm $C\|\partial_x^2f\|_{L^\infty}$, we have
$$\left\|q^{(1)}(\partial_xf(x))P.V.\int_{\mathbb{R}}\frac{\tilde{\Delta}\partial_xf}{\alpha^2}\partial_x^kh(x-\alpha)d\alpha\right\|_{L^2_{\gamma(t)}}\lesssim\|\partial_xf\|_{L^\infty_{\gamma(t)}}\|\partial_x^2f\|_{L^\infty_{\gamma(t)}}\|\partial_x^kh\|_{L^2_{\gamma(t)}}.$$
The third term is a regular integral. Notice that
$$\left|\alpha^{-1}\tilde{\Delta} f-\partial_xf(x)\right|\lesssim|\alpha|\|\partial_x^2f\|_{L^\infty_\gamma(t)}\mathrm{1}_{|\alpha|\leq1}+\|\partial_xf\|_{L^\infty_\gamma(t)}\mathrm{1}_{|\alpha|>1},$$
$$\left|\frac{\tilde{\Delta}\partial_xf}{\alpha}\right|\leq\|\partial_x^2f\|_{L^\infty_{\gamma(t)}}\mathrm{1}_{|\alpha|\leq1}+\alpha^{-1}\|\partial_xf\|_{L^\infty_{\gamma(t)}}\mathrm{1}_{|\alpha|>1}.$$
Hence we have
$$\left|\left(q^{(1)}\left(\frac{\tilde{\Delta} f}{\alpha}\right)-q^{(1)}(\partial_xf(x))\right)\frac{\tilde{\Delta}\partial_xf}{\alpha^2}\right|\lesssim \left(\|\partial_x^2f\|_{L^\infty_{\gamma(t)}}\mathrm{1}_{|\alpha|\leq1}+\alpha^{-1}\|\partial_xf\|_{L^\infty_{\gamma(t)}}\mathrm{1}_{|\alpha|>1}\right)^2.$$
Since the right-hand side of the above inequality is a $L^1$ smooth function of $\alpha$, and its $L^1$ norm is bounded by $C\left(\|\partial_xf\|_{L^\infty_{\gamma(t)}}+\|\partial_x^2f\|_{L^\infty_{\gamma(t)}}\right)^2$, using Young's inequality yields
$$\begin{aligned}
&\left\|\int_{\mathbb{R}}\left(q^{(1)}\left(\frac{\tilde{\Delta} f}{\alpha}\right)-q^{(1)}(\partial_xf(x))\right)\frac{\tilde{\Delta}\partial_xf}{\alpha^2}\left(\partial_x^kh(x)-\partial_x^kh(x-\alpha)\right)d\alpha\right\|_{L^2_{\gamma(t)}}\\
\lesssim&\left(\|\partial_xf\|_{L^\infty_{\gamma(t)}}+\|\partial_x^2f\|_{L^\infty_{\gamma(t)}}\right)^2\|\partial_x^kh\|_{L^2_{\gamma(t)}}.
\end{aligned}$$
Collecting the controls for the three terms above, we obtain
$$\begin{aligned}
&\|Safe_{h,1}^{11}\|_{L^2_{\gamma(t)}} \\
\lesssim&\left[\|\partial_xf\|_{L^\infty_{\gamma(t)}}\left(\|\partial_x^2f\|_{H^1_{\gamma(t)}}+\|\partial_x^2f\|_{L^\infty_{\gamma(t)}}+\|\partial_xf\|_{L^\infty_{\gamma(t)}}\right)+\|\partial_x^2f\|_{L^\infty_{\gamma(t)}}^2\right]\|\partial_x^kh\|_{L^2_{\gamma(t)}}\\
\lesssim&\left(\|\partial_xf\|_{L^\infty_{\gamma(t)}}+\|\partial_x^2f\|_{H^1_{\gamma(t)}}\right)^2\|\partial_x^kh\|_{L^2_{\gamma(t)}}.
\end{aligned}$$
Likewise, the other safe terms in class (\romannumeral1) can be controlled in the same way. To sum up, for $w=h$ or $w=\theta$ 
\begin{equation}\begin{aligned}
&\sum_{i\in\{1,2\}}\left(\|Safe_{w,i}^{11}\|_{L^2_{\gamma(t)}}+\|Safe_{w,i}^{22}\|_{L^2_{\gamma(t)}}\right)\\
\lesssim&\left(\|\partial_xf\|_{L^\infty_{\gamma(t)}}+\|\partial_x^2f\|_{H^1_{\gamma(t)}}+\|\partial_xg\|_{L^\infty_{\gamma(t)}}+\|\partial_x^2g\|_{H^1_{\gamma(t)}}\right)\\
&\cdot\left[\left(\|\partial_xf\|_{L^\infty_{\gamma(t)}}+\|\partial_x^2f\|_{H^1_{\gamma(t)}}+\|\partial_xg\|_{L^\infty_{\gamma(t)}}+\|\partial_x^2g\|_{H^1_{\gamma(t)}}\right)\|\partial_x^kw\|_{L^2_{\gamma(t)}}\right.\\
&\left.+\left(\|\partial_xw\|_{L^\infty_{\gamma(t)}}+\|\partial_x^2w\|_{H^1_{\gamma(t)}}\right)\left(\|\partial_x^kf\|_{L^2_{\gamma(t)}}+\|\partial_x^kg\|_{L^2_{\gamma(t)}}\right)\right].
\end{aligned}\label{2.64}\end{equation}
\textbf{Controls for class (\romannumeral2)}. First, we decompose $q^{(1)}\left(\frac{\tilde{\Delta} g+2\sigma+\theta(x)}{\alpha}\right)$ into odd and even parts:
$$\begin{aligned}
&q^{(1)}\left(\frac{\tilde{\Delta} g+2\sigma+\theta(x)}{\alpha}\right)=q^{(1)}_{odd}(x,x-\alpha)+q^{(1)}_{even}(x,x-\alpha),
\end{aligned}$$
where
$$\begin{aligned}
&q^{(1)}_{even}(x,x-\alpha)\\
:=&\frac{1}{2}q^{(1)}\left(\frac{\tilde{\Delta} g+2\sigma+\theta(x)}{\alpha}\right)+\frac{1}{2}q^{(1)}\left(\frac{\tilde{\Delta} g-2\sigma-\theta(x-\alpha)}{\alpha}\right)\\
=&-\frac{\alpha^3(\tilde{\Delta} f+\tilde{\Delta} g)}{2}\left[\left((\tilde{\Delta} g+2\sigma+\theta(x))^2+\alpha^2\right)^{-2}+\left((\tilde{\Delta} g-2\sigma-\theta(x-\alpha))^2+\alpha^2\right)^{-2}\right]\\
&-\frac{\alpha^3(4\sigma+\theta(x)+\theta(x-\alpha))}{2}\left[\left((\tilde{\Delta} g+2\sigma+\theta(x))^2+\alpha^2\right)^{-2}-\left((\tilde{\Delta} g-2\sigma-\theta(x-\alpha))^2+\alpha^2\right)^{-2}\right],
\end{aligned}$$
$$\begin{aligned}
&q^{(1)}_{odd}(x,x-\alpha)\\
:=&\frac{1}{2}q^{(1)}\left(\frac{\tilde{\Delta} g+2\sigma+\theta(x)}{\alpha}\right)-\frac{1}{2}q^{(1)}\left(\frac{\tilde{\Delta} g-2\sigma-\theta(x-\alpha)}{\alpha}\right)\\
=&-\frac{\alpha^3(\tilde{\Delta} f+\tilde{\Delta} g)}{2}\left[\left((\tilde{\Delta} g+2\sigma+\theta(x))^2+\alpha^2\right)^{-2}-\left((\tilde{\Delta} g-2\sigma-\theta(x-\alpha))^2+\alpha^2\right)^{-2}\right]\\
&-\frac{\alpha^3(4\sigma+\theta(x)+\theta(x-\alpha))}{2}\left[\left((\tilde{\Delta} g+2\sigma+\theta(x))^2+\alpha^2\right)^{-2}+\left((\tilde{\Delta} g-2\sigma-\theta(x-\alpha))^2+\alpha^2\right)^{-2}\right].
\end{aligned}$$
Let $K_{g}^{12}(x,x-\alpha):=q^{(1)}_{even}(x,x-\alpha)\frac{\tilde{\Delta}\partial_xg}{\alpha^2}$,  $L_{g}^{12}(x,x-\alpha):=q^{(1)}_{odd}(x,x-\alpha)\frac{\tilde{\Delta}\partial_xg}{\alpha^2},$ and define the operator $(T_{g}^{12}h)(x):=\int_{\mathbb{R}}K_{g}^{12}(x,x-\alpha)h(x-\alpha)d\alpha.$ To derive the control for $Safe_{h,1}^{12}$, we apply the $T1$ theorem \cite[Page 236-253]{grafakos2014modern} upon $T_{g}^{12}$. To this end, we check in the following lemma that the operator $T_{g}^{12}$ and its kernel $K_{g}^{12}$ satisfy the conditions required by the $T1$ theorem.
\begin{lem}
Suppose the solution satisfies the condition
\begin{equation}
\|\partial_xf\|_{L^\infty_{\gamma(t)}}+\|\partial_xg\|_{L^\infty_{\gamma(t)}}+\|\partial_x^2f\|_{H^1_{\gamma(t)}}+\|\partial_x^2g\|_{H^1_{\gamma(t)}}+\sigma^{-1}\|\partial_x\theta\|_{L^2_{\gamma(t)}}\leq\delta_1
\label{2.65}\end{equation} for a universal small constant $\delta_1>0$. Then $K_{g}^{12}$ is a standard antisymmetric kernel, and $T_{g}^{12}(1)$ is $L^\infty$.  
\label{lem2.6}\end{lem}
\begin{proof}
By the construction of $q^{(1)}_{even}$, we immediately have $q^{(1)}_{even}(x,y)=q^{(1)}_{even}(y,x)$, and thus $K_{g}^{12}(x,y)=-K_{g}^{12}(y,x)$, i.e. $K_{g}^{12}$ is antisymmetric. For the size condition, by an elementary computation using condition (\ref{2.65}), there is 
$$\begin{aligned}
&\left|\left((\tilde{\Delta} g+2\sigma+\theta(x))^2+\alpha^2\right)^{-2}-\left((\tilde{\Delta} g-2\sigma-\theta(x-\alpha))^2+\alpha^2\right)^{-2}\right|\\
\lesssim&\left(\|\partial_xf\|_{L^\infty_{\gamma(t)}}+\|\partial_xg\|_{L^\infty_{\gamma(t)}}\right)(\alpha^2+(2\sigma)^2)^{-2},
\end{aligned}$$
so that 
\begin{equation}|q^{(1)}_{even}(x,y)|\lesssim\left(\|\partial_xf\|_{L^\infty_{\gamma(t)}}+\|\partial_xg\|_{L^\infty_{\gamma(t)}}\right)\frac{|x-y|^4}{(|x-y|^2+(2\sigma)^2)^2},
\label{2.66}\end{equation}
$$\begin{aligned}|K_{g}^{12}(x,y)|
&\lesssim\left(\|\partial_xf\|_{L^\infty_{\gamma(t)}}+\|\partial_xg\|_{L^\infty_{\gamma(t)}}\right)\|\partial_x^2g\|_{L^\infty_{\gamma(t)}}|x-y|^{-1}.
\end{aligned}$$
For the regularity condition, suppose now $|y_1-y_2|\leq\frac{1}{2}\max\left\{|x-y_1|,|x-y_2|\right\}$ and consequently $|x-y_1|\simeq|x-y_2|$. Through a lengthy but elementary computation with the help of Lemma \ref{lem2.2}, it can be shown that 
$$\begin{aligned}
&|q^{(1)}_{even}(x,y_1)-q^{(1)}_{even}(x,y_2)|\\
\lesssim&\left(1+\sigma^{-1}\|\theta\|_{L^\infty_{\gamma(t)}}+\|\partial_xf\|_{L^\infty_{\gamma(t)}}+\|\partial_xg\|_{L^\infty_{\gamma(t)}}\right)\left(\|\partial_xf\|_{L^\infty_{\gamma(t)}}+\|\partial_xg\|_{L^\infty_{\gamma(t)}}\right)\\
&\cdot|y_1-y_2|\left(|x-y_1|^3+|x-y_2|^3\right)\left[(|x-y_1|^2+(2\sigma)^2)^{-2}+(|x-y_2|^2+(2\sigma)^2)^{-2}\right].
\end{aligned}$$
Meanwhile, we have
$|\partial_xg(y_1)-\partial_xg(y_2)|\leq\|\partial_x^2g\|_{L^\infty}|y_1-y_2|$ and $|x-y_1|^{-1}-|x-y_2|^{-1}\leq\frac{|y_1-y_2|}{|x-y_1||x-y_2|}$. By Lemma \ref{lem2.2} again, we conclude that
$$|K_{g}^{12}(x,y_1)-K_{g}^{12}(x,y_2)|\lesssim\|\partial_x^2g\|_{L^\infty_{\gamma(t)}}\left(\|\partial_xf\|_{L^\infty_{\gamma(t)}}+\|\partial_xg\|_{L^\infty_{\gamma(t)}}\right)|y_1-y_2|(|x-y_1|^{-2}+|x-y_2|^{-2}).$$
Notice that $K_{g}^{12}(x,y)$ is well-defined on the diagonal $x=y$ since $q^{(1)}_{even}$ provides the factor $\alpha^3$. We write
$$\begin{aligned}
&(T_{g}^{12}1)(x)\\
=&\int_{\mathbb{R}}q^{(1)}_{even}(x,x-\alpha)\frac{\tilde{\Delta}\partial_xg}{\alpha^2}d\alpha\\
=&\int_{\alpha>0}\alpha^{-2}\left(q^{(1)}_{even}(x,x-\alpha)(\partial_xg(x)-\partial_xg(x-\alpha))+q^{(1)}_{even}(x,x+\alpha)(\partial_xg(x)-\partial_xg(x+\alpha))\right)d\alpha\\
=&\int_{\alpha>0}q^{(1)}_{even}(x,x-\alpha)\frac{2\partial_xg(x)-\partial_xg(x-\alpha)-\partial_xg(x+\alpha)}{\alpha^2}d\alpha\\
&+\int_{\alpha>0}\frac{q^{(1)}_{even}(x,x+\alpha)-q^{(1)}_{even}(x,x-\alpha)}{\alpha}\frac{\partial_xg(x)-\partial_xg(x+\alpha)}{\alpha}d\alpha.
\end{aligned}$$
For the first term, we use (\ref{2.66}) and
$$\begin{aligned}
&\left|\alpha^{-2}(2\partial_xg(x)-\partial_xg(x-\alpha)-\partial_xg(x+\alpha))\right|\\
\lesssim&\left(M[\partial_x^3g](x-\alpha)+M[\partial_x^3g(x+\alpha)]\right)\mathrm{1}_{|\alpha|\leq1}+\alpha^{-1}\left(M[\partial_x^2g](x-\alpha)+M[\partial_x^2g](x+\alpha)\right)\mathrm{1}_{|\alpha|>1}
\end{aligned}$$
to obtain
$$\begin{aligned}
&\left|\int_{\alpha>0}q^{(1)}_{even}(x,x-\alpha)\frac{2\partial_xg(x)-\partial_xg(x-\alpha)-\partial_xg(x+\alpha)}{\alpha^2}d\alpha\right|\\
\lesssim&\left(\|\partial_xf\|_{L^\infty_{\gamma(t)}}+\|\partial_xg\|_{L^\infty_{\gamma(t)}}\right)\|\partial_x^2g\|_{H^1_{\gamma(t)}}.
\end{aligned}$$
For the second term, a direct computation with the help of Lemma \ref{lem2.2} gives
$$\begin{aligned}
&\left|(g(x)-g(x-\alpha)+2\sigma+\theta(x))-(g(x+\alpha)-g(x)+2\sigma+\theta(x+\alpha))\right|\\
=&\frac{1}{2}\left|\left(2f(x)+2g(x)-f(x-\alpha)-g(x-\alpha)-f(x+\alpha)-g(x+\alpha)\right)+(\theta(x-\alpha)-\theta(x+\alpha))\right|\\
\lesssim&\alpha^2\left(M[\partial_x^2f+\partial_x^2g](x+\alpha)+M[\partial_x^2f+\partial_x^2g](x-\alpha)\right)+|\alpha|M[\partial_x\theta](x+\alpha).
\end{aligned}$$ 
and likewise,
$$\begin{aligned}
&\left|(g(x)-g(x-\alpha)-2\sigma-\theta(x-\alpha))-(g(x+\alpha)-g(x)-2\sigma-\theta(x))\right|\\
\lesssim&\alpha^2\left(M[\partial_x^2f+\partial_x^2g](x+\alpha)+M[\partial_x^2f+\partial_x^2g](x-\alpha)\right)+|\alpha|M[\partial_x\theta](x+\alpha).
\end{aligned}$$
Then it follows by  Lemma \ref{lem2.2} that
$$\begin{aligned}
&\left|\alpha^{-1}\left(q^{(1)}_{even}(x,x+\alpha)-q^{(1)}_{even}(x,x-\alpha)\right)\right|\\
\lesssim
&\left(\|\partial_xf\|_{L^\infty_{\gamma(t)}}+\|\partial_xg\|_{L^\infty}\right)\left[\frac{\alpha^2}{\alpha^2+(2\sigma)^2}\left(M[\partial_x^2f+\partial_x^2g](x-\alpha)+M[\partial_x^2f+\partial_x^2g](x+\alpha)\right)\right.\\
&+\left.\frac{\alpha}{\alpha^2+(2\sigma)^2}M[\partial_x\theta](x+\alpha)\right]+M[\partial_x^2f+\partial_x^2g](x+\alpha)+M[\partial_x^2f+\partial_x^2g](x-\alpha)\\
\lesssim&(1+\|\partial_xf\|_{L^\infty_{\gamma(t)}}+\|\partial_xg\|_{L^\infty_{\gamma(t)}})\left(M[\partial_x^2f+\partial_x^2g](x-\alpha)+M[\partial_x^2f+\partial_x^2g](x+\alpha)\right)\\
&+(\|\partial_xf\|_{L^\infty_{\gamma(t)}}+\|\partial_xg\|_{L^\infty_{\gamma(t)}})\sigma^{-1}M[\partial_x\theta](x+\alpha).
\end{aligned}$$
Together with $\left|\alpha^{-1}(\partial_xg(x)-\partial_xg(x+\alpha))\right|\leq M[\partial_x^2g](x+\alpha)$, we obtain
$$\begin{aligned}
&\left|\int_{\alpha>0}\frac{q^{(1)}_{even}(x,x+\alpha)-q^{(1)}_{even}(x,x-\alpha)}{\alpha}\frac{\partial_xg(x)-\partial_xg(x+\alpha)}{\alpha}d\alpha\right|\\
\lesssim&\left(\|\partial_x^2f\|_{L^2_{\gamma(t)}}+\|\partial_x^2g\|_{L^2_{\gamma(t)}}\right)\|\partial_x^2g\|_{L^2_{\gamma(t)}}+\left(\|\partial_xf\|_{L^\infty_{\gamma(t)}}+\|\partial_xg\|_{L^\infty_{\gamma(t)}}\right)\sigma^{-1}\|\partial_x\theta\|_{L^2_{\gamma(t)}}\|\partial_x^2g\|_{L^2_{\gamma(t)}}.
\end{aligned}$$
Collecting the above bounds gives
$$\begin{aligned}|T_{g}^{12}1|\lesssim&\left[\left(\|\partial_xf\|_{L^\infty_{\gamma(t)}}+\|\partial_xg\|_{L^\infty_{\gamma(t)}}\right)\left(1+\sigma^{-1}\|\partial_x\theta\|_{L^2_{\gamma(t)}}\right)+\|\partial_x^2f\|_{L^2_{\gamma(t)}}+\|\partial_x^2g\|_{L^2_{\gamma(t)}}\right]\|\partial_x^2g\|_{H^1_{\gamma(t)}}\\
\lesssim&\left(\|\partial_xf\|_{L^\infty_{\gamma(t)}}+\|\partial_xg\|_{L^\infty_{\gamma(t)}}+\|\partial_x^2f\|_{H^1_{\gamma(t)}}+\|\partial_x^2g\|_{H^1_{\gamma(t)}}\right)\|\partial_x^2g\|_{H^1_{\gamma(t)}}.\end{aligned}$$
\end{proof}
Now applying the $T1$ theorem upon $T_{g}^{12}$ yields
$$\begin{aligned}
\|T_{g}^{12}\partial_x^kh\|_{L^2_{\gamma(t)}}
\lesssim&\left(\|\partial_xf\|_{L^\infty_{\gamma(t)}}+\|\partial_xg\|_{L^\infty_{\gamma(t)}}+\|\partial_x^2f\|_{L^2_{\gamma(t)}}+\|\partial_x^2g\|_{L^2_{\gamma(t)}}\right)\|\partial_x^2g\|_{H^1_{\gamma(t)}}\|\partial_x^kh\|_{L^2_{\gamma(t)}}.\end{aligned}$$
Meanwhile, since $T_{g}^{12}1$ is $L^\infty$, we also have $$\left\|\int_{\mathbb{R}}K_{g}^{12}(x,x-\alpha)\partial_x^kh(x)d\alpha\right\|_{L^2_{\gamma(t)}}\lesssim\|T_{g}^{12}1\|_{L^\infty_{\gamma(t)}}\|\partial_x^kh\|_{L^2_{\gamma(t)}},$$
and it follows
$$\begin{aligned}
&\left\|\int_{\mathbb{R}}q^{(1)}_{even}(x,x-\alpha)\frac{\tilde{\Delta}\partial_xg}{\alpha}\frac{\tilde{\Delta}\partial_x^kh}{\alpha}d\alpha\right\|_{L^2_{\gamma(t)}}\\
\lesssim&\left(\|\partial_xf\|_{L^\infty_{\gamma(t)}}+\|\partial_xg\|_{L^\infty_{\gamma(t)}}+\|\partial_x^2f\|_{L^2_{\gamma(t)}}+\|\partial_x^2g\|_{L^2_{\gamma(t)}}\right)\|\partial_x^2g\|_{H^1_{\gamma(t)}}\|\partial_x^kh\|_{L^2_{\gamma(t)}}.
\end{aligned}$$
In order to control $Safe_{h,1}^{12}$, it remains to bound $\int_{\mathbb{R}}L^{12}_{g}(x,x-\alpha)(\partial_x^kh(x)-\partial_x^kh(x-\alpha))d\alpha$. By observing that
\begin{equation}|q^{(1)}_{odd}(x,x-\alpha)|\lesssim\left(1+\|\partial_xf\|_{L^\infty_{\gamma(t)}}+\|\partial_xg\|_{L^\infty_{\gamma(t)}}\right)\frac{|\alpha^3|\sigma}{(\alpha^2+(2\sigma)^2)^2},\label{2.67}\end{equation}
Young's inequality then gives
$$\begin{aligned}
&\left\|\int_{\mathbb{R}}L^{12}_{g}(x,x-\alpha)\left(\partial_x^kh(x)-\partial_x^kh(x-\alpha)\right)d\alpha\right\|_{L^2_{\gamma(t)}}\\
\lesssim&\|\partial_x^2g\|_{L^\infty_{\gamma(t)}}\left\|\int_{\mathbb{R}}\frac{\alpha^2\sigma}{(\alpha^2+(2\sigma)^2)^2}\left(|\partial_x^kh(x)|+|\partial_x^kh(x-\alpha)|\right)d\alpha\right\|_{L^2_{\gamma(t)}}\\
\lesssim&\|\partial_x^2g\|_{L^\infty_{\gamma(t)}}\|\partial_x^kh\|_{L^2_{\gamma(t)}}.
\end{aligned}$$
Therefore, under the condition (\ref{2.65}), we conclude that
$$\begin{aligned}
\|Safe_{h,1}^{12}\|_{L^2_{\gamma(t)}}\lesssim\|\partial_x^2g\|_{H^1_{\gamma(t)}}\|\partial_x^kh\|_{L^2_{\gamma(t)}}.
\end{aligned}$$
Since $Safe_{w,i}^{12}$, $Safe_{w,i}^{21}$, $w\in\{h,\theta\}$, $i\in\{1,3\}$ share the same structure, repeating the above arguments yields for $w=h,\theta$ that
\begin{equation}\begin{aligned}
&\sum_{i\in\{1,3\}}\left(\|Safe_{w,i}^{12}\|_{L^2_{\gamma(t)}}+\|Safe^{21}_{w,i}\|_{L^2_{\gamma(t)}}\right)\\
\lesssim&\left(\|\partial_x^2f\|_{H^1_{\gamma(t)}}+\|\partial_x^2g\|_{H^1_{\gamma(t)}}\right)\|\partial_x^kw\|_{L^2_{\gamma(t)}}+\left(\|\partial_x^kf\|_{L^2_{\gamma(t)}}+\|\partial_x^kg\|_{L^2_{\gamma(t)}}\right)\|\partial_x^2w\|_{H^1_{\gamma(t)}}.
\end{aligned}\label{2.68}\end{equation}
Next, we are devoted to the control over $Safe_{h,4}^{12}$. Introduce the decomposition
$$q^{(1)}\left(\frac{\tilde{\Delta} g+2\sigma+\theta(x)}{\alpha}\right)\frac{\tilde{\Delta}\partial_xh}{\alpha^2}=K^{12}_h(x,x-\alpha)+L_h^{12}(x,x-\alpha),$$
$$K^{12}_h(x,x-\alpha)=q^{(1)}_{even}(x,x-\alpha)\frac{\tilde{\Delta}\partial_xh}{\alpha^2},\quad L^{12}_h(x,x-\alpha)=q^{(1)}_{odd}(x,x-\alpha)\frac{\tilde{\Delta}\partial_xh}{\alpha^2}.$$
By mimicking the estimate of $T_{g}^{12}1$ as in the proof of Lemma \ref{lem2.6}, it can be shown that
$$\left|\int_{\mathbb{R}}K^{12}_h(x,x-\alpha)d\alpha\right|\lesssim\left(\|\partial_xf\|_{L^\infty_{\gamma(t)}}+\|\partial_xg\|_{L^\infty_{\gamma(t)}}+\|\partial_x^2f\|_{L^2_{\gamma(t)}}+\|\partial_x^2g\|_{L^2_{\gamma(t)}}\right)\|\partial_x^2h\|_{H^1_{\gamma(t)}}.$$
Meanwhile, using (\ref{2.67}) again, we have
$$\left|\int_{\mathbb{R}}L^{12}_h(x,x-\alpha)d\alpha\right|\lesssim\|\partial_x^2h\|_{L^\infty_{\gamma(t)}}\int_{\mathbb{R}}\frac{\alpha^2\sigma}{(\alpha^2+(2\sigma)^2)^2}d\alpha\lesssim\|\partial_x^2h\|_{H^1_{\gamma(t)}}.$$
Taking the sum yields $\|Safe_{h,4}^{12}\|_{L^2_{\gamma(t)}}\lesssim\|\partial_x^2h\|_{H^1_{\gamma(t)}}\|\partial_x^k\theta\|_{L^2_{\gamma(t)}}.$ In the same manner, we can obtain the controls over $Safe_{h,4}^{21}$, $Safe_{\theta,4}^{12}$ and $Safe_{\theta,4}^{21}$. For $w=h,\theta$,
\begin{equation}\begin{aligned}
\|Safe_{w,4}^{12}\|_{L^2_{\gamma(t)}}+\|Safe_{w,4}^{21}\|_{L^2_{\gamma(t)}}\lesssim\|\partial_x^2w\|_{H^1_{\gamma(t)}}\|\partial_x^k\theta\|_{L^2_{\gamma(t)}}.
\end{aligned}\label{2.69}\end{equation}
We decompose $Safe_{h,2}^{12}(x)$ as
$$\begin{aligned}
&\int_{\mathbb{R}}q^{(1)}\left(\frac{\tilde{\Delta} g+2\sigma+\theta(x)}{\alpha}\right)\frac{\partial_x\theta(x)}{\alpha}\frac{\tilde{\Delta}\partial_x^kh}{\alpha}d\alpha\\
=&-2\int_{\mathbb{R}}\frac{\alpha\tilde{\Delta} g\cdot\partial_x\theta(x)}{(\alpha^2+(\tilde{\Delta} g+2\sigma+\theta(x))^2)^2}\tilde{\Delta}\partial_x^kh\,d\alpha-2\int_{\mathbb{R}}\frac{\alpha(2\sigma+\theta(x))\partial_x\theta(x)}{(\alpha^2+(\tilde{\Delta} g+2\sigma+\theta(x))^2)^2}\tilde{\Delta}\partial_x^kh\,d\alpha.
\end{aligned}$$
The first term can be controlled easily by
$$\left|\int_{\mathbb{R}}\frac{\alpha\tilde{\Delta} g\cdot\partial_x\theta(x)}{(\alpha^2+(\tilde{\Delta} g+2\sigma+\theta(x))^2)^2}\tilde{\Delta}\partial_x^kh\,d\alpha\right|\lesssim\sigma^{-1}\|\partial_x\theta\|_{L^\infty_{\gamma(t)}}\|\partial_xg\|_{L^\infty_{\gamma(t)}}\|\partial_x^kh\|_{L^2_{\gamma(t)}}.$$
To control the second term, let $\varphi\in L^2_{\gamma(t)}$ be arbitrary. By H\"{o}lder inequality, it holds
$$\begin{aligned}
&\left|\int_{\Gamma_{\pm}(t)}\left(\int_{\mathbb{R}}\frac{\alpha(2\sigma+\theta(x))\partial_x\theta(x)}{(\alpha^2+(\tilde{\Delta} g+2\sigma+\theta(x))^2)^2}\tilde{\Delta}\partial_x^kh\,d\alpha\right)\varphi(x)dx\right|\\
\lesssim&\sigma^{-1}\|\partial_x\theta\|_{L^\infty_{\gamma(t)}}\left(1+\sigma^{-1}\|\theta\|_{L^\infty_{\gamma(t)}}\right)\left(\int_{\Gamma_{\pm}(t)}\int_{\mathbb{R}}\frac{\sigma^2}{\alpha^2+(2\sigma)^2}\left|\frac{\tilde{\Delta}\partial_x^kh}{\alpha}\right|^2d\alpha dx\right)^\frac{1}{2}\\
&\cdot\left(\int_{\Gamma_{\pm}(t)}\int_{\mathbb{R}}\frac{\alpha^4\sigma^2}{(\alpha^2+(2\sigma)^2)^3}|\varphi(x)|^2d\alpha dx\right)^\frac{1}{2}\\
\lesssim&\sigma^{-\frac{1}{2}}\|\partial_x\theta\|_{L^\infty_{\gamma(t)}}\left(\int_{\Gamma_{\pm}(t)}\int_{\mathbb{R}}\frac{\sigma^2}{\alpha^2+(2\sigma)^2}\left|\frac{\tilde{\Delta}\partial_x^kh}{\alpha}\right|^2d\alpha dx\right)^\frac{1}{2}\|\varphi\|_{L^2_{\gamma(t)}}.
\end{aligned}$$
Hence
$$\begin{aligned}
&\left\|\int_{\mathbb{R}}\frac{\alpha(2\sigma+\theta(x))\partial_x\theta(x)}{(\alpha^2+(\tilde{\Delta} g+2\sigma+\theta(x))^2)^2}\tilde{\Delta}\partial_x^kh\,d\alpha\right\|_{L^2_{\gamma(t)}}\\
\lesssim&\sigma^{-\frac{1}{2}}\|\partial_x\theta\|_{L^\infty_{\gamma(t)}}\left(\int_{\Gamma_{\pm}(t)}\int_{\mathbb{R}}\frac{\sigma^2}{\alpha^2+(2\sigma)^2}\left|\frac{\tilde{\Delta}\partial_x^kh}{\alpha}\right|^2d\alpha dx\right)^\frac{1}{2}.
\end{aligned}$$
Recall from Proposition \ref{prop2.3} that $D_{22}^0(x,x-\alpha)\simeq\frac{1}{\alpha^2}\frac{(2\sigma)^2}{\alpha^2+(2\sigma)^2}$. Summing the above two bounds yields
$$\begin{aligned}
\|Safe_{h,2}^{12}\|_{L^2_{\gamma(t)}}\lesssim&\sigma^{-1}\|\partial_x\theta\|_{L^\infty_{\gamma(t)}}\|\partial_xg\|_{L^\infty_{\gamma(t)}}\|\partial_x^kh\|_{L^2_{\gamma(t)}}\\
&+\sigma^{-\frac{1}{2}}\|\partial_x\theta\|_{L^\infty_{\gamma(t)}}\left(\int_{\Gamma_{\pm}(t)}\int_{\mathbb{R}}D^0_{22}(x,x-\alpha)|\tilde{\Delta}\partial_x^kh|^2d\alpha dx\right)^\frac{1}{2}.
\end{aligned}$$
The controls over $Safe_{h,2}^{21}$, $Safe_{\theta,2}^{12}$ and $Safe_{\theta,2}^{21}$ can be obtained in the same way. To sum up, we have for $w=h,\theta$
\begin{equation}\begin{aligned}
&\|Safe_{w,2}^{12}\|_{L^2_{\gamma(t)}}+\|Safe_{w,2}^{21}\|_{L^2_{\gamma(t)}}\\
\lesssim&\sigma^{-1}\|\partial_x\theta\|_{L^\infty_{\gamma(t)}}\left(\|\partial_xg\|_{L^\infty_{\gamma(t)}}+\|\partial_xf\|_{L^\infty_{\gamma(t)}}\right)\|\partial_x^kw\|_{L^2_{\gamma(t)}}\\
&+\sigma^{-\frac{1}{2}}\|\partial_x\theta\|_{L^\infty_{\gamma(t)}}\left(\int_{\Gamma_{\pm}(t)}\int_{\mathbb{R}}D^0_{22}(x,x-\alpha)|\tilde{\Delta}\partial_x^kw|^2d\alpha dx\right)^\frac{1}{2}.\end{aligned}\label{2.70}\end{equation}
\textbf{Controls for class (\romannumeral3)}.
Note that $Safe_{\theta,5}^{12}$ has the same structure as $Safe_{h,4}^{12}$, and $Safe_{\theta,7}^{12}$ has the same structure as $Safe_{h,2}^{12}$. Hence by repeating the corresponding arguments, similar estimates can be obtained for $Safe_{\theta,5}^{12}$, $Safe_{\theta,7}^{12}$. Moreover, by replacing $g$ by $f$, we also have the estimates of $Safe_{\theta,5}^{21}$, $Safe_{\theta,7}^{21}$.
\begin{equation}\begin{aligned}
\|Safe_{\theta,5}^{12}\|_{L^2_{\gamma(t)}}+\|Safe_{\theta,5}^{21}\|_{L^2_{\gamma(t)}}\lesssim\left(\|\partial_x^2g\|_{H^1_{\gamma(t)}}+\|\partial_x^2f\|_{H^1_{\gamma(t)}}\right)\|\partial_x^k\theta\|_{L^2_{\gamma(t)}},
\end{aligned}\label{2.71}\end{equation}
\begin{equation}\begin{aligned}
&\|Safe_{\theta,7}^{12}\|_{L^2_{\gamma(t)}}+\|Safe_{\theta,7}^{21}\|_{L^2_{\gamma(t)}}\\
\lesssim& \sigma^{-1}\|\partial_x\theta\|_{L^\infty_{\gamma(t)}}\|\partial_xg\|_{L^\infty_{\gamma(t)}}\|\partial_x^kg\|_{L^2_{\gamma(t)}}+\sigma^{-1}\|\partial_x\theta\|_{L^\infty_{\gamma(t)}}\|\partial_xf\|_{L^\infty_{\gamma(t)}}\|\partial_x^kf\|_{L^2_{\gamma(t)}}\\
&+\sigma^{-\frac{1}{2}}\|\partial_x\theta\|_{L^\infty_{\gamma(t)}}\left(\int_{\Gamma_{\pm}(t)}\int_{\mathbb{R}}D_{22}^0(x,x-\alpha)|\tilde{\Delta}\partial_x^kf|^2d\alpha dx\right)^{\frac{1}{2}}\\
&+\sigma^{-\frac{1}{2}}\|\partial_x\theta\|_{L^\infty_{\gamma(t)}}\left(\int_{\Gamma_{\pm}(t)}\int_{\mathbb{R}}D_{22}^0(x,x-\alpha)|\tilde{\Delta}\partial_x^kg|^2d\alpha dx\right)^{\frac{1}{2}}.
\end{aligned}\label{2.72}\end{equation}
In order to control $Safe_{\theta,6}^{12}$ and $Safe_{\theta,6}^{21}$, it suffices to bound $\int_{\mathbb{R}}q^{(1)}\left(\frac{\tilde{\Delta} g+2\sigma+\theta(x)}{\alpha}\right)\alpha^{-2}d\alpha$. Since $q^{(1)}$ is an odd function, we have
$$\int_{\mathbb{R}}q^{(1)}\left(\frac{\tilde{\Delta} g+2\sigma+\theta(x)}{\alpha}\right)\alpha^{-2}d\alpha=\int_{\mathbb{R}}\left[q^{(1)}\left(\frac{\tilde{\Delta} g+2\sigma+\theta(x)}{\alpha}\right)-q^{(1)}\left(\frac{2\sigma+\theta(x)}{\alpha}\right)\right]\alpha^{-2}d\alpha.$$
In view of $\|\partial_xg\|_{L^\infty_{\gamma(t)}}\leq w_0\lesssim1$, it holds
$$\begin{aligned}
&\left|q^{(1)}\left(\frac{\tilde{\Delta} g+2\sigma+\theta(x)}{\alpha}\right)-q^{(1)}\left(\frac{2\sigma+\theta(x)}{\alpha}\right)\right|\\
\leq&\frac{|\tilde{\Delta} g||\alpha^3|}{\left(\alpha^2+(\tilde{\Delta} g+2\sigma+\theta(x))\right)^2}+2|2\sigma+\theta(x)||\alpha^3|\left|\frac{1}{\left(\alpha^2+(\tilde{\Delta} g+2\sigma+\theta(x))\right)^2}-\frac{1}{(\alpha^2+(2\sigma+\theta(x)))^2}\right|\\
\lesssim&\|\partial_xg\|_{L^\infty_{\gamma(t)}}\frac{\alpha^4}{(\alpha^2+(2\sigma)^2)^2}.
\end{aligned}$$
Hence $\left|\int_{\mathbb{R}}q^{(1)}\left(\frac{\tilde{\Delta} g+2\sigma+\theta(x)}{\alpha}\right)\alpha^{-2}d\alpha\right|\lesssim\sigma^{-1}\|\partial_xg\|_{L^\infty_{\gamma(t)}}$, and it follows that
$$\|Safe_{\theta,6}^{12}\|_{L^2_{\gamma(t)}}\lesssim\sigma^{-1}\|\partial_x\theta\|_{L^\infty_{\gamma(t)}}\|\partial_xg\|_{L^\infty_{\gamma(t)}}\|\partial_x^k\theta\|_{L^2_{\gamma(t)}}.$$ The same bound with $g$ replaced by $f$ holds for $\|Safe_{\theta,6}^{21}\|_{L^2_{\gamma(t)}}$. That is
\begin{equation}
\|Safe_{\theta,6}^{12}\|_{L^2_{\gamma(t)}}+\|Safe_{\theta,6}^{21}\|_{L^2_{\gamma(t)}}\lesssim\sigma^{-1}\|\partial_x\theta\|_{L^\infty_{\gamma(t)}}\left(\|\partial_xf\|_{L^\infty_{\gamma(t)}}+\|\partial_xg\|_{L^\infty_{\gamma(t)}}\right)\|\partial_x^k\theta\|_{L^2_{\gamma(t)}}.
\label{2.73}\end{equation}
Recall that $f=h+\mu_1\theta$, $g=h-\mu_2\theta$, and $\|\partial_x^jf\|_{X}+\|\partial_x^jg\|_{X}\lesssim\|\partial_x^jh\|_{X}+\|\partial_x^j\theta\|_X$ for all $j\geq0$ and $X=L^2_{\gamma(t)}, H^1_{\gamma(t)}$ or $L^\infty_{\gamma(t)}$. We collect (\ref{2.64}) and (\ref{2.68}-\ref{2.73}) to conclude with the lemma below.
\begin{lem}\label{summary3.6}
Suppose that the solution satisfies conditions (\ref{2.17}) and (\ref{2.65}). Then all the safe terms satisfy the bound
$$\begin{aligned}
&\sum_{w\in\{h,\theta\}}\sum_{i\in\{1,2\}}\left(\|Safe_{w,i}^{11}\|_{L^2_{\gamma(t)}}+\|Safe_{w,i}^{22}\|_{L^2_{\gamma(t)}}\right)+\sum_{5\leq i \leq7}\left(\|Safe_{\theta,i}^{12}\|_{L^2_{\gamma(t)}}+\|Safe_{\theta,i}^{21}\|_{L^2_{\gamma(t)}}\right)\\
&+\sum_{w\in\{h,\theta\}}\sum_{1\leq i\leq4}\left(\|Safe_{w,i}^{12}\|_{L^2_{\gamma(t)}}+\|Safe_{w,i}^{21}\|_{L^2_{\gamma(t)}}\right)\\
\lesssim&\left(1+\sigma^{-1}\|\partial_x\theta\|_{L^\infty_{\gamma(t)}}\right)\left(\sum_{w\in\{h,\theta\}}\|\partial_xw\|_{L^\infty_{\gamma(t)}}+\|\partial_x^2w\|_{H^1_{\gamma(t)}}\right)\left(\|\partial_x^kh\|_{L^2_{\gamma(t)}}+\|\partial_x^k\theta\|_{L^2_{\gamma(t)}}\right)\\
&+\sigma^{-\frac{1}{2}}\|\partial_x\theta\|_{L^\infty_{\gamma(t)}}\left(\int_{\Gamma_{\pm}(t)}\int_{\mathbb{R}}D_{22}^0(x,x-\alpha)|\tilde{\Delta}\partial_x^kh|^2d\alpha dx\right)^\frac{1}{2}\\
&+\sigma^{-\frac{1}{2}}\|\partial_x\theta\|_{L^\infty_{\gamma(t)}}\left(\int_{\Gamma_{\pm}(,l,rt)}\int_{\mathbb{R}}D_{22}^0(x,x-\alpha)|\tilde{\Delta}\partial_x^k\theta|^2d\alpha dx\right)^\frac{1}{2}
\end{aligned}$$
\end{lem}
\subsection{Commutators-Easy terms}\label{commutators2}
In the next, we are devoted to the controls over the easy terms, which correspond to the case $k-1\geq j\geq2$ or $j=k$, $l\geq2$. For convenience, denote the set of such index $(j,l)$ by $B_k:=\left\{(j,l)\mid 2\leq j\leq k,\,1\leq l\leq j,\,(j,l)\neq(k,1) \right\}.$  \\
\textbf{Controls for class (\romannumeral1)}.
First, in view of condition (\ref{2.17}) and the fact that $q^{(l)}$, $l\geq0$ are smooth functions, we always have $\left|q^{(l)}\left(\frac{\tilde{\Delta} f}{\alpha}\right)\right|+\left|q^{(l)}\left(\frac{\tilde{\Delta} g}{\alpha}\right)\right|\lesssim1$. Let $\tilde{k}:=\left\lfloor\frac{k+1}{2}\right\rfloor$. If $r_a=0$ for all $a>\tilde{k}$, we write
$$\begin{aligned}
&\left|\int_{\mathbb{R}}q^{(l)}\left(\frac{\tilde{\Delta} f}{\alpha}\right)\prod_{a=1}^j\left(\frac{\tilde{\Delta}\partial_x^af}{\alpha}\right)^{r_a}\frac{\tilde{\Delta}\partial_x^{k+1-j}h}{\alpha}d\alpha\right|\\
\lesssim&\int_{|\alpha|\leq1}\left|\prod_{a=1}^j\left(\frac{\tilde{\Delta}\partial_x^af}{\alpha}\right)^{r_a}\right|d\alpha\cdot M[\partial_x^{k+2-j}h](x)+\int_{|\alpha|>1}\left|\alpha^{-1}\prod_{a=1}^j\left(\frac{\tilde{\Delta}\partial_x^af}{\alpha}\right)^{r_a}(\tilde{\Delta}\partial_x^{k+1-j}h)\right|d\alpha.
\end{aligned}$$
By virtue of $\left|\frac{\tilde{\Delta}\partial_x^af}{\alpha}\right|\leq\|\partial_x^{a+1}f\|_{L^\infty_{\gamma(t)}}$, it holds
$$\begin{aligned}
\int_{|\alpha|\leq1}\left|\prod_{a=1}^j\left(\frac{\tilde{\Delta}\partial_x^af}{\alpha}\right)^{r_a}\right|d\alpha\lesssim\prod_{a=1}^j\|\partial_x^{a+1}f\|_{L^\infty_{\gamma(t)}}^{r_a}.  
\end{aligned}$$
Since $r_a=0$ for $a>\tilde{k}$ and $\sum_{a=1}^jr_a=l$, we have by Sobolev embedding that
$$\prod_{a=1}^j\|\partial_x^{a+1}f\|^{r_a}_{L^\infty_{\gamma(t)}}\lesssim\|\partial_x^2f\|_{H^{\tilde{k}}_{\gamma(t)}}^l.$$
Hence the $L^2_{\gamma(t)}$ norm of the part $|\alpha|\leq1$ can be controlled by $\|\partial_x^2f\|_{H^{\tilde{k}}_{\gamma(t)}}^l\|\partial_x^{k+2-j}h\|_{L^2_{\gamma(t)}}$.
For the part $|\alpha|>1$, we have
$$\left|\alpha^{-1}\prod_{a=1}^j\left(\frac{\tilde{\Delta}\partial_x^af}{\alpha}\right)^{r_a}\right|\lesssim|\alpha|^{-1-l}\prod_{a=1}^j\|\partial_x^af\|_{L^\infty_{\gamma(t)}}^{r_a}\lesssim|\alpha|^{-1-l}\left(\|\partial_xf\|_{L^\infty_{\gamma(t)}}+\|\partial_x^2f\|_{H^{\tilde{k}-1}_{\gamma(t)}}\right)^l.$$
Hence by Minkowski 
$$\left\|\int_{|\alpha|>1}\left|\alpha^{-1}\prod_{a=1}^j\left(\frac{\tilde{\Delta}\partial_x^af}{\alpha}\right)^{r_a}(\tilde{\Delta}\partial_x^{k+1-j}h)\right|d\alpha\right\|_{L^2_{\gamma(t)}}\lesssim\left(\|\partial_xf\|_{L^\infty_{\gamma(t)}}+\|\partial_x^2f\|_{H^{\tilde{k}-1}_{\gamma(t)}}\right)^l\|\partial_x^{k+1-j}h\|_{L^2_{\gamma(t)}}.$$
Since $2\leq k+2-j\leq k$, we remark that $\|\partial_x^{k+1-j}h\|_{L^2_{\gamma(t)}}+\|\partial_x^{k+2-j}h\|_{L^2_{\gamma(t)}}\leq
\|h\|_{H^k_{\gamma(t)}}.$
Now suppose otherwise that there exists $a_0>\tilde{k}$ such that $r_{a_0}\neq0$. Recall that $\sum_{a=1}^jar_a=j\leq k$ and $(j,l)\neq
(k,1)$. Hence we must have $a_0+1\leq k$, $j>\tilde
{k}$, $r_{a_0}=1$ and $r_a=0$ for all $a>\tilde{k}$ with $a\neq a_0$. Moreover, $a_0\geq5$ since we choose $k\geq 10$. In this case, we write
$$\begin{aligned}
&\int_{\mathbb{R}}\left|q^{(l)}\left(\frac{\tilde{\Delta} f}{\alpha}\right)\prod_{a=1}^j\left(\frac{\tilde{\Delta}\partial_x^af}{\alpha}\right)^{r_a}\frac{\tilde{\Delta}\partial_x^{k+1-j}h}{\alpha}\right|d\alpha\\
\lesssim&\int_{|\alpha|>1}\left|\alpha^{-1}\prod_{a\neq a_0}\left(\frac{\tilde{\Delta}\partial_x^af}{\alpha}\right)^{r_a}\frac{\tilde{\Delta}\partial_x^{k+1-j}h}{\alpha}(\tilde{\Delta}\partial_x^{a_0}f)\right|d\alpha\\
&+\int_{|\alpha|\leq1}\left|\prod_{a\neq a_0}\left(\frac{\tilde{\Delta}\partial_x^af}{\alpha}\right)^{r_a}\frac{\tilde{\Delta}\partial_x^{k+1-j}h}{\alpha}\right|d\alpha\cdot M[\partial_x^{a_0+1}f](x).
\end{aligned}$$
Since $|\alpha^{-1}\tilde{\Delta}\partial_x^{k+1-j}h|\leq\|\partial_x^{k+2-j}h\|_{L^\infty_{\gamma(t)}}$ and $\tilde{k}+1\geq k+2-j\geq2$, we have
$$\int_{|\alpha|\leq1}\left|\prod_{a\neq a_0}\left(\frac{\tilde{\Delta}\partial_x^af}{\alpha}\right)^{r_a}\frac{\tilde{\Delta}\partial_x^{k+1-j}h}{\alpha}\right|d\alpha\lesssim\|\partial_x^2f\|_{H^{\tilde{k}}_{\gamma(t)}}^{l-1}\|\partial_x^{k+2-j}h\|_{L^\infty_{\gamma(t)}}\lesssim\|\partial_x^2f\|_{H^{\tilde{k}}_{\gamma(t)}}^{l-1}\|\partial_x^2h\|_{H^{\tilde{k}}_{\gamma(t)}}.$$
Hence the $L^2_{\gamma(t)}$ norm of the part $|\alpha|\leq
1$ can be controlled by $\|\partial_x^2f\|_{H^{\tilde{k}}_{\gamma(t)}}^{l-1}\|\partial_x^2h\|_{H^{\tilde{k}}_{\gamma(t)}}\|\partial_x^{a_0+1}f\|_{L^2_{\gamma(t)}}.$ For the part $|\alpha|>1$, there is
$$\begin{aligned}
&\left|\alpha^{-1}\prod_{a\neq a_0}\left(\frac{\tilde{\Delta}\partial_x^af}{\alpha}\right)^{r_a}\frac{\tilde{\Delta}\partial_x^{k+1-j}h}{\alpha}\right|\\
\lesssim&|\alpha|^{-1-l}\prod_{a\neq a_0}\|\partial_x^af\|_{L^\infty_{\gamma(t)}}^{r_a}\|\partial_x^{k+1-j}h\|_{L^\infty_{\gamma(t)}}\\
\lesssim&|\alpha|^{-1-l}\left(\|\partial_xf\|_{L^\infty_{\gamma(t)}}+\|\partial_x^2f\|_{H^{\tilde{k}-1}_{\gamma(t)}}\right)^{l-1}\left(\|\partial_xh\|_{L^\infty_{\gamma(t)}}+\|\partial_x^2h\|_{H^{\tilde{k}-1}_{\gamma(t)}}\right).
\end{aligned}$$
Then using Minkowski and noticing $5\leq a_0\leq k-1$ yields
$$\begin{aligned}
&\left\|\int_{|\alpha|>1}\left|\alpha^{-1}\prod_{a\neq a_0}\left(\frac{\tilde{\Delta}\partial_x^af}{\alpha}\right)^{r_a}\frac{\tilde{\Delta}\partial_x^{k+1-j}h}{\alpha}(\tilde{\Delta}\partial_x^{a_0}f)\right|d\alpha\right\|_{L^2_{\gamma(t)}}\\
\lesssim&\left(\|\partial_xf\|_{L^\infty_{\gamma(t)}}+\|\partial_x^2f\|_{H^{\tilde{k}-1}_{\gamma(t)}}\right)^{l-1}\left(\|\partial_xh\|_{L^\infty_{\gamma(t)}}+\|\partial_x^2h\|_{H^{\tilde{k}-1}_{\gamma(t)}}\right)\|\partial_x^{a_0}f\|_{L^2_{\gamma(t)}}\\
\lesssim&\left(\|\partial_xf\|_{L^\infty_{\gamma(t)}}+\|\partial_x^2f\|_{H^{\tilde{k}-1}_{\gamma(t)}}\right)^{l-1}\left(\|\partial_xh\|_{L^\infty_{\gamma(t)}}+\|\partial_x^2h\|_{H^{\tilde{k}-1}_{\gamma(t)}}\right)\|\partial_x^{5}f\|_{H^{k-6}_{\gamma(t)}}.
\end{aligned}$$
We can also obtain similar controls for other easy terms in class (\romannumeral1) with $f$ replaced by $g$ or $h$ replaced by $\theta$. Using that $\|\partial_x^jf\|_{X}+\|\partial_x^jg\|_{X}\lesssim\|\partial_x^jh\|_X+\|\partial_x^j\theta\|_X$ for all $j\geq0$ and $X=L^\infty_{\gamma(t)}$ or $H^{m}_{\gamma(t)}$, $m\geq0$, we obtain the following control for easy terms of class (i):
\begin{lem}\label{summary3.7-1}
Suppose that condition (\ref{2.17}) holds. Then
\begin{equation}\begin{aligned}
&\sum_{B_k}\sum_{S_{j,l}}\sum_{w^\prime\in\{f,g\}}\left\|\int_{\mathbb{R}}q^{(l)}\left(\frac{\tilde{\Delta} w^\prime}{\alpha}\right)\prod_{a=1}^j\left(\frac{\tilde{\Delta}\partial_x^aw^\prime}{\alpha}\right)^{r_a}\frac{\tilde{\Delta}\partial_x^{k+1-j}w}{\alpha}d\alpha\right\|_{L^2_{\gamma(t)}}\\
\lesssim&(1+\delta_2)^{k-1}\left[\delta_2\|w\|_{H^k_{\gamma(t)}}+\left(\|\partial_x^2w\|_{H^{\tilde{{k}}}_{\gamma(t)}}+\|\partial_xw\|_{L^\infty_{\gamma(t)}}\right)\left(\|\partial_x^5f\|_{H^{k-6}_{\gamma(t)}}+\|\partial_x^5g\|_{H^{k-6}_{\gamma(t)}}\right)\right]\\
\lesssim&\delta_2(1+\delta_2)^{k-1}\left(\|h\|_{H^k_{\gamma(t)}}+\|\theta\|_{H^k_{\gamma(t)}}\right),
\label{2.74}\end{aligned}\end{equation}
where we denote for simplicity
$\delta_2(t):=\|\partial_xh\|_{L^\infty_{\gamma(t)}}+\|\partial_x\theta\|_{L^\infty_{\gamma(t)}}+\|\partial_x^2h\|_{H^{\tilde{k}}_{\gamma(t)}}+\|\partial_x^2\theta\|_{H^{\tilde{k}}_{\gamma(t)}}$.
\end{lem}
\noindent\textbf{Controls for class (\romannumeral2)}. To begin with, we observe that
$q^{(l)}(x)=\frac{l!}{2}\left(\frac{(-i)^l}{(1+ix)^{l+1}}+\frac{i^l}{(1-ix)^{l+1}}\right)$. 
Suppose first that $r_a=0$ for all $a>\tilde{k}$. In view of condition (\ref{2.17}), for any $a\in[1,\tilde{k}]$, it holds
$$\begin{aligned}
\left|\frac{\tilde{\Delta}\partial_x^ag+\partial_x^a\theta(x)}{\alpha+i(\tilde{\Delta} g+2\sigma+\theta(x))}\right|+\left|\frac{\tilde{\Delta}\partial_x^ag+\partial_x^a\theta(x)}{\alpha-i(\tilde{\Delta} g+2\sigma+\theta(x))}\right|\lesssim&\|\partial_x^{a+1}g\|_{L^\infty_{\gamma(t)}}+\sigma^{-1}\|\partial_x^a\theta\|_{L^\infty_{\gamma(t)}}\\
\lesssim&\|\partial_x^2g\|_{H^{\tilde{k}}_{\gamma(t)}}+\sigma^{-1}\|\partial_x\theta\|_{H^{\tilde{k}}_{\gamma(t)}}.    
\end{aligned}$$
Then using Lemma \ref{lem2.2} we obtain
$$\begin{aligned}
&\left|\prod_{a=1}^j\left(\frac{\tilde{\Delta}\partial_x^a g+\partial_x^a\theta(x)}{\alpha+i(\tilde{\Delta} g+2\sigma+\theta(x))}\right)^{r_a}-\prod_{a=1}^j\left(\frac{\partial_x^a\theta(x)}{\alpha+i(\tilde{\Delta} g+2\sigma+\theta(x))}\right)^{r_a}\right|\\
&+\left|\prod_{a=1}^j\left(\frac{\tilde{\Delta}\partial_x^a g+\partial_x^a\theta(x)}{\alpha-i(\tilde{\Delta} g+2\sigma+\theta(x))}\right)^{r_a}-\prod_{a=1}^j\left(\frac{\partial_x^a\theta(x)}{\alpha-i(\tilde{\Delta} g+2\sigma+\theta(x))}\right)^{r_a}\right|\\
\lesssim&\left(\|\partial_x^2g\|_{H^{\tilde{k}}_{\gamma(t)}}+\sigma^{-1}\|\partial_x\theta\|_{H^{\tilde{k}}_{\gamma(t)}}\right)^{l-1}[\alpha^2+(2\sigma)^2]^{-\frac{1}{2}}\sum_{r_a\neq 0}|\tilde{\Delta}\partial_x^ag|.
\end{aligned}$$
Since $|\alpha^{-1}|\geq[\alpha^2+(2\sigma)^2]^{-\frac{1}{2}}$ and $\left|\frac{\alpha}{\alpha\pm i(\tilde{\Delta} g+2\sigma+\theta(x))}\right|\lesssim1$, it follows that
\begin{equation}\begin{aligned}
&\left|\int_{\mathbb{R}}q^{(l)}\left(\frac{\tilde{\Delta} g+2\sigma+\theta(x)}{\alpha}\right)\prod_{a=1}^j\left(\frac{\tilde{\Delta}\partial_x^ag+\partial_x^a\theta(x)}{\alpha}\right)^{r_a}\frac{\tilde{\Delta}\partial_x^{k+1-j}h}{\alpha}d\alpha\right|\\
\lesssim&\left(\|\partial_x^2g\|_{H^{\tilde{k}}_{\gamma(t)}}+\sigma^{-1}\|\partial_x\theta\|_{H^{\tilde{k}}_{\gamma(t)}}\right)^{l-1}\sum_{r_a\neq0}\int_{\mathbb{R}}\left|\frac{\tilde{\Delta}\partial_x^{a}g}{\alpha}\frac{\tilde{\Delta}\partial_x^{k+1-j}h}{\alpha}\right|d\alpha.\\
&+\left|\int_{\mathbb{R}}\left[\frac{\alpha}{\alpha+i(\tilde{\Delta} g+2\sigma+\theta(x))}\prod_{a=1}^j\left(\frac{\partial_x^a\theta(x)}{(\alpha+i(\tilde{\Delta} g+2\sigma+\theta(x))}\right)^{r_a}\right.\right.\\
&\left.\left.+(-1)^l\frac{\alpha}{\alpha-i(\tilde{\Delta} g+2\sigma+\theta(x))}\prod_{a=1}^j\left(\frac{\partial_x^a\theta(x)}{(\alpha-i(\tilde{\Delta} g+2\sigma+\theta(x))}\right)^{r_a}\right]\frac{\tilde{\Delta}\partial_x^{k+1-j}h}{\alpha}d\alpha\right|
\end{aligned}\label{2.75}\end{equation}
To control the first term, by repeating the arguments in the controls for class (\romannumeral1), for each $a\leq\tilde{k}$ we have
\begin{equation}\begin{aligned}
\left\|\int_{\mathbb{R}}\left|\frac{\tilde{\Delta}\partial_x^{a}g}{\alpha}\frac{\tilde{\Delta}\partial_x^{k+1-j}h}{\alpha}\right|d\alpha\right\|_{L^2_{\gamma(t)}}
\lesssim&\left(\|\partial_xg\|_{L^\infty_{\gamma(t)}}+\|\partial_x^2g\|_{H^{\tilde{k}}_{\gamma(t)}}\right)\|\partial_x^{k+1-j}h\|_{H^1_{\gamma(t)}}\\
\lesssim&\left(\|\partial_xg\|_{L^\infty_{\gamma(t)}}+\|\partial_x^2g\|_{H^{\tilde{k}}_{\gamma(t)}}\right)\|h\|_{H^k_{\gamma(t)}}.    
\end{aligned}\label{2.76}\end{equation}
For the second term, if $l\geq2$, using H\"{o}lder's inequality with arbitrary $\varphi\in L^2_{\gamma(t)}$ gives
\allowdisplaybreaks[4]\begin{align*}
&\left|\int_{\Gamma_{\pm}}\left(\int_{\mathbb{R}}\frac{\alpha}{\alpha+i(\tilde{\Delta} g+2\sigma+\theta(x))}\prod_{a=1}^j\left(\frac{\partial_x^a\theta(x)}{\alpha+i(\tilde{\Delta} g+2\sigma+\theta(x))}\right)^{r_a}\frac{\tilde{\Delta}\partial_x^{k+1-j}h}{\alpha}d\alpha\right)\overline{\varphi(x)}dx\right|\\
\lesssim&\prod_{a=1}^j\left(\sigma^{-1}\|\partial_x^a\theta\|_{L^\infty_{\gamma(t)}}\right)^{r_a}\left(\int_{\Gamma_{\pm}(t)}\int_{\mathbb{R}}\frac{(2\sigma)^l}{[\alpha^2+(2\sigma)^2]^\frac{l}{2}}\left|\frac{\tilde{\Delta}\partial_x^{k+1-j}h}{\alpha}\right||\varphi(x)|d\alpha dx\right)\\
\lesssim&\prod_{a=1}^j\left(\sigma^{-1}\|\partial_x^a\theta\|_{L^\infty_{\gamma(t)}}\right)^{r_a}\left(\int_{\Gamma_{\pm}(t)}\int_{\mathbb{R}}D_{22}^0(x,x-\alpha)\left|\tilde{\Delta}\partial_x^{k+1-j}h\right|^2d\alpha dx\right)^\frac{1}{2}\\
&\cdot\left(\int_{\Gamma_{\pm}(t)}\int_{\mathbb{R}}\frac{(2\sigma)^{2l-2}}{(\alpha^2+(2\sigma)^2)^{l-1}}|\varphi(x)|^2dxd\alpha\right)^\frac{1}{2}\\
\lesssim&\sigma^{\frac{1}{2}}\prod_{a=1}^j\left(\sigma^{-1}\|\partial_x^a\theta\|_{L^\infty_{\gamma(t)}}\right)^{r_a}\left(\int_{\Gamma_{\pm}(t)}\int_{\mathbb{R}}D_{22}^0(x,x-\alpha)\left|\tilde{\Delta}\partial_x^{k+1-j}h\right|^2d\alpha dx\right)^\frac{1}{2}\|\varphi\|_{L^2_{\gamma(t)}},
\end{align*}\allowdisplaybreaks[0]
which, in view of $\|\partial_x^a\theta\|_{L^\infty_{\gamma(t)}}\lesssim\|\partial_x\theta\|_{H^{\tilde{k}}_{\gamma(t)}}$, implies that
\begin{equation}\begin{aligned}
&\left\|\int_{\mathbb{R}}\frac{\alpha}{\alpha+i(\tilde{\Delta} g+2\sigma+\theta(x))}\prod_{a=1}^j\left(\frac{\partial_x^a\theta(x)}{(\alpha+i(\tilde{\Delta} g+2\sigma+\theta(x))}\right)^{r_a}\frac{\tilde{\Delta}\partial_x^{k+1-j}h}{\alpha}d\alpha\right\|_{L^2_{\gamma(t)}}\\
\lesssim&\sigma^\frac{1}{2}\left(\sigma^{-1}\|\partial_x\theta\|_{H^{\tilde{k}}_{\gamma(t)}}\right)^l\left(\int_{\Gamma_{\pm}(t)}\int_{\mathbb{R}}D_{22}^0(x,x-\alpha)\left|\tilde{\Delta}\partial_x^{k+1-j}h\right|^2d\alpha dx\right)^\frac{1}{2}.
\end{aligned}\label{2.77}\end{equation}
Likewise,
\begin{equation}\begin{aligned}
&\left\|\int_{\mathbb{R}}\frac{\alpha}{\alpha-i(\tilde{\Delta} g+2\sigma+\theta(x))}\prod_{a=1}^j\left(\frac{\partial_x^a\theta(x)}{(\alpha-i(\tilde{\Delta} g+2\sigma+\theta(x))}\right)^{r_a}\frac{\tilde{\Delta}\partial_x^{k+1-j}h}{\alpha}d\alpha\right\|_{L^2_{\gamma(t)}}\\
\lesssim&\sigma^\frac{1}{2}\left(\sigma^{-1}\|\partial_x\theta\|_{H^{\tilde{k}}_{\gamma(t)}}\right)^l\left(\int_{\Gamma_{\pm}(t)}\int_{\mathbb{R}}D_{22}^0(x,x-\alpha)\left|\tilde{\Delta}\partial_x^{k+1-j}h\right|^2d\alpha dx\right)^\frac{1}{2}.
\end{aligned}\label{2.78}\end{equation}
If $l=1$, then $2\leq j\leq\tilde{k}$, $r_{j}=1$ and $r_a=0$, $1\leq a\leq j-1$. 
In this case, we use the identity
$$\begin{aligned}
&\frac{-1}{4i}\int_{\mathbb{R}}\left[\frac{\alpha\partial_x^j\theta(x)}{\left(\alpha+i(\tilde{\Delta} g+2\sigma+\theta(x))\right)^2}-\frac{\alpha\partial_x^j\theta(x)}{(\alpha-i(\tilde{\Delta} g+2\sigma+\theta(x)))^2}\right]\frac{\tilde{\Delta}\partial_x^{k+1-j}h}{\alpha}d\alpha\\
=&\int_{\mathbb{R}}\frac{\alpha^2\tilde{\Delta} g\cdot\partial_x^j\theta(x)}{(\alpha^2+(\tilde{\Delta} g+2\sigma+\theta(x))^2)^2}\frac{\tilde{\Delta}\partial_x^{k+1-j}h}{\alpha}d\alpha+\int_{\mathbb{R}}\frac{\alpha^2\partial_x^j\theta(x)(2\sigma+\theta(x))}{(\alpha^2+(\tilde{\Delta} g+2\sigma+\theta(x))^2)^2}\frac{\tilde{\Delta}\partial_x^{k+1-j}h}{\alpha}d\alpha.
\end{aligned}$$
The right-hand side of the above identity has the same structure as $Safe_{h,2}^{12}$ (see (\ref{2.70})), so by repetition of the proof for (\ref{2.70}), we obtain
\begin{equation}\begin{aligned}
&\left\|\int_{\mathbb{R}}\left[\frac{\alpha\partial_x^j\theta(x)}{\left(\alpha+i(\tilde{\Delta} g+2\sigma+\theta(x))\right)^2}-\frac{\alpha\partial_x^j\theta(x)}{(\alpha-i(\tilde{\Delta} g+2\sigma+\theta(x)))^2}\right]\frac{\tilde{\Delta}\partial_x^{k+1-j}h}{\alpha}d\alpha\right\|_{L^2_{\gamma(t)}}\\
\lesssim&\sigma^{-\frac{1}{2}}\|\partial_x^j\theta\|_{L^\infty_{\gamma(t)}}\left(\int_{\Gamma_{\pm}(t)}\int_{\mathbb{R}}D_{22}^0(x,x-\alpha)|\tilde{\Delta}\partial_x^{k+1-j}h|^2d\alpha dx\right)^\frac{1}{2}\\
&+\sigma^{-1}\|\partial_x^j\theta\|_{L^\infty_{\gamma(t)}}\|\partial_xg\|_{L^\infty_{\gamma(t)}}\|\partial_x^{k+1-j}h\|_{L^2_{\gamma(t)}}.
\end{aligned}\label{2.79}\end{equation}
Collecting the inequalities (\ref{2.75}-\ref{2.79}) and noticing that $\|\partial_x^j\theta\|_{L^\infty_{\gamma(t)}}\lesssim\|\partial_x\theta\|_{H^{\tilde{k}}_{\gamma(t)}}$ for $j\leq\tilde{k}$, we conclude that
\begin{equation}\begin{aligned}
&\left\|\int_{\mathbb{R}}q^{(l)}\left(\frac{\tilde{\Delta} g+2\sigma+\theta(x)}{\alpha}\right)\prod_{a=1}^j\left(\frac{\tilde{\Delta}\partial_x^ag+\partial_x^a\theta(x)}{\alpha}\right)^{r_a}\frac{\tilde{\Delta}\partial_x^{k+1-j}h}{\alpha}d\alpha\right\|_{L^2_{\gamma(t)}}\\
\lesssim&\left(1+\sigma^{-1}\|\partial_x\theta\|_{H^{\tilde{k}}_{\gamma(t)}}\right)\left(\|\partial_x^2g\|_{H^{\tilde{k}}_{\gamma(t)}}+\sigma^{-1}\|\partial_x\theta\|_{H^{\tilde{k}}_{\gamma(t)}}\right)^{l-1}\left(\|\partial_xg\|_{L^\infty_{\gamma(t)}}+\|\partial_x^2g\|_{H^{\tilde{k}}_{\gamma(t)}}\right)\|h\|_{H^k_{\gamma(t)}}\\
&+\sigma^{-\frac{1}{2}}\|\partial_x\theta\|_{H^{\tilde{k}}_{\gamma(t)}}\left[1+\sigma^{1-l}\|\partial_x\theta\|_{H^{\tilde{k}}_{\gamma(t)}}^{l-1}\right]\left(\int_{\Gamma_{\pm}(t)}\int_{\mathbb{R}}D_{22}^0(x,x-\alpha)\left|\tilde{\Delta}\partial_x^{k+1-j}h\right|^2d\alpha dx\right)^\frac{1}{2}.
\end{aligned}\label{2.80}\end{equation}
Now suppose otherwise that there exists $a_0>\tilde{k}$ such that $r_{a_0}\neq0$. Then $j>\tilde{k}$, $a_0+1\leq
k$, $r_{a_0}=1$ and $r_a=0$ for all $a>\tilde{k}$ with $a\neq a_0$. In this case, we use the decomposition
$$\begin{aligned}
&q^{(l)}\left(\frac{\tilde{\Delta} g+2\sigma+\theta(x)}{\alpha}\right)\prod_{a=1}^j\left(\frac{\tilde{\Delta}\partial_x^ag+\partial_x^a\theta(x)}{\alpha}\right)^{r_a}\frac{\tilde{\Delta}\partial_x^{k+1-j}h}{\alpha}\\  
=&q^{(l)}\left(\frac{\tilde{\Delta} g+2\sigma+\theta(x)}{\alpha}\right)\prod_{a\neq a_0}\left(\frac{\tilde{\Delta}\partial_x^ag+\partial_x^a\theta(x)}{\alpha}\right)^{r_a}\frac{\tilde{\Delta}\partial_x^{k+1-j}h}{\alpha}\frac{\tilde{\Delta}\partial_x^{a_0}g}{\alpha}\\
&+q^{(l)}\left(\frac{\tilde{\Delta} g+2\sigma+\theta(x)}{\alpha}\right)\prod_{a\neq a_0}\left(\frac{\tilde{\Delta}\partial_x^ag+\partial_x^a\theta(x)}{\alpha}\right)^{r_a}\frac{\tilde{\Delta}\partial_x^{k+1-j}h}{\alpha}\frac{\partial_x^{a_0}\theta(x)}{\alpha}
\end{aligned}$$
Note that $\left|q^{(l)}\left(\frac{\tilde{\Delta} g+2\sigma+\theta(x)}{\alpha}\right)\right|\lesssim\frac{|\alpha|^{l+1}}{(\alpha^2+(2\sigma)^2)^\frac{l+1}{2}}$, and an analogue of (\ref{2.76}) gives
$$\begin{aligned}
\left\|\int_{\mathbb{R}}\left|\frac{\tilde{\Delta}\partial_x^{k+1-j}h}{\alpha}\frac{\tilde{\Delta}\partial_x^{a_0}g}{\alpha}\right|d\alpha\right\|_{L^2_{\gamma(t)}}\lesssim\left(\|\partial_xh\|_{L^\infty_{\gamma(t)}}+\|\partial_x^2h\|_{H^{\tilde{k}}_{\gamma(t)}}\right)\|g\|_{H^k_{\gamma(t)}}.
\end{aligned}$$
Hence we get
\begin{equation}\begin{aligned}
&\left\|\int_{\mathbb{R}}q^{(l)}\left(\frac{\tilde{\Delta} g+2\sigma+\theta(x)}{\alpha}\right)\prod_{a\neq a_0}\left(\frac{\tilde{\Delta}\partial_x^ag+\partial_x^a\theta(x)}{\alpha}\right)^{r_a}\frac{\tilde{\Delta}\partial_x^{k+1-j}h}{\alpha}\frac{\tilde{\Delta}\partial_x^{a_0}g}{\alpha}d\alpha\right\|_{L^2_{\gamma(t)}}\\
\lesssim&\left(\|\partial_x^2g\|_{H^{\tilde{k}}_{\gamma(t)}}+\sigma^{-1}\|\partial_x\theta\|_{H^{\tilde{k}}_{\gamma(t)}}\right)^{l-1}\left(\|\partial_xh\|_{L^\infty_{\gamma(t)}}+\|\partial_x^2h\|_{H^{\tilde{k}}_{\gamma(t)}}\right)\|\partial_x^{a_0}g\|_{H^1_{\gamma(t)}}.
\end{aligned}\label{2.81}\end{equation}
For the other term, we further decompose it into two parts:
$$\begin{aligned}
&q^{(l)}\left(\frac{\tilde{\Delta} g+2\sigma+\theta(x)}{\alpha}\right)\prod_{a\neq a_0}\left(\frac{\tilde{\Delta}\partial_x^ag+\partial_x^a\theta(x)}{\alpha}\right)^{r_a}\frac{\tilde{\Delta}\partial_x^{k+1-j}h}{\alpha}\frac{\partial_x^{a_0}\theta(x)}{\alpha}\\  
=&\alpha^{-l}q^{(l)}\left(\frac{\tilde{\Delta} g+2\sigma+\theta(x)}{\alpha}\right)\left[\prod_{a\neq a_0}\left(\tilde{\Delta}\partial_x^ag+\partial_x^a\theta(x)\right)^{r_a}-\prod_{a\neq a_0}\left(\tilde{\Delta}\partial_x^ag\right)^{r_a}\right]\frac{\tilde{\Delta}\partial_x^{k+1-j}h}{\alpha}\cdot\partial_x^{a_0}\theta(x)\\
&+\alpha^{-1}q^{(l)}\left(\frac{\tilde{\Delta} g+2\sigma+\theta(x)}{\alpha}\right)\prod_{a\neq a_0}\left(\frac{\tilde{\Delta}\partial_x^ag}{\alpha}\right)^{r_a}\frac{\tilde{\Delta}\partial_x^{k+1-j}h}{\alpha}\cdot\partial_x^{a_0}\theta(x),
\end{aligned}$$
The first term on the right-hand side disappears if $l=1$. If otherwise $l>1$,  using Lemma \ref{lem2.2} and $\left\|\partial_x^a\theta\right\|_{L^\infty_{\gamma(t)}}\lesssim\|\partial_x\theta\|_{H^{\tilde{k}}_{\gamma(t)}}$, we get
\begin{equation}\begin{aligned}
&\left\|\int_{\mathbb{R}}\frac{\left[\prod_{a\neq a_0}(\tilde{\Delta}\partial_x^ag+\partial_x^a\theta(x))^{r_a}-\prod_{a\neq a_0}(\tilde{\Delta}\partial_x^ag)^{r_a}\right]\tilde{\Delta}\partial_x^{k+1-j}h}{(\alpha\pm i(\tilde{\Delta} g+2\sigma+\theta(x)))^{l+1}}d\alpha\cdot\partial_x^{a_0}\theta(x)\right\|_{L^2_{\gamma(t)}}\\
\lesssim&\left(\|\partial_x^2g\|_{H^{\tilde{k}}_{\gamma(t)}}+\sigma^{-1}\|\partial_x\theta\|_{H^{\tilde{k}}_{\gamma(t)}}\right)^{l-2}\|\partial_x\theta\|_{H^{\tilde{k}}_{\gamma(t)}}\|\partial_x^{k+2-j}h\|_{L^\infty_{\gamma(t)}}\int_{\mathbb{R}}\frac{1}{\alpha^2+(2\sigma)^2}d\alpha\|\partial_x^{a_0}\theta\|_{L^2_{\gamma(t)}}\\
\lesssim&\left(\|\partial_x^2g\|_{H^{\tilde{k}}_{\gamma(t)}}+\sigma^{-1}\|\partial_x\theta\|_{H^{\tilde{k}}_{\gamma(t)}}\right)^{l-2}\sigma^{-1}\|\partial_x\theta\|_{H^{\tilde{k}}_{\gamma(t)}}\|\partial_x^{k+2-j}h\|_{L^\infty_{\gamma(t)}}\|\partial_x^{a_0}\theta\|_{L^2_{\gamma(t)}}.
\end{aligned}\label{2.82}\end{equation}
To control the second term, it suffices to bound the $L^\infty_{\gamma(t)}$ norm of
$$\int_{\mathbb{R}}\alpha^{-1}q^{(l)}\left(\frac{\tilde{\Delta} g+2\sigma+\theta(x)}{\alpha}\right)\prod_{a\neq a_0}\left(\frac{\tilde{\Delta}\partial_x^ag}{\alpha}\right)^{r_a}\frac{\tilde{\Delta}\partial_x^{k+1-j}h}{\alpha}d\alpha.$$
To this end, we decompose the integral over $\alpha$ into the parts $|\alpha|\leq1$ and $|\alpha|>1$. Denote $R_l(x,\alpha;g):=\alpha^{-1}q^{(l)}\left(\frac{\tilde{\Delta} g+2\sigma+\theta(x)}{\alpha}\right)$ for simplicity. The part $|\alpha|>1$ can be controlled easily:
$$\begin{aligned}
\left|\int_{|\alpha|>1}\prod_{a\neq a_0}\left(\frac{\tilde{\Delta}\partial_x^{a}g}{\alpha}\right)^{r_a}\frac{\tilde{\Delta}\partial_x^{k+1-j}h}{\alpha}R_l(x,\alpha;g)d\alpha\right|
\lesssim&\|\partial_x^2g\|_{H^{\tilde{k}}_{\gamma(t)}}^{l-1}\int_{|\alpha|>1}\frac{1}{|\alpha|}M[\partial_x^{k+2-j}h](x-\alpha)d\alpha\\
\lesssim&\|\partial_x^2g\|_{H^{\tilde{k}}_{\gamma(t)}}^{l-1}\|\partial_x^{k+2-j}h\|_{L^2_{\gamma(t)}},
\end{aligned}$$
where we used the $L^2$ boundedness of the maximal functional, $2\leq k+2-j\leq1+\tilde{k}$ and that $|R_{l}(x,\alpha;g)|\lesssim|\alpha|^{-1}$. As to the part $|\alpha|\leq1$, we write
$$\begin{aligned}
&\int_{|\alpha|\leq1}\alpha^{-1}q^{(l)}\left(\frac{\tilde{\Delta} g+2\sigma+\theta(x)}{\alpha}\right)\prod_{a\neq a_0}\left(\frac{\tilde{\Delta}\partial_x^{a}g}{\alpha}\right)^{r_a}\frac{\tilde{\Delta}\partial_x^{k+1-j}h}{\alpha}d\alpha\\
=&\int_{|\alpha|\leq1}R_l(x,\alpha;g)\left[\prod_{a\neq a_0}\left(\frac{\tilde{\Delta}\partial_x^{a}g}{\alpha}\right)^{r_a}\frac{\tilde{\Delta}\partial_x^{k+1-j}h}{\alpha}-\prod_{a\neq a_0}(\partial_x^{a+1}g(x))^{r_a}\partial_x^{k+2-j}h(x)\right]d\alpha\\
&+\int_{|\alpha|\leq1}R_l(x,\alpha;g)d\alpha\prod_{a\neq a_0}(\partial_x^{a+1}g(x))^{r_a}\partial_x^{k+2-j}h(x).
\end{aligned}$$
Observe that $|R_l(x,\alpha;g)|\lesssim(\alpha^2+(2\sigma)^2)^{-\frac{1}{2}}$ and that $\left|\alpha^{-1}\tilde{\Delta}\partial_x^ag-\partial_x^{a+1}g(x)\right|\lesssim|\alpha|M[\partial_x^{a+2}g](x-\alpha)$, $\left|\alpha^{-1}\tilde{\Delta}\partial_x^{k+1-j}h-\partial_x^{k+2-j}h(x)\right|\lesssim|\alpha|M[\partial_x^{k+3-j}h](x-\alpha)$.
By Lemma \ref{lem2.2} and noticing $k+3-j\leq\tilde{k}+2$, it holds
\allowdisplaybreaks[4]\begin{align*}
&\left|\int_{|\alpha|\leq1}R_{l}(x,\alpha;g)\left[\prod_{a\neq a_0}\left(\frac{\tilde{\Delta}\partial_x^{a}g}{\alpha}\right)^{r_a}\frac{\tilde{\Delta}\partial_x^{k+1-j}h}{\alpha}-\prod_{a\neq a_0}(\partial_x^{a+1}g(x))^{r_a}\partial_x^{k+2-j}h(x)\right]d\alpha\right|\\
\lesssim&\mathrm{1}_{l\geq2}\|\partial_x^{k+2-j}h\|_{L^\infty_{\gamma(t)}}\|\partial_x^2g\|_{H^{\tilde{k}}_{\gamma(t)}}^{l-2}\int_{|\alpha|\leq1}\sum_{a\neq a_0}M[\partial_x^{a+2}g](x-\alpha)d\alpha\\
&+\|\partial_x^2g\|_{H^{\tilde{k}}_{\gamma(t)}}^{l-1}\int_{|\alpha|\leq1}M[\partial_x^{k+3-j}h](x-\alpha)d\alpha\\
\lesssim&\|\partial_x^2g\|^{l-1}_{H^{\tilde{k}}_{\gamma(t)}}\|\partial_x^{k+2-j}h\|_{H^1_{\gamma(t)}}.
\end{align*}\allowdisplaybreaks[0]
Meanwhile,  we write
$\int_{|\alpha|\leq1}R_l(x,\alpha;g)d\alpha=\int_{0\leq\alpha\leq1}\left(R_{l}(x,\alpha;g)+R_{l}(x,-\alpha;g)\right)d\alpha,$
and direct computation using condition (\ref{2.17}) shows
$$\begin{aligned}
&\left|R_{l}(x,\alpha;g)+R_{l}(x,-\alpha;g)\right|\\
=&\frac{l!}{2}\left|\frac{\alpha^l}{(\alpha+i(g(x)-g(x-\alpha)+2\sigma+\theta(x)))^{l+1}}-\frac{\alpha^l}{(\alpha-i(g(x)-g(x+\alpha)+2\sigma+\theta(x)))^{l+1}}\right.\\
&\left.+\frac{(-\alpha)^l}{(\alpha-i(g(x)-g(x-\alpha)+2\sigma+\theta(x)))^{l+1}}-\frac{(-\alpha)^l}{(\alpha+i(g(x)-g(x+\alpha)+2\sigma+\theta(x)))^{l+1}}\right|\\
\lesssim&M[\partial_x^2g](x-\alpha)+M[\partial_x^2g](x+\alpha)+\frac{\sigma}{\alpha^2+(2\sigma)^2}.
\end{aligned}$$
Using the above inequality, we obtain
$\left|\int_{|\alpha|\leq1}R_l(x,\alpha;g)d\alpha\right|\lesssim1+\|\partial_x^2g\|_{L^2_{\gamma(t)}}$, and consequently,
\begin{equation}\begin{aligned}
&\left\|\int_{|\alpha|\leq1}R_l(x,\alpha;g)d\alpha\prod_{a\neq a_0}(\partial_x^{a+1}g(x))^{r_a}\partial_x^{k+2-j}h(x)\right\|_{L^\infty_{\gamma(t)}}\\
\lesssim&\left(1+\|\partial_x^2g\|_{L^2_{\gamma(t)}}\right)\|\partial_x^2g\|_{H^{\tilde{k}}_{\gamma(t)}}^{l-1}\|\partial_x^{k+2-j}h\|_{H^1_{\gamma(t)}},
\end{aligned}\label{2.83}\end{equation}
\begin{equation}\begin{aligned}
&\left\|\int_{|\alpha|\leq1}\alpha^{-1}q^{(l)}\left(\frac{\tilde{\Delta} g+2\sigma+\theta(x)}{\alpha}\right)\prod_{a\neq a_0}\left(\frac{\tilde{\Delta}\partial_x^ag}{\alpha}\right)^{r_a}\frac{\tilde{\Delta}\partial_x^{k+1-j}h}{\alpha}d\alpha\cdot\partial_x^{a_0}\theta(x)\right\|_{L^2_{\gamma(t)}}\\
\lesssim&\left(1+\|\partial_x^2g\|_{L^2_{\gamma(t)}}\right)\|\partial_x^2g\|_{H^{\tilde{k}}_{\gamma(t)}}^{l-1}\|\partial_x^{k+2-j}h\|_{H^1_{\gamma(t)}}\|\partial_x^{a_0}\theta\|_{L^2_{\gamma(t)}}.
\end{aligned}\label{2.84}\end{equation}
Collecting (\ref{2.82}-\ref{2.84}) yields
\begin{equation}\begin{aligned}
&\left\|\int_{\mathbb{R}}q^{(l)}\left(\frac{\tilde{\Delta} g+2\sigma+\theta(x)}{\alpha}\right)\prod_{a\neq a_0}\left(\frac{\tilde{\Delta}\partial_x^ag+\partial_x^a\theta(x)}{\alpha}\right)^{r_a}\frac{\tilde{\Delta}\partial_x^{k+1-j}h}{\alpha}\frac{\partial_x^{a_0}\theta(x)}{\alpha}d\alpha\right\|_{L^2_{\gamma(t)}}\\
\lesssim&\left(1+\|\partial_x^2g\|_{L^2_{\gamma(t)}}\right)\left(\|\partial_x^2g\|_{H^{\tilde{k}}_{\gamma(t)}}+\sigma^{-1}\|\partial_x\theta\|_{H^{\tilde{k}}_{\gamma(t)}}\right)^{l-1}\|\partial_x^{k+2-j}h\|_{H^{1}_{\gamma(t)}}\|\partial_x^{a_0}\theta\|_{L^2_{\gamma(t)}}.
\end{aligned}\label{2.85}\end{equation}
Then  using $k+2-j\leq\tilde{k}+1$, $5\leq a_0\leq k-1$ and summing (\ref{2.81})(\ref{2.85}) gives
\begin{equation}\begin{aligned}
&\left\|\int_{\mathbb{R}}q^{(l)}\left(\frac{\tilde{\Delta} g+2\sigma+\theta(x)}{\alpha}\right)\prod_{a=1}^j\left(\frac{\tilde{\Delta}\partial_x^ag+\partial_x^a\theta(x)}{\alpha}\right)^{r_a}\frac{\tilde{\Delta}\partial_x^{k+1-j}h}{\alpha}d\alpha\right\|_{L^2_{\gamma(t)}}\\
\lesssim&\left(1+\|\partial_x^2g\|_{L^2_{\gamma(t)}}\right)\left(\|\partial_x^2g\|_{H^{\tilde{k}}_{\gamma(t)}}+\sigma^{-1}\|\partial_x\theta\|_{H^{\tilde{k}}_{\gamma(t)}}\right)^{l-1}\\
&\cdot\left(\|\partial_xh\|_{L^\infty_{\gamma(t)}}+\|\partial_x^2h\|_{H^{\tilde{k}}_{\gamma(t)}}\right)\left(\|\partial_x^5g\|_{H^{k-5}_{\gamma(t)}}+\|\partial_x^5\theta\|_{H^{k-5}_{\gamma(t)}}\right).
\end{aligned}\label{2.86}\end{equation}
Moreover, with $\tilde{\Delta} g+2\sigma+\theta(x)$ as well as its derivatives replaced by $\tilde{\Delta} f-2\sigma-\theta(x)$ and its derivatives respectively, or with $h$ replaced by $\theta$, similar controls as in (\ref{2.80}) and (\ref{2.86}) can be established for any other easy terms in class (\romannumeral2). Summing up all these controls and using $f=h+\mu_1\theta$, $g=h-\mu_2\theta$, we get
\begin{lem}\label{summary3.7-2}
Suppose that condition (\ref{2.17}) holds. Then for $w=h$ or $w=\theta$
\begin{equation}\begin{aligned}
&\sum_{B_k}\sum_{S_{j,l}}\left\|\int_{\mathbb{R}}q^{(l)}\left(\frac{\tilde{\Delta} g+2\sigma+\theta(x)}{\alpha}\right)\prod_{a=1}^j\left(\frac{\tilde{\Delta}\partial_x^ag+\partial_x^a\theta(x)}{\alpha}\right)^{r_a}\frac{\tilde{\Delta}\partial_x^{k+1-j}w}{\alpha}d\alpha\right\|_{L^2_{\gamma(t)}}\\
&+\sum_{B_k}\sum_{S_{j,l}}\left\|\int_{\mathbb{R}}q^{(l)}\left(\frac{\tilde{\Delta} f-2\sigma-\theta(x)}{\alpha}\right)\prod_{a=1}^j\left(\frac{\tilde{\Delta}\partial_x^af-\partial_x^a\theta(x)}{\alpha}\right)^{r_a}\frac{\tilde{\Delta}\partial_x^{k+1-j}w}{\alpha}d\alpha\right\|_{L^2_{\gamma(t)}}\\
\lesssim&\left(1+\|\partial_x^2h\|_{H^{\tilde{k}}_{\gamma(t)}}+\|\partial_x^2\theta\|_{H^{\tilde{k}}_{\gamma(t)}}+\sigma^{-1}\|\partial_x\theta\|_{H^{\tilde{k}}_{\gamma(t)}}\right)^{k}\\
&\cdot\left[\left(\|\partial_xw\|_{L^\infty_{\gamma(t)}}+\|\partial_x^2w\|_{H^{\tilde{k}}_{\gamma(t)}}\right)\left(\|\partial_x^5h\|_{H^{k-5}_{\gamma(t)}}+\|\partial_x^5\theta\|_{H^{k-5}_{\gamma(t)}}\right)\right.\\
&\left.+\left(\|\partial_xh\|_{L^\infty_{\gamma(t)}}+\|\partial_x^2h\|_{H^{\tilde{k}}_{\gamma(t)}}+\|\partial_x\theta\|_{L^\infty_{\gamma(t)}}+\|\partial_x^2\theta\|_{H^{\tilde{k}}_{\gamma(t)}}\right)\|w\|_{H^k_{\gamma(t)}}\right]\\
&+\sigma^{-\frac{1}{2}}\|\partial_x\theta\|_{H^{\tilde{k}}_{\gamma(t)}}\left(1+\sigma^{-1}\|\partial_x\theta\|_{H^{\tilde{k}}_{\gamma(t)}}\right)^{k-1}\\
&\cdot\sum_{j=2}^k\left(\int_{\Gamma_{\pm}(t)}\int_{\mathbb{R}}D_{22}^0(x,x-\alpha)\left|\tilde{\Delta}\partial_x^{k+1-j}w\right|^2d\alpha dx\right)^\frac{1}{2}.
\end{aligned}\label{2.87}\end{equation}
\end{lem}
\noindent\textbf{Controls for class (\romannumeral3)}. As before, we first look at the case $r_a=0$ for all $a>\tilde{k}.$ For any $(j,l)\in B_k$, $r\in S_{j,l}$, choose an arbitrary $a_1$ such that $r_{a_1}\neq0$. Then we use the decomposition
$$\begin{aligned}
&\int_{\mathbb{R}}\alpha^{-1}q^{(l)}\left(\frac{\tilde{\Delta} g+2\sigma+\theta(x)}{\alpha}\right)\prod_{a=1}^j\left(\frac{\tilde{\Delta}\partial_x^ag+\partial_x^a\theta(x)}{\alpha}\right)^{r_a}\partial_x^{k+1-j}\theta(x)\,d\alpha\\
=&\int_{\mathbb{R}}q^{(l)}\left(\frac{\tilde{\Delta} g+2\sigma+\theta(x)}{\alpha}\right)\prod_{a\neq a_1}\left(\frac{\tilde{\Delta}\partial_x^ag+\partial_x^a\theta(x)}{\alpha}\right)^{r_a}\frac{\tilde{\Delta}\partial_x^{a_1}g}{\alpha}\frac{\partial_x^{k+1-j}\theta(x)}{\alpha}\,d\alpha\\
&+\int_{\mathbb{R}}q^{(l)}\left(\frac{\tilde{\Delta} g+2\sigma+\theta(x)}{\alpha}\right)\prod_{a\neq a_1}\left(\frac{\tilde{\Delta}\partial_x^ag+\partial_x^a\theta(x)}{\alpha}\right)^{r_a}\frac{\partial_x^{a_1}\theta(x)}{\alpha}\frac{\partial_x^{k+1-j}\theta(x)}{\alpha}\,d\alpha.
\end{aligned}$$
Observing that the first term has the same structure as the left-hand side of (\ref{2.85}), by adjusting the proof of (\ref{2.85}), it can be shown that
\begin{equation}\begin{aligned}
&\left\|\int_{\mathbb{R}}q^{(l)}\left(\frac{\tilde{\Delta} g+2\sigma+\theta(x)}{\alpha}\right)\prod_{a\neq a_1}\left(\frac{\tilde{\Delta}\partial_x^ag+\partial_x^a\theta(x)}{\alpha}\right)^{r_a}\frac{\tilde{\Delta}\partial_x^{a_1}g}{\alpha}\frac{\partial_x^{k+1-j}\theta(x)}{\alpha}\,d\alpha\right\|_{L^2_{\gamma(t)}}\\
\lesssim&\left(1+\|\partial_x^2g\|_{L^2_{\gamma(t)}}\right)\left(\|\partial_x^2g\|_{H^{\tilde{k}}_{\gamma(t)}}+\sigma^{-1}\|\partial_x\theta\|_{H^{\tilde{k}}_{\gamma(t)}}\right)^{l-1}\left(\|\partial_xg\|_{L^\infty_{\gamma(t)}}+\|\partial_x^2g\|_{H^{\tilde{k}}_{\gamma(t)}}\right)\|\theta\|_{H^k_{\gamma(t)}}.
\end{aligned}\label{2.88}\end{equation}
We further decompose the second term into two parts $M_1$ and $M_2$:
$$M_1:=\int_{\mathbb{R}}\alpha^{-l-1}q^{(l)}\left(\frac{\tilde{\Delta} g+2\sigma+\theta(x)}{\alpha}\right)d\alpha\prod_{a=1}^j\left(\partial_x^{r_a}\theta(x)\right)^{r_a}\cdot\partial_x^{k+1-j}\theta(x),$$
$$\begin{aligned}
M_2:=&\int_{\mathbb{R}}\alpha^{-1}q^{(l)}\left(\frac{\tilde{\Delta} g+2\sigma+\theta(x)}{\alpha}\right)\left[\prod_{a\neq a_1}\left(\frac{\tilde{\Delta}\partial_x^ag+\partial_x^a\theta(x)}{\alpha}\right)^{r_a}-\prod_{a\neq a_1}\left(\frac{\partial_x^a\theta(x)}{\alpha}\right)^{r_a}\right]\\
&\cdot\frac{\partial_x^{a_1}\theta(x)}{\alpha}\,d\alpha\cdot\partial_x^{k+1-j}\theta(x).
\end{aligned}$$
If $l=1$, then $M_2$ disappears. If otherwise $l\geq1$, using Lemma \ref{lem2.2}, we obtain
\begin{equation}\begin{aligned}
&\|M_2\|_{L^2_{\gamma(t)}}\\
\lesssim&\|\partial_x\theta\|_{H^{\tilde{k}}_{\gamma(t)}}\left(\|\partial_x^2g\|_{H^{\tilde{k}}_{\gamma(t)}}+\sigma^{-1}\|\partial_x\theta\|_{H^{\tilde{k}}_{\gamma(t)}}\right)^{l-2}\|\partial_x^2g\|_{H^{\tilde{k}}_{\gamma(t)}}\int_{\mathbb{R}}\frac{d\alpha}{\alpha^2+(2\sigma)^2}\|\partial_x^{k+1-j}\theta\|_{L^2_{\gamma(t)}}\\
\lesssim&\sigma^{-1}\|\partial_x\theta\|_{H^{\tilde{k}}_{\gamma(t)}}\left(\|\partial_x^2g\|_{H^{\tilde{k}}_{\gamma(t)}}+\sigma^{-1}\|\partial_x\theta\|_{H^{\tilde{k}}_{\gamma(t)}}\right)^{l-2}\|\partial_x^2g\|_{H^{\tilde{k}}_{\gamma(t)}}\|\theta\|_{H^k_{\gamma(t)}}.
\end{aligned}\label{2.89}\end{equation}
To control $M_1$, it suffices to bound $\int_{\mathbb{R}}\alpha^{-l-1}q^{(l)}\left(\frac{\tilde{\Delta} g+2\sigma+\theta(x)}{\alpha}\right)d\alpha$. A key observation is that $\alpha^{-1-l}q^{(l)}\left(\frac{2\sigma+\theta(x)}{\alpha}\right)$ is an odd function of $\alpha$ for fixed $x$ and all $l\geq
1$. Meanwhile, a direct computation shows
$$\begin{aligned}
&\left|\alpha^{-l-1}q^{(l)}\left(\frac{\tilde{\Delta} g+2\sigma+\theta(x)}{\alpha}\right)-\alpha^{-l-1}q^{(l)}\left(\frac{2\sigma+\theta(x)}{\alpha}\right)\right|
\lesssim|\tilde{\Delta} g|\left(\alpha^2+(2\sigma)^2\right)^{-\frac{l+2}{2}}.
\end{aligned}$$
It follows that $\left|\int_{\mathbb{R}}\alpha^{-l-1}q^{(l)}\left(\frac{\tilde{\Delta} g+2\sigma+\theta(x)}{\alpha}\right)d\alpha\right|\lesssim\sigma^{-l}\|\partial_xg\|_{L^\infty_{\gamma(t)}}$, and thus
\begin{equation}\label{2.90}\begin{aligned}
\|M_1\|_{L^2_{\gamma(t)}}\lesssim\|\partial_xg\|_{L^\infty_{\gamma(t)}}\left(\sigma^{-1}\|\partial_x\theta\|_{H^{\tilde{k}}_{\gamma(t)}}\right)^l\|\theta\|_{H^k_{\gamma(t)}}.
\end{aligned}\end{equation}
Summing up (\ref{2.89})(\ref{2.90}) yields
\begin{equation}\begin{aligned}
&\left\|\int_{\mathbb{R}}q^{(l)}\left(\frac{\tilde{\Delta} g+2\sigma+\theta(x)}{\alpha}\right)\prod_{a\neq a_1}\left(\frac{\tilde{\Delta}\partial_x^ag+\partial_x^a\theta(x)}{\alpha}\right)^{r_a}\frac{\partial_x^{a_1}\theta(x)}{\alpha}\frac{\partial_x^{k+1-j}\theta(x)}{\alpha}\,d\alpha\right\|_{L^2_{\gamma(t)}}\\
\lesssim&\left(1+\sigma^{-1}\|\partial_x\theta\|_{H^{\tilde{k}}_{\gamma(t)}}\right)\left(\|\partial_xg\|_{L^\infty_{\gamma(t)}}+\|\partial_x^2g\|_{H^{\tilde{k}}_{\gamma(t)}}\right)\left(\|\partial_x^2g\|_{H^{\tilde{k}}_{\gamma(t)}}+\sigma^{-1}\|\partial_x\theta\|_{H^{\tilde{k}}_{\gamma(t)}}\right)^{l-1}\|\theta\|_{H^k_{\gamma(t)}},
\end{aligned}\label{2.91}\end{equation}
and together with (\ref{2.88}) we conclude
\begin{equation}\begin{aligned}
&\left\|\int_{\mathbb{R}}\alpha^{-1}q^{(l)}\left(\frac{\tilde{\Delta} g+2\sigma+\theta(x)}{\alpha}\right)\prod_{a=1}^j\left(\frac{\tilde{\Delta}\partial_x^ag+\partial_x^a\theta(x)}{\alpha}\right)^{r_a}\partial_x^{k+1-j}\theta(x)\,d\alpha\right\|_{L^2_{\gamma(t)}}\\
\lesssim&\left(1+\|\partial_x^2g\|_{L^2_{\gamma(t)}}+\sigma^{-1}\|\partial_x\theta\|_{H^{\tilde{k}}_{\gamma(t)}}\right)\left(\|\partial_xg\|_{L^\infty_{\gamma(t)}}+\|\partial_x^2g\|_{H^{\tilde{k}}_{\gamma(t)}}\right)\\&\cdot\left(\|\partial_x^2g\|_{H^{\tilde{k}}_{\gamma(t)}}+\sigma^{-1}\|\partial_x\theta\|_{H^{\tilde{k}}_{\gamma(t)}}\right)^{l-1}\|\theta\|_{H^k_{\gamma(t)}},
\end{aligned}\label{2.92}\end{equation}
Suppose now that there exists $a_0>\tilde{k}$ with $r_{a_0}\neq 0$ so that $r_{a_0}=1$, $a_0+1\leq k$, $j>\tilde{k}$ and $r_a=0$ for all $a>\tilde{k}$ with $a\neq a_0$. In this case, we write
$$\begin{aligned}
&\int_{\mathbb{R}}\alpha^{-1}q^{(l)}\left(\frac{\tilde{\Delta} g+2\sigma+\theta(x)}{\alpha}\right)\prod_{a=1}^j\left(\frac{\tilde{\Delta}\partial_x^ag+\partial_x^a\theta(x)}{\alpha}\right)^{r_a}\partial_x^{k+1-j}\theta(x)\,d\alpha\\
=&\int_{\mathbb{R}}q^{(l)}\left(\frac{\tilde{\Delta} g+2\sigma+\theta(x)}{\alpha}\right)\prod_{a\neq a_0}\left(\frac{\tilde{\Delta}\partial_x^ag+\partial_x^a\theta(x)}{\alpha}\right)^{r_a}\frac{\partial_x^{k+1-j}\theta(x)}{\alpha}\frac{\tilde{\Delta}\partial_x^{a_0}g}{\alpha}\,d\alpha\\
&+\int_{\mathbb{R}}q^{(l)}\left(\frac{\tilde{\Delta} g+2\sigma+\theta(x)}{\alpha}\right)\prod_{a\neq a_0}\left(\frac{\tilde{\Delta}\partial_x^ag+\partial_x^a\theta(x)}{\alpha}\right)^{r_a}\frac{\partial_x^{k+1-j}\theta(x)}{\alpha}\frac{\partial_x^{a_0}\theta(x)}{\alpha}\,d\alpha
\end{aligned}$$
Note that the second term has the same structure as the left-hand side of (\ref{2.91}). Hence we have
\begin{equation}\begin{aligned}
&\left\|\int_{\mathbb{R}}q^{(l)}\left(\frac{\tilde{\Delta} g+2\sigma+\theta(x)}{\alpha}\right)\prod_{a\neq a_0}\left(\frac{\tilde{\Delta}\partial_x^ag+\partial_x^a\theta(x)}{\alpha}\right)^{r_a}\frac{\partial_x^{k+1-j}\theta(x)}{\alpha}\frac{\partial_x^{a_0}\theta(x)}{\alpha}\,d\alpha\right\|_{L^2_{\gamma(t)}}\\
\lesssim&\left(1+\sigma^{-1}\|\partial_x\theta\|_{H^{\tilde{k}}_{\gamma(t)}}\right)\left(\|\partial_xg\|_{L^\infty_{\gamma(t)}}+\|\partial_x^2g\|_{H^{\tilde{k}}_{\gamma(t)}}\right)\left(\|\partial_x^2g\|_{H^{\tilde{k}}_{\gamma(t)}}+\sigma^{-1}\|\partial_x\theta\|_{H^{\tilde{k}}_{\gamma(t)}}\right)^{l-1}\|\theta\|_{H^k_{\gamma(t)}}.
\end{aligned}\label{2.93}\end{equation}
We further decompose the first term into two parts $M_3$ and $M_4$:
$$M_3:=\int_{\mathbb{R}}q^{(l)}\left(\frac{\tilde{\Delta} g+2\sigma+\theta(x)}{\alpha}\right)\prod_{a\neq a_0}\left(\frac{\partial_x^{a}\theta(x)}{\alpha}\right)^{r_a}\frac{\partial_{x}^{k+1-j}\theta(x)}{\alpha}\frac{\tilde{\Delta}\partial_x^{a_0}g}{\alpha}d\alpha,$$
$$\begin{aligned}M_4:=&\int_{\mathbb{R}}q^{(l)}\left(\frac{\tilde{\Delta} g+2\sigma+\theta(x)}{\alpha}\right)\left[\prod_{a\neq a_0}\left(\frac{\tilde{\Delta}\partial_x^ag+\partial_x^a\theta(x)}{\alpha}\right)^{r_a}-\prod_{a\neq a_0}\left(\frac{\partial_x^a\theta(x)}{\alpha}\right)^{r_a}\right]\\
&\cdot
\frac{\partial_x^{k+1-j}\theta(x)}{\alpha}\frac{\tilde{\Delta}\partial_x^{a_0}g}{\alpha}\,d\alpha.\end{aligned}$$
If $l=1$, $M_4$ disappears. If otherwise $l\geq2$, using Lemma \ref{lem2.2}, Minkowski and the fact $\|\partial_x^{a_0}g\|_{L^2_{\gamma(t)}}\lesssim\|g\|_{H^k_{\gamma(t)}}$ gives
\begin{equation}\begin{aligned}
\|M_4\|_{L^2_{\gamma(t)}}\lesssim&\|\partial_x\theta\|_{H^{\tilde{k}}_{\gamma(t)}}\left(\|\partial_x^2g\|_{H^{\tilde{k}}_{\gamma(t)}}+\sigma^{-1}\|\partial_x\theta\|_{H^{\tilde{k}}_{\gamma(t)}}\right)^{l-2}\|\partial_x^2g\|_{H^{\tilde{k}}_{\gamma(t)}}\left\|\int_{\mathbb{R}}\frac{|\tilde{\Delta}\partial_x^{a_0}g|d\alpha}{\alpha^2+(2\sigma)^2}\right\|_{L^2_{\gamma(t)}}\\ \lesssim&\sigma^{-1}\|\partial_x\theta\|_{H^{\tilde{k}}_{\gamma(t)}}\left(\|\partial_x^2g\|_{H^{\tilde{k}}_{\gamma(t)}}+\sigma^{-1}\|\partial_x\theta\|_{H^{\tilde{k}}_{\gamma(t)}}\right)^{l-2}\|\partial_x^2g\|_{H^{\tilde{k}}_{\gamma(t)}}\|g\|_{H^k_{\gamma(t)}}. 
\end{aligned}\label{2.94}\end{equation}
$M_3$ can be estimated in the same way as in (\ref{2.77})(\ref{2.78}) if $l\geq2$ or (\ref{2.79}) if $l=1$. To be specific, we have in both cases
\begin{equation}\begin{aligned}
\|M_3\|_{L^2_{\gamma(t)}}
\lesssim&\sigma^{-\frac{1}{2}}\|\partial_x\theta\|_{H^{\tilde{k}}_{\gamma(t)}}\left[1+\sigma^{1-l}\|\partial_x\theta\|_{H^{\tilde{k}}_{\gamma(t)}}^{l-1}\right]\left(\int_{\Gamma_{\pm}(t)}\int_{\mathbb{R}}D_{22}^0(x,x-\alpha)\left|\tilde{\Delta}\partial_x^{a_0}g\right|^2d\alpha dx\right)^\frac{1}{2}\\
&+\left(\sigma^{-1}\|\partial_x\theta\|_{H^{\tilde{k}}_{\gamma(t)}}\right)^l\|\partial_xg\|_{L^\infty_{\gamma(t)}}\|g\|_{H^k_{\gamma(t)}}.
\end{aligned}\label{2.95}\end{equation}
Collecting (\ref{2.93}-\ref{2.95}), we conclude that
\begin{equation}\begin{aligned}
&\left\|\int_{\mathbb{R}}\alpha^{-1}q^{(l)}\left(\frac{\tilde{\Delta} g+2\sigma+\theta(x)}{\alpha}\right)\prod_{a=1}^j\left(\frac{\tilde{\Delta}\partial_x^ag+\partial_x^a\theta(x)}{\alpha}\right)^{r_a}\partial_x^{k+1-j}\theta(x)\,d\alpha\right\|_{L^2_{\gamma(t)}}\\
\lesssim&\left(1+\sigma^{-1}\|\partial_x\theta\|_{H^{\tilde{k}}_{\gamma(t)}}\right)\left(\|\partial_x^2g\|_{H^{\tilde{k}}_{\gamma(t)}}+\sigma^{-1}\|\partial_x\theta\|_{H^{\tilde{k}}_{\gamma(t)}}\right)^{l-1}\left(\|\partial_x^2g\|_{H^{\tilde{k}}_{\gamma(t)}}+\|\partial_xg\|_{L^\infty_{\gamma(t)}}\right)\|g\|_{H^k_{\gamma(t)}}\\
&+\sigma^{-\frac{1}{2}}\|\partial_x\theta\|_{H^{\tilde{k}}_{\gamma(t)}}\left[1+\sigma^{1-l}\|\partial_x\theta\|_{H^{\tilde{k}}_{\gamma(t)}}^{l-1}\right]\left(\int_{\Gamma_{\pm}(t)}\int_{\mathbb{R}}D_{22}^0(x,x-\alpha)\left|\tilde{\Delta}\partial_x^{a_0}g\right|^2d\alpha dx\right)^\frac{1}{2}.
\end{aligned}\label{2.96}\end{equation}
With $\tilde{\Delta} g+2\sigma+\theta(x)$ and its derivatives replaced by $\tilde{\Delta} f-2\sigma-\theta(x)$ and its derivatives, respectively, similar controls as in (\ref{2.92}) or (\ref{2.96}) for other easy terms in class (\romannumeral3) can be established. In view of $f=h+\mu_1\theta$ and $g=h-\mu_2\theta$, we finally obtain

\begin{lem}\label{summary3.7-3}
Suppose that the solution satisfies conditions (\ref{2.17}). Then all the easy terms of class (iii) satisfy the bound
\begin{equation}\begin{aligned}
&\sum_{B_k}\sum_{S_{j,l}}\left\|\int_{\mathbb{R}}\alpha^{-1}q^{(l)}\left(\frac{\tilde{\Delta} g+2\sigma+\theta(x)}{\alpha}\right)\prod_{a=1}^j\left(\frac{\tilde{\Delta}\partial_x^ag+\partial_x^a\theta(x)}{\alpha}\right)^{r_a}\partial_x^{k+1-j}\theta(x)\,d\alpha\right\|_{L^2_{\gamma(t)}}\\
&+\sum_{B_k}\sum_{S_{j,l}}\left\|\int_{\mathbb{R}}\alpha^{-1}q^{(l)}\left(\frac{\tilde{\Delta} f-2\sigma-\theta(x)}{\alpha}\right)\prod_{a=1}^j\left(\frac{\tilde{\Delta}\partial_x^af-\partial_x^a\theta(x)}{\alpha}\right)^{r_a}\partial_x^{k+1-j}\theta(x)\,d\alpha\right\|_{L^2_{\gamma(t)}}\\
\lesssim&\left(\|\partial_xh\|_{L^\infty_{\gamma(t)}}+\|\partial_x^2h\|_{H^{\tilde{k}}_{\gamma(t)}}+\|\partial_x\theta\|_{L^\infty_{\gamma(t)}}+\|\partial_x^2\theta\|_{H^{\tilde{k}}_{\gamma(t)}}\right)\\
&\cdot\left(1+\|\partial_x^2h\|_{H^{\tilde{k}}_{\gamma(t)}}+\|\partial_x^2\theta\|_{H^{\tilde{k}}_{\gamma(t)}}+\sigma^{-1}\|\partial_x\theta\|_{H^{\tilde{k}}_{\gamma(t)}}\right)^{k}\left(\|h\|_{H^k_{\gamma(t)}}+\|\theta\|_{H^k_{\gamma(t)}}\right)\\
&+\sigma^{-\frac{1}{2}}\|\partial_x\theta\|_{H^{\tilde{k}}_{\gamma(t)}}\left(1+\sigma^{-1}\|\partial_x\theta\|_{H^{\tilde{k}}_{\gamma(t)}}\right)^{k-1}\\
&\cdot\sum_{a_0=\tilde{k}+1}^{k-1}\left(\int_{\Gamma_{\pm}(t)}\int_{\mathbb{R}}D_{22}^0(x,x-\alpha)\left(\left|\tilde{\Delta}\partial_x^{a_0}h\right|^2+\left|\tilde{\Delta}\partial_x^{a_0}\theta\right|^2\right)d\alpha dx\right)^\frac{1}{2}.
\end{aligned}\label{2.97}\end{equation}
\end{lem}

\subsection{Energy estimate uniform in $\sigma$}\label{energy1}
We hereby collect all the estimates obtained in section \ref{apriori} to conclude with the a priori energy estimate for the solution. For simplicity, we introduce the notations
$$\text{Diss}_k:=\left(\int_{\Gamma_{\pm}(t)}\int_{\Gamma_{\pm}(t)}D^0_{11}(x,x_1)|\Delta\partial_x^kh|^2dxdx_1+\mu_1^2\mu_2^2\int_{\Gamma_{\pm}(t)}\int_{\Gamma_{\pm}(t)}D_{22}^0(x,x_1)|\Delta\partial_x^k\theta|^2dxdx_1\right)^\frac{1}{2},$$
$$B_5(t):=\left(1+\delta_2+\sigma^{-1}\|\partial_x\theta\|_{H^{\tilde{k}}_{\gamma(t)}}\right)^k,$$
and recall the notation $B_i$ in (\ref{B_i}).
By symmetry, it holds
$$\begin{aligned}
\int_{\Gamma_{\pm}(t)}\int_{\mathbb{R}}\left|\frac{\Delta\partial_x^kh(x,x-\alpha)}{\alpha}\right|^2dxd\alpha=&2\int_{\Gamma_{\pm}(t)}P.V.\int_{\mathbb{R}}\frac{\partial_x^kh(x)-\partial_x^kh(x-\alpha)}{\alpha^2}d\alpha\cdot\overline{\partial_x^kh(x)}\,dx\\
=&2\pi\int_{\Gamma_{\pm}(t)}\Lambda\partial_x^kh\cdot\overline{\partial_x^kh}\,dx.
\end{aligned}$$
Hence $\int_{\Gamma_{\pm}(t)}\int_{\Gamma_{\pm}(t)}D_{11}^0(x,x_1)|\Delta\partial_x^kh|^2dxdx_1\simeq
\|h\|_{\dot{H}^{k+\frac{1}{2}}_{\gamma(t)}},$ and consequently 
$$\begin{aligned}
&\sum_{j=0}^k\int_{\Gamma_{\pm(t)}}\int_{\Gamma_{\pm(t)}}D_{11}^0(x,x_1)|\Delta\partial_x^jh|^2dxdx_1
\lesssim\int_{\Gamma_{\pm(t)}}\int_{\Gamma_{\pm(t)}}D_{11}^0(x,x_1)|\left(|\Delta\partial_x^kh|^2+|\Delta h|^2\right)dxdx_1.
\end{aligned}$$
As to the dissipation in $\theta$, we write
$$\begin{aligned}
&\int_{\Gamma_{\pm(t)}}\int_{\mathbb{R}}\frac{(2\sigma)^2}{\alpha^2+(2\sigma)^2}\left|\frac{\Delta\partial_x^k\theta(x,x-\alpha)}{\alpha}\right|^2dxd\alpha\\
=&2\int_{\Gamma_{\pm}(t)}\int_{\mathbb{R}}P.V.\frac{(2\sigma)^2}{\alpha^2+(2\sigma)^2}\frac{1}{\alpha^2}(\partial_x^k\theta(x)-\partial_x^k\theta(x-\alpha))d\alpha\cdot\overline{\partial_x^k\theta(x)}\,dx\\
=&2\int_{\Gamma_{\pm}(t)}P.V.\int_{\mathbb{R}}\left(\frac{1}{\alpha}-\frac{1}{2\sigma}\arctan\frac{2\sigma}{\alpha}\right)\partial_x^{k+1}\theta(x-\alpha)d\alpha\cdot\overline{\partial_x^k\theta(x)}\,dx.
\end{aligned}$$
The Fourier transform of $\frac{1}{\alpha}-\frac{1}{2\sigma}\arctan\frac{2\sigma}{\alpha}$ can be obtained through a simple computation:
$$\mathcal{F}\left(\frac{1}{\alpha}-\frac{1}{2\sigma}\arctan\frac{2\sigma}{\alpha}\right)(\xi)=-i\pi \text{sign}(\xi)\left(1-\frac{1-e^{-2\sigma|\xi|}}{2\sigma|\xi|}\right).$$
Therefore, Plancheral identity yields
\begin{equation}\begin{aligned}
&\int_{\Gamma_{\pm(t)}}\int_{\mathbb{R}}\frac{(2\sigma)^2}{\alpha^2+(2\sigma)^2}\left|\frac{\Delta\partial_x^k\theta(x,x-\alpha)}{\alpha}\right|^2dxd\alpha\\
=&4\pi\int_{\mathbb{R}}|\xi|^{2k}\left(|\xi|-\frac{1-e^{-2\sigma|\xi|}}{2\sigma}\right)|\hat{\theta}(\xi)|
^2\cosh\left(2\gamma(t)\xi\right)d\xi.
\end{aligned}\label{2.99}\end{equation}
The above inequality shows 
$$\begin{aligned}
&\sum_{j=0}^k\int_{\Gamma_{\pm(t)}}\int_{\Gamma_{\pm(t)}}D_{22}^0(x,x_1)|\Delta\partial_x^j\theta|^2dxdx_1
\lesssim\int_{\Gamma_{\pm(t)}}\int_{\Gamma_{\pm(t)}}D_{22}^0(x,x_1)|\left(|\Delta\partial_x^k\theta|^2+|\Delta\theta|^2\right)dxdx_1,
\end{aligned}$$
and the claim is proved. We also remark that (\ref{2.99}) indicates that we do not gain any derivatives of $\theta$ uniformly in $\sigma$ from the dissipation. Actually, if we suppose $|\xi|\leq\sigma^{-\frac{1}{3}}$ for example, then 
$|\xi|-\frac{1-e^{-2\sigma|\xi|}}{2\sigma}\leq\sigma|\xi|^2\leq\sigma^\frac{1}{3}.$
Therefore, as $\sigma$ approaches $0$, $|\xi|-\frac{1-e^{-2\sigma|\xi|}}{2\sigma}$ goes to $0$ for $\xi$ in the region $\left[-\sigma^{-\frac{1}{3}},\sigma^{-\frac{1}{3}}\right]$, which tends to $\mathbb{R}$.\\
\indent Now, gathering the controls given by Lemma \ref{summary3.2}, \ref{summary3.3}, \ref{summary3.4}, \ref{summary3.5}, \ref{summary3.6} and Lemma \ref{summary3.7-1}-\ref{summary3.7-3} in the equation resulted from
multiplying equation (\ref{2.1}) by $\overline{\partial^k_xh}$, multiplying equation (\ref{2.2}) by $\mu_1\mu_2\overline{\partial^k_x\theta}$, and integrating the real part of the sum along $x\in\Gamma_{\pm}(t)$ yields the following inequality as long as conditions (\ref{2.17}) and (\ref{2.65}) hold.
\begin{equation}\begin{aligned}
&\frac{d}{dt}\left(\|\partial_x^kh\|^2_{L^2_{\gamma(t)}}+\mu_1\mu_2\|\partial_x^k\theta\|^2_{L^2_{\gamma(t)}}\right)+c_0\text{Diss}_k^2\\
&-2\gamma^\prime(t)\tanh(2\gamma(t))\left(\|\Lambda^\frac{1}{2}\partial_x^kh\|_{L^2_{\gamma(t)}}^2+\mu_1\mu_2\|\Lambda^\frac{1}{2}\partial_x^k\theta\|_{L^2_{\gamma(t)}}^2\right)\\
\lesssim&(B_5+w_0)\left(\delta_2+\sigma^{\frac{1}{2}}+\|\partial_xf\|_{L^2_{\gamma(t)}}+\|\partial_xg\|_{L^2_{\gamma(t)}}\right)\left(\|h\|_{H^k_{\gamma(t)}}+\|\theta\|_{H^k_{\gamma(t)}}\right)\sum_{j=1}^k\text{Diss}_j\\
&+\left(B_4+B_5\delta_2-2\gamma^\prime(t)\tanh(2\gamma(t))\right)\left(\|h\|_{H^k_{\gamma(t)}}^2+\|\theta\|_{H^k_{\gamma(t)}}^2\right)\\
&+B_3\left(\|\partial_xf\|_{L^\infty_{\gamma(t)}}+\|\partial_xg\|_{L^\infty_{\gamma(t)}}\right)\left(\|\Lambda^\frac{1}{2}\partial_x^kh\|_{L^2_{\gamma(t)}}^2+\|\Lambda^\frac{1}{2}\partial_x^k\theta\|_{L^2_{\gamma(t)}}^2\right)\\
&+\left[\|\partial_x^2f\|_{L^\infty_{\gamma(t)}}+\|\partial_x^2g\|_{L^\infty_{\gamma(t)}}+\sigma^{-1}\|\partial_x\theta\|_{L^\infty_{\gamma(t)}}\left(\|\partial_xf\|_{L^\infty_{\gamma(t)}}+\|\partial_xg\|_{L^\infty_{\gamma(t)}}\right)\right]\\
&\cdot\left(\|\Lambda^\frac{1}{2}\partial_x^kh\|_{L^2_{\gamma(t)}}+\|\Lambda^\frac{1}{2}\partial_x^k\theta\|_{L^2_{\gamma(t)}}\right)\left(\|\partial_x^kh\|_{L^2_{\gamma(t)}}+\|\partial_x^k\theta\|_{L^2_{\gamma(t)}}\right),
\end{aligned}\label{2.100}\end{equation}
where $c_0$ is a universal constant. Meanwhile, if we set $k=0$ so that there is no commutator, then
\begin{equation}\begin{aligned}
&\frac{d}{dt}\left(\|h\|^2_{L^2_{\gamma(t)}}+\mu_1\mu_2\|\theta\|^2_{L^2_{\gamma(t)}}\right)+c_0\text{Diss}_0^2\\
&-2\gamma^\prime(t)\tanh(2\gamma(t))\left(\|\Lambda^\frac{1}{2}h\|_{L^2_{\gamma(t)}}^2+\mu_1\mu_2\|\Lambda^\frac{1}{2}\theta\|_{L^2_{\gamma(t)}}^2\right)\\
\lesssim&\left[w_0\left(\|\partial_x^2f\|_{L^2_{\gamma(t)}}+\|\partial_x^2g\|_{L^2_{\gamma(t)}}\right)+\sigma^{-1}\|\partial_x\theta\|_{L^\infty_{\gamma(t)}}\left(\sigma^\frac{1}{2}+\|\partial_xf\|_{L^2_{\gamma(t)}}+\|\partial_xg\|_{L^2_{\gamma(t)}}\right)\right]\\
&\cdot\left(\|h\|_{L^2_{\gamma(t)}}+\|\theta\|_{L^2_{\gamma(t)}}\right)\text{Diss}_0\\
&+\left[B_4+\left(1+\sigma^{-1}\|\partial_x\theta\|_{L^\infty_{\gamma(t)}}\right)\left(\|\partial_xf\|_{L^\infty_{\gamma(t)}}+\|\partial_xg\|_{L^\infty_{\gamma(t)}}\right)-2\gamma^\prime\tanh(2\gamma(t))\right]\\
&\cdot\left(\|h\|_{L^2_{\gamma(t)}}^2+\|\theta\|_{L^2_{\gamma(t)}}^2\right)+B_3\left(\|\partial_xf\|_{L^\infty_{\gamma(t)}}+\|\partial_xg\|_{L^\infty_{\gamma(t)}}\right)\left(\|\Lambda^\frac{1}{2}h\|^2_{L^2_{\gamma(t)}}+\|\Lambda^\frac{1}{2}\theta\|_{L^2_{\gamma(t)}}^2\right)\\
&+\left[\|\partial_x^2f\|_{L^\infty_{\gamma(t)}}+\|\partial_x^2g\|_{L^\infty_{\gamma(t)}}+\sigma^{-1}\|\partial_x\theta\|_{L^\infty_{\gamma(t)}}\left(\|\partial_xf\|_{L^\infty_{\gamma(t)}}+\|\partial_xg\|_{L^\infty_{\gamma(t)}}\right)\right]\\
&\cdot\left(\|\Lambda^\frac{1}{2}h\|_{L^2_{\gamma(t)}}+\|\Lambda^\frac{1}{2}\theta\|_{L^2_{\gamma(t)}}\right)\left(\|h\|_{L^2_{\gamma(t)}}+\|\theta\|_{L^2_{\gamma(t)}}\right).
\end{aligned}\label{2.101}\end{equation}
Use Cauchy-Schwarz inequality and (\ref{2.99}) for the first terms on the right-hand sides of (\ref{2.100}) and (\ref{2.101}) to get
\begin{equation}\begin{aligned}
&(B_5+w_0)\left(\delta_2+\sigma^{\frac{1}{2}}+\|\partial_xf\|_{L^2_{\gamma(t)}}+\|\partial_xg\|_{L^2_{\gamma(t)}}\right)\left(\|h\|_{H^k_{\gamma(t)}}+\|\theta\|_{H^k_{\gamma(t)}}\right)\sum_{j=1}^k\text{Diss}_j\\
&+\left[w_0\left(\|\partial_x^2f\|_{L^2_{\gamma(t)}}+\|\partial_x^2g\|_{L^2_{\gamma(t)}}\right)+\sigma^{-1}\|\partial_x\theta\|_{L^\infty_{\gamma(t)}}\left(\sigma^\frac{1}{2}+\|\partial_xf\|_{L^2_{\gamma(t)}}+\|\partial_xg\|_{L^2_{\gamma(t)}}\right)\right]\\
&\cdot\left(\|h\|_{L^2_{\gamma(t)}}+\|\theta\|_{L^2_{\gamma(t)}}\right)\text{Diss}_0\\
\leq&Cb_1^{-1}(B_5+w_0)^2\left(\delta_2^2+\sigma+\|\partial_xf\|^2_{L^2_{\gamma(t)}}+\|\partial_xg\|^2_{L^2_{\gamma(t)}}\right)\left(\|h\|^2_{H^k_{\gamma(t)}}+\|\theta\|^2_{H^k_{\gamma(t)}}\right)+b_1\left(\text{Diss}_0^2+\text{Diss}_k^2\right).
\end{aligned}\label{2.102}\end{equation}
Next, using the pointwise interpolation inequality  $\|\partial^2_xf\|_{L^\infty}\lesssim\|\partial_xf\|_{L^\infty}^\frac{1}{2}\|\partial_x^3f\|_{L^\infty}^\frac{1}{2}$ (see \cite[Theorem 1]{Mazya1999}) and Cauchy-Schwarz inequality for the last terms on the right-hand sides of (\ref{2.100})(\ref{2.101}) yields 
$$\begin{aligned}
&\left[\|\partial_x^2f\|_{L^\infty_{\gamma(t)}}+\|\partial_x^2g\|_{L^\infty_{\gamma(t)}}+\sigma^{-1}\|\partial_x\theta\|_{L^\infty_{\gamma(t)}}\left(\|\partial_xf\|_{L^\infty_{\gamma(t)}}+\|\partial_xg\|_{L^\infty_{\gamma(t)}}\right)\right]\\
&\cdot\left[\left(\|\Lambda^\frac{1}{2}\partial_x^kh\|_{L^2_{\gamma(t)}}+\|\Lambda^\frac{1}{2}\partial_x^k\theta\|_{L^2_{\gamma(t)}}\right)\left(\|\partial_x^kh\|_{L^2_{\gamma(t)}}+\|\partial_x^k\theta\|_{L^2_{\gamma(t)}}\right)\right.\\
&\left.+\left(\|\Lambda^\frac{1}{2}h\|_{L^2_{\gamma(t)}}+\|\Lambda^\frac{1}{2}\theta\|_{L^2_{\gamma(t)}}\right)\left(\|h\|_{L^2_{\gamma(t)}}+\|\theta\|_{L^2_{\gamma(t)}}\right)\right]\\
\lesssim&\left(\|\partial_xf\|_{L^\infty_{\gamma(t)}}+\|\partial_xg\|_{L^\infty_{\gamma(t)}}\right)\left(\|\Lambda^\frac{1}{2}h\|^2_{H^{k}_{\gamma(t)}}+\|\Lambda^\frac{1}{2}\theta\|^2_{H^k_{\gamma(t)}}\right)\\
&+\left[\|\partial_x^3f\|_{L^\infty_{\gamma(t)}}+\|\partial_x^3g\|_{L^\infty_{\gamma(t)}}+\sigma^{-2}\|\partial_x\theta\|_{L^\infty_{\gamma(t)}}^2\left(\|\partial_xf\|_{L^\infty_{\gamma(t)}}+\|\partial_xg\|_{L^\infty_{\gamma(t)}}\right)\right]\left(\|h\|^2_{H^k_{\gamma(t)}}+\|\theta\|_{H^k_{\gamma(t)}}^2\right).
\end{aligned}$$
Summing (\ref{2.100})(\ref{2.101}) and choosing $b_1$ small enough shows 
\begin{equation}\begin{aligned}
&\frac{d}{dt}\left(\|h\|^2_{H^k_{\gamma(t)}}+\mu_1\mu_2\|\theta\|^2_{H^k_{\gamma(t)}}\right)+\frac{c_0}{2}\left(\text{Diss}_0^2+\text{Diss}_k^2\right)\\
&-\left(2\gamma^\prime(t)\tanh(2\gamma(t))+C_3 B_7\right)\left(\|\Lambda^\frac{1}{2}h\|_{H^k_{\gamma(t)}}^2+\mu_1\mu_2\|\Lambda^\frac{1}{2}\theta\|_{H^k_{\gamma(t)}}^2\right)\\
\lesssim&\left(B_6-2\gamma^\prime(t)\tanh(2\gamma(t))\right)\left(\|h\|^2_{H^k_{\gamma(t)}}+\|\theta\|^2_{H^k_{\gamma(t)}}\right),
\end{aligned}\label{2.103}\end{equation}
where $C_3$ denotes the implicit constant and
$$\begin{aligned}
B_6(t):=&(B_5+w_0)^2\left(\delta_2^2+\sigma+\|\partial_xf\|^2_{L^2_{\gamma(t)}}+\|\partial_xg\|^2_{L^2_{\gamma(t)}}\right)+B_4+B_5\delta_2,
\end{aligned}$$
$$B_7(t):=B_3\left(\|\partial_xh\|_{L^\infty_{\gamma(t)}}+\|\partial_x\theta\|_{L^\infty_{\gamma(t)}}\right).$$
Here we also used $\|\partial_xf\|_{L^\infty_{\gamma(t)}}+\|\partial_xg\|_{L^\infty_{\gamma(t)}}\simeq\|\partial_xh\|_{L^\infty_{\gamma(t)}}+\|\partial_x\theta\|_{L^\infty_{\gamma(t)}}.$
\subsection{Energy estimate dependent on $\sigma$}\label{energy2}
In order to construct the solution, we also establish an energy estimate without losing analyticity. We treat the last terms on the right-hand sides of (\ref{2.100})(\ref{2.101}) in an alternative way:
\begin{equation}\begin{aligned}
&\left[\|\partial_x^2f\|_{L^\infty_{\gamma(t)}}+\|\partial_x^2g\|_{L^\infty_{\gamma(t)}}+\sigma^{-1}\|\partial_x\theta\|_{L^\infty_{\gamma(t)}}\left(\|\partial_xf\|_{L^\infty_{\gamma(t)}}+\|\partial_xg\|_{L^\infty_{\gamma(t)}}\right)\right]\\
&\cdot\left[\left(\|\Lambda^\frac{1}{2}\partial_x^kh\|_{L^2_{\gamma(t)}}+\|\Lambda^\frac{1}{2}\partial_x^k\theta\|_{L^2_{\gamma(t)}}\right)\left(\|\partial_x^kh\|_{L^2_{\gamma(t)}}+\|\partial_x^k\theta\|_{L^2_{\gamma(t)}}\right)\right.\\
&\left.+\left(\|\Lambda^\frac{1}{2}h\|_{L^2_{\gamma(t)}}+\|\Lambda^\frac{1}{2}\theta\|_{L^2_{\gamma(t)}}\right)\left(\|h\|_{L^2_{\gamma(t)}}+\|\theta\|_{L^2_{\gamma(t)}}\right)\right]\\
\leq
&Cb_2^{-1}\left[\|\partial_x^2f\|_{L^\infty_{\gamma(t)}}+\|\partial_x^2g\|_{L^\infty_{\gamma(t)}}+\sigma^{-1}\|\partial_x\theta\|_{L^\infty_{\gamma(t)}}\left(\|\partial_xf\|_{L^\infty_{\gamma(t)}}+\|\partial_xg\|_{L^\infty_{\gamma(t)}}\right)\right]^2\\
&\cdot\left(\|h\|^2_{H^k_{\gamma(t)}}+\|\theta\|_{H^k_{\gamma(t)}}^2\right)+b_2\left(\|\Lambda^\frac{1}{2}h\|^2_{H^{k}_{\gamma(t)}}+\|\Lambda^\frac{1}{2}\theta\|^2_{H^k_{\gamma(t)}}\right)\\
\leq&Cb_2^{-1}(B_5\delta_2)^2\left(\|h\|^2_{H^k_{\gamma(t)}}+\|\theta\|^2_{H^k_{\gamma(t)}}\right)+b_2\left(\|\Lambda^\frac{1}{2}h\|^2_{H^{k}_{\gamma(t)}}+\|\Lambda^\frac{1}{2}\theta\|^2_{H^k_{\gamma(t)}}\right).
\end{aligned}\label{2.104}\end{equation}
Meanwhile, in view of (\ref{2.99}), it holds
$$\begin{aligned}
&\int_{\Gamma_{\pm(t)}}\int_{\mathbb{R}}\frac{(2\sigma)^2}{\alpha^2+(2\sigma)^2}\left|\frac{\Delta\partial_x^k\theta(x,x-\alpha)}{\alpha}\right|^2dxd\alpha\geq 2\pi\|\Lambda^\frac{1}{2}\partial_x^k\theta\|_{L^2_{\gamma(t)}}^2-\pi\sigma^{-1}\|\partial_x^k\theta\|_{L^2_{\gamma(t)}}^2.
\end{aligned}$$
This implies
\begin{equation}\text{Diss}_0^2+\text{Diss}_k^2\gtrsim\|\Lambda^\frac{1}{2}h\|_{H^k_{\gamma(t)}}^2+\|\Lambda^\frac{1}{2}\theta\|_{H^k_{\gamma(t)}}^2-C\sigma^{-1}\|\theta\|^2_{H^k_{\gamma(t)}}.
\label{2.105}\end{equation}
Choosing $b_2$ small enough and inserting (\ref{2.104})(\ref{2.105}) into (\ref{2.100}) and (\ref{2.101}) yields that
$$\begin{aligned}
&\frac{d}{dt}\left(\|h\|^2_{H^k_{\gamma(t)}}+\mu_1\mu_2\|\theta\|^2_{H^k_{\gamma(t)}}\right)-2\gamma^\prime(t)\tanh(2\gamma(t))\left(\|\Lambda^\frac{1}{2}h\|_{H^k_{\gamma(t)}}^2+\mu_1\mu_2\|\Lambda^\frac{1}{2}\theta\|_{H^k_{\gamma(t)}}^2\right)\\
\lesssim&\left(B_6-2\gamma^\prime(t)\tanh(2\gamma(t))\right)\left(\|h\|^2_{H^k_{\gamma(t)}}+\|\theta\|^2_{H^k_{\gamma(t)}}\right)+\sigma^{-1}\|\theta\|^2_{H^k_{\gamma(t)}}\\
&+\left(B_7-\frac{c_0}{C_4}\right)\left(\|\Lambda^\frac{1}{2}h\|^2_{H^k_{\gamma(t)}}+\|\Lambda^\frac{1}{2}\theta\|^2_{H^k_{\gamma(t)}}\right),
\end{aligned}$$
where $C_4$ is a constant. If we let $\gamma(t)=\tilde{\gamma}$ be a positive constant, then
\begin{equation}\begin{aligned}
&\frac{d}{dt}\left(\|h\|^2_{H^k_{\tilde{\gamma}}}+\mu_1\mu_2\|\theta\|^2_{H^k_{\tilde{\gamma}}}\right)\\
\lesssim& B_6\left(\|h\|^2_{H^k_{\tilde{\gamma}}}+\|\theta\|^2_{H^k_{\tilde{\gamma}}}\right)+\sigma^{-1}\|\theta\|^2_{H^k_{\tilde{\gamma}}}+\left(B_7-\frac{c_0}{C_4}\right)\left(\|\Lambda^\frac{1}{2}h\|^2_{H^k_{\tilde{\gamma}}}+\|\Lambda^\frac{1}{2}\theta\|^2_{H^k_{\tilde{\gamma}}}\right).
\end{aligned}\label{2.106}\end{equation}

\section{Control for distance function}\label{distance}
As presented by the coefficients $B_6$ and $B_7$ of (\ref{2.103}), to close the energy estimate, we still need to find a control for $\sigma^{-1}\theta$. To this end, we set $\theta_1:=\sigma^{-1}\theta$ and let $k_1:=k-3$. Notice that $k_1\geq\tilde{k}+2$ since $k\geq10$.  Denote by ${N_i^\theta}^\prime$ the commutators obtained from replacing $k$ by $k_1$ in the expressions ${N_i^\theta}$. Then we replace $k$ by $k_1$ in the equation (\ref{2.2}) and multiply it by $\sigma^{-1}$ to obtain
\begin{equation}\begin{aligned}
&\partial_t\partial_x^{k_1}\theta_1+(\mu_1u_++\mu_2u_-)\partial_x^{k_1+1}\theta_1\\
=&\Delta\rho\,\partial_x^{k_1}\left[P.V.\int_{\mathbb{R}}\sigma^{-1}\left(\mu_2(P_{11}-P_{21})-\mu_1(P_{22}-P_{12})\right)(x,x-\alpha)\tilde{\Delta}\partial_xh \,d\alpha\right]\\
&-\mu_1\mu_2\Delta\rho\,P.V.\int_{\mathbb{R}}\left(P_{11}+P_{22}-P_{12}-P_{21}\right)(x,x-\alpha)\partial_x^{k_1+1}\theta_1(x-\alpha)d\alpha\\
&+\sigma^{-1}\left({N_2^{\theta}}^\prime+{N_4^{\theta}}^\prime\right).
\end{aligned}\label{3.1}\end{equation}
Here we note that
$$\begin{aligned}
&\Delta\rho\,\partial_x^{k_1}\left[P.V.\int_{\mathbb{R}}\left(\mu_2(P_{11}-P_{21})-\mu_1(P_{22}-P_{12})\right)(x,x-\alpha)\tilde{\Delta}\partial_xh \,d\alpha\right]\\
=&-\Delta\rho\,P.V.\int_{\mathbb{R}}\left(\mu_2(P_{11}-P_{21})-\mu_1(P_{22}-P_{12})\right)(x,x-\alpha)\partial_x^{k_1+1}h(x-\alpha)d\alpha\\
&-(u_+-u_-)\partial_x^{k_1+1}h+{N^{\theta}_{1}}^\prime+{N^{\theta}_3}^\prime.
\end{aligned}$$ 
Then we multiply (\ref{3.1}) by $\overline{\partial_x^{k_1}\theta_1}$ and integrate the real part over $\Gamma_{\pm(t)}$ to get the equation of $\frac{d}{dt}\|\partial_x^{k_1}\theta_1\|_{L^2_{\gamma(t)}}^2$.
In fact, the estimates for the contribution of the terms $(\mu_1u_++\mu_2u_-)\partial_x^{{k_1}+1}\theta_1$, $\sigma^{-1}\left({N_2^\theta}^\prime+{N_4^\theta}^\prime\right)$ and $-\mu_1\mu_2\Delta\rho\,P.V.\int_{\mathbb{R}}\left(P_{11}+P_{22}-P_{12}-P_{21}\right)(x,x-\alpha)\partial_x^{k_1+1}\theta_1(x-\alpha)d\alpha$  are already available. First, using (\ref{2.61}) with $u=u_+$ and $u=u_-$ immediately yields
\begin{equation}\begin{aligned}
&\left|\int_{\Gamma_{\pm}(t)}(\mu_1u_++\mu_2u_-)\partial_x^{{k_1}+1}\theta_1\overline{\partial_x^{k_1}\theta_1}\,dx\right|\\
\lesssim&B_3\left(\|\partial_xf\|_{L^\infty_{\gamma(t)}}+\|\partial_xg\|_{L^\infty_{\gamma(t)}}\right)\|\partial_x^{k_1}\theta_1\|_{\dot{H}^\frac{1}{2}_{\gamma(t)}}^2+B_4\|\partial_x^{k_1}\theta_1\|_{L^2_{\gamma(t)}}^2,
\end{aligned}\label{3.2}\end{equation}
In view of (\ref{2.6})(\ref{2.7})(\ref{2.8}), we have
\begin{equation}\begin{aligned}
&-\Re\int_{\Gamma_{\pm}(t)}P.V.\int_{\Gamma_{\pm}(t)}\left(P_{11}+P_{22}-P_{12}-P_{21}\right)(x,x_1)\partial_x^{k_1+1}\theta_1(x_1)dx_1\,\overline{\partial_x^{k_1}\theta_1(x)}dx\\
=&-\frac{1}{2}\int_{\Gamma_{\pm}(t)}\int_{\Gamma_{\pm}(t)}\Re\left(K_{11}-\frac{1}{2}K_{12}-\frac{1}{2}K_{21}\right)(x,x_1)|\Delta\partial_x^{k_1}\theta_1|^2dxdx_1,\\
&-\frac{1}{2}\int_{\Gamma_{\pm}(t)}\int_{\Gamma_{\pm}(t)}\Re\left(K_{22}-\frac{1}{2}\tilde{K}_{12}-\frac{1}{2}\tilde{K}_{21}\right)(x,x_1)|\Delta\partial_x^{k_1}\theta_1|^2dxdx_1.
\\
&-\Re\int_{\Gamma_{\pm}(t)}\int_{\Gamma_{\pm}(t)}\left(J_{11}+J_{22}\right)(x,x_1)\cdot\Delta\partial_x^{k_1}\theta_1\cdot\overline{\partial_x^{k_1}\theta_1(x)}dxdx_1.\\
&+\frac{1}{2}\Re\int_{\Gamma_{\pm}(t)}\int_{\Gamma_{\pm}(t)}(J_{12}+J_{21}+\tilde{J}_{12}+\tilde{J}_{21})(x,x_1)\cdot\Delta\partial_x^{k_1}\theta_1\cdot\overline{\partial_x^{k_1}\theta_1(x)}dxdx_1.
\end{aligned}\label{3.3}\end{equation}
From the proof of (\ref{2.34}), the first and second terms can be estimated as follows.
\allowdisplaybreaks[4]\begin{align*}
&\frac{1}{2}\int_{\Gamma_{\pm}(t)}\int_{\Gamma_{\pm}(t)}\Re\left(K_{11}-\frac{1}{2}K_{12}-\frac{1}{2}K_{21}\right)(x,x_1)|\Delta\partial_x^{k_1}\theta_1|^2dxdx_1,\\
&+\frac{1}{2}\int_{\Gamma_{\pm}(t)}\int_{\Gamma_{\pm}(t)}\Re\left(K_{22}-\frac{1}{2}\tilde{K}_{12}-\frac{1}{2}\tilde{K}_{21}\right)(x,x_1)|\Delta\partial_x^{k_1}\theta_1|^2dxdx_1\\
\geq&\frac{c_0}{2}\int_{\Gamma_{\pm}(t)}\int_{\Gamma_{\pm}(t)}D_{22}^0(x,x_1)|\Delta\partial_x^{k_1}\theta_1|^2dxdx_1\\
&-C\|\partial_x\theta_1\|_{L^\infty_{\gamma(t)}}\left(\|\partial_xf\|_{L^\infty_{\gamma(t)}}+\|\partial_xg\|_{L^\infty_{\gamma(t)}}\right)\|\partial_x^{k_1}\theta_1\|_{L^2_{\gamma(t)}}^2.
\end{align*}\allowdisplaybreaks[0]
By (\ref{2.38})(\ref{2.39}), the absolute value of the last two terms can be bounded by
$$\begin{aligned}
&\left[\left(\|\partial_xf\|_{L^\infty_{\gamma(t)}}+\|\partial_xg\|_{L^\infty_{\gamma(t)}}+\left\|\theta_1\right\|_{L^\infty_{\gamma(t)}}\right)\left(\|\partial_x^2f\|_{L^2_{\gamma(t)}}+\|\partial_x^2g\|_{L^2_{\gamma(t)}}\right)+\|\partial_x\theta_1\|_{L^\infty_{\gamma(t)}}\right.\\
&\left.\cdot\left(\|\partial_xf\|_{L^2_{\gamma(t)}}+\|\partial_xg\|_{L^2_{\gamma(t)}}+\sigma^{\frac{1}{2}}\right)\right]\|\partial_x^{k_1}\theta_1\|_{L^2_{\gamma(t)}}\left(\int_{\Gamma_{\pm}(t)}\int_{\Gamma_{\pm}(t)}D_{22}^0|\Delta\partial_x^{k_1}\theta_1|^2dxdx_1\right)^\frac{1}{2}.
\end{aligned}$$
Taking the sum yields
\begin{equation}\begin{aligned}
&-\Re\int_{\Gamma_{\pm}(t)}P.V.\int_{\Gamma_{\pm}(t)}\left(P_{11}+P_{22}-P_{12}-P_{21}\right)(x,x_1)\partial_x^{k_1+1}\theta_1(x_1)dx_1\,\overline{\partial_x^{k_1}\theta_1(x)}dx\\
&+\frac{c_0}{2}\int_{\Gamma_{\pm(t)}}\int_{\Gamma_{\pm(t)}}D_{22}^0(x,x_1)|\Delta\partial_x^{k_1}\theta_1|^2dxdx_1\\
\lesssim&\left[\left(\|\partial_xf\|_{L^\infty_{\gamma(t)}}+\|\partial_xg\|_{L^\infty_{\gamma(t)}}+\left\|\theta_1\right\|_{L^\infty_{\gamma(t)}}\right)\left(\|\partial_x^2f\|_{L^2_{\gamma(t)}}+\|\partial_x^2g\|_{L^2_{\gamma(t)}}\right)+\|\partial_x\theta_1\|_{L^\infty_{\gamma(t)}}\right.\\
&\left.\cdot\left(\|\partial_xf\|_{L^2_{\gamma(t)}}+\|\partial_xg\|_{L^2_{\gamma(t)}}+\sigma^{\frac{1}{2}}\right)\right]\|\partial_x^{k_1}\theta_1\|_{L^2_{\gamma(t)}}\left(\int_{\Gamma_{\pm}(t)}\int_{\Gamma_{\pm}(t)}D_{22}^0|\Delta\partial_x^{k_1}\theta_1|^2dxdx_1\right)^\frac{1}{2}\\
&+\|\partial_x\theta_1\|_{L^\infty_{\gamma(t)}}\left(\|\partial_xf\|_{L^\infty_{\gamma(t)}}+\|\partial_xg\|_{L^\infty_{\gamma(t)}}\right)\|\partial_x^{k_1}\theta_1\|_{L^2_{\gamma(t)}}^2.
\end{aligned}\label{3.4}\end{equation}
For ${N_2^\theta}^\prime+{N_4^\theta}^\prime$, recall that 
$$\begin{aligned}
&(\Delta\rho)^{-1}({N_2^\theta}^\prime+{N_4^\theta}^\prime)\\
=&\mu_1\mu_2\,P.V.\int_{\mathbb{R}}[\partial_x^{k_1},(P_{11}+P_{22}-P_{12}-P_{21})(x,x-\alpha)](\partial_x\theta(x)-\partial_x\theta(x-\alpha))d\alpha\\
&+\mu_1\mu_2\,P.V.\int_{\mathbb{R}}[\partial_x^{k_1},(\mu_2P_{21}+\mu_1P_{12})(x,x-\alpha)]\partial_x\theta(x)d\alpha.
\end{aligned}$$The first term on the right-hand side can be controlled as follows by completely repeating the argument for the safe and easy terms in class (\romannumeral1) and class (\romannumeral2), see (\ref{2.64})(\ref{2.68}-\ref{2.70})(\ref{2.74})(\ref{2.87}).
\begin{equation}\begin{aligned}
&\left\|P.V.\int_{\mathbb{R}}[\partial_x^{k_1},(P_{11}+P_{22}-P_{12}-P_{21})(x,x-\alpha)](\partial_x\theta(x)-\partial_x\theta(x-\alpha))d\alpha\right\|_{L^2_{\gamma(t)}}\\
\lesssim&B_5\delta_2\|\theta\|_{H^{k_1}_{\gamma(t)}}+B_5\sigma^\frac{1}{2}\sum_{j=0}^{k_1}\left(\int_{\Gamma_{\pm}(t)}\int_{\mathbb{R}}D_{22}^0(x,x-\alpha)|\tilde{\Delta}\partial_x^j\theta|^2d\alpha dx\right)^\frac{1}{2}\\
&+B_5\left(\|\partial_x^2\theta\|_{H^{\tilde{k}_1}_{\gamma(t)}}+\|\partial_x\theta\|_{L^\infty_{\gamma(t)}}\right)\left(\|\partial_x^5h\|_{H^{k_1-5}_{\gamma(t)}}+\|\partial_x^5\theta\|_{H^{k_1-5}_{\gamma(t)}}\right).
\end{aligned}\label{3.5}\end{equation}
Here $\tilde{k}_1:=\left\lfloor\frac{k_1+1}{2}\right\rfloor$ , and we used $\tilde{k}_1\leq\tilde{k}$.
The second term is a linear combination of 
$$\int_{\mathbb{R}}\alpha^{-1}q^{(l)}\left(\frac{\tilde{\Delta} g+2\sigma+\theta(x)}{\alpha}\right)\prod_{a=1}^j\left(\frac{\tilde{\Delta}\partial_x^ag+\partial_x^a\theta(x)}{\alpha}\right)^{r_a}d\alpha\cdot\partial_x^{k_1+1-j}\theta(x),$$
$$\int_{\mathbb{R}}\alpha^{-1}q^{(l)}\left(\frac{\tilde{\Delta} f-2\sigma-\theta(x)}{\alpha}\right)\prod_{a=1}^j\left(\frac{\tilde{\Delta}\partial_x^af-\partial_x^a\theta(x)}{\alpha}\right)^{r_a}d\alpha\cdot\partial_x^{k_1+1-j}\theta(x)$$
for $1\leq l\leq j\leq k_1$ and $r\in S_{j,l}$. Suppose first $r_a=0$ for all $a>\tilde{k}_1$. Completely repeating the argument for the safe and easy terms in class (\romannumeral3) in the case $r_a=0$ for all $a>\tilde{k}$  (see (\ref{2.71})(\ref{2.73})(\ref{2.92})), we already have
\begin{equation}\begin{aligned}
&\left\|\int_{\mathbb{R}}\alpha^{-1}q^{(l)}\left(\frac{\tilde{\Delta} g+2\sigma+\theta(x)}{\alpha}\right)\prod_{a=1}^j\left(\frac{\tilde{\Delta}\partial_x^ag+\partial_x^a\theta(x)}{\alpha}\right)^{r_a}\partial_x^{k_1+1-j}\theta(x)\,d\alpha\right\|_{L^2_{\gamma(t)}}\\
\lesssim&\left(1+\|\partial_x^2g\|_{L^2_{\gamma(t)}}+\sigma^{-1}\|\partial_x\theta\|_{H^{\tilde{k}_1}_{\gamma(t)}}\right)\left(\|\partial_xg\|_{L^\infty_{\gamma(t)}}+\|\partial_x^2g\|_{H^{\tilde{k}_1}_{\gamma(t)}}\right)\\&\cdot\left(\|\partial_x^2g\|_{H^{\tilde{k}_1}_{\gamma(t)}}+\sigma^{-1}\|\partial_x\theta\|_{H^{\tilde{k}_1}_{\gamma(t)}}\right)^{l-1}\|\theta\|_{H^{k_1}_{\gamma(t)}}.
\end{aligned}\label{3.6}\end{equation}
Otherwise, there exists $a_0>\tilde{k}_1$ such that $r_{a_0}\neq0$ and thus $j>\tilde{k}_1$, $r_{a_0}=1$, $r_a=0$ for all $a>\tilde{k}_1$ with $a\neq a_0$. In this case,
$$\begin{aligned}
&\int_{\mathbb{R}}\alpha^{-1}q^{(l)}\left(\frac{\tilde{\Delta} g+2\sigma+\theta(x)}{\alpha}\right)\prod_{a=1}^j\left(\frac{\tilde{\Delta}\partial_x^ag+\partial_x^a\theta(x)}{\alpha}\right)^{r_a}\partial_x^{k_1+1-j}\theta(x)\,d\alpha\\
=&\int_{\mathbb{R}}q^{(l)}\left(\frac{\tilde{\Delta} g+2\sigma+\theta(x)}{\alpha}\right)\prod_{a\neq a_0}\left(\frac{\tilde{\Delta}\partial_x^ag+\partial_x^a\theta(x)}{\alpha}\right)^{r_a}\frac{\partial_x^{k_1+1-j}\theta(x)}{\alpha}\frac{\tilde{\Delta}\partial_x^{a_0}g}{\alpha}\,d\alpha\\
&+\int_{\mathbb{R}}q^{(l)}\left(\frac{\tilde{\Delta} g+2\sigma+\theta(x)}{\alpha}\right)\prod_{a\neq a_0}\left(\frac{\tilde{\Delta}\partial_x^ag+\partial_x^a\theta(x)}{\alpha}\right)^{r_a}\frac{\partial_x^{k_1+1-j}\theta(x)}{\alpha}\frac{\partial_x^{a_0}\theta(x)}{\alpha}\,d\alpha.
\end{aligned}$$
The $L^2_{\gamma(t)}$ norm of the second term can be controlled in the same way as in (\ref{2.93}) or (\ref{2.91}) by
$$\left(1+\sigma^{-1}\|\partial_x\theta\|_{H^{\tilde{k}_1}_{\gamma(t)}}\right)\left(\|\partial_xg\|_{L^\infty_{\gamma(t)}}+\|\partial_x^2g\|_{H^{\tilde{k}_1}_{\gamma(t)}}\right)\left(\|\partial_x^2g\|_{H^{\tilde{k}_1}_{\gamma(t)}}+\sigma^{-1}\|\partial_x\theta\|_{H^{\tilde{k}_1}_{\gamma(t)}}\right)^{l-1}\|\theta\|_{H^{k_1}_{\gamma(t)}}.$$
For the first term, we will not follow (\ref{2.96}) since it cannot provide the factor $\|\theta\|_{H^{k_1}_{\gamma(t)}}$ in order to get a $\sigma-$independent estimate of $\theta_1$. Instead, we have to use the technique in (\ref{2.82}) and (\ref{2.85}). To be specific, we use the decomposition
$$\begin{aligned}
&q^{(l)}\left(\frac{\tilde{\Delta} g+2\sigma+\theta(x)}{\alpha}\right)\prod_{a\neq a_0}\left(\frac{\tilde{\Delta}\partial_x^ag+\partial_x^a\theta(x)}{\alpha}\right)^{r_a}\frac{\partial_x^{k_1+1-j}\theta(x)}{\alpha}\frac{\tilde{\Delta}\partial_x^{a_0}g}{\alpha}\\
=&\alpha^{-l}q^{(l)}\left(\frac{\tilde{\Delta} g+2\sigma+\theta(x)}{\alpha}\right)\left[\prod_{a\neq a_0}\left(\tilde{\Delta}\partial_x^ag+\partial_x^a\theta(x)\right)^{r_a}-\prod_{a\neq a_0}\left(\tilde{\Delta}\partial_x^ag\right)^{r_a}\right]\frac{\tilde{\Delta}\partial_x^{a_0}g}{\alpha}\cdot\partial_x^{k_1+1-j}\theta(x)\\
&+\alpha^{-1}q^{(l)}\left(\frac{\tilde{\Delta} g+2\sigma+\theta(x)}{\alpha}\right)\prod_{a\neq a_0}\left(\frac{\tilde{\Delta}\partial_x^ag}{\alpha}\right)^{r_a}\frac{\tilde{\Delta}\partial_x^{a_0}g}{\alpha}\cdot\partial_x^{k_1+1-j}\theta(x),
\end{aligned}$$
The first term disappears if $l=1$. Recall that
$$\begin{aligned}
&q^{(l)}\left(\frac{\tilde{\Delta} g+2\sigma+\theta(x)}{\alpha}\right)
=\frac{(-i)^ll!}{2}\left[\frac{\alpha^{l+1}}{(\alpha+i(\tilde{\Delta} g+2\sigma+\theta(x)))^{l+1}}+\frac{(-1)^l\alpha^{l+1}}{(\alpha-i(\tilde{\Delta} g+2\sigma+\theta(x)))^{l+1}}\right].
\end{aligned}$$
As an analogue of (\ref{2.82}), if $l\geq2$,
there is
\begin{equation}\begin{aligned}
&\left\|\int_{\mathbb{R}}\frac{\left[\prod_{a\neq a_0}\left(\tilde{\Delta}\partial_x^ag+\partial_x^a\theta(x)\right)^{r_a}-\prod_{a\neq a_0}\left(\tilde{\Delta}\partial_x^ag\right)^{r_a}\right]\tilde{\Delta}\partial_x^{a_0}g}{(\alpha\pm i(\tilde{\Delta} g+2\sigma+\theta(x)))^{l+1}}d\alpha\cdot\partial_x^{k_1+1-j}\theta(x)\right\|_{L^2_{\gamma(t)}}\\
\lesssim&\left(\|\partial_x^2g\|_{H^{\tilde{k}_1}_{\gamma(t)}}+\sigma^{-1}\|\partial_x\theta\|_{H^{\tilde{k}_1}_{\gamma(t)}}\right)^{l-2}\sigma^{-1}\|\partial_x\theta\|_{H^{\tilde{k}_1}_{\gamma(t)}}\|\partial_x^{a_0+1}g\|_{L^\infty_{\gamma(t)}}\|\partial_x^{k_1+1-j}\theta\|_{L^2_{\gamma(t)}},
\end{aligned}\label{3.7}\end{equation}
For the second term, replacing $\partial_x^{k+2-j}h$ by $\partial_x^{a_0+1}g$ and replacing $\partial_x^{a_0}\theta$ by $\partial_x^{k_1+1-j}\theta$ in (\ref{2.85}) gives
\begin{equation}\begin{aligned}
&\left\|\int_{|\alpha|\leq1}\alpha^{-1}q^{(l)}\left(\frac{\tilde{\Delta} g+2\sigma+\theta(x)}{\alpha}\right)\prod_{a\neq a_0}\left(\frac{\tilde{\Delta}\partial_x^ag}{\alpha}\right)^{r_a}\frac{\tilde{\Delta}\partial_x^{a_0}g}{\alpha}d\alpha\cdot\partial_x^{k_1+1-j}\theta(x)\right\|_{L^2_{\gamma(t)}}\\
\lesssim&\left(1+\|\partial_x^2g\|_{L^2_{\gamma(t)}}\right)\|\partial_x^2g\|_{H^{\tilde{k}_1}_{\gamma(t)}}^{l-1}\|\partial_x^{a_0+1}g\|_{H^1_{\gamma(t)}}\|\partial_x^{k_1+1-j}\theta\|_{L^2_{\gamma(t)}}.
\end{aligned}\label{3.8}\end{equation}
Note that $a_0$ reaches $k_1$ if $l=1$ and $j=k_1$. Collecting (\ref{3.6}-\ref{3.8}) yields
$$\begin{aligned}
&\left\|\int_{\mathbb{R}}\alpha^{-1}q^{(l)}\left(\frac{\tilde{\Delta} g+2\sigma+\theta(x)}{\alpha}\right)\prod_{a=1}^j\left(\frac{\tilde{\Delta}\partial_x^ag+\partial_x^a\theta(x)}{\alpha}\right)^{r_a}\partial_x^{k_1+1-j}\theta(x)\,d\alpha\right\|_{L^2_{\gamma(t)}}\\
\lesssim&\left(1+\|\partial_x^2g\|_{L^2_{\gamma(t)}}+\sigma^{-1}\|\partial_x\theta\|_{H^{\tilde{k}_1}_{\gamma(t)}}\right)\left(\|\partial_x^2g\|_{H^{\tilde{k}_1}_{\gamma(t)}}+\sigma^{-1}\|\partial_x\theta\|_{H^{\tilde{k}_1}_{\gamma(t)}}\right)^{l-1}\\
&\cdot\left(\|\partial_x^2g\|_{H^{k_1}_{\gamma(t)}}+\|\partial_xg\|_{L^\infty_{\gamma(t)}}\right)\|\theta\|_{H^{k_1}_{\gamma(t)}}.
\end{aligned}$$
The similar control holds for $\int_{\mathbb{R}}\alpha^{-1}q^{(l)}\left(\frac{\tilde{\Delta} f-2\sigma-\theta(x)}{\alpha}\right)\prod_{a=1}^j\left(\frac{\tilde{\Delta}\partial_x^af-\partial_x^a\theta(x)}{\alpha}\right)^{r_a}\partial_x^{k_1+1-j}\theta(x)\,d\alpha$ if we replace the derivatives of $g$ by the derivatives of $f$ on the right-hand side of the above inequality.
Then we conclude
\begin{equation}\begin{aligned}
&\left\|\int_{\mathbb{R}}[\partial_x^{k_1},(\mu_2P_{21}+\mu_1P_{12})(x,x-\alpha)]\partial_x\theta(x)d\alpha\right\|_{L^2_{\gamma(t)}}\\
\lesssim&B_5\left(\|\partial_xh\|_{L^\infty_{\gamma(t)}}+\|\partial_x\theta\|_{L^\infty_{\gamma(t)}}+\|\partial_x^2h\|_{H^{k_1}_{\gamma(t)}}+\|\partial_x^2\theta\|_{H^{k_1}_{\gamma(t)}}\right)\|\theta\|_{H^{k_1}_{\gamma(t)}}.
\end{aligned}\label{3.9}\end{equation}
To bound $\partial_x^{k_1}\left[P.V.\int_{\mathbb{R}}\sigma^{-1}\left(\mu_2(P_{11}-P_{21})-\mu_1(P_{22}-P_{12})\right)(x,x-\alpha)\tilde{\Delta}\partial_xh \,d\alpha\right]$, observe that
$$\left(P_{11}-P_{21}\right)(x,x-\alpha)=\frac{\alpha(2\sigma+\theta(x))^2-2\alpha(2\sigma+\theta(x))\tilde{\Delta} f}{(\alpha^2+(\tilde{\Delta} f)^2)(\alpha^2+(\tilde{\Delta} f-2\sigma-\theta(x))^2)},$$
$$\left(P_{22}-P_{12}\right)(x,x-\alpha)=\frac{\alpha(2\sigma+\theta(x))^2+2\alpha(2\sigma+\theta(x))\tilde{\Delta} g}{(\alpha^2+(\tilde{\Delta} g)^2)(\alpha^2+(\tilde{\Delta} g+2\sigma+\theta(x))^2)}.$$
We present the details of the $L^2_{\gamma(t)}$ bound over $\partial_x^{k_1}\left(\int_{\mathbb{R}}(P_{11}-P_{21})(x,x-\alpha)\tilde{\Delta}\partial_xh\,d\alpha\right)$, and the bound over $\partial_x^{k_1}\left(\int_{\mathbb{R}}(P_{22}-P_{12})(x,x-\alpha)\tilde{\Delta}\partial_xh\,d\alpha\right)$ can be obtained in a similar way. For simplicity, denote
$$Q_1(x,x-\alpha):=\frac{\alpha(2\sigma+\theta(x))^2}{(\alpha^2+(\tilde{\Delta} f)^2)(\alpha^2+(\tilde{\Delta} f-2\sigma-\theta(x))^2)},$$
$$Q_2(x,x-\alpha):=\frac{-2\alpha(2\sigma+\theta(x))\tilde{\Delta} f}{(\alpha^2+(\tilde{\Delta} f)^2)(\alpha^2+(\tilde{\Delta} f-2\sigma-\theta(x))^2)}.$$
We also note that 
$$Q_1(x,x-\alpha)=\alpha^{-3}(2\sigma+\theta(x))^2q\left(\frac{\tilde{\Delta} f}{\alpha}\right)q\left(\frac{\tilde{\Delta} f-2\sigma-\theta(x)}{\alpha}\right),$$
$$Q_2(x,x-\alpha)=-2\alpha^{-3}\tilde{\Delta} f(2\sigma+\theta(x))q\left(\frac{\tilde{\Delta} f}{\alpha}\right)q\left(\frac{\tilde{\Delta} f-2\sigma-\theta(x)}{\alpha}\right).$$
It suffices to bound the $L^2_{\gamma(t)}$ norm of $\int_{\mathbb{R}}\partial_x^{k_1}\left(Q_{i}(x,x-\alpha)\tilde{\Delta}\partial_xh\right)d\alpha$, $i=1,2$. First, we write
$$\partial_x^{k_1}(Q_i(x,x-\alpha)\tilde{\Delta}\partial_xh)=\sum_{j=0}^{k_1}\binom{k_1}{j}\partial_x^{j}(Q_i(x,x-\alpha))\tilde{\Delta}\partial_x^{k_1+1-j}h.$$
$\partial_x^j(Q_1(x,x-\alpha))$ is a linear combination of terms in the form
$$\begin{aligned}
\alpha^{-3}\partial_x^{j_1}\left[(2\sigma+\theta(x))^{2}\right]\partial_x^{j_2}\left[q\left(\frac{\tilde{\Delta} f}{\alpha}\right)\right]\partial_x^{j_3}\left[q\left(\frac{\tilde{\Delta} f-2\sigma-\theta(x)}{\alpha}\right)\right]
\end{aligned}$$
with $j_1+j_2+j_3=j$. Hence it suffices to bound
$$\begin{aligned}
Q_{1}^{k_1,j_1,j_2,j_3}(x):=&\int_{\mathbb{R}}\alpha^{-3}\partial_x^{j_1}\left[(2\sigma+\theta(x))^{2}\right]\partial_x^{j_2}\left[q\left(\frac{\tilde{\Delta} f}{\alpha}\right)\right]\partial_x^{j_3}\left[q\left(\frac{\tilde{\Delta} f-2\sigma-\theta(x)}{\alpha}\right)\right]\tilde{\Delta}\partial_x^{k_1+1-j}h\,d\alpha
\end{aligned}$$
for each possible index $(j_1,j_2,j_3)$. To start with, we prepare some inequalities. As before, assume that (\ref{2.17}) holds. 
A simple computation gives
$$\begin{aligned}
\partial_x^{j}\left[q\left(\frac{\tilde{\Delta} f-2\sigma-\theta(x)}{\alpha}\right)\right]
=&\sum_{1\leq l\leq j}\sum_{S_{j,l}}C_{j,l,r}\,q^{(l)}\left(\frac{\tilde{\Delta} f-2\sigma-\theta(x)}{\alpha}\right)\prod_{a=1}^j\left(\frac{\tilde{\Delta}\partial_x^af-\partial_x^a\theta(x)}{\alpha}\right)^{r_a}\\
=&\sum_{1\leq l\leq j}\sum_{S_{j,l}}C_{j,l,r}\,\frac{(-i)^ll!}{2}\prod_{a=1}^j\left(\tilde{\Delta}\partial_x^a f-\partial_x^a\theta(x)\right)^{r_a}\\
&\cdot\left(\frac{\alpha}{(\alpha+i(\tilde{\Delta} f-2\sigma-\theta(x))^{l+1}}+\frac{(-1)^l\alpha}{(\alpha-i(\tilde{\Delta} f-2\sigma-\theta(x))^{l+1}}\right).
\end{aligned}$$
$$\left|\frac{\alpha}{(\alpha+i(\tilde{\Delta} f-2\sigma-\theta(x))^{l+1}}+\frac{(-1)^l\alpha}{(\alpha-i(\tilde{\Delta} f-2\sigma-\theta(x))^{l+1}}\right|\lesssim\frac{\alpha^2}{\alpha^2+(2\sigma)^2}\frac{1}{(\alpha^2+(2\sigma)^2)^\frac{l}{2}}.$$
If $(j,l)\neq(k_1,1)$ and $j\geq1$, then $ j\leq k_1-1$ or $r_{k_1}=0$, and subsequently
\begin{equation}\begin{aligned}
&\left|\frac{1}{(\alpha^2+(2\sigma)^2)^\frac{l}{2}}\prod_{a=1}^j\left(\tilde{\Delta}\partial_x^a f-\partial_x^a\theta(x)\right)^{r_a}\right|\\
\lesssim&\left(\|\partial_x^2f\|_{H^{k_1-1}_{\gamma(t)}}+\sigma^{-1}\|\partial_x\theta\|_{H^{k_1-1}_{\gamma(t)}}\right)\left(1+\|\partial_x^2f\|_{H^{k_1-1}_{\gamma(t)}}+\sigma^{-1}\|\partial_x\theta\|_{H^{k_1-1}_{\gamma(t)}}\right)^{j-1}.
\end{aligned}\label{3.10}\end{equation}
If $(j,l)=(k_1,1)$, 
\begin{equation}\begin{aligned}
&\left|\frac{\tilde{\Delta}\partial_x^{k_1} f-\partial_x^{k_1}\theta(x)}{(\alpha^2+(2\sigma)^2)^\frac{l}{2}}\right|\lesssim\left|\frac{\tilde{\Delta}\partial_x^{k_1}f}{\alpha}\right|+\sigma^{-1}|\partial_x^{k_1}\theta(x)|
\lesssim\|\partial_x^2f\|_{H^{k_1}_{\gamma(t)}}+\sigma^{-1}|\partial_x^{k_1}\theta(x)|.
\end{aligned}\label{3.11}\end{equation}
For all $1\leq j\leq k_1$, using the embedding $\|\cdot\|_{L^\infty_{\gamma(t)}}\lesssim\|\cdot\|_{H^1_{\gamma(t)}}$, it can be shown that
\begin{equation}\left|\partial_x^{j}\left[q\left(\frac{\tilde{\Delta} f}{\alpha}\right)\right]\right|\lesssim\|\partial_x^2f\|_{H^j_{\gamma(t)}}\left(1+\|\partial_x^2f\|_{H^j_{\gamma(t)}}\right)^{j-1}.\label{3.12}\end{equation}
$\partial_x^{j}(2\sigma+\theta(x))^2$ can be controlled as follows:
\begin{equation}
\left|\partial_x^{j}(2\sigma+\theta(x))^2\right|\lesssim\|\partial_x\theta\|_{H^{j}_{\gamma(t)}}\left(\sigma+\|\partial_x\theta\|_{H^j_{\gamma(t)}}\right),\;  1\leq j\leq k_1-1.
\label{3.13}\end{equation}
\begin{equation}
\left|\partial_x^{j}(2\sigma+\theta(x))^2\right|\lesssim\|\partial_x\theta\|_{H^{k_1-1}_{\gamma(t)}}\left(\sigma+\|\partial_x\theta\|_{H^{k_1-1}_{\gamma(t)}}\right)+\sigma|\partial_x^{k_1}\theta(x)|,\quad j=k_1.
\label{3.14}\end{equation}
When $j=0$, we simply use $\left|q\left(\frac{\tilde{\Delta} f}{\alpha}\right)\right|\lesssim1$, $|2\sigma+\theta(x)|^2\lesssim\sigma^2$ and $\left|q\left(\frac{\tilde{\Delta} f-2\sigma-\theta(x)}{\alpha}\right)\right|\lesssim\frac{\alpha^2}{\alpha^2+(2\sigma)^2}$. Suppose now $j_3\neq k_1$, $j_1\neq k_1$. In this case, we use (\ref{3.10})(\ref{3.12})(\ref{3.13}) to obtain
$$\begin{aligned}
\|Q_1^{k_1,j_1,j_2,j_3}\|_{L^2_{\gamma(t)}}
\lesssim&\left(1+\sigma^{-1}\|\partial_x\theta\|_{H^{k_1-1}_{\gamma(t)}}\right)^2\left(1+\|\partial_x^2f\|_{H^{k_1-1}_{\gamma(t)}}+\sigma^{-1}\|\partial_x\theta\|_{H^{k_1-1}_{\gamma(t)}}\right)^{k_1}\\
&\cdot\left(1+\|\partial_x^2f\|_{H^{k_1}_{\gamma(t)}}\right)^{k_1}\int_{\mathbb{R}}\frac{\sigma^2}{\alpha^2+(2\sigma)^2}d\alpha\left\|M[\partial_x^{k_1+2-j}h]\right\|_{L^2_{\gamma(t)}}\\
\lesssim&\sigma B_8\|\partial_x^2h\|_{H^{k_1}_{\gamma(t)}},
\end{aligned}$$
where we note that  $\|\partial_x^2f\|_{H^{j}_{\gamma(t)}}+\|\partial_x^2g\|_{H^j_{\gamma(t)}}\lesssim\|\partial_x^2h\|_{H^j_{\gamma(t)}}+\|\partial_x^2\theta\|_{H^{j}_{\gamma(t)}}$ for $j\geq0$ and denote
$$\begin{aligned}
B_8(t):=&\left(1+\sigma^{-1}\|\partial_x\theta\|_{H^{k_1-1}_{\gamma(t)}}\right)^2\left(1+\|\partial_x^2h\|_{H^{k_1}_{\gamma(t)}}+\|\partial_x^2\theta\|_{H^{k_1}_{\gamma(t)}}\right)^{k_1}\\
&\cdot\left(1+\|\partial_x^2h\|_{H^{k_1-1}_{\gamma(t)}}+\|\partial_x^2\theta\|_{H^{k_1-1}_{\gamma(t)}}+\sigma^{-1}\|\partial_x\theta\|_{H^{k_1-1}_{\gamma(t)}}\right)^{k_1}.
\end{aligned}$$ 
In the case $j_3=k_1$, we must have $j_1=j_2=0$, $j=k_1$. Using the inequality
$$\begin{aligned}
&\left\|\int_{\mathbb{R}}\alpha^{-3}(2\sigma+\theta(x))^2q\left(\frac{\tilde{\Delta} f}{\alpha}\right)\frac{\alpha^2}{\alpha^2+(2\sigma)^2}\left(\left|\frac{\tilde{\Delta}\partial_x^{k_1}f}{\alpha}\right|+\sigma^{-1}|\partial_x^{k_1}\theta(x)|\right)\tilde{\Delta}\partial_xh\,d\alpha\right\|_{L^2_{\gamma(t)}}\\
\lesssim&\sigma\|\partial_x^2h\|_{L^\infty_{\gamma(t)}}\left(\|\partial_x^{k_1+1}f\|_{L^2_{\gamma(t)}}+\sigma^{-1}\|\partial_x^{k_1}\theta\|_{L^2_{\gamma(t)}}\right)\lesssim\sigma B_8\|\partial_x^2h\|_{L^\infty_{\gamma(t)}}
\end{aligned}$$
and (\ref{3.10})(\ref{3.11}) gives that
$\begin{aligned}
\|Q_1^{k_1,0,0,k_1}\|_{L^2_{\gamma(t)}}
\lesssim&\sigma B_8\left(\|\partial_x^2h\|_{L^2_{\gamma(t)}}+\|\partial_x^2h\|_{L^\infty_{\gamma(t)}}\right)
\lesssim\sigma B_8\|\partial_x^2h\|_{H^{k_1}_{\gamma(t)}}.\end{aligned}$
In the case $j_1=j=k_1$, we must have $j_2=j_3=0$. By (\ref{3.14}), it holds
$$\begin{aligned}
\|Q_1^{k_1,k_1,0,0}\|_{L^2_{\gamma(t)}}\lesssim&\left(1+\sigma^{-1}\|\partial_x\theta\|_{H^{k_1-1}_{\gamma(t)}}\right)\|\partial_x\theta\|_{H^{k_1-1}_{\gamma(t)}}\int_{\mathbb{R}}\frac{\sigma}{\alpha^2+(2\sigma)^2}d\alpha\|M[\partial_x^2h]\|_{L^2_{\gamma(t)}}\\
&+\left\|\int_{\mathbb{R}}\frac{\sigma|\partial_x^{k_1}\theta(x)|}{\alpha^2+(2\sigma)^2}\left|\frac{\tilde{\Delta}\partial_xh}{\alpha}\right|d\alpha\right\|_{L^2_{\gamma(t)}}\\
\lesssim&\left(1+\sigma^{-1}\|\partial_x\theta\|_{H^{k_1-1}_{\gamma(t)}}\right)\|\partial_x\theta\|_{H^{k_1-1}_{\gamma(t)}}\|\partial_x^2h\|_{L^2_{\gamma(t)}}+\|\partial_x^2h\|_{L^\infty_{\gamma(t)}}\|\partial_x^{k_1}\theta\|_{L^2_{\gamma(t)}}.
\end{aligned}$$
Taking the sum, we conclude that
\begin{equation}\begin{aligned}
&\left\|\int_{\mathbb{R}}\partial_x^{k_1}(Q_1(x,x-\alpha)\tilde{\Delta}\partial_xh)d\alpha\right\|_{L^2_{\gamma(t)}}\lesssim\sigma B_8\|\partial_x^2h\|_{H^{k_1}_{\gamma(t)}}.
\end{aligned}\label{3.15}\end{equation}
In addition, since $|Q_1(x,x-\alpha)|\lesssim\frac{1}{|\alpha|}\frac{\sigma^2}{\alpha^2+(2\sigma)^2}$, we also have
\begin{equation}\begin{aligned}
&\left\|\int_{\mathbb{R}}Q_1(x,x-\alpha)\tilde{\Delta}\partial_xh\,d\alpha\right\|_{L^2_{\gamma(t)}}\lesssim\sigma \|\partial_x^2h\|_{L^2_{\gamma(t)}}.
\end{aligned}\label{3.16}\end{equation}
Next, we are devoted to the control of $\left\|\int_{\mathbb{R}}\partial_x^{k_1}(Q_2(x,x-\alpha)\tilde{\Delta}\partial_xh)d\alpha\right\|_{L^2_{\gamma(t)}}$. $\partial_x^j(Q_2(x,x-\alpha))$ is a linear combination of terms in the form
$$q_{2}^{j_1,j_2,j_3,j_4}(x,x-\alpha):=\alpha^{-3}\tilde{\Delta}\partial_x^{j_4}f\cdot\partial_x^{j_1}(2\sigma+\theta(x))\partial_x^{j_2}\left[q\left(\frac{\tilde{\Delta} f}{\alpha}\right)\right]\partial_x^{j_3}\left[q\left(\frac{\tilde{\Delta} f-2\sigma-\theta(x)}{\alpha}\right)\right]$$
with $j_1+j_2+j_3+j_4=j$.
We decompose the integral over $\alpha$ into the parts $|\alpha|\leq1$ and $|\alpha|>1$. To be specific, we need to control
$$\begin{aligned}
Q_{2,|\alpha|\leq1}^{k_1,j_1,j_2,j_3,j_4}(x):=\int_{|\alpha|\leq1}q_2^{j_1,j_2,j_3,j_4}(x,x-\alpha)\tilde{\Delta}\partial_x^{k_1+1-j}h\,d\alpha,    
\end{aligned}$$
$$\begin{aligned}
Q_{2,|\alpha|>1}^{k_1,j_1,j_2,j_3,j_4}(x):=\int_{|\alpha|>1}q_2^{j_1,j_2,j_3,j_4}(x,x-\alpha)\tilde{\Delta}\partial_x^{k_1+1-j}h\,d\alpha.    
\end{aligned}$$
In the region $|\alpha|>1$, if $j_1\neq k_1$ and $j_3\neq k_1$, by (\ref{3.10})(\ref{3.12}) and 
$\left|\alpha^{-1}\tilde{\Delta}\partial_x^{j_4}f\right|\lesssim M[\partial_x^{j_4+1}f](x-\alpha)$, there is
$$\begin{aligned}
\left|q_2^{j_1,j_2,j_3,j_4}(x,x-\alpha)\right|
\lesssim&\left(\sigma+\|\partial_x\theta\|_{H^{k_1-1}_{\gamma(t)}}\right)\left(1+\|\partial_x^2f\|_{H^{k_1}_{\gamma(t)}}\right)^{k_1}\\
&\cdot\left(1+\|\partial_x^2f\|_{H^{k_1-1}_{\gamma(t)}}+\sigma^{-1}\|\partial_x\theta\|_{H^{k_1-1}_{\gamma(t)}}\right)^{k_1}\frac{M[\partial_x^{j_4+1}f](x-\alpha)}{\alpha^2+(2\sigma)^2},
\end{aligned}$$
$$\begin{aligned}
\|Q_{2,|\alpha|>1}^{k_1,j_1,j_2,j_3,j_4}\|_{L^2_{\gamma(t)}}
\lesssim\sigma B_8\left\|\int_{|\alpha|>1}\frac{|\alpha|}{\alpha^2+(2\sigma)^2}M[\partial_x^{j_4+1}f](x-\alpha)M[\partial_x^{k_1+2-j}h](x)d\alpha\right\|_{L^2_{\gamma(t)}}.
\end{aligned}$$
Since $0\leq j_4\leq j\leq k_1$, it follows
$\|Q_{2,|\alpha|>1}^{k_1,j_1,j_2,j_3,j_4}\|_{L^2_{\gamma(t)}}\lesssim\sigma B_8\|\partial_xf\|_{H^{k_1}_{\gamma(t)}}\|\partial_x^2h\|_{H^{k_1}_{\gamma(t)}}.$
If $j_3=j=k_1$, then $j_1=j_2=j_4=0$. Using (\ref{3.10})(\ref{3.11}) and that
$$\begin{aligned}
&\left\|\int_{|\alpha|>1}\alpha^{-3}\tilde{\Delta} f(2\sigma+\theta(x))q\left(\frac{\tilde{\Delta} f}{\alpha}\right)\frac{\alpha^2}{\alpha^2+(2\sigma)^2}\left(\left|\frac{\tilde{\Delta}\partial_x^{k_1}f}{\alpha}\right|+\sigma^{-1}|\partial_x^{k_1}\theta(x)|\right)\tilde{\Delta}\partial_xh\,d\alpha\right\|_{L^2_{\gamma(t)}}\\
\lesssim&\sigma\left\|\int_{|\alpha|>1}\frac{|\alpha|}{\alpha^2+(2\sigma)^2}M[\partial_xf](x-\alpha)M[\partial_x^2h](x-\alpha)\left(M[\partial_x^{k_1+1}f](x)+\sigma^{-1}|\partial_x^{k_1}\theta(x)|\right)d\alpha\right\|_{L^2_{\gamma(t)}}\\
\lesssim&\sigma\|\partial_xf\|_{L^\infty_{\gamma(t)}}\|\partial_x^2h\|_{L^2_{\gamma(t)}}\left(\|\partial_x^{k_1+1}f\|_{L^2_{\gamma(t)}}+\sigma^{-1}\|\partial_x^{k_1}\theta\|_{L^2_{\gamma(t)}}\right)\\
\lesssim&\sigma B_8\|\partial_xf\|_{L^\infty_{\gamma(t)}}\|\partial^2h\|_{L^2_{\gamma(t)}},
\end{aligned}$$
we can also obtain $\|Q_{2,|\alpha|>1}^{k_1,0,0,k_1,0}\|_{L^2_{\gamma(t)}}\lesssim\sigma B_8\|\partial_xf\|_{H^{k_1}_{\gamma(t)}}\|\partial_x^2h\|_{H^{k_1}_{\gamma(t)}}.$
If $j_1=j=k_1$, then $j_2=j_3=j_4=0$, and
$$\begin{aligned}
\|Q_{2,|\alpha|>1}^{k_1,k_1,0,0,0}\|_{L^2_{\gamma(t)}}
\lesssim&\left\|\int_{|\alpha|>1}\frac{|\alpha|}{\alpha^2+(2\sigma)^2}M[\partial_xf](x-\alpha)M[\partial_x^2h](x-\alpha)|\partial_x^{k_1}\theta(x)|d\alpha\right\|_{L^2_{\gamma(t)}}\\
\lesssim&\|\partial_xf\|_{L^\infty_{\gamma(t)}}\|\partial_x^2h\|_{L^2_{\gamma(t)}}\|\partial_x^{k_1}\theta\|_{L^2_{\gamma(t)}}.
\end{aligned}$$ 
Taking the sum yields
\begin{equation}
\left\|\int_{|\alpha|>1}\partial_x^{k_1}(Q_2(x,x-\alpha)\tilde{\Delta}\partial_xh)d\alpha\right\|_{L^2_{\gamma(t)}}\lesssim\sigma B_8\|\partial_xf\|_{H^{k_1}_{\gamma(t)}}\|\partial_x^2h\|_{H^{k_1}_{\gamma(t)}}.
\label{3.17}\end{equation}
In addition, it can be shown easily that
\begin{equation}\begin{aligned}
\left\|\int_{|\alpha|>1}Q_2(x,x-\alpha)\tilde{\Delta}\partial_xh\,d\alpha\right\|_{L^2_{\gamma(t)}}\lesssim\sigma \|\partial_xf\|_{L^2_{\gamma(t)}}\|\partial_x^2h\|_{L^2_{\gamma(t)}}.
\end{aligned}\label{3.18}\end{equation}
In the region $|\alpha|\leq1$, we use the decomposition 
$$\begin{aligned}
Q_{2,|\alpha|\leq1}^{k_1,j_1,j_2,j_3,j_4}(x)=&\int_{|\alpha|\leq1}q_2^{j_1,j_2,j_3,j_4}(x,x-\alpha)(\tilde{\Delta}\partial_x^{k_1+1-j}h-\alpha\partial_x^{k_1+2-j}h(x))\,d\alpha\\
&+\int_{|\alpha|\leq1}\alpha q_2^{j_1,j_2,j_3,j_4}(x,x-\alpha)d\alpha\cdot\partial_x^{k_1+2-j}h(x)\\
:=&Q_{2,1,|\alpha|\leq1}^{k_1,j_1,j_2,j_3,j_4}(x)+Q_{2,2,|\alpha|\leq1}^{k_1,j_1,j_2,j_3,j_4}(x).
\end{aligned}$$
For the first term, in the case $j_1\neq k_1$, $j_3\neq k_1$, (\ref{3.10}) and (\ref{3.12})  yield that 
$$\begin{aligned}
\|Q_{2,1,|\alpha|\leq1}^{k_1,j_1,j_2,j_3,j_4}\|_{L^2_{\gamma(t)}}
\lesssim&\left(1+\|\partial_x^2f\|_{H^{k_1-1}_{\gamma(t)}}+\sigma^{-1}\|\partial_x\theta\|_{H^{k_1-1}_{\gamma(t)}}\right)^{k_1}\left(1+\|\partial_x^2f\|_{H^{k_1}_{\gamma(t)}}\right)^{k_1}\\
&\cdot\left(\sigma+\|\partial_x\theta\|_{H^{k_1-1}_{\gamma(t)}}\right)\left\|\int_{|\alpha|\leq1}\frac{M[\partial_x^{j_4+1}f](x-\alpha)}{\alpha^2+(2\sigma)^2}\alpha^2M[\partial_x^{k_1+3-j}h](x)d\alpha\right\|_{L^2_{\gamma(t)}}\\
\lesssim&\sigma B_8\|\partial_xf\|_{H^{k_1}_{\gamma(t)}}\|\partial_x^3h\|_{H^{k_1}_{\gamma(t)}}
\end{aligned}$$
If $j_3=j=k_1$, then $j_1=j_2=j_4=0$. By (\ref{3.11}), it holds
$$\begin{aligned}
\|Q_{2,1,|\alpha|\leq1}^{k_1,0,0,k_1,0}\|_{L^2_{\gamma(t)}}
\lesssim&\sigma B_8\|\partial_xf\|_{L^2_{\gamma(t)}}\|\partial_x^3h\|_{L^2_{\gamma(t)}}\\&+\sigma\|\partial_xf\|_{L^2_{\gamma(t)}}\|\partial_x^3h\|_{L^\infty_{\gamma(t)}}\left(\|\partial_x^{k_1+1}f\|_{L^2_{\gamma(t)}}
+\sigma^{-1}\|\partial_x^{k_1}\theta\|_{L^2_{\gamma(t)}}\right)\\
\lesssim&\sigma B_8\|\partial_xf\|_{H^{k_1}_{\gamma(t)}}\|\partial_x^3h\|_{H^{k_1}_{\gamma(t)}}.
\end{aligned}$$
If $j_1=j=k_1$, then $j_2=j_3=j_4=0$, and
$$\begin{aligned}
\|Q_{2,1,|\alpha|\leq1}^{k_1,k_1,0,0,0}\|_{L^2_{\gamma(t)}}\lesssim&\left\|\int_{|\alpha|\leq1}\frac{\alpha^2|\partial_x^{k_1}\theta(x)|}{\alpha^2+(2\sigma)^2}M[\partial_xf](x-\alpha)M[\partial_x^3h](x)d\alpha\right\|_{L^2_{\gamma(t)}}\\
\lesssim&\|\partial_xf\|_{L^2_{\gamma(t)}}\|\partial_x^3h\|_{L^\infty_{\gamma(t)}}\|\partial_x^{k_1}\theta\|_{L^2_{\gamma(t)}}.
\end{aligned}$$
Summing the bounds for the above three cases gives
\begin{equation}\begin{aligned}
\left\|\int_{|\alpha|\leq1}\partial_x^{k_1}\left(Q_2(x,x-\alpha)\left(\tilde{\Delta}\partial_xh-\partial_x^2h(x)\right)\right)d\alpha\right\|_{L^2_{\gamma(t)}}\lesssim\sigma B_8\|\partial_xf\|_{H^{k_1}_{\gamma(t)}}\|\partial_x^3h\|_{H^{k_1}_{\gamma(t)}}.
\end{aligned}\label{3.19}\end{equation}
Meanwhile, we also have
\begin{equation}\begin{aligned}
\left\|\int_{|\alpha|\leq1}Q_2(x,x-\alpha)\left(\tilde{\Delta}\partial_xh-\partial_x^2h(x)\right)\,d\alpha\right\|_{L^2_{\gamma(t)}}\lesssim\sigma \|\partial_xf\|_{L^2_{\gamma(t)}}\|\partial_x^3h\|_{L^2_{\gamma(t)}}.
\end{aligned}\label{3.20}\end{equation}
To bound $Q_{2,2,|\alpha|\leq1}^{k_1,j_1,j_2,j_3,j_4}$ in $L^2_{\gamma(t)}$ norm, we write
$$\int_{|\alpha|\leq1}\alpha q_2^{j_1,j_2,j_3,j_4}(x,x-\alpha)d\alpha=\int_{0\leq\alpha\leq1}\alpha\left(q_2^{j_1,j_2,j_3,j_4}(x,x-\alpha)-q_2^{j_1,j_2,j_3,j_4}(x,x+\alpha)\right)d\alpha.$$
To use Lemma \ref{lem2.2}, we need the control over the difference between each factor of $q_2^{j_1,j_2,j_3,j_4}(x,x\pm\alpha)$.
\begin{lem}
Suppose condition (\ref{2.17}) holds. Then for all $1\leq j_2\leq k_1$,
\begin{equation}\begin{aligned}
&\left|\partial_x^{j_2}\left[q\left(\frac{f(x)-f(x-\alpha)}{\alpha}\right)\right]-\partial_x^{j_2}\left[q\left(\frac{f(x+\alpha)-f(x)}{\alpha}\right)\right]\right|\\
\lesssim&|\alpha|\left(1+\|\partial_x^2f\|_{H^{j_2}_{\gamma(t)}}\right)^{j_2}\sum_{a=0}^{j_2}\left(M[\partial_x^{a+2}f](x-\alpha)+M[\partial_x^{a+2}f](x+\alpha)\right).
\end{aligned}\label{3.21}\end{equation}
For $j_3=k_1$,
\begin{equation}\begin{aligned}
&\left|\partial_x^{k_1}\left[q\left(\frac{f(x)-f(x-\alpha)-2\sigma-\theta(x)}{\alpha}\right)\right]-\partial_x^{k_1}\left[q\left(\frac{f(x+\alpha)-f(x)+2\sigma+\theta(x)}{\alpha}\right)\right]\right|\\
\lesssim&\frac{|\alpha|}{\alpha^2+(2\sigma)^2}\left(1+\|\partial_x^2f\|_{H^{k_1-1}_{\gamma(t)}}+\sigma^{-1}\|\partial_x\theta\|_{H^{k_1-1}_{\gamma(t)}}\right)^{k_1}\\
&\cdot\left[\sigma+\alpha^2\sum_{a=0}^{k_1}\left(M[\partial_x^{a+2}f](x-\alpha)+M[\partial_x^{a+2}f](x+\alpha)\right)+\sum_{a=1}^{k_1}|\partial_x^a\theta(x)|\right]\\
&+\frac{|\alpha|}{\left(\alpha^2+(2\sigma)^2\right)^\frac{3}{2}}\left[\alpha^2\left(M[\partial_x^2f](x-\alpha)+M[\partial_x^2f](x+\alpha)\right)+\sigma\right]\\
&\cdot\left[|\alpha|\left(M[\partial_x^{k_1+1}f](x-\alpha)+M[\partial_x^{k_1+1}f](x+\alpha)\right)+|\partial_x^{k_1}\theta(x)|\right].
\end{aligned}\label{3.22}\end{equation}
For $1\leq j_3\leq k_1-1$,
\begin{equation}\begin{aligned}
&\left|\partial_x^{j_3}\left[q\left(\frac{f(x)-f(x-\alpha)-2\sigma-\theta(x)}{\alpha}\right)\right]-\partial_x^{j_3}\left[q\left(\frac{f(x+\alpha)-f(x)+2\sigma+\theta(x)}{\alpha}\right)\right]\right|\\
\lesssim&\frac{|\alpha|}{\alpha^2+(2\sigma)^2}\left(1+\|\partial_x^2f\|_{H^{k_1-1}_{\gamma(t)}}+\sigma^{-1}\|\partial_x\theta\|_{H^{k_1-1}_{\gamma(t)}}\right)^{j_3}\\
&\cdot\left[\sigma+\alpha^2\sum_{a=0}^{j_3}\left(M[\partial_x^{a+2}f](x-\alpha)+M[\partial_x^{a+2}f]](x+\alpha)\right)+\sum_{a=1}^{j_3}|\partial_x^a\theta(x)|\right].
\end{aligned}\label{3.23}\end{equation}
\label{lem3.1}\end{lem}
\begin{proof}
To prove (\ref{3.21}), first recall that 
$\partial_x^{j_2}\left[q\left(\frac{\tilde{\Delta} f}{\alpha}\right)\right]=\sum_{l=1}^{j_2}\sum_{S_{j_2,l}}q^{(l)}\left(\frac{\tilde{\Delta} f}{\alpha}\right)\prod_{a=1}^{j_2}\left(\frac{\tilde{\Delta}\partial_x^a f}{\alpha}\right)^{r_a}.$
Then using Lemma \ref{lem2.2}, the embedding $\|\cdot\|_{L^\infty_{\gamma(t)}}\lesssim\|\cdot\|_{H^{1}_{\gamma(t)}}$ and the inequalities $$\left|\frac{\partial_x^af(x)-\partial_x^af(x-\alpha)}{\alpha}-\frac{\partial_x^af(x+\alpha)-\partial_x^af(x)}{\alpha}\right|\lesssim|\alpha|\left(M[\partial_x^{a+2}f](x+\alpha)+M[\partial_x^{a+2}f](x-\alpha)\right),$$
$$\begin{aligned}
&\left|q^{(l)}\left(\frac{f(x)-f(x-\alpha)}{\alpha}\right)-q^{(l)}\left(\frac{f(x+\alpha)-f(x)}{\alpha}\right)\right|\lesssim|\alpha|\left(M[\partial_x^{2}f](x-\alpha)+M[\partial_x^{2}f](x-\alpha)\right),
\end{aligned}$$
we obtain
$$\begin{aligned}
&\left|\partial_x^{j_2}\left[q\left(\frac{f(x)-f(x-\alpha)}{\alpha}\right)\right]-\partial_x^{j_2}\left[q\left(\frac{f(x+\alpha)-f(x)}{\alpha}\right)\right]\right|
\\\lesssim&|\alpha|\left(1+\|\partial_x^2f\|_{H^{j_2}_{\gamma(t)}}\right)^{j_2}\sum_{a=0}^{j_2}\left(M[\partial_x^{a+2}f](x-\alpha)+M[\partial_x^{a+2}f](x+\alpha)\right).
\end{aligned}$$
To show (\ref{3.22}), we note that
$$\partial_x^{j_3}\left[q\left(\frac{\tilde{\Delta} f-2\sigma-\theta(x)}{\alpha}\right)\right]=\sum_{l=1}^{j_3}\sum_{S_{j_3,l}}q^{(l)}\left(\frac{\tilde{\Delta} f-2\sigma-\theta(x)}{\alpha}\right)\prod_{a=1}^{j_3}\left(\frac{\tilde{\Delta}\partial_x^af-2\sigma-\theta(x)}{\alpha}\right)^{r_a},$$
$$\begin{aligned}
&q^{(l)}\left(\frac{\tilde{\Delta} f-2\sigma-\theta(x)}{\alpha}\right)
=\frac{(-i)^ll!}{2}\left[\frac{\alpha^{l+1}}{(\alpha+i(\tilde{\Delta} f-2\sigma-\theta(x)))^{l+1}}+\frac{(-1)^l\alpha^{l+1}}{(\alpha-i(\tilde{\Delta} f-2\sigma-\theta(x)))^{l+1}}\right],
\end{aligned}$$
\begin{equation}\begin{aligned}
&\left|\frac{\partial_x^af(x)-\partial_x^af(x-\alpha)-\partial_x^a\theta(x)}{\alpha}-\frac{\partial_x^{a}f(x+\alpha)-\partial_x^af(x)+\partial_x^a\theta(x)}{\alpha}\right|\\
\lesssim&|\alpha|\left(M[\partial_x^{a+2}f](x-\alpha)+M[\partial_x^{a+2}f](x+\alpha)\right)+\left|\frac{\partial_x^a\theta(x)}{\alpha}\right|.
\end{aligned}\label{3.24}\end{equation}
Moreover, in order to bound $q^{(l)}\left(\frac{f(x)-f(x-\alpha)-2\sigma-\theta(x)}{\alpha}\right)-q^{(l)}\left(\frac{f(x+\alpha)-f(x)+2\sigma+\theta(x)}{\alpha}\right)$, we have
$$\begin{aligned}
&\left|\frac{\alpha^{l+1}}{(\alpha+i(f(x)-f(x-\alpha)-2\sigma-\theta(x)))^{l+1}}-\frac{\alpha^{l+1}}{(\alpha-i(f(x)-f(x+\alpha)-2\sigma-\theta(x)))^{l+1}}\right|\\
&+\left|\frac{(-1)^{l+1}\alpha^{l+1}}{(\alpha-i(f(x)-f(x-\alpha)-2\sigma-\theta(x)))^{l+1}}-\frac{(-1)^{l+1}\alpha^{l+1}}{(\alpha+i(f(x)-f(x+\alpha)-2\sigma-\theta(x)))^{l+1}}\right|\\
\lesssim&|\alpha|^{l+1}\left(\alpha^2+(2\sigma)^2\right)^{-\frac{l+2}{2}}\left|2f(x)-f(x-\alpha)-f(x+\alpha)-4\sigma-2\theta(x)\right|.
\end{aligned}$$
Therefore,
\begin{equation}\begin{aligned}
&\left|q^{(l)}\left(\frac{f(x)-f(x-\alpha)-2\sigma-\theta(x)}{\alpha}\right)-q^{(l)}\left(\frac{f(x+\alpha)-f(x)+2\sigma+\theta(x)}{\alpha}\right)\right|\\
\lesssim&|\alpha|^{l+1}\left(\alpha^2+(2\sigma)^2\right)^{-\frac{l+2}{2}}\left[|\alpha|^2\left(M[\partial_x^2f](x-\alpha)+M[\partial_x^2f](x+\alpha)\right)+\sigma\right].
\end{aligned}\label{3.25}\end{equation}
If $(j_3,l)\neq(k_1,1)$, then $j_3\leq k_1-1$ or $r_{k_1}=0$, using (\ref{3.24})(\ref{3.25}) and Lemma \ref{lem2.2}, we obtain
\begin{equation}\begin{aligned}
&\left|q^{(l)}\left(\frac{f(x)-f(x-\alpha)-2\sigma-\theta(x)}{\alpha}\right)\prod_{a=1}^{j_3}\left(\frac{\partial_x^af(x)-\partial_x^af(x-\alpha)-\partial_x^a\theta(x)}{\alpha}\right)^{r_a}\right.\\
&\left.-q^{(l)}\left(\frac{f(x+\alpha)-f(x)+2\sigma+\theta(x)}{\alpha}\right)\prod_{a=1}^{j_3}\left(\frac{\partial_x^af(x+\alpha)-\partial_x^af(x)+\partial_x^a\theta(x)}{\alpha}\right)^{r_a}\right|\\
\lesssim&\frac{|\alpha|}{\alpha^2+(2\sigma)^2}\left(\|\partial_x^2f\|_{H^{k_1-1}_{\gamma(t)}}+\sigma^{-1}\|\partial_x\theta\|_{H^{k_1-1}_{\gamma(t)}}\right)^{l}\left[\alpha^2\left(M[\partial_x^2f](x-\alpha)+M[\partial_x^2f](x+\alpha)\right)+\sigma\right]\\
&+\frac{|\alpha|}{\alpha^2+(2\sigma)^2}\left(\|\partial_x^2f\|_{H^{k_1-1}_{\gamma(t)}}+\sigma^{-1}\|\partial_x\theta\|_{H^{k_1-1}_{\gamma(t)}}\right)^{l-1}\\
&\cdot\sum_{r_a\neq0}\left[\alpha^2\left(M[\partial_x^{a+2}f](x-\alpha)+M[\partial_x^{a+2}f]](x+\alpha)\right)+|\partial_x^a\theta(x)|\right].
\end{aligned}\label{3.26}\end{equation}
In particular, when $1\leq j_3\leq k_1-1$, it follows
$$\begin{aligned}
&\left|\partial_x^{j_3}\left[q\left(\frac{f(x)-f(x-\alpha)-2\sigma-\theta(x)}{\alpha}\right)\right]-\partial_x^{j_3}\left[q\left(\frac{f(x+\alpha)-f(x)+2\sigma+\theta(x)}{\alpha}\right)\right]\right|\\
\lesssim&\frac{|\alpha|}{\alpha^2+(2\sigma)^2}\left(1+\|\partial_x^2f\|_{H^{k_1-1}_{\gamma(t)}}+\sigma^{-1}\|\partial_x\theta\|_{H^{k_1-1}_{\gamma(t)}}\right)^{j_3}\\
&\cdot\left[\sigma+\alpha^2\sum_{a=0}^{j_3}\left(M[\partial_x^{a+2}f](x-\alpha)+M[\partial_x^{a+2}f]](x+\alpha)\right)+\sum_{a=1}^{j_3}|\partial_x^a\theta(x)|\right].
\end{aligned}$$
If $(j_3,l)=(k_1,1)$, then $r_a=0$ for all $a\leq k_1-1$ and $r_{k_1}=1$,
$$\begin{aligned}
&\left|q^{(1)}\left(\frac{f(x)-f(x-\alpha)-2\sigma-\theta(x)}{\alpha}\right)\left(\frac{\partial_x^{k_1}f(x)-\partial_x^{k_1}f(x-\alpha)-\partial_x^{k_1}\theta(x)}{\alpha}\right)\right.\\
&\left.-q^{(1)}\left(\frac{f(x+\alpha)-f(x)+2\sigma+\theta(x)}{\alpha}\right)\left(\frac{\partial_x^{k_1}f(x+\alpha)-\partial_x^{k_1}f(x)+\partial_x^{k_1}\theta(x)}{\alpha}\right)\right|\\
\lesssim&\frac{|\alpha|}{\alpha^2+(2\sigma)^2}\left[\alpha^2\left(M[\partial_x^{k_1+2}f](x-\alpha)+M[\partial_x^{k_1+2}f](x+\alpha)\right)+|\partial_x^{k_1}\theta(x)|\right]\\
&+\frac{|\alpha|}{\left(\alpha^2+(2\sigma)^2\right)^\frac{3}{2}}\left[\alpha^2\left(M[\partial_x^2f](x-\alpha)+M[\partial_x^2f](x+\alpha)\right)+\sigma\right]\\
&\cdot\left[|\alpha|\left(M[\partial_x^{k_1+1}f](x-\alpha)+M[\partial_x^{k_1+1}f](x+\alpha)\right)+|\partial_x^{k_1}\theta(x)|\right].
\end{aligned}$$
Taking the sum with (\ref{3.26}) yields (\ref{3.22}).
\end{proof}
Now, using Lemma \ref{lem2.2}, Lemma \ref{lem3.1} together with the bound (\ref{3.10}-\ref{3.12}) and
$$\begin{aligned}
&\left|\frac{\partial_x^{j_4}f(x)-\partial_x^{j_4}f(x-\alpha)}{\alpha}-\frac{\partial_x^{j_4}f(x+\alpha)-\partial_x^{j_4}f(x)}{\alpha}\right|\\
\lesssim&|\alpha|\left(M[\partial_x^{j_4+2}f](x-\alpha)+M[\partial_x^{j_4+2}f](x+\alpha)\right),
\end{aligned}$$
we are able to obtain the control over $q_2^{j_1,j_2,j_3,j_4}(x,x-\alpha)-q_2^{j_1,j_2,j_3,j_4}(x,x+\alpha)$. In the case $j_1\neq k_1$, $j_3\neq k_1$, by the Sobolev embedding $|\partial_x^{j}\theta(x)|\lesssim\|\partial_x\theta\|_{H^{k_1-1}_{\gamma(t)}}$, $k_1-1\geq j\geq1$, there is
$$\begin{aligned}
&|\alpha|\left|q_2^{j_1,j_2,j_3,j_4}(x,x-\alpha)-q_2^{j_1,j_2,j_3,j_4}(x,x+\alpha)\right|\\
\lesssim&\sigma\left(1+\sigma^{-1}\|\partial_x\theta\|_{H^{j_1-1}_{\gamma(t)}}\right)\left(1+\|\partial_x^2f\|_{H^{j_2}_{\gamma(t)}}\right)^{j_2}\left(1+\|\partial_x^2f\|_{H^{j_2}_{\gamma(t)}}+\sigma^{-1}\|\partial_x\theta\|_{H^{j_3}_{\gamma(t)}}\right)^{j_3}\\
&\cdot\left\{\frac{\|\partial_xf\|_{H^{j_4+1}_{\gamma(t)}}}{\alpha^2+(2\sigma)^2}\alpha^2\sum_{a=0}^{j_2}\left(M[\partial_x^{a+2}f](x-\alpha)+M[\partial_x^{a+2}f](x+\alpha)\right)\right.\\
&\left.+\frac{\|\partial_xf\|_{H^{j_4+1}_{\gamma(t)}}}{\alpha^2+(2\sigma)^2}\left[\sigma+\alpha^2\sum_{a=0}^{j_3}\left(M[\partial_x^{a+2}f](x-\alpha)+M[\partial_x^{a+2}f]](x+\alpha)\right)+\sum_{a=1}^{j_3}|\partial_x^a\theta(x)|\right]\right.\\
&\left.+\frac{\alpha^2}{\alpha^2+(2\sigma)^2}\left(M[\partial_x^{j_4+2}f](x-\alpha)+M[\partial_x^{j_4+2}f](x+\alpha)\right)\right\},
\end{aligned}$$
and subsequently
\begin{equation}\begin{aligned}
\|Q_{2,2,|\alpha|\leq1}^{k_1,j_1,j_2,j_3,j_4}\|_{L^2_{\gamma(t)}}
\lesssim&\sigma B_8\|\partial_xf\|_{H^{j_4+1}_{\gamma(t)}}\|\partial_x^{k_1+2-j}h\|_{L^2_{\gamma(t)}}.
\end{aligned}\label{3.27}\end{equation}
In the case $j_1=j=k_1, j_2=j_3=j_4=0$, we have
$$\begin{aligned}
&|\alpha|\left|q_2^{k_1,0,0,0}(x,x-\alpha)-q_2^{k_1,0,0,0}(x,x+\alpha)\right|\\
\lesssim&|\partial_x^{k_1}\theta(x)|\left\{\frac{\|\partial_xf\|_{H^{1}_{\gamma(t)}}}{\alpha^2+(2\sigma)^2}\left[|\alpha|^2\left(M[\partial_x^2f](x-\alpha)+M[\partial_x^2f](x+\alpha)\right)+\sigma\right]\right.\\
&\left.+\frac{\alpha^2}{\alpha^2+(2\sigma)^2}\left(M[\partial_x^2f](x-\alpha)+M[\partial_x^2f](x+\alpha)\right)\right\},
\end{aligned}$$
\begin{equation}\begin{aligned}
\|Q_{2,2,|\alpha|\leq1}^{k_1,k_1,0,0,0}\|_{L^2_{\gamma(t)}}
\lesssim&\|\partial_x^{k_1}\theta\|_{L^2_{\gamma(t)}}\|\partial_x^2h\|_{L^\infty_{\gamma(t)}}\|\partial_xf\|_{H^{1}_{\gamma(t)}}\left(1+\|\partial_x^2f\|_{L^2_{\gamma(t)}}\right).
\end{aligned}\label{3.28}\end{equation}
In the case $j_3=j=k_1$, $j_1=j_2=j_4=0$, by (\ref{3.10})(\ref{3.11}) and (\ref{3.22}), 
\begin{align*}
&|\alpha|\left|q_{2}^{0,0,k_1,0}(x,x-\alpha)-q_2^{0,0,k_1,0}(x,x+\alpha)\right|\\
\lesssim&\sigma\left(1+\|\partial_x^2f\|_{H^{k_1-1}_{\gamma(t)}}+\sigma^{-1}\|\partial_x\theta\|_{H^{k_1-1}_{\gamma(t)}}\right)^{k_1}\frac{\|\partial_xf\|_{H^1_{\gamma(t)}}}{\alpha^2+(2\sigma)^2}\\
&\left\{\sigma+\alpha^2\sum_{a=0}^{k_1}\left(M[\partial_x^{a+2}f](x-\alpha)+M[\partial_x^{a+2}f](x+\alpha)\right)+\sum_{a=1}^{k_1}|\partial_x^a\theta(x)|\right.\\
&\left.+\left(\alpha^2+(2\sigma)^2\right)^{-\frac{1}{2}}\left[\alpha^2\left(M[\partial_x^2f](x-\alpha)+M[\partial_x^2f](x+\alpha)\right)+\sigma\right]\right.\\
&\left.\cdot\left[|\alpha|\left(M[\partial_x^{k_1+1}f](x-\alpha)+M[\partial_x^{k_1+1}f](x+\alpha)\right)+|\partial_x^{k_1}\theta(x)|\right].
\right\}\\
&+\frac{\sigma\alpha^2}{\alpha^2+(2\sigma)^2}\left[M[\partial_x^2f](x-\alpha)+M[\partial_x^2f](x+\alpha)\right]\\
&\cdot\left[\left(1+\|\partial_x^2f\|_{H^{k_1-1}_{\gamma(t)}}+\sigma^{-1}\|\partial_x\theta\|_{H^{k_1-1}_{\gamma(t)}}\right)^{k_1}+\left|\frac{\tilde{\Delta}\partial_x^{k_1}f}{\alpha}\right|+\sigma^{-1}|\partial_x^{k_1}\theta(x)|\right]
\end{align*}
Then we bound the contribution from each term on the right-hand side:
$$\begin{aligned}
&\left\|\int_{0\leq\alpha\leq1}\frac{1}{\alpha^2+(2\sigma)^2}\left[\sigma+\alpha^2\sum_{a=0}^{k_1}\left(M[\partial_x^{a+2}f](x-\alpha)+M[\partial_x^{a+2}f](x+\alpha)\right)\right]d\alpha\cdot\partial_x^2h(x)\right\|_{L^2_{\gamma(t)}}\\
\lesssim&\|\partial_x^2h\|_{L^2_{\gamma(t)}}\left(1+\|\partial_x^2f\|_{H^{k_1}_{\gamma(t)}}\right),
\end{aligned}$$
$$\begin{aligned}
&\left\|\int_{0\leq\alpha\leq1}\frac{\sigma}{\alpha^2+(2\sigma)^2}\sum_{a=1}^{k_1}|\partial_x^a\theta(x)|d\alpha\cdot\partial_x^2h(x)\right\|_{L^2_{\gamma(t)}}
\lesssim\|\partial_x^2h\|_{L^\infty_{\gamma(t)}}\|\partial_x\theta\|_{H^{k_1-1}_{\gamma(t)}},
\end{aligned}$$
$$\begin{aligned}
&\left\|\int_{0\leq\alpha\leq1}\left(\alpha^2+(2\sigma)^2\right)^{-\frac{3}{2}}\left[\alpha^2\left(M[\partial_x^2f](x-\alpha)+M[\partial_x^2f](x+\alpha)\right)+\sigma\right]\right.\\
&\left.\cdot\left[|\alpha|\left(M[\partial_x^{k_1+1}f](x-\alpha)+M[\partial_x^{k_1+1}f](x+\alpha)\right)+|\partial_x^{k_1}\theta(x)|\right]d\alpha\cdot\partial_x^2h(x)\right\|_{L^2_{\gamma(t)}}\\
\lesssim&\left(\|\partial_x^2f\|_{L^\infty_{\gamma(t)}}+1\right)\left(\|\partial_x^{k_1+1}f\|_{L^2_{\gamma(t)}}\|\partial_x^2h\|_{L^2_{\gamma(t)}}+\sigma^{-1}\|\partial_x^2h\|_{L^\infty_{\gamma(t)}}\|\partial_x^{k_1}\theta\|_{L^2_{\gamma(t)}}\right),
\end{aligned}$$
$$\begin{aligned}
&\left\|\int_{0\leq\alpha\leq1}\frac{\sigma\alpha^2\left[M[\partial_x^2f](x-\alpha)+M[\partial_x^2f](x+\alpha)\right]}{\alpha^2+(2\sigma)^2}d\alpha\partial_x^2h(x)\right\|_{L^2_{\gamma(t)}}
\lesssim\sigma\|\partial_x^2f\|_{L^2_{\gamma(t)}}\|\partial_x^2h\|_{L^2_{\gamma(t)}},
\end{aligned}$$
$$\begin{aligned}
&\left\|\int_{0\leq\alpha\leq1}\frac{\sigma\alpha^2\left[M[\partial_x^2f](x-\alpha)+M[\partial_x^2f](x+\alpha)\right]}{\alpha^2+(2\sigma)^2}\left(\left|\frac{\tilde{\Delta}\partial_x^{k_1}f}{\alpha}\right|+\sigma^{-1}|\partial_x^{k_1}\theta(x)|\right)d\alpha\cdot\partial_x^2h(x)\right\|_{L^2_{\gamma(t)}}\\
\lesssim&\sigma\|\partial_x^2f\|_{L^2_{\gamma(t)}}\|\partial_x^{k_1+1}f\|_{L^2_{\gamma(t)}}\|\partial_x^2h\|_{L^2_{\gamma(t)}}+\|\partial_x^2f\|_{L^2_{\gamma(t)}}\|\partial_x^2h\|_{L^\infty_{\gamma(t)}}\|\partial_x^{k_1}\theta\|_{L^2_{\gamma(t)}}.
\end{aligned}$$
Taking the sum and using Sobolev embedding yields
\begin{equation}\begin{aligned}
\|Q_{2,2,|\alpha|\leq1}^{k_1,0,0,k_1,0}\|_{L^2_{\gamma(t)}}
\lesssim&\sigma B_8\|\partial_xf\|_{H^1_{\gamma(t)}}\left(\|\partial_x^2h\|_{L^2_{\gamma(t)}}+\|\partial_x^2h\|_{L^\infty_{\gamma(t)}}\right).
\end{aligned}\label{3.29}\end{equation}
(\ref{3.27}-\ref{3.29}) imply that
$\left\|\int_{|\alpha|\leq1}\alpha\partial_x^{k_1}\left(Q_2(x,x-\alpha)\partial_x^2h(x)\right)d\alpha\right\|_{L^2_{\gamma(t)}}
\lesssim\sigma B_8\|\partial_xf\|_{H^{k_1+1}_{\gamma(t)}}\|\partial_x^2h\|_{H^{k_1}_{\gamma(t)}}$,
which together with (\ref{3.15}) and (\ref{3.19}) gives
\begin{equation}\begin{aligned}
\left\|\partial_x^{k_1}\left(\int_{\mathbb{R}}(P_{11}-P_{21})(x,x-\alpha)\tilde{\Delta}\partial_xh\,d\alpha\right)\right\|_{L^2_{\gamma(t)}}\lesssim\sigma B_8\left(1+\|\partial_xf\|_{H^{k_1+1}_{\gamma(t)}}\right)\|\partial_x^2h\|_{H^{k_1+1}_{\gamma(t)}}.
\end{aligned}\label{3.30}\end{equation}
Moreover, by Lemma \ref{lem2.2} again, we have
$$\begin{aligned}
&\left|\alpha||Q_2(x,x-\alpha)-Q_2(x,x+\alpha)\right|\\
\lesssim&\frac{\sigma}{\alpha^2+(2\sigma)^2}\left[\alpha^2\left(M[\partial_x^2f](x-\alpha)+M[\partial_x^2f](x+\alpha)\right)\left(1+\|\partial_xf\|_{L^\infty_{\gamma(t)}}\right)+\sigma\|\partial_xf\|_{L^\infty_{\gamma(t)}}\right],
\end{aligned}$$
and subsequently
\begin{equation}\begin{aligned}
&\left\|\int_{0\leq\alpha\leq1}\alpha\left(Q_2(x,x-\alpha)-Q_2(x,x+\alpha)\right)d\alpha\cdot\partial_x^2h(x)\right\|_{L^2_{\gamma(t)}}\\
\lesssim&\sigma\|\partial_x^2f\|_{L^2_{\gamma(t)}}\left(1+\|\partial_xf\|_{L^\infty_{\gamma(t)}}\right)\|\partial_x^2h\|_{L^2_{\gamma(t)}}+\sigma\|\partial_xf\|_{L^\infty_{\gamma(t)}}\|\partial_x^2h\|_{L^2_{\gamma(t)}}.
\end{aligned}\label{3.31}\end{equation}
Summing up (\ref{3.16})(\ref{3.18})(\ref{3.20})(\ref{3.31}) yields
\begin{equation}\begin{aligned}
&\left\|\int_{\mathbb{R}}\left(P_{11}-P_{21}\right)(x,x-\alpha)\tilde{\Delta}\partial_xh\,d\alpha\right\|_{L^2_{\gamma(t)}}
\lesssim\sigma\left(1+\|\partial_xf\|_{L^2_{\gamma(t)}}\right)\left(1+\|\partial_xf\|_{L^\infty_{\gamma(t)}}\right)\|\partial_x^2h\|_{H^1_{\gamma(t)}}.
\end{aligned}\label{3.32}\end{equation}
As an analogue, it can also be shown 
\begin{equation}\begin{aligned}
\left\|\partial_x^{k_1}\left(\int_{\mathbb{R}}(P_{22}-P_{12})(x,x-\alpha)\tilde{\Delta}\partial_xh\,d\alpha\right)\right\|_{L^2_{\gamma(t)}}\lesssim\sigma B_8\left(1+\|\partial_xg\|_{H^{k_1+1}_{\gamma(t)}}\right)\|\partial_x^2h\|_{H^{k_1+1}_{\gamma(t)}}.
\end{aligned}\label{3.33}\end{equation}
\begin{equation}\begin{aligned}
&\left\|\int_{\mathbb{R}}\left(P_{22}-P_{12}\right)(x,x-\alpha)\tilde{\Delta}\partial_xh\,d\alpha\right\|_{L^2_{\gamma(t)}}
\lesssim\sigma\left(1+\|\partial_xg\|_{L^2_{\gamma(t)}}\right)\left(1+\|\partial_xg\|_{L^\infty_{\gamma(t)}}\right)\|\partial_x^2h\|_{H^1_{\gamma(t)}}.
\end{aligned}\label{3.34}\end{equation}
Finally, we insert the controls (\ref{3.2})(\ref{3.4})(\ref{3.5})(\ref{3.9})(\ref{3.30})(\ref{3.33}) into (\ref{3.1}) and conclude with the estimate over $\frac{d}{dt}\|\partial_x^{k_1}\theta_1\|_{H^{k_1}_{\gamma(t)}}^2$:
$$\begin{aligned}
&\frac{d}{dt}\|\partial_x^{k_1}\theta_1\|_{L^2_{\gamma(t)}}^2+c_0\mu_1\mu_2\Delta\rho\int_{\Gamma_{\pm}(t)}\int_{\Gamma_{\pm}(t)}D_{22}^0(x,x_1)\left|\Delta\partial_x^{k_1}\theta_1\right|^2dxdx_1\\
&-2\gamma^\prime(t)\tanh(2\gamma(t))\|\partial_x^{k_1}\theta_1\|_{\dot{H}^\frac{1}{2}_{\gamma(t)}}^2\\
\lesssim&(B_5+w_0)\left(\delta_2+\sigma^\frac{1}{2}+\|\partial_xf\|_{L^2_{\gamma(t)}}+\|\partial_xg\|_{L^2_{\gamma(t)}}\right)\|\partial_x^{k_1}\theta_1\|_{L^2_{\gamma(t)}}\\
&\cdot\sum_{j=0}^{k_1}\left(\int_{\Gamma_{\pm}(t)}\int_{\Gamma_{\pm}(t)}D_{22}^0(x,x_1)|\Delta\partial_x^{k_1}\theta_1|^2dxdx_1\right)^\frac{1}{2}+B_7\|\partial_x^k\theta_1\|_{\dot{H}^\frac{1}{2}_{\gamma(t)}}^2\\
&
+B_9\left(\|\partial_x^2h\|_{H^{k_1+1}_{\gamma(t)}}+\|\partial_x^2\theta\|_{H^{k_1+1}_{\gamma(t)}}\right)\|\theta_1\|_{H^{k_1}_{\gamma(t)}}+(B_5\delta_2+B_4-2\gamma^\prime(t)\tanh(2\gamma(t)))\|\theta_1\|_{H^{k_1}_{\gamma(t)}}^2,
\end{aligned}$$
where 
$$B_9(t):=B_5\left(\|\partial_x^2\theta_1\|_{H^{\tilde{k}_1}_{\gamma(t)}}+\|\partial_x\theta_1\|_{L^\infty_{\gamma(t)}}+\|\theta_1\|_{H^{k_1}_{\gamma(t)}}\right)+B_8\left(1+\|\partial_xf\|_{H^{k_1+1}_{\gamma(t)}}+\|\partial_xg\|_{H^{k_1+1}_{\gamma(t)}}\right).$$
Meanwhile, we can also estimate $\frac{d}{dt}\|\theta_1\|_{L^2_{\gamma(t)}}^2$ through (\ref{3.32})(\ref{3.34}):
$$\begin{aligned}
&\frac{d}{dt}\|\theta_1\|_{L^2_{\gamma(t)}}+c_0\mu_1\mu_2\Delta\rho\int_{\Gamma_{\pm(t)}}\int_{\Gamma_{\pm}(t)}D_{22}^0(x,x_1)|\tilde{\Delta}\theta_1|^2dxdx_1\\
&-2\gamma^\prime(t)\tanh(2\gamma(t))\|\theta_1\|_{\dot{H}^\frac{1}{2}_{\gamma(t)}}^2\\
\lesssim&(B_5+w_0)\left(\delta_2+\sigma^\frac{1}{2}+\|\partial_xf\|_{L^2_{\gamma(t)}}+\|\partial_xg\|_{L^2_{\gamma(t)}}\right)\|\theta_1\|_{L^2_{\gamma(t)}}\\
&\cdot\left(\int_{\Gamma_{\pm}(t)}\int_{\Gamma_{\pm}(t)}D_{22}^0(x,x_1)|\tilde{\Delta}\partial_x^{k_1}\theta_1|^2dxdx_1\right)^\frac{1}{2}+B_7\|\theta_1\|_{\dot{H}^\frac{1}{2}_{\gamma(t)}}^2\\
&+\left(1+\|\partial_xf\|_{L^2_{\gamma(t)}}+\|\partial_xg\|_{L^2_{\gamma(t)}}\right)\left(1+\|\partial_xf\|_{L^\infty_{\gamma(t)}}+\|\partial_xg\|_{L^\infty_{\gamma(t)}}\right)\|\partial_x^2h\|_{H^1_{\gamma(t)}}\|\theta_1\|_{L^2_{\gamma(t)}}\\
&+(B_5\delta_2+B_4-2\gamma^\prime(t)\tanh(2\gamma(t)))\|\theta_1\|_{H^{k_1}_{\gamma(t)}}^2.
\end{aligned}$$
Summing the above two inequalities and using Cauchy-Schwarz inequality for the first terms on the right-hand sides like what we have done in (\ref{2.102}), we obtain
\begin{equation}\begin{aligned}
&\frac{d}{dt}\|\theta_1\|_{H^{k_1}_{\gamma(t)}}^2+\frac{c_0}{2}\mu_1\mu_2\Delta\rho\int_{\Gamma_{\pm(t)}}\int_{\Gamma_{\pm(t)}}D_{22}^0(x,x_1)\left(\Delta\partial_x^{k_1}\theta_1|^2+|\Delta\theta_1|^2\right)dxdx_1\\
&-\left(2\gamma^\prime(t)\tanh(2\gamma(t))+C_3 B_7\right)\|\Lambda^\frac{1}{2}\theta_1\|_{H^{k_1}_{\gamma(t)}}^2\\
\lesssim&\left(B_6-2\gamma^\prime(t)\tanh(2\gamma(t))\right)\|\theta_1\|_{H^{k_1}_{\gamma(t)}}^2+B_9\left(\|\partial_x^2h\|_{H^{k_1+1}_{\gamma(t)}}+\|\partial_x^2\theta\|_{H^{k_1+1}_{\gamma(t)}}\right)\|\theta_1\|_{H^{k_1}_{\gamma(t)}},
\end{aligned}\label{3.35}\end{equation}
In addition, if we let $\gamma(t)=\tilde{\gamma}$ be a positive constant, then by (\ref{2.105}) 
\begin{equation}\begin{aligned}
\frac{d}{dt}\|\theta_1\|_{H^{k_1}_{\gamma(t)}}^2
\lesssim&\left(B_6+\sigma^{-1}\right)\|\theta_1\|_{H^{k_1}_{\gamma(t)}}^2+\left(B_7-\frac{c_0}{C_4}\right)\|\Lambda^\frac{1}{2}\theta_1\|_{H^{k_1}_{\gamma(t)}}^2\\
&+B_9\left(\|\partial_x^2h\|_{H^{k_1+1}_{\gamma(t)}}+\|\partial_x^2\theta\|_{H^{k_1+1}_{\gamma(t)}}\right)\|\theta_1\|_{H^{k_1}_{\gamma(t)}},
\end{aligned}\label{3.36}\end{equation}

\section{Proof of theorem \ref{thm1.1}}\label{proof}
The proof of Theorem \ref{thm1.1} will be completed through a bootstrap argument encompassing conditions (\ref{1.18}) and (\ref{1.19}). For simplicity, denote $E(t):=\|h(t)\|_{H^{k}_{\gamma(t)}}^2+\mu_1\mu_2\|\theta(t)\|_{H^{k}_{\gamma(t)}}^2+\|\theta_1(t)\|_{H^{k_1}_{\gamma(t)}}^2$. First, since we choose $k\geq10$, it holds $k_1=k-3\geq\tilde{k}+2$ and consequently
$$
B_3\leq C_5\left(1+E^\frac{1}{2}\right),\; B_4\lesssim\left(1+E^\frac{1}{2}\right)E^\frac{1}{2},\;B_5\lesssim\left(1+E^\frac{1}{2}\right)^k,\;\delta_2\lesssim E^\frac{1}{2},\;  B_8\lesssim\left(1+E^\frac{1}{2}\right)^{2k_1+2},$$
where we used Sobolev embedding, and $C_5$ denotes the implicit constant. Again, by Sobolev embedding, it holds that $\|\partial_xh\|_{L^\infty_{\gamma(t)}}+\|\partial_x\theta\|_{L^\infty_{\gamma(t)}}\leq C_6E(t)^\frac{1}{2}$ with a constant $C_6$, from which it follows
\begin{equation}\begin{aligned}
B_6\lesssim&\left(1+E^\frac{1}{2}\right)^{2k}\left(\delta_2^2+\sigma+\|\partial_xf\|_{L^2_{\gamma(t)}}^2+\|\partial_xg\|_{L^2_{\gamma(t)}}^2\right)+B_4+\left(1+E^\frac{1}{2}\right)^k\delta_2\\
\lesssim&\left(1+E^\frac{1}{2}\right)^{2k+1}\left(E^\frac{1}{2}+\sigma\right).
\end{aligned}\label{4.1}\end{equation}
\begin{equation}
B_7\leq C_5\left(1+E^\frac{1}{2}\right)\left(\|\partial_xh\|_{L^\infty_{\gamma(t)}}+\|\partial_x\theta\|_{L^\infty_{\gamma(t)}}\right)\leq C_5C_6\left(1+E^\frac{1}{2}\right)E^\frac{1}{2},
\label{4.2}\end{equation}
\begin{equation}B_9\lesssim\left(1+E^\frac{1}{2}\right)^{2k-3}.\label{4.3}\end{equation}
In one hand, if we let $\gamma(t)=\gamma(0)$ be invariant in time and define $$\tilde{E}(t):=\|h(t)\|_{H^{k}_{\gamma(0)}}^2+\mu_1\mu_2\|\theta(t)\|_{H^{k}_{\gamma(0)}}^2+\|\theta_1(t)\|_{H^{k_1}_{\gamma(0)}}^2,$$
then by (\ref{2.106})(\ref{3.36}), as long as (\ref{2.17})(\ref{2.65}) hold with $\gamma(t)=\gamma(0)$, there is
$$\begin{aligned}
\frac{d\tilde{E}}{dt}\lesssim&\left(B_6+B_9+\sigma^{-1}\right)\tilde{E}+\left(B_7-\frac{c_0}{C_4}\right)\left(\|\Lambda^\frac{1}{2}h\|_{H^{k}_{\gamma(t)}}+\|\Lambda^\frac{1}{2}\theta\|_{H^{k}_{\gamma(t)}}+\|\Lambda^\frac{1}{2}\theta_1\|_{H^{k_1}_{\gamma(t)}}\right),
\end{aligned}$$
where $B_6$, $B_7$, $B_9$ satisfy the bounds (\ref{4.1}-\ref{4.3}) with $E$ replaced by $\tilde{E}$. Moreover, if we assume in addition that
\begin{equation}
B_7-\frac{c_0}{C_4}\leq0,
\label{4.4}\end{equation}
then it follows that
\begin{equation}
\frac{d\tilde{E}}{dt}\lesssim\left(\mathcal{C}(\tilde{E})+\sigma^{-1}\right)\tilde{E},
\label{4.5}\end{equation} 
where $\mathcal{C}(E)$ is a smooth function of $E$ with the bound $\mathcal{C}(E)\lesssim\left(1+E^\frac{1}{2}\right)^{2k+2}$. On the other hand, since $f=h+\mu_1\theta$, $g=h-\mu_2\theta$, there exists a constant $C_7$ such that $$\|\partial_xf\|_{L^\infty_{\gamma(t)}}+\|\partial_xg\|_{L^\infty_{\gamma(t)}}+\sigma^{-1}\|\theta\|_{L^\infty_{\gamma(t)}}+\|\partial_x^2f\|_{H^1_{\gamma(t)}}+\|\partial_x^2g\|_{H^1_{\gamma(t)}}+\sigma^{-1}\|\partial_x\theta\|_{L^2_{\gamma(t)}}\leq C_7E^\frac{1}{2}.
$$
In view of (\ref{1.16})(\ref{1.17}), by choosing $\epsilon_0$, $\epsilon_1$ small enough, we can guarantee that 
\begin{equation}
E(t)^\frac{1}{2}<c_1:=\min\left\{1,\frac{\delta_1}{C_7},\frac{w_0}{C_7},\frac{c_0}{2C_4C_5C_6}\right\}
\label{4.6}\end{equation}
holds for $t=0$, and thus (\ref{2.17})(\ref{2.65})(\ref{4.4}) hold at $t=0$. Then using the a priori estimate (\ref{4.5}), it can be shown in a standard way of iteration and compactness argument that there exists a unique solution $(h,\theta)\in H^{k}_{\gamma(0)}\times H^k_{\gamma(0)}$ on $[0,T_0]$ with $T_0\gtrsim\sigma$. In fact, if we consider the maximal interval $[0,T_0]$ on which (\ref{4.6}) holds, then for $t\in[0, T_0]$, we have $\mathcal{C}(E)\lesssim1$. Subsequently, for certain constant $C^\prime$
$$\tilde{E}(t)\leq\exp\left(C^\prime\int_0^t(1+\sigma^{-1})ds\right)E(0)\leq\exp(C^\prime(1+\sigma^{-1})t)E(0).$$
Choose $T_{\sigma,0}=\frac{\sigma}{3C^\prime}\log\left(\frac{c_1^2}{E(0)}\right)$, then $\tilde{E}(t)^\frac{1}{2}<c_1$ for all $t\in[0,T_{\sigma,0}]$, and thus $T_0\geq T_{\sigma,0}$. Such control can be used to construct the solution $(h,\theta)\in L^\infty\left([0,T_{\sigma,0}];H^k_{\gamma(0)}\right)\times L^\infty\left([0,T_{\sigma,0}];H^k_{\gamma(0)}\right)$.\\
\indent Next, observing that $(h,\theta)\in L^\infty\left([0,T_{\sigma,0}];H^k_{\gamma(t)}\right)\times L^\infty\left([0,T_{\sigma,0}];H^k_{\gamma(t)}\right)$ and
$E(t)\leq\tilde{E}(t)<c_1$ for any decreasing positive function $\gamma(t)$, we are able to define the function $\gamma(t)$ by solving (\ref{1.20}):
$$-\frac{d\log(\cosh(2\gamma(t))}{dt}=-2\gamma^\prime(t)\tanh(2\gamma(t))=C_2\left(\|\partial_xh\|_{L^\infty_{\gamma(t)}}+\|\partial_x\theta\|_{L^\infty_{\gamma(t)}}\right).$$
with $\gamma(0)=\gamma_0$. Here $C_2:=2C_3C_5$ so that $$C_3B_7\leq C_2\left(\|\partial_xh\|_{L^\infty_{\gamma(t)}}+\|\partial_x\theta\|_{L^\infty_{\gamma(t)}}\right)=-2\gamma^\prime(t)\tanh(2\gamma(t)).$$
Since $\|\partial_xh\|_{L^\infty_{\gamma(t)}}+\|\partial_x\theta\|_{L^\infty_{\gamma(t)}}\leq C_6E(t)^\frac{1}{2}\leq C_6\tilde{E}^\frac{1}{2}(t)<C_6$ on $[0,T_{\sigma,0}]$, such $\gamma$ can be defined at least on $[0,T_{\sigma,0}]\cap[0, (C_2C_6)^{-1}\log(\cosh(2\gamma_0))].$\\
\indent Now, we impose the bootstrap assumptions. Suppose that the solution and $\gamma(t)$ exist on $[0,\tilde{T}]$ with the following assumptions:
\begin{itemize}
\item [(\romannumeral1)] The function $\gamma(t)$ satisfies (\ref{1.20}):
$$2\gamma^\prime(t)\tanh(2\gamma(t))=-C_2\left(\|\partial_x h(t)\|_{L^\infty_{\gamma(t)}}+\|\partial_x\theta(t)\|_{L^\infty_{\gamma(t)}}\right).$$
\item [(\romannumeral2)] $E(t)$ satisfies the bound (\ref{4.6}).
\end{itemize}
By (\ref{2.103})(\ref{3.35}) and (\ref{1.20}), we have
\begin{equation}\begin{aligned}
\frac{d}{dt}\left(\|h\|_{H^k_{\gamma(t)}}^2+\mu_1\mu_2\|\theta\|^2_{H^k_{\gamma(t)}}\right)\lesssim\left(B_6+B_7\right)\left(\|h\|_{H^k_{\gamma(t)}}^2+\mu_1\mu_2\|\theta\|^2_{H^k_{\gamma(t)}}\right)
\end{aligned}\label{4.7}\end{equation}
\begin{equation}\begin{aligned}
\frac{d}{dt}\|\theta_1\|_{H^{k-3}_{\gamma(t)}}^2\lesssim\left(B_6+B_7+B_9\right)E(t).
\end{aligned}\label{4.8}\end{equation}
In view of (\ref{4.1}-\ref{4.3}) and assumption (\ref{4.6}), we have $B_6+B_7+B_9\lesssim1$ on $[0,\tilde{T}]$. Therefore, \begin{equation}\|h(t)\|_{H^k_{\gamma(t)}}^2+\mu_1\mu_2\|\theta(t)\|_{H^k_{\gamma(t)}}^2\leq e^{C_0t}\epsilon_0^2,\quad\|\theta_1(t)\|_{H^{k-3}_{\gamma(t)}}^2\leq E(t)\leq e^{C_1t} \left(\epsilon_0^2+\epsilon_1^2\right)
\label{4.9}\end{equation}
for $t\in[0,\tilde{T}]$ and large enough $C_0$, $C_1$. Now we let $T_1$ be such that  $e^{C_0T_1}\epsilon_0^2+e^{C_1T_1} \left(\epsilon_0^2+\epsilon_1^2\right)=\frac{c_1^2}{2}.$
If we already have $\tilde{T}\geq T_1$, then the proof is complete. Hence we assume without loss of generality that $\tilde{T}<T_1$. In this case, we have $E(t)\leq\frac{c_1^2}{2}$ for all $t\in[0,\tilde{T}]$. Now we let $\tilde{E}(t):=\|h(t)\|_{H^{k}_{\gamma(\tilde{T})}}^2+\mu_1\mu_2\|\theta(t)\|_{H^k_{\gamma(\tilde{T})}}^2+\|\theta_1(t)\|_{H^{k_1}_{\gamma(\tilde{T})}}^2$and invoke (\ref{2.106})(\ref{3.36}) and thus (\ref{4.5}). By repeating the argument we used to construct the solution on $[0,T_{\sigma,0}]$, we can extend the solution to $[0,\tilde{T}+T_{\sigma}]$ with $T_{\sigma}=\frac{\sigma}{3C^\prime}\log\left(\frac{c_1^2}{E(\tilde{T})}\right)\geq\frac{\sigma\log2}{3C^\prime}$. Moreover, it follows that for $t\in[\tilde{T},\tilde{T}+T_{\sigma}]$
$$\tilde{E}(t)\leq\exp\left(C^\prime\int_{\tilde{T}}^{\tilde{T}+T_\sigma}(1+\sigma^{-1})ds\right)E(\tilde{T})<c_1.$$
To extend $\gamma$, we integrate (\ref{1.20}) to get 
$$\begin{aligned}
\log\left(\cosh(2\gamma(\tilde{T}))\right)-\log\left(\cosh(2\gamma_0)\right)=&-\int_0^{\tilde{T}}C_2\left(\|\partial_xh\|_{L^\infty_{\gamma(t)}}+\|\partial_x\theta\|_{L^\infty_{\gamma(t)}}\right)dt\\
\geq&-C_2C_6\int_0^{\tilde{T}}
E^\frac{1}{2}dt\\
\geq&-C_2C_6c_1\tilde{T}.
\end{aligned}$$
Let $T_2>0$ be such that $C_2C_6c_1T_2=\frac{1}{2}\log\left(\cosh(2\gamma_0)\right)$. Without loss of generality, we assume $\tilde{T}<T_2$, and subsequently $\log\left(\cosh(2\gamma(\tilde{T}))\right)>\frac{1}{2}\log\left(\cosh(2\gamma_0)\right)$. As before, noting that
$\|\partial_xh\|_{L^\infty_{\gamma(t)}}+\|\partial_x\theta\|_{L^\infty_{\gamma(t)}}\leq C_6E(t)^\frac{1}{2}\leq C_6\tilde{E}^\frac{1}{2}(t)<C_6$ in $[\tilde{T},\tilde{T}+T_\sigma]$, we can then extend $\gamma(t)$ to $[\tilde{T},\tilde{T}+T_\sigma]\cap[\tilde{T},\tilde{T}+(2C_2C_6)^{-1}\log(\cosh(2\gamma_0))]$ by solving (\ref{1.20}) in this interval. So far, we have completed the extension of the bootstrap assumptions (\romannumeral1) and (\romannumeral2) to $[0,\tilde{T}+T_\sigma]\cap[0,\tilde{T}+(2C_2C_6)^{-1}\log(\cosh(2\gamma_0))].$ Therefore, we conclude that (\romannumeral1) and (\romannumeral2) hold until $\tilde{T}$ reaches $T:=\min\left\{T_1,T_2\right\}$. Moreover, (\ref{4.9}) holds for all $t\leq T$, which gives (\ref{1.18})(\ref{1.19}).\qed

\subsection*{Acknowledgments}
\'Angel Castro acknowledges financial support from  2023 Leonardo Grant for Researchers and Cultural Creators, BBVA Foundation. The BBVA Foundation
accepts no responsibility for the opinions, statements, and contents included in the project and/or the results thereof, which are entirely the responsibility of the authors. AC was partially supported by the Severo Ochoa Programme for Centers of Excellence Grant CEX2019-000904-S and CEX-2023-001347-S funded by MCIN/AEI/10.13039/501100011033 and by the MICINN through the grant PID2020-114703GB-100. 	

This work was realized during the stay that Liangchen Zou did in the ICMAT at Madrid. Liangchen thanks the institute for its hospitality. This stay was supported by the China Scholarship Council Program (Project ID 202306340142).

\bibliographystyle{amsplain}		
\bibliography{Muskat}

\end{document}